\documentclass[11pt]{article}
\usepackage{amssymb}
\usepackage{graphicx}
\usepackage{xcolor} 
\usepackage{tensor}
\usepackage{fullpage} 
\usepackage{amsmath}
\usepackage{amsthm}
\usepackage{verbatim}
\usepackage{hyperref}
\AtBeginDocument{
	\label{CorrectFirstPageLabel}
	
}

\DeclareFontFamily{U}{mathx}{\hyphenchar\font45}
\DeclareFontShape{U}{mathx}{m}{n}{
	<5> <6> <7> <8> <9> <10>
	<10.95> <12> <14.4> <17.28> <20.74> <24.88>
	mathx10
}{}
\DeclareSymbolFont{mathx}{U}{mathx}{m}{n}
\DeclareFontSubstitution{U}{mathx}{m}{n}
\DeclareMathAccent{\widecheck}{0}{mathx}{"71}
\DeclareMathAccent{\wideparen}{0}{mathx}{"75}

\usepackage{enumitem}
\setlist[enumerate]{leftmargin=1.5em}
\setlist[itemize]{leftmargin=1.5em}

\definecolor{green}{rgb}{0,0.8,0} 

\newtheorem{maintheorem}{Theorem}

\newtheorem{theorem}{Theorem}[section]
\newtheorem{corollary}[theorem]{Corollary}
\newtheorem{lemma}[theorem]{Lemma}
\newtheorem{proposition}[theorem]{Proposition}

\theoremstyle{definition}
\newtheorem{definition}[theorem]{Definition}

\theoremstyle{remark}
\newtheorem{remark}[theorem]{Remark}

\numberwithin{equation}{section}
\newcommand{\relphantom}[1]{\mathrel{\phantom{#1}}}

\makeatletter
\newcommand{\nrm}{\@ifstar{\nrmb}{\nrmi}}
\newcommand{\nrmi}[1]{\Vert{#1}\Vert}
\newcommand{\nrmb}[1]{\left\Vert{#1}\right\Vert}
\newcommand{\abs}{\@ifstar{\absb}{\absi}}
\newcommand{\absi}[1]{\vert{#1}\vert}
\newcommand{\absb}[1]{\left\vert{#1}\right\vert}
\newcommand{\brk}{\@ifstar{\brkb}{\brki}}
\newcommand{\brki}[1]{\langle{#1}\rangle}
\newcommand{\brkb}[1]{\left\langle{#1}\right\rangle}
\newcommand{\set}{\@ifstar{\setb}{\seti}}
\newcommand{\seti}[1]{\{#1\}}
\newcommand{\setb}[1]{\left\{ #1\right\}}
\makeatother

\newcommand{\tld}[1]{\widetilde{#1}}
\newcommand{\br}[1]{\overline{#1}}
\newcommand{\ubr}[1]{\underline{#1}}

\newcommand{\nnrm}[1]{{\vert\kern-0.25ex\vert\kern-0.25ex\vert #1 
    \vert\kern-0.25ex\vert\kern-0.25ex\vert}}

\newcommand{\VERT}[1]{{\left\vert\kern-0.25ex\left\vert\kern-0.25ex\left\vert #1 
    \right\vert\kern-0.25ex\right\vert\kern-0.25ex\right\vert}}

\DeclareMathOperator{\supp}{supp}

\let\Re\relax
\DeclareMathOperator{\Re}{Re}

\newcommand{\aeq}{\simeq}
\newcommand{\aleq}{\lesssim}
\newcommand{\ageq}{\gtrsim}

\newcommand{\lap}{\Delta}

\newcommand{\ud}{\mathrm{d}}
\newcommand{\rd}{\partial}
\newcommand{\nb}{\nabla}

\newcommand{\peq}{\relphantom{=}}			
\newcommand{\pleq}{\relphantom{\leq}}			
\newcommand{\paleq}{\relphantom{\aleq}}			

\newcommand{\alp}{\alpha}
\newcommand{\bt}{\beta}
\newcommand{\gmm}{\gamma}
\newcommand{\Gmm}{\Gamma}
\newcommand{\dlt}{\delta}
\newcommand{\Dlt}{\Delta}
\newcommand{\eps}{\epsilon}

\newcommand{\kpp}{\kappa}
\newcommand{\lmb}{\lambda}
\newcommand{\Lmb}{\Lambda}
\newcommand{\sgm}{\sigma}

\newcommand{\tht}{\theta}
\newcommand{\Tht}{\Theta}
\newcommand{\vtht}{\vartheta}

\newcommand{\Omg}{\Omega}
\newcommand{\ups}{\upsilon}
\newcommand{\Ups}{\Upsilon}
\newcommand{\zt}{\zeta}

\newcommand{\bbC}{\mathbb C}

\newcommand{\bbN}{\mathbb N}

\newcommand{\bbR}{\mathbb R}

\newcommand{\bbT}{\mathbb T}

\newcommand{\bbZ}{\mathbb Z}

\newcommand{\calL}{\mathcal L}
\newcommand{\calM}{\mathcal M}
\newcommand{\calN}{\mathcal N}

\newcommand{\err}{\boldsymbol{\epsilon}}		

\newcommand{\rst}[1]{\left. #1 \right\vert}

\newcommand{\ollmb}{\overline{\lmb}}

\newcommand{\bgtht}{\mathring{\tht}}
\newcommand{\dgx}{\mathring{x}}

\newcommand{\bfPhi}{\mathbf{\Phi}}				
\newcommand{\wpT}{T}						

\begin{document}

\title{ {Illposedness via degenerate dispersion} for generalized surface quasi-geostrophic equations with singular velocities}
\author{Dongho Chae\thanks{Department of Mathematics, Chung-ang University. E-mail: dchae@cau.ac.kr} \and In-Jee Jeong\thanks{Department of Mathematical Sciences and RIM, Seoul National University. E-mail: injee\_j@snu.ac.kr}\and Sung-Jin Oh\thanks{Department of Mathematics, UC Berkeley and School of Mathematics, Korea Institute for Advanced Study. E-mail: sjoh@math.berkeley.edu}}

\date{\today}



\maketitle


\begin{abstract}
We prove strong nonlinear illposedness results for the generalized SQG equation \begin{equation*}
	\begin{split}
	\rd_t \tht + \nb^\perp \Gmm[\tht] \cdot \nb \theta = 0 
	\end{split}
	\end{equation*} in any sufficiently regular Sobolev spaces, when $\Gmm$ is a singular multiplier in the sense that its symbol satisfies  $|\Gmm(\xi)|\rightarrow\infty$ as $|\xi|\rightarrow\infty$ together some mild regularity assumptions. The key mechanism is degenerate dispersion, i.e., the rapid growth of frequencies of solutions around certain shear states, as in the second and third author's earlier work on Hall-magnetohydrodynamics \cite{JO1}. The robustness of our method allows one to extend linear and nonlinear illposedness to fractionally dissipative systems, as long as the order of dissipation is lower than that of $\Gmm$. Our illposedness results are completely sharp in view of various existing wellposedness statements as well as those from our companion paper \cite{CJNO}. 
	
	Key to our proofs is a novel construction of degenerating wave packets for the class of linear equations \begin{equation*}
	\begin{split}
	\rd_t \phi + ip(t,X,D)\phi = 0 
	\end{split}
	\end{equation*} possibly with lower order terms, where $p(t,X,D)$ is a possibly time dependent pseudo-differential operator which is formally self-adjoint in $L^2$, degenerate, and dispersive. Degenerating wave packets are approximate solutions to the above linear equation with spatial and frequency support localized at $(X(t),\Xi(t))$, which are solutions to the bicharacteristic ODE system associated with $p(t,x,\xi)$. These wave packets explicitly show degeneration as $X(t)$ approaches a point where $p$ vanishes, which in particular allows us to prove illposedness in topologies finer than $L^2$. While the equation for the wave packet can be formally obtained from a Taylor expansion of the symbol near $\xi=\Xi(t)$, the difficult part is to rigorously control the error in sufficiently long timescales within which significant degeneration occurs. To achieve this task, we develop a systematic way to obtain sharp estimates for not only degenerating wave packets but also for oscillatory integrals which naturally appear in the error estimate.
 \end{abstract}

\tableofcontents

\section{Introduction}

\subsection{Generalized SQG equations}

In two spatial dimensions, the \textit{generalized surface quasi-geostrophic} (gSQG) equations are given by 
\begin{equation} \label{eq:ssqg}
	\left\{
	\begin{aligned}
		&\rd_{t} \tht + u \cdot \nb \tht = 0, \\
		&u = \nb^{\perp} \Gmm \tht,
	\end{aligned}
	\right.
\end{equation}
where $\nb^{\perp} = (-\rd_{x_2}, \rd_{x_1})^{\top}$ and $\Gmm$ is a Fourier multiplier with a real-valued symbol $\gmm$. Here, $\tht(t,\cdot): \Omg\rightarrow \bbR$ and $u(t,\cdot):\Omg\rightarrow \bbR^2$ with $\Omg$ a two-dimensional domain without boundaries (e.g. $\Omg = \bbT^2$, $\bbR^2$, or $\bbT\times\bbR$). The system \eqref{eq:ssqg} says that the scalar $\theta$ is being advected by the flow of $u$, which is determined from $\theta$ at each moment of time by the ``Biot--Savart'' law $u=\nb^\perp\Gmm\tht$. For this reason, \eqref{eq:ssqg} is sometimes referred to as an \textit{active scalar} system. 

The gSQG system generalizes several important partial differential equations (PDEs) arising in hydrodynamics and magnetohydrodynamics (MHD) and has been intensively studied over the past few decades. A well-known special case occurs when $\Gamma = (-\Delta)^{-1}$, in which case \eqref{eq:ssqg} reduces to the vorticity equation for the two-dimensional incompressible Euler equations, where $\theta$ corresponds to the vorticity of the fluid, which is $\nabla \times u$. Another fundamental case is $\Gamma = \Lambda^{-1}$, which yields the (standard) SQG equations, describing the evolution of atmospheric fronts. Here, $\Lambda = (-\Delta)^{1/2}$ denotes the Zygmund operator. The SQG model was introduced by Constantin--Majda--Tabak \cite{CMT1, CMT2} to describe Boussinesq dynamics at the boundary of the upper half-space (see \cite{Ped} for further details). {At the opposite extreme, the cases $\Gamma = \Lambda$ and $\Gamma = \Lambda^2$ arise naturally in MHD and large-scale atmospheric dynamics, respectively, and we will return to these examples below in Section~\ref{subsec:singular}.} 

{The pioneering works \cite{CMT1,CMT2} suggested the possibility of rapid small scale creation for the SQG equation, and there have been several works demonstrating large gradient growth for smooth solutions (\cite{DEJ,HeKi,Kis,Kis2,KiNa,DE}). 
However, it is still unclear whether smooth solutions of the SQG equation could blow up in finite time, in stark contrast with the Euler case $\Gmm= \Lmb^{-2}$ for which global regularity is a classical result. This global regularity is based upon the fact that the conservation of $\nrm{\tht(t,\cdot)}_{L^{\infty}}$ ``almost'' controls the Lipschitz norm of the velocity. Therefore, as the multiplier $\Gmm$ becomes more singular, the control over the velocity one obtains from $\nrm{\tht(t,\cdot)}_{L^{\infty}}$ becomes weaker, and one may expect that smooth solutions are more likely to blow up in finite time. Towards this goal, the global regularity question has been extensively studied for the  {generalized} models where $\Gmm$ is taken to be $\Lmb^{\bt}$ for {$- 1 \geq \bt > -2$ (i.e., between the SQG and 2D Euler equations)}.  When the equation is posed in the half-plane, some finite time singularity formation results were obtained in the range $-1.5>\bt>-2$ (\cite{KRYZ,KYZ,GaPa,MTXX,Z2,JKY}) {using in a crucial way the presence of the boundary (which, roughly speaking, stabilizes the singularity formation mechanism)}. The numerical simulations in \cite{ScDr} report self-similar type singularity formation for SQG patches. On the other hand, some global solutions were constructed in the works \cite{CCG-annpde,CCG-duke,CCW3,CGI,CQZZ,CQZZ-gSQG,GIP,HH-sqg,HMsqg,HXX,HSZ,HSZ2}. 
}

As $\beta$ increases, however, even the question of \textit{local} regularity for smooth solutions becomes nontrivial. This difficulty arises as soon as the multiplier becomes more singular than in the SQG case (i.e., for $\beta > -1$). In this regime, the velocity field $u$ is more singular than $\theta$, which creates serious obstacles in closing energy estimates in Sobolev spaces. Specifically, one must carefully analyze terms where all derivatives in the $H^m$ estimate fall on $u$:
\begin{equation*}
	\begin{split}
		\langle \partial^m( \nabla^\perp \Lambda^\beta \theta ) \cdot \nabla\theta, \partial^m\theta \rangle = \langle \nabla\theta \cdot \nabla^\perp\Lambda^\beta g, g \rangle,
	\end{split}
\end{equation*}
where $g := \partial^m\theta$, and $\langle \cdot, \cdot \rangle$ denotes the standard $L^2$ inner product.

A crucial observation here is that the principal part of the operator $\nabla\theta \cdot \nabla^\perp \Lambda^\beta$ is antisymmetric, allowing one to gain a derivative. This idea was exploited by Chae--Constantin--C\'{o}rdoba--Gancedo--Wu in \cite{CCCGW} to establish local well-posedness for $\beta < 0$, and this observation can be generalized to obtain local wellposedness when $|\gamma(\xi)| \lesssim 1$ as $|\xi| \to 0$, assuming some natural regularity conditions on the derivatives of $\gamma$. Since then, the behavior of solutions in the range $\beta \in (-1,0)$ was considered by many authors  (\cite{CQZZ-gSQG,LiSi,MR4302173,CGI,KR,CMZ,GIP}).

\subsection{The case of singular multipliers}\label{subsec:singular}

Turning to the case of \textit{singular multipliers}, by which we simply mean that $\gmm(\xi)\rightarrow\infty$ as $|\xi|\rightarrow\infty$, one first sees that in the ``borderline'' case when $\beta = 0$ ($\Gmm = \Lmb^\bt$), the nonlinearity vanishes completely: $u\cdot\nb\tht = \nb^\perp\tht\cdot\nb\tht\equiv 0$. This could make one speculate that there might be some additional cancellation which gives local regularity even when $\beta>0$. Furthermore, the (formally) conserved quantity $\nrm{\Gmm^{\frac12}\tht(t,\cdot)}_{L^{2}}$ for \eqref{eq:ssqg} becomes stronger than the other $L^2$-based conservation law $\nrm{\tht(t,\cdot)}_{L^{2}}$ as soon as $\bt>0$. 

Despite these facts, our main result shows \emph{strong illposedness} in Sobolev spaces (with arbitrarily high regularity) for singular $\Gmm$ satisfying a few reasonable assumptions. Interestingly, the generalized SQG equations with singular multipliers naturally appear in a variety of situations, as we shall now explain. 
\begin{itemize}
	\item \textbf{Ohkitani model}. In the papers \cite{Oh1,Oh2}, Ohkitani considered the collective behavior of solutions to \eqref{eq:ssqg} which are obtained by varying $\bt<0$ with the same initial data, towards the goal of settling the question of global regularity versus finite time singularity formation. Numerical simulations in \cite{Oh1,Oh2} did not show any singular behavior of solutions in the limit $\bt\to 0^-$, and based on these, Ohkitani conjectured global regularity of the limiting model \begin{equation}\label{eq:Oh}
		\begin{split}
			\rd_t\tht - \nb^\perp \ln(\Lmb) \tht \cdot \nb \tht = 0
		\end{split}
	\end{equation} which is \eqref{eq:ssqg} with $\Gmm = -\ln(\Lmb)$. This has been referred to as \textit{Ohkitani model} in \cite{CCCGW}. To see how \eqref{eq:Oh} arises, one can simply rewrite \eqref{eq:ssqg} with $\Gmm=\Lmb^\bt$ with $\bt<0$ as \begin{equation}\label{eq:bt-sqg-rewrite}
	\begin{split}
		\frac{1}{\bt} \rd_t \tht + \nb^\perp \left( \frac{\Lmb^\bt -1}{\bt} \right) \tht \cdot\nb\tht = 0, 
	\end{split}
\end{equation} and formally we have that as $\bt\to0^-$, \eqref{eq:bt-sqg-rewrite} converges to \eqref{eq:Oh} in the rescaled timescale $\bt t$. This limit was made rigorous in our companion paper \cite{CJNO}, {and in particular, gave a long-time existence result for solutions in the limit $\bt\to0^-$.}
	\item \textbf{E-MHD system}. The electron magnetohydrodynamics (E-MHD) system takes the form \begin{equation}  \label{eq:e-MHD}
		\left\{
		\begin{aligned} 
			&\rd_t B + \nb \times( (\nb\times B) \times B) = 0, \\
			&\nb\cdot B = 0, 
		\end{aligned}
		\right.
	\end{equation} where $B(t,x) :\bbR\times \bbR^3\rightarrow\bbR^3$. This is a leading order model for the Hall--MHD system (\cite{Light,Pecseli}). Under the so-called $2+\frac{1}{2}$-dimensional assumption, $B$ can be written as \begin{equation*}
		\begin{split}
			B = \nb\psi \times e_z + \Lmb\phi \, e_z
		\end{split}
	\end{equation*} for some \textit{scalar} functions $\psi, \phi$ independent of the third coordinate $z$, and \eqref{eq:e-MHD} reduces to the following system in two dimensions (\cite{LM}): \begin{equation}  \label{eq:e-MHD-21/2}
		\left\{
		\begin{aligned}
			&\rd_t\psi + \nb^\perp\Lmb\phi \cdot\nb\psi = 0 ,\\
			&\rd_t\phi + \Lmb^{-1}( \nb^\perp\Lmb (\Lmb\psi)\cdot\nb\psi )= 0.
		\end{aligned}
		\right.
	\end{equation} Up to leading order, the ansatz $\psi\simeq \phi$ propagates in time, which simply corresponds to \eqref{eq:ssqg} with $\Gmm=\Lmb$. 

	\item \textbf{Asymptotic model for the large-scale quasi-geostrophic (AM-LQG) equation}. The following equation is referred to as the asymptotic model (AM) for the large-scale quasi-geostrophic equation (see \cite{LaMc,Burg,TrDr,BurgDrits,SHF1,SHF2}): \begin{equation}\label{eq:LQG}
		\begin{split}
			\rd_t \psi + \nb^\perp\lap\psi\cdot\nb\psi = 0. 
		\end{split}
	\end{equation} Notice that this is nothing but \eqref{eq:ssqg} with the choice $\Gmm = \lap$. The AM equation has received quite a bit of attention from physicists as the solutions exhibit very different features from the usual 2D turbulence \cite{BurgDrits}. 
	One can arrive at \eqref{eq:LQG} by starting from the so-called Charney--Hasegawa--Mima (CHM) equation \begin{equation}\label{eq:CHM}
		\begin{split}
			\rd_t q + \nb^\perp \psi \cdot \nb q = 0,
		\end{split}
	\end{equation} which is relevant for shallow water quasi-geostrophic dynamics. It is argued that this equation governs ocean front dynamics and planetary atmospheric pattern including Great Red Spot (\cite{PrattStern,Ped,HaseMima}). Here, $q$ denotes the \textit{potential vorticity} and is related by the stream function $\psi$ by $q=(\lap - L_D^{-2})\psi$ (so that \eqref{eq:CHM} is nothing but \eqref{eq:ssqg} with $\Gmm=(\lap-L_D^{-2})^{-1}$). Here, $L_D$ is the so-called  \textit{Rossby deformation length} and it is argued in Burgess and Dritschel \cite{BurgDrits} that when the characteristic length-scale of the flow $L$ satisfies $L \gg L_D$, \eqref{eq:LQG} can be obtained from \eqref{eq:CHM} in the rescaled timescale $t L_D^2$. {Very recently, Svirsky--Herbert--Frishman \cite{SHF1,SHF2} reported detailed statistical properties of the condensate for \eqref{eq:LQG}.}
\end{itemize}

{Let us discuss the physical relevance of illposedness for E-MHD and AM-LQG in Section~\ref{subsec:interpret} below, after precisely stating our main results.}

\subsection{Rough statement of the results}

Our main result, which is stated roughly for now, gives strong illposedness for a large class of singular symbols, including all of the above three examples. 
\begin{theorem} \label{thm:main-simple}
	Consider the following symbols $\gmm$ and pairs of exponents $s, s'$:
	\begin{center}
		\begin{tabular}{| l | l | l |} 
			\hline
			Multiplier
			& Sobolev regularity exponents  \\
			\hline
			$\gmm = \brk{\xi}^{\bt}$, $\bt > 1$
			&  {$s' = s > 3 + \frac{3}{2}\bt$} \\
			$\gmm = \brk{\xi}$
			&  {$s' > \frac{9}{2}$} \\
			$\gmm = \brk{\xi}^{\bt}$, $\bt < 1$
			&  {$\frac{1}{1-\bt} s' > \max\set{s + \frac{\bt^2}{2(1-\bt)}, \frac{3}{2} \frac{2+\beta}{1-\bt}}$} \\
			$\gmm = \log^{\bt} (10+\abs{\xi})$, $\bt > 0$
			&  {$s' = s > 3$} \\
			$\gmm = \log^{\bt} (10+\log (10+\abs{\xi}))$, $\bt > 0$
			&  {$s' = s > 3$} \\
			\hline
		\end{tabular}
	\end{center}
	In each of the above cases, the Cauchy problem for \eqref{eq:ssqg} on the domain {$\Omg = \bbT^{2}$ or $\bbT \times \bbR$} is \emph{$H^{s}$-$H^{s'}$ ill-posed} in the following sense:
For any $\eps, \dlt, A>0$, there exists initial data $\tht_0 \in C^\infty_c(\Omg)$ with $\nrm{\tht_0}_{H^s}<\eps $ such that either \begin{itemize}
		\item there exists no solution $\tht \in L^\infty([0,\dlt]; H^{s'})$ to \eqref{eq:ssqg} with $\rst{\tht}_{t=0} = \tht_0$, or 
		\item any solution $\tht$ belonging to $L^\infty([0,\dlt]; H^{s'})$ satisfy the growth \begin{equation*}
		\begin{split}
		\sup_{t \in [0,\dlt]} \nrm{\tht(t, \cdot)}_{H^{s'}} > A. 
		\end{split}
		\end{equation*}
	\end{itemize}
\end{theorem}
Theorem~\ref{thm:main-simple} is a \emph{norm inflation} result.  {When $\Omg = \bbT \times \bbR$ and $s = s'$, we may furthermore establish a \emph{non-existence} result.}

\begin{theorem} \label{thm:main-nonexist-simple}
Consider a symbol $\gmm$ and a real number $s$ such that $\gmm$ and $s = s'$ satisfy the hypothesis of Theorem~\ref{thm:main-simple}. Then the Cauchy problem for \eqref{eq:ssqg} on the domain $\Omg = \bbT \times \bbR$ is \emph{ill-posed} in the following sense: There exists an initial data set $\tht_{0} \in H^{s}$ with arbitrarily small $H^{s}$ norm, for which there does \emph{not} exist a solution to the Cauchy problem for \eqref{eq:ssqg} in $L^{\infty}_{t}([0, \dlt], H^{s'})$.
\end{theorem} 

The above illposedness results suggest that one needs to be careful when working with models with singular multipliers: either an appropriate dissipative term must be supplied,\footnote{We note that while these models are often written without any dissipative terms in many physics texts, numerical simulations are always performed by adding very strong dissipative terms.} or one should restrict to an appropriate class of functions (see Remark \ref{rem:wellposed} below for more discussion). Indeed, local wellposedness of all the above singular examples, namely \eqref{eq:Oh}, \eqref{eq:e-MHD}, \eqref{eq:e-MHD-21/2}, and \eqref{eq:LQG}, has been obtained with appropriate dissipation terms: consider now
 \begin{equation}\label{eq:ssqg-diss}
	\left\{
	\begin{split}
		&\rd_t \tht + u\cdot\nb \tht + \kappa \Upsilon(\tht) = 0, \\
		& u = \nb^\perp\Gmm[\tht],
	\end{split}
	\right.
\end{equation} where $\kpp>0$ and $\Upsilon$ is a multiplier with strictly positive symbol. The authors of \cite{CCCGW} have shown local regularity of the Ohkitani model with arbitrarily fractional dissipation (i.e. $\Upsilon=(-\lap)^\eps$ for any $\eps>0$), which have been improved to any super-logarithmic dissipation in \cite{CJNO}. In the case of the E-MHD (and Hall-MHD), the works \cite{CDL,CWW} obtained local regularity with magnetic dissipation stronger than $\Lmb$. A similar computation can be done for the AM, which then requires a dissipation term strictly stronger than $\Lmb^2$. 

While these wellposedness results are obtained by rather standard Sobolev and commutator estimates, our illposedness results for the dissipative systems \eqref{eq:ssqg-diss} show that these existing results are completely sharp:

\begin{theorem} \label{thm:main-diss-simple}
	Consider the following pairs of symbols $\gmm,\upsilon$ and exponents $s, s'$:
		\begin{center} 
			\begin{tabular}{ | l | l | l | l |} 
				\hline
				Multiplier $\Gmm$ & Dissipation $\Upsilon$ & Sobolev regularity exponents \\ 
				\hline
				$\gmm = \brk{\xi}^{\bt}$, $\bt > 1$ 
				& $\ups = \brk{\xi}^{\alp}$, $\alp < \bt$  
				& $\frac{s'}{1-(\bt-\alp)} > s + \frac{\bt(\bt-\alp)}{2(1-(\bt-\alp))} $ \\
				$\gmm = \brk{\xi}$
				& $\ups = \brk{\xi}^{\alp}$, $\alp < 1$ 
				& $\frac{s'}{\alp} > s + \frac{1-\alp}{2\alp} $ \\
				$\gmm = \brk{\xi}^{\bt}$, $\bt < 1$
				& $\ups = \brk{\xi}^{\alp}$, $\alp < \bt$ 
				& $\frac{s'}{1-(\bt-\alp)} > s + \frac{\bt(\bt-\alp)}{2(1-(\bt-\alp))} $ \\
				$\gmm = \log^{\bt} (10+\abs{\xi})$, $\bt > 0$
				& $\ups = \log^{\alp} (10+\abs{\xi})$, $\alp < \bt$ 
				& $s' = s$ \\
				\hline
			\end{tabular}
		\end{center} 
	The restrictions on $s$ and $s'$ in the table are \emph{in addition} to those from the nondissipative case. In each of the above cases, the Cauchy problem for \eqref{eq:ssqg-diss} on the domain $\Omg = \bbT^{2}$ or $\bbT \times \bbR$ is \emph{$H^{s}$-$H^{s'}$ ill-posed} in the same sense as in Theorem \ref{thm:main-simple}.
\end{theorem}

\begin{remark}\label{rem:wellposed}
Our results establish illposedness of \eqref{eq:ssqg} for a large class of singular symbols in standard function spaces near the trivial solution $\bgtht = 0$. Nevertheless, wellposedness in standard function spaces may still be possible around nontrivial background solutions, which are sometimes physically motivated. For instance, small data local wellposedness of E-MHD in weighted Sobolev space around $\bgtht = x_{2}$ -- corresponding to a uniform magnetic field -- may be established as an application of the techniques in \cite{MMT}.
\end{remark}

\subsection{Degenerate dispersive equations}\label{subsec:prev-work}

It turns out that the linearization around degenerate shear steady states for \eqref{eq:ssqg} shows \textit{degenerate dispersion}, which is the mechanism behind strong illposedness in the singular regime. Indeed, a large part of this work is devoted to the construction of \textit{degenerating wave packets} for the class of linear equations \begin{equation}\label{eq:dd}
	\begin{split}
		\rd_t\phi + i p(t,X,D)\phi = 0
	\end{split}
\end{equation} possibly with lower order terms, where $p$ is a time-dependent real pseudo-differential operator which is degenerate and dispersive. Once degenerating wave packets are constructed for all large frequencies, their time evolution essentially governs in which topologies the initial value problem for \eqref{eq:dd} could be well-posed. While our framework works for the general class of equations \eqref{eq:dd}, we have chosen to focus on its applications towards the family of singular generalized SQG equations in this work, as the class of linearized equations arising from this family by varying $\Gmm$ and the steady state profile forms a representative class of \eqref{eq:dd}. 

Degenerate dispersive equations appear in a variety of physical contexts, besides those related with the gSQG equations described in the above. Primary examples include shallow water wave (\cite{CaHo,BCS,GrNa}) and sedimentation models (\cite{BBKT,CaPa,Zum,Rub,RuKe}). Many of these models, most notably the Camassa--Holm and $abcd$-Boussinesq equations (introduced in \cite{CaHo} and \cite{BCS} respectively and extensively studied since) feature principal terms which involve non-local and non-linear dispersion. More comprehensive list of physical systems involving degenerate dispersion (as well as related mathematical progress) is given in \cite{ASWY,GHM,JO3}. We shall review a few recent developments on the Cauchy problem for these type of equations, which are most directly relevant for the current work. 

\medskip

\noindent \textbf{Well/Illposedness of $K(m,n)$ equations and their variants}. The family of  $K(m,n)$ equations introduced in \cite{RoHy,Ro} is given by \begin{equation}\label{eq:Kmn}
	\begin{split}
		\rd_t u + (u^n)_{xxx} + (u^m)_{x} = 0.
	\end{split}
\end{equation} For $n>1$, this model could be considered as the simplest equations featuring a quasilinear dispersive principal term (see \cite{Zum} where this type of term appears for a model of particle suspensions). Various numerical simulations for this equation hinted at illposedness in strong topologies (see \cite{ASWY,IsTa,dLS}). In the case of degenerate Airy equation $\rd_t u + 2u u_{xxx} = 0 $ (which is a further simplified model for $n=2$), \cite{ASWY} gave illposedness in $H^{2}$ of the initial value problem using explicit self-similar solution with scaling symmetries of the equation. We note that (uni-directional in time) illposedness for $\rd_t u \pm x u_{xxx} = 0$ was obtained earlier in \cite{CrGm}, based on the explicit solution formula. A general illposedness result, which works not only for \eqref{eq:Kmn} but also for many variants, was obtained in our recent work \cite{JO3}. Here, the illposedness is deduced from the construction of degenerating wave packets for the linearized equation around degenerate solutions. On the other hand, \cite{GHM} obtained a well-posedness result for certain variant of \eqref{eq:Kmn} in the case of ``subcritically'' degenerate (cf. \cite[Section 1.5]{GHM}) data. This well-posedness result is not contradictory to illposedness results from \cite{JO3}; the solutions of \cite{GHM} live in a suitably weighted space, which takes into account the rate of degeneracy of the solution.

\medskip

\noindent \textbf{Illposedness of the Hall- and electron-MHD systems}. Based on construction of degenerating wave packets, the system \eqref{eq:e-MHD} (as well as the Hall--MHD system) was shown to be strongly ill-posed near degenerate shear magnetic backgrounds, in  {the recent work \cite{JO1,JO4} of the second and third authors}. Indeed, the mechanism of illposedness for \eqref{eq:e-MHD} is the same with the current paper. Additional difficulties arising in this work is that the system is \textit{non-local} (opposed to \eqref{eq:e-MHD}) and the symbol of $\Gmm$ could be only slightly singular, where the issue of well/illposedness becomes very delicate (as demonstrated explicitly in Theorem \ref{thm:wp-log}).

\bigskip

\noindent The remainder of the introduction is organized as follows.
In {\bf Section~\ref{subsec:results}}, we give precise statements of the main results of this paper, of which Theorems~\ref{thm:main-simple}--\ref{thm:main-nonexist-simple} are special cases. 
In {\bf Section~\ref{subsec:toy}}, we present a toy model for the linearized equation which is almost explicitly solvable in the Fourier space yet contains the main features of the linearized dynamics. The model is obtained simply by dropping the sub-principal term and replacing the principal coefficient with a linear function. This solvable toy model demonstrates that the illposedness behavior is caused by degenerate dispersion, and gives the optimal growth rate of Sobolev norms that can be achieved. Furthermore, by comparing the toy model with actual linear equations, we explain the main difficulties in understanding the dynamics of linear equations. 
Then we end the introduction with an outline of the {\bf Organization of the Paper}.

\subsection{Main Results} \label{subsec:results}
We now give precise formulation of the main results of this paper.

\medskip 

\noindent \textbf{Assumptions on $\Gmm$}. 
%
In what follows, we assume that $\gmm$ is a \textbf{smooth even positive symbol} that satisfies the following properties for some $\Xi_{0} > 1$:
\begin{enumerate}
	\item {\bf $\gmm$ is elliptic and slowly varying:} $\abs{\rd_{\xi}^{I} \gmm(\xi)} \aleq_{I} \brk{\xi}^{-\abs{I}} \gmm(\xi)$ for any multi-index $I$ and $\xi \in \bbR^{2}$.
	\item {\bf $\gmm \nearrow \infty$ as $\abs{\xi} \nearrow \infty$:}  {$\gmm(\xi_{1}, \xi_{2}) \to \infty$ as $\abs{\xi} \to \infty$.}
	\item  {{\bf $\xi_{2} \rd_{\xi_{2}} \gmm$ is elliptic and slowly varying:} 
	$\abs{\rd_{\xi}^{I} (\xi_{2} \rd_{\xi_{2}} \gmm(\xi))} \aleq_{I} \brk{\xi}^{-\abs{I}} \xi_{2} \rd_{\xi_{2}} \gmm(\xi)$} for any multi-index $I$ and $\Xi_{0} \leq \abs{\xi_{1}} \leq \abs{\xi_{2}}$.
	\item  {{\bf $\xi_{2} \rd_{\xi_{2}} \gmm$ is almost comparable to $\gmm$:} 
	$\xi_{2} \rd_{\xi_{2}} \gmm(\xi) \ageq \frac{1}{(\log \abs{\xi_{2}})^{2}} \gmm(\xi)$
	for $\Xi_{0} \leq \abs{\xi_{1}} \leq \abs{\xi_{2}}$.}
\end{enumerate}

A brief explanation of each assumption is in order:
\begin{itemize}
\item Assumption~1 is a natural assumption that justifies, in particular, symbolic calculus (see Section~\ref{sec:psdo}). 
\item Assumption~2 is a basic requirement for an arbitrarily fast frequency growth.
\item Assumptions~3 and 4 arise naturally in the control of the focusing of nearby bicharacteristics, which are put together to construct a suitable approximate solution to the linearized equation, called \emph{degenerating wave packets} (see Sections~\ref{sec:psdo}--\ref{sec:wavepacket}). In particular, Assumption~4 allows us to quantify the scale of the degenerating wave packets (denoted by $\mu^{-1}$ later) in terms of its frequency (denoted by $\lmb$ later). The factor $(\log \abs{\xi_{2}})^{-2}$ is somewhat arbitrary but fixed for simplicity.
\end{itemize}

Many natural choices of $\gmm$ satisfy the above assumptions, including $\gmm(\xi) = \brk{\xi}^{\bt}$ for any $\bt > 0$, $\gmm(\xi) = \exp(\bt \log^{\alp} (10 + \abs{\xi}))$ for any $\bt > 0$ and $0 < \alp < 1$, $\gmm(\xi) = \log^{\bt}(10 + \abs{\xi})$ for any $\bt > 0$, $\gmm(\xi) = \log^{\bt}(10 + \log^{\alp} (10 + \abs{\xi}))$ for any $\alp, \bt > 0$ etc. 

By Assumption~1, there exists $\bt_{0} > 0$ such that 
\begin{equation} \label{eq:slow-var}
\sup_{\xi' : \abs{\xi} \leq \abs{\xi'} \leq 2 \abs{\xi}} \gmm(\xi') \leq 2^{\bt_{0}} \gmm(\xi) \quad \hbox{ for } \abs{\xi} \geq \Xi_{0}.
\end{equation}
For the remainder of this paper, {\bf we fix one such $\bt_{0} > 0$} and let other constants depend on it. Iterating this bound, it follows that
\begin{equation*}
\gmm(\xi) \aleq \abs{\xi}^{\bt_{0}} \abs{\Xi_{0}}^{-\bt_{0}} \gmm(\Xi_{0}) \quad \hbox{ for } \abs{\xi} \geq \Xi_{0}.
\end{equation*}

The infimum of possible $\bt_{0}$'s (among all real numbers) is called the \emph{order} of $\gmm$. The justification for this terminology comes from the property that if $\gmm = \abs{\xi}^{\bt_{0}}$, then $\bt_{0}$ is its order in the usual sense.

\medskip

\noindent \textbf{Assumptions on $\Ups$.}
For the dissipative operator $\Ups$, we simply assume that its symbol $\ups$ is a \textbf{smooth even positive symbol} that is  {elliptic and slowly varying, in the sense that}
\begin{equation} \label{eq:slow-var-ups}
\abs{\rd_{\xi}^{ {I}} \ups(\xi)} \aleq_{ {I}} \abs{\xi}^{-\abs{ {I}}} \ups(\xi).
\end{equation}
As in the case of $\gmm$, there exists $\alp_{0} > 0$ such that 
\begin{equation}\label{eq:alp}
\sup_{\xi' : \abs{\xi} \leq \abs{\xi'} \leq 2 \abs{\xi}} \ups(\xi') \leq 2^{\alp_{0}} \ups(\xi) \quad \hbox{ for } \abs{\xi} \geq \Xi_{0}.
\end{equation}
For the remainder of this paper, {\bf we fix one such $\alp_{0} > 0$} and let other constants depend on it.

\medskip

\noindent \textbf{Assumptions on $\bgtht$}. Next, we specify the class of shear states that will be proved to be unstable, in linear/nonlinear settings and with/without dissipation, in high-regularity Sobolev spaces. 

In the absence of dissipation, recall that any shear steady state $\tht = f(x_{2})$ (with reasonable regularity assumptions) solves \eqref{eq:ssqg}. We shall {\bf assume that $f$ is smooth, bounded and has a quadratic degeneracy} in the following sense:
\begin{definition} \label{def:degen-shear}
	We say that a shear steady state $\bgtht = f(x_{2})$ for \eqref{eq:ssqg} is \emph{quadratically degenerate} at $\dgx_{2} \in (\bbT, \bbR)$ if 
	\begin{equation*}
	f'(\dgx_{2}) = 0, \qquad f''(\dgx_{2}) \neq 0.
	\end{equation*}
\end{definition}

\begin{remark}
	Our method easily extends to the case when the order of vanishing of $f'(x_{2})$ at $x_{2} = \dgx_{2}$ is generalized to any positive real number. The general heuristic principle is that the slower the vanishing (i.e., the lower the order), the faster the frequency growth. The quadratically degenerate case considered in Definition~\ref{def:degen-shear} is distinguished by the fact that it is the generic order for a smooth $f'(x_{2})$.  {We also note that the boundedness assumption can be readily generalized to a polynomial growth condition at infinity.}
\end{remark}

In the presence of dissipation, equation \eqref{eq:ssqg-diss} for a shear state $\bgtht = f(t, x_{2})$ reduces to
\begin{equation} \label{eq:ssqg-diss-shear}
(\rd_{t} + \kpp \Ups)f(t, x_{2}) = 0,
\end{equation}
where $\Ups$ obeys the assumptions made above. By Fourier analysis, \eqref{eq:ssqg-diss-shear} is clearly well-posed forward in time in $H^{s}((\bbT, \bbR)_{x_{2}})$ for any $s \in \bbR$. Moreover, for any well-posed solution $f$, if $f(0, x_{2})$ is even then so is $f(t, x_{2})$ for each $t > 0$. In what follows, we shall take as our background shear state a {\bf smooth bounded solution $\bgtht = f(t, x_{2})$} such that $f_{0}(x_{2}) = f(0, x_{2})$ is \textbf{even} and $ {f_{0}''(0)} \neq 0$. 

\begin{remark}
	The evenness assumption brings a technical simplification in our argument, as the degenerate point $x_{2} = 0$ is then fixed in $t$. Our methods may be extended to the case when this assumption is removed, in which case the degenerate point may move in time, but at the price of additional technical constraints on the length of the time interval. 
\end{remark}

\medskip

\noindent \textbf{Linear results.} 
We begin by stating our main results concerning the linearization of \eqref{eq:ssqg} and \eqref{eq:ssqg-diss} around $\bgtht$ introduced above. The direct linearization of \eqref{eq:ssqg} around $\bgtht$ is given by $L_{\bgtht} \phi = 0$, where
\begin{equation} \label{eq:direct-L-bgtht}
	L_{\bgtht} \phi = \rd_{t} \phi - \nb^{\perp} \bgtht \cdot \nb \Gmm \phi + \nb^{\perp} \Gmm \bgtht \cdot \nb \phi,
\end{equation}
where as for \eqref{eq:ssqg-diss}, the linearization around $\bgtht$ takes the form $L^{(\kpp)}_{\bgtht} \phi = 0$, where
\begin{equation} \label{eq:direct-L-bgtht-diss}
	L^{(\kpp)}_{\bgtht} \phi = \rd_{t} \phi - \nb^{\perp} \bgtht \cdot \nb \Gmm \phi + \nb^{\perp} \Gmm \bgtht \cdot \nb \phi + \kpp \Ups \phi.
\end{equation}

To formulate a linear illposedness result of the desired generality and precision, it is convenient to introduce the following set of parameters. Given $\lmb_{0} > 1$, called the \emph{initial frequency parameter} and $M > 1$, called the \emph{frequency growth factor}, define the corresponding (normalized) \emph{frequency growth time} $\tau_{M}$ to be 
\begin{equation} \label{eq:gf-time}
\tau_{M} := \int_{\lmb_{0}}^{M \lmb_{0}} \frac{1}{\gmm(\lmb_{0}, \lmb)} \frac{\ud \lmb}{\lmb_{0}}.
\end{equation}
Justification of the formula \eqref{eq:gf-time} shall be given later in Section~\ref{sec:ideas}, {but let us briefly explain the terminology we use here. Our wave packets will have frequency near $\lmb_{0}$, in both variables, at the initial time. The consideration of bicharacteristic ODE system (see Section~\ref{sec:ideas}, \eqref{eq:lin-L2}) predicts that wave packets would take time equal to $\tau_{M}$ for their initial frequency $\lmb_{0}$ to grow by the factor of $M$.} In addition to these parameters, we also fix an arbitrarily small parameter $0 < \dlt_{0} < \frac{1}{100}$ and a nonnegative parameter $0 \leq \sgm_{0} \leq \frac{1}{2} - \dlt_{0}$, which shall be used in the conditions that $\lmb_{0}, M$ need to satisfy. 

\medskip

\noindent \textit{Non-dissipative case}. Roughly speaking, our main linear result in the non-dissipative case, Theorem~\ref{thm:norm-growth}, states that given $\lmb_{0}$ and $M$ obeying a suitable condition (see \eqref{eq:gf-condition-1}--\eqref{eq:gf-condition-3} below), there exists an initial data set  {$\phi_{0}$} for the linearized equation  {$L_{\bgtht} \phi = 0$} with frequency $O(\lmb_{0})$ such that any corresponding solution exhibits frequency growth by factor $M$ in time $O_{f}(\tau_{M})$.

\begin{maintheorem}[Linear illposedness, non-dissipative case] \label{thm:norm-growth}
	Let $\bgtht = f(x_{2})$ be a smooth bounded shear steady state that is quadratically degenerate at $x_{2} = \dgx_{2}$ and fix a small parameter $0 < \dlt_{0} < \frac{1}{100}$  {and $0 \leq \sgm_{0} \leq \frac{1}{3}(1-2\dlt_{0})$}. Then there exist $\Lmb_{0} = \Lmb_{0}(f, \gmm, \dlt_{0}, \sgm_{0}) > 1$ and $T_0 = T_0(f,\gmm, \dlt_{0}, \sgm_{0})>0$ such that the following holds: For each $\lmb_{0} \in \bbN$ such that $\lmb_{0} \geq \Lmb_{0}$, $\tau_{M} \leq T_{0}$ and $M > 1$ satisfying the \emph{growth conditions}
\begin{align} 
\sup_{M' \in [1, M]} \frac{\gmm(\lmb_{0}, M' \lmb_{0})}{ {M'}} \tau_{M'} &\leq  {\min \set{\gmm(\lmb_{0}, \lmb_{0})^{1-\dlt_{0}}, \lmb_{0}^{\sgm_{0}}}}, \label{eq:gf-condition-1} \\
\tau_{M} &\leq \min \set*{\frac{\lmb_{0}^{1-\dlt_{0} {-3\sgm_{0}}}}{\gmm(\lmb_{0}, \lmb_{0})}, 1}, \label{eq:gf-condition-2} \\
M &\leq \lmb_{0}^{\frac{1}{\dlt_{0}}}, \label{eq:gf-condition-3} 
\end{align}
	there exists a smooth function  {$\phi_{0}$} such that
	\begin{equation} \label{eq:norm-growth-ini}
	\nrm{ {\phi_{0}}}_{H^{s'}} \leq  {C_{0; s', s}} \lmb_{0}^{s'-s} \nrm{ {\phi_{0}}}_{H^{s}} \quad \hbox{ for any } s' \geq s,
	\end{equation}
	yet any  {$\Gmm^{-\frac{1}{2}} L^{2}$}-solution  {$\phi$ to $L_{\bgtht} \phi = 0$} on $\left[0, \tfrac{100}{99} \tfrac{1}{\abs{f''(\dgx_{2})}}\tau_{M}\right]$ with  {$\rst{\phi}_{t=0}= \phi_{0}$} obeys
	\begin{equation} \label{eq:norm-growth}
	\sup_{t \in \left[ 0, \tfrac{100}{99} \tfrac{1}{\abs{f''(\dgx_{2})}}\tau_{M} \right]} \nrm{ {\phi}(t, \cdot)}_{H^{s'}} \geq C_{s', s}  {\frac{\gmm(\lmb_{0}, \lmb_{0})^{\frac{1}{2}}}{\gmm(\lmb_{0}, M \lmb_{0})^{\frac{1}{2}}}} M^{s'} \lmb_{0}^{s'-s} \nrm{ {\phi_{0}}}_{H^{s}} \quad \hbox{ for any } s, s' > 0.
	\end{equation}
\end{maintheorem} 

A $\Gmm^{-\frac{1}{2}} L^{2}$ solution to $L_{\bgtht} \phi$ is a natural notion of a weak solution in view of the energy structure of $L_{\bgtht}$; we postpone its precise definition to Section~\ref{sec:ideas} below.

\begin{remark}
Condition~\eqref{eq:gf-condition-1} is used to control the evolution of the frequency of the wave packet. The condition $\tau_{M} \leq \gmm(\lmb_{0}, \lmb_{0})^{-1} \lmb_{0}^{1-\dlt_{0}-3\sgm_{0}}$ in \eqref{eq:gf-condition-2} arises naturally from the error estimate (see Section~\ref{subsec:wp-error}), while $\tau_{M} \leq 1$ is natural in view of the local-in-time energy argument employed in the proof (see Section~\ref{sec:proofs}). Condition~\eqref{eq:gf-condition-3} is a mild technical condition that is assumed to simplify the control of some non-main error terms; we expect that it can be removed for specific $\gmm$'s (see, for instance, \cite{JO1} for electron and Hall MHD, where the frequency growth factor could be of size $e^{c \lmb_{0}}$).
\end{remark}

\begin{remark}
We note that the LHS of \eqref{eq:gf-condition-1} is \emph{uniformly bounded} when \emph{the order $\bt_{0}$ of $\gmm$ is less than $1$}, in which case we only need \eqref{eq:gf-condition-2} with $\sgm_{0} = 0$. In fact, the parameter $\sgm_{0}$ is simply a device introduced to handle (in an non-optimal way) the case $\bt_{0} \geq 1$, which is ``supercritical'' in many ways; see the discussion following Corollary~\ref{cor:norm-growth-ex}. If we assume that $f'''$ vanishes to a high order at $\dgx_{2}$, then the factor $3$ in front of $\sgm_{0}$ in \eqref{eq:gf-condition-2} can be lowered to any number greater than $2$ , and the condition $\sgm_{0} \leq \frac{1}{3}(1-2\dlt_{0})$ would be relaxed accordingly; see Remark~\ref{rem:x0-x1}.
\end{remark}

\begin{remark}
	{The above statement may give the impression that illposeness occurs only for a very carefully selected initial data, but this is not the case. We have stated the result in this way both for simplicity and also to achieve the sharp possible growth rate dictated by the bicharacteristic ODE system. In the proof, we shall see that the same growth rate (possibly up to multiplicative constants) persists for all initial data which are sufficiently close in $L^{2}$ to $\phi_{0}$ we constructed.
	
}
\end{remark}

We now discuss the implications of Theorem~\ref{thm:norm-growth}. We first demonstrate that under our assumptions on $\gmm$, \emph{the linearized equation around any smooth shear steady state with a quadratic degeneracy is always ill-posed in $H^{s}$ for any $s > {\frac{\bt_{0}}{2}}$.} As $\lmb \mapsto \gmm(\lmb_{0}, \lmb)$ is increasing on $[\lmb_{0}, \infty)$, we have the obvious bound
\begin{equation*}
\tau_{M} \leq \frac{M}{\gmm(\lmb_{0}, \lmb_{0})} \quad \hbox{ for $\lmb_{0}$ sufficiently large.}
\end{equation*}
As a result, if we set
\begin{equation*}
M(\lmb_{0}) = { \min\set*{\gmm(\lmb_{0}, \lmb_{0})^{1-\dlt_{0}}, \lmb_{0}^{\dlt_{0}}, \, \gmm(\lmb_{0}, \lmb_{0})^{\frac{1-2\dlt_{0}}{\bt_{0}}}, \lmb_{0}^{\frac{1-5\dlt_{0}}{3\bt_{0}}}}}, 
\end{equation*}
then \eqref{eq:gf-condition-2} (with $\sgm_{0} = \frac{1}{3}(1-2\dlt_{0})$) and \eqref{eq:gf-condition-3} are satisfied for sufficiently large $\lmb_{0}$. To check \eqref{eq:gf-condition-1}, we estimate 
\begin{align*}
\frac{\gmm(\lmb_0,M'\lmb_0)}{M' \min\set{\gmm(\lmb_{0}, \lmb_{0})^{1-\dlt_{0}}, \lmb_{0}^{\frac{1}{3}(1-2\dlt_{0})}}}  \tau_{M'} 
&\leq C(M')^{\beta_0} \frac{\gmm(\lmb_0,\lmb_0)}{\min\set{\gmm(\lmb_{0}, \lmb_{0})^{1-\dlt_{0}}, \lmb_{0}^{\frac{1}{3}(1-2\dlt_{0})}}} \frac{1}{\gmm(\lmb_0,\lmb_0)} \\ &\leq \frac{C}{\min\set{\gmm(\lmb_{0}, \lmb_{0})^{\dlt_{0}}, \lmb_{0}^{\dlt_{0}}}} \le 1,
\end{align*} 
using the above bound for $\tau_{M'}$ and {$M(\lmb_0) \le \gmm(\lmb_{0}, \lmb_{0})^{\frac{1-2\dlt_{0}}{\bt_{0}}}, \lmb_{0}^{\frac{1-5\dlt_{0}}{3\bt_{0}}}$}. Furthermore, observe that $M(\lmb_{0}) \to \infty$ and $\tau_{M(\lmb_{0})} \to 0$ as $\lmb_{0} \to \infty$. {Finally, by \eqref{eq:slow-var}, note that
\begin{equation} \label{eq:gmm-ratio}
\frac{\gmm(\lmb_{0}, \lmb_{0})^{\frac{1}{2}}}{\gmm(\lmb_{0}, M \lmb_{0})^{\frac{1}{2}}} \ageq M^{-\frac{\bt_{0}}{2}},
\end{equation}
so the RHS of \eqref{eq:norm-growth} is increasing in $M$ for $s' > \frac{\bt_{0}}{2}$. We have therefore proved:
\begin{corollary}[Linear $H^{s}$ illposedness for any $s > {\frac{\bt_{0}}{2}}$] \label{cor:norm-growth-all}
	Let $\bgtht = f(x_{2})$ be a smooth bounded shear steady state that is quadratically degenerate at $x_{2} = \dgx_{2}$. For any $s > {\frac{\bt_{0}}{2}}$, there exists a sequence ${\phi_{(n) 0}}$ of initial data sets and times $T_{(n)}$ such that
	\begin{equation*}
	\nrm{{\phi_{(n) 0}}}_{H^{s}} \leq 1, \qquad
	0 < T_{(n)} \leq 1, \qquad 
	T_{(n)} \searrow 0,
	\end{equation*}
	yet for any sequence ${\phi_{(n)}}$ of {$\Gmm^{-\frac{1}{2}} L^{2}$}-solutions to {$L_{\bgtht} \phi = 0$} on $[0, T_{(n)}]$ with {$\rst{\phi_{(n)}}_{t=0} = \phi_{(n) 0}$}, we have
	\begin{equation*}
	\frac{\sup_{t \in [0, T_{(n)}]}\nrm{{\phi_{(n)}}(t, \cdot)}_{H^{s}}}{\nrm{{\phi_{(n) 0}}}_{H^{s}}} \to \infty \quad \hbox{ as } n \to \infty.
	\end{equation*}
\end{corollary}

On the other hand, an inspection of the size of the optimal frequency growth factor $M$ for various model cases (see Corollary~\ref{cor:norm-growth-ex}) shows that the instability mechanism in hand is stronger for a multiplier $\gmm$ with faster growth. One way to make quantify this idea is to introduce the following notion:
\begin{definition} \label{def:HsHs'-illposed}
	We say that the shear steady state $\bgtht$ with a quadratic degeneracy is \emph{linearly $H^{s}$-$H^{s'}$ unstable by degenerate dispersion according to Theorem~\ref{thm:norm-growth}} if there exist a sequence $(M_{(n)}, \lmb_{0 (n)}, {\tau_{M (n)}})$ satisfying \eqref{eq:gf-condition-1}--\eqref{eq:gf-condition-3}, $\lmb_{0 (n)} \geq \Lmb_{0}$ {and $\tau_{M (n)} \leq T_{0}$} such that 
	\begin{equation*}
	{\frac{\gmm(\lmb_{0 (n)}, \lmb_{0 (n)})^{\frac{1}{2}}}{\gmm(\lmb_{0 (n)}, M_{(n)} \lmb_{0 (n)})^{\frac{1}{2}}}} M_{(n)}^{s'} \lmb_{0(n)}^{s' - s} \to \infty, \qquad \tau_{M_{(n)}} \to 0.
	\end{equation*}
\end{definition}
In what follows, we shall often drop the proviso \emph{according to Theorem~\ref{thm:norm-growth}}. Given a shear steady state $\bgtht$ that is linearly $H^{s}$-$H^{s'}$ unstable by degenerate dispersion, no mapping from $H^{s}$ to a set of {$\Gmm^{-\frac{1}{2}} L^{2}$}-solutions to {$L_{\bgtht} \phi = 0$} in $L^{\infty}_{t}([0, \dlt], H^{s'})$ can be bounded thanks to \eqref{eq:norm-growth} in Theorem~\ref{thm:norm-growth}. Note that if Definition~\ref{def:HsHs'-illposed} holds for one shear steady state $\bgtht$ with a quadratic degeneracy, then it holds for any other such shear steady states. {Indeed, the key conditions  \eqref{eq:gf-condition-1}--\eqref{eq:gf-condition-3} depend only on $\gmm$; the only dependence on the specific shear steady state $\bgtht$ is through the lower and upper bounds $\Lmb_{0}$ and $T_{0}$ on $\lmb_{0 (n)}$ and $\tau_{M (n)}$, respectively, which can be trivially enforced by restricting $n$ to be large enough}. Hence, Definition~\ref{def:HsHs'-illposed} is a property of the system \eqref{eq:ssqg} (more precisely, of $\gmm$). 

Next, we specialize $\gmm$ to several model cases and compute the (essentially) optimal growth factor given by Theorem~\ref{thm:norm-growth}.
\begin{corollary} \label{cor:norm-growth-ex}
	Let $\bgtht = f(x_{2})$ be a smooth shear steady state that is quadratically degenerate at $x_{2} = \dgx_{2}$.
	In each of the following cases, Theorem~\ref{thm:norm-growth} applies with the specified choice of $M$ for any $0 \leq \sgm < 1$, provided that {$\dlt_{0}$ and $\sgm_{0}$ are chosen appropriately depending on $\sgm$ and $\bt$}; moreover, in each case, $\tau_{M} \to 0$ as $\lmb_{0} \to \infty$. As a result, \eqref{eq:ssqg} is linearly $H^{s}$-$H^{s'}$ unstable by degenerate dispersion for the specified values of $s, s'$.  		\begin{center}
			\begin{tabular}{| l | l | l |} 
				\hline
				Multiplier $\Gmm$
				& Freq.~growth $M$
				& Lin.~$H^{s}$-$H^{s'}$ inst.  \\
				\hline
				$\gmm = \brk{\xi}^{\bt}$, $\bt > 1$
				& $M = \lmb_{0}^{\sgm \frac{1}{3(\bt-1)} }$ 
				& $s', s > \frac{\bt}{2}$, ${(1 + \frac{1}{3(\bt-1)})s' > s + \frac{\bt}{6(\bt-1)}}$ \\
				$\gmm = \brk{\xi}$
				& $M = \lmb_{0}^{\frac{1}{1-\sgm}}$ 
				& $s', s > \frac{1}{2}$ \\
				$\gmm = \brk{\xi}^{\bt}$, $\bt < 1$
				& $M = \lmb_{0}^{\sgm \frac{\bt}{1-\bt}}$ 
				& $s', s > \frac{\bt}{2}$, ${\frac{1}{1-\bt} s' > s + \frac{\bt^2}{2(1-\bt)}}$ \\
				$\gmm = \log^{\bt} (10+\abs{\xi})$, $\bt > 0$
				& $M = \log^{\sgm \bt} \lmb_{0}$ 
				& $s' = s > 0$ \\
				$\gmm = \log^{\bt} (10+\log (10+\abs{\xi}))$, $\bt > 0$
				& $M = \log^{\sgm \bt} \log \lmb_{0}$
				& $s' = s > 0$ \\
				\hline
			\end{tabular}
		\end{center}
\end{corollary}
Corollary~\ref{cor:norm-growth-ex} makes quantitative the expectation that the faster growth of $\gmm$, the stronger the instability by degenerate dispersion.\footnote{The case $\brk{\xi}^{\bt}$ for $\bt > 1$ is an exception, but this seems to be due to the inefficiency of our method in this case.}

\medskip

\noindent \textit{Dissipative case}. We now state the main linear result in the dissipative case. Let $\lmb_{0}$, $M$ and $\tau_{M}$ be defined as before. We introduce two small constants $0 < \dlt_{1} \leq \dlt_{0} < \frac{1}{100}$. In place of $\abs{f''(0)}^{-1} {\tau_{M}}$ that arose in Theorem~\ref{thm:norm-growth}, we define the function $t_{f}(\tau_{M})$ by the relation
\begin{equation*}
	\int_{0}^{t_{f}(\tau_{M})} \abs{f''(t', 0)} \, \ud t' = \tau_{M}.
\end{equation*}
Note that $f''$ is time-independent, then $t_{f}(\tau_{M}) = \abs{f''(0)}^{-1} \tau_{M}$.

\begin{maintheorem}[Linear illposedness, dissipative case] \label{thm:norm-growth-diss} 
Let $\bgtht_{0} = f_{0}(x_{2})$ a smooth {even} function with {$f_{0}''(0) \neq 0$}, let $\bgtht = f(t, x_{2})$ be the smooth solution to \eqref{eq:ssqg-diss-shear} with $f(0, x_{2}) = f_{0}(x_{2})$ and fix small parameters $0 < \dlt_{1} \leq \dlt_{0} < \frac{1}{100}$ and a {parameter $0 \leq \sgm_{0} \leq \frac{1}{3}(1-2\dlt_{0})$.} 
Then there exist $\Lmb_{1} = \Lmb_{1}({f, \gmm, \ups, \dlt_{0}, \dlt_{1}, \sgm_{0}})$ and {$T_1 = T_1({f, \gmm, \ups, \dlt_{0}, \dlt_{1}, \sgm_{0}})>0$} such that the following holds. For each $\lmb_{0} \in \bbN$ such that $\lmb_{0} \geq \Lmb_{1}$, $\tau_{M} \leq T_{1}$ and $M > 1$ satisfying the nondissipative growth conditions \eqref{eq:gf-condition-1}--\eqref{eq:gf-condition-3}, as well as the conditions 
\begin{gather} 
\int_{\lmb_{0}}^{M \lmb_{0}} \frac{\ups(\lmb_{0}, \lmb)}{\gmm(\lmb_{0}, \lmb)}\frac{\ud \lmb}{\lmb_{0}} 
=  o_{\lmb_{0}}(1),\label{eq:gf-condition-diss-1} \\
\sup_{M' \in [1, M]} \int_{\lmb_{0}}^{M' \lmb_{0}} \left(\frac{\rd_{\xi_{2}} \gmm(\lmb_{0}, \lmb)}{\rd_{\xi_{2}} \gmm(\lmb_{0}, M' \lmb_{0})}\right)^{1-\dlt_{1}} \frac{\gmm(\lmb_{0}, \lmb_{0})}{\gmm(\lmb_{0}, \lmb)^{2}} \frac{\ud \lmb}{\lmb_{0}} 
\leq 1,
\label{eq:gf-condition-diss-2}
\end{gather}
there exists a smooth function ${\phi_{0}}$ such that
\begin{equation} \label{eq:norm-growth-ini-diss}
	\nrm{{\phi_{0}}}_{H^{s'}} \leq {C_{s, s'}} \lmb_{0}^{s'-s} \nrm{{\phi_{0}}}_{H^{s}} \quad \hbox{ for any } s' \geq s,
\end{equation}
yet any {$\Gmm^{-\frac{1}{2}} L^{2}$}-solution {$\phi$ to $L^{(\kpp)}_{\bgtht} \phi = 0$} on $\left[ 0, \tfrac{100}{99} t_{f}(\tau_{M}) \right]$ with {$\phi(t=0) = \phi_{0}$} obeys
\begin{equation} \label{eq:norm-growth-diss}
	\sup_{t \in \left[ 0, \tfrac{100}{99} t_{f} (\tau_{M}) \right]} \nrm{{\phi}(t, \cdot)}_{H^{s'}} \geq C_{s', s} {\frac{\gmm(\lmb_{0}, \lmb_{0})^{\frac{1}{2}}}{\gmm(\lmb_{0}, M \lmb_{0})^{\frac{1}{2}}}} M^{s'} \lmb_{0}^{s'-s} \nrm{{\phi_{0}}}_{H^{s}} \quad \hbox{ for any } s' > 0.
\end{equation}
\end{maintheorem}
\begin{remark}
Condition \eqref{eq:gf-condition-diss-1} arises naturally from the contribution of the dissipative term in the error estimate in our degenerating wave packet construction. Condition \eqref{eq:gf-condition-diss-2} is a technical condition arising due to the time dependence of $\bgtht = f(t, x_{2})$. We note that when $\dlt_{1} = 0$ and $\tau_{M} \leq 1$, \eqref{eq:gf-condition-diss-2} already holds up to a logarithmic power of $\lmb_{0}$ (from Assumption~4 for $\gmm$); see \eqref{eq:HJE-h-int-s-re} and \eqref{eq:HJE-h-simple-0}.
\end{remark}

We introduce the following analogue of Definition~\ref{def:HsHs'-illposed}:
\begin{definition} \label{def:HsHs'-illposed-diss}
We say that the shear state $\bgtht$ satisfying the assumptions of Theorem~\ref{thm:norm-growth-diss} is \emph{linearly $H^{s}$-$H^{s'}$ unstable by degenerate dispersion according to Theorem~\ref{thm:norm-growth-diss}} if there exist a sequence $(M_{(n)}, \lmb_{0 (n)})$ satisfying \eqref{eq:gf-condition-1}--\eqref{eq:gf-condition-3}, \eqref{eq:gf-condition-diss-1}--\eqref{eq:gf-condition-diss-2} and $\lmb_{0 (n)} \geq \Lmb_{1}$ such that 
\begin{equation*}
	{\frac{\gmm(\lmb_{0 (n)}, \lmb_{0 (n)})^{\frac{1}{2}}}{\gmm(\lmb_{0 (n)}, M_{(n)} \lmb_{0 (n)})^{\frac{1}{2}}}} M_{(n)}^{s'} \lmb_{0(n)}^{s' - s} \to \infty, \qquad \tau_{M_{(n)}} \to 0.
\end{equation*}
\end{definition}
In what follows, we shall often drop the proviso \emph{according to Theorem~\ref{thm:norm-growth-diss}}. As before, Theorem~\ref{thm:norm-growth-diss} implies that given a shear state $\bgtht$ that is linearly $H^{s}$-$H^{s'}$ unstable by degenerate dispersion, no mapping from $H^{s}$ to a set of {$\Gmm^{-\frac{1}{2}} L^{2}$}-solutions to {$L_{\bgtht} \phi = 0$} in $L^{\infty}_{t}([0, \dlt], H^{s'})$ can be bounded. Moreover, Definition~\ref{def:HsHs'-illposed-diss} is a property of the system \eqref{eq:ssqg-diss} (more precisely, of $\gmm$, $\kpp \ups$) in the same sense as before.

We now specialize $\gmm$ and $\ups$ to several model cases and exhibit instances of illposedness in the dissipative case given by Theorem~\ref{thm:norm-growth-diss}.
\begin{corollary} \label{cor:norm-growth-diss-ex}
Fix $\kpp > 0$, let $\bgtht_{0} = f_{0}(x_{2})$ a smooth odd function with $f_{0}'(0) \neq 0$, and let $\bgtht = f(t, x_{2})$ be the smooth solution to \eqref{eq:ssqg-diss-shear}. In each of the following cases, Theorem~\ref{thm:norm-growth-diss} applies with the specified choice of $M$ for any $0 \leq \sgm < 1$, provided that $\lmb_{0}$ is sufficiently large depending on $\sgm$; moreover, in each case, $\tau_{M} \to 0$ as $\lmb_{0} \to \infty$. As a result, \eqref{eq:ssqg-diss} is linearly $H^{s}$-$H^{s'}$ unstable by degenerate dispersion for the specified values of $s'$. In the table below, the restrictions on $M$ and $(s,s')$ are \emph{in addition} to those from the nondissipative case.
\begin{center} 
	\begin{tabular}{ | l | l | l | l |} 
		\hline
		Multiplier $\Gmm$ & Dissipation $\Upsilon$ & Freq.~growth $M$ & Lin.~$H^{s}$-$H^{s'}$ inst.~ \\ 
		\hline
		$\gmm = \brk{\xi}^{\bt}$, $\bt > 1$ 
		& $\ups = \brk{\xi}^{\alp}$, $\alp < \bt$ 
		& $M =  \lmb_{0}^{\sigma\frac{\bt-\alp}{1-(\bt-\alp)} } $ 
		& $\frac{s'}{1-(\bt-\alp)} > s + \frac{\bt(\bt-\alp)}{2(1-(\bt-\alp))} $ \\
		$\gmm = \brk{\xi}$
		& $\ups = \brk{\xi}^{\alp}$, $\alp < 1$
		& $M = \lmb_{0}^{\sigma \frac{1-\alp}{\alp}}$ 
		& $\frac{s'}{\alp} > s + \frac{1-\alp}{2\alp} $ \\
		$\gmm = \brk{\xi}^{\bt}$, $\bt < 1$
		& $\ups = \brk{\xi}^{\alp}$, $\alp < \bt$
		& $M =  \lmb_{0}^{\sgm \frac{\bt-\alp}{1-(\bt-\alp)}} $ 
		& $\frac{s'}{1-(\bt-\alp)} > s + \frac{\bt(\bt-\alp)}{2(1-(\bt-\alp))} $ \\
		$\gmm = \log^{\bt} (10+\abs{\xi})$, $\bt > 0$
		& $\ups = \log^{\alp} (10+\abs{\xi})$, $\alp < \bt$
		& $M = \log^{\sigma(\bt-\alp)}\lmb_{0}$
		& $s' = s$ \\
		\hline
	\end{tabular}
\end{center} 
\end{corollary}

The illposedness results in Corollary~\ref{cor:norm-growth-diss-ex} are essentially sharp, in that it is not difficult to prove using standard energy estimates that in each case, {$L^{(\kpp)}_{\bgtht} \phi = 0$} is locally well-posed in $H^{s}$ for any $s \geq 0$ if the dissipative exponent $\bt$ is greater than $\alp$.

\medskip

\noindent\textbf{Nonlinear illposedness results.}
Remarkably, the linear norm growth results (Theorems~\ref{thm:norm-growth} and \ref{thm:norm-growth-diss}) may be extended to corresponding norm inflation properties of the nonlinear Cauchy problem. 

\begin{maintheorem}[Nonlinear illposedness] \label{thm:nonlin-norm-infl} 
Assume that a shear steady state $\bgtht$ with a quadratic degeneracy (resp.~a shear state $\bgtht$ satisfying the assumptions of Theorem~\ref{thm:norm-growth-diss}) is linearly $H^{s}$-$H^{s'}$ unstable by degenerate dispersion according to Theorem~\ref{thm:norm-growth} (resp.~Theorem~\ref{thm:norm-growth-diss}) with $s' >  {\frac{3}{2}} \beta_{0}+3$.
Then $\bgtht$ is nonlinearly {$H^{s} \to H^{s'}$} ill-posed with respect to \eqref{eq:ssqg} (resp.~\eqref{eq:ssqg-diss}) in the following sense: For any $\eps, \dlt, A>0$, there exists initial data $\tht_0 \in C^\infty_c(\Omg)$ with ${\nrm{\tht_0}_{H^s}}<\eps $ such that either \begin{itemize}
		\item there exists no solution $\tht \in \bgtht + L^\infty([0,\dlt]; {H^{s'}})$ to \eqref{eq:ssqg} (resp.~\eqref{eq:ssqg-diss}) with $\rst{\tht}_{t=0} = \rst{\bgtht}_{t=0} + \tht_0$, or 
		\item any solution $\tht$ belonging to $\bgtht + L^\infty([0,\dlt]; {H^{s'}})$ satisfy the growth \begin{equation*}
		\begin{split}
		\sup_{t \in [0,\dlt]} \nrm{{(\tht - \bgtht)}(t, \cdot)}_{H^{s'}} > A. 
		\end{split}
		\end{equation*}
	\end{itemize}
\end{maintheorem}
 
Since $\bgtht(t, \cdot)$ can be chosen to have an arbitrarily small $H^{s'} \cap H^{s}$ norm on a sufficiently short time interval in both the dissipative and non-dissipative cases, we immediately obtain the following illposedness statement with $\bgtht = 0$:
\begin{corollary} \label{cor:nonlin-norm-infl-0}
Assume that \eqref{eq:ssqg} (resp.~\eqref{eq:ssqg-diss}) is linearly $H^{s}$-$H^{s'}$ unstable by degenerate dispersion with {$s' > 3+ {\frac{3}{2}} \beta_0$}. Then \eqref{eq:ssqg} (resp.~\eqref{eq:ssqg-diss}) is nonlinearly {$H^{s} \to H^{s'}$} ill-posed in the following sense: For any $\eps, \dlt, A>0$, there exists initial data $\tht_0 \in C^\infty_c(\Omg)$ with ${\nrm{\tht_0}_{H^s}}<\eps $ such that either \begin{itemize}
		\item there exists no solution $\tht \in L^\infty([0,\dlt];{H^{s'}})$ to \eqref{eq:ssqg} (resp.~\eqref{eq:ssqg-diss}) with $\rst{\tht}_{t=0} = \tht_0$, or 
		\item any solution $\tht$ belonging to $L^\infty([0,\dlt];{H^{s'}})$ satisfy the growth \begin{equation*}
		\begin{split}
		\sup_{t \in [0,\dlt]} \nrm{{\tht}(t, \cdot)}_{H^{s'}} > A. 
		\end{split}
		\end{equation*}
	\end{itemize}
\end{corollary}
{
\begin{remark}
	While we shall refrain from giving details, we note that the assumption $s' > 3+ {\frac{3}{2}}\beta_0$ can be lowered (up to $s'>2$) when $\Gmm$ becomes less singular. The issue is to control the nonlinearity in $L^{2}$ using the hypothesis that $\tht \in H^{s'}$, and more precisely, to have an estimate of the form \begin{equation*}
		\begin{split}
			\nrm{ \nb^\perp \Gmm[\tht] \cdot \nb \tht }_{L^{2}} \le C(\nrm{\tht}_{H^{s'}}) \nrm{ \tht }_{L^2}. 
		\end{split}
	\end{equation*} When $\Gmm = \Lmb^{\bt}$ (which means $\beta_0 = \bt$), using Sobolev inequalities we have that for any $s>2$, \begin{equation*}
		\begin{split}
			\nrm{ \nb^\perp \Lmb^\bt[\tht] \cdot \nb \tht }_{L^{2}} \lesssim \nrm{\nb\tht}_{L^{\infty}} \nrm{  \nb^\perp \Lmb^\bt[\tht] }_{L^{2}} \lesssim \nrm{\tht}_{L^{2}}^{ 2 - \frac{2+\bt}{s'} } \nrm{\tht}_{H^{s'}}^{ \frac{2+\bt}{s'} }
		\end{split}
	\end{equation*} and hence we see that the requirement is precisely $s' \ge 2+\bt$. 
\end{remark}
}

By Corollary~\ref{cor:norm-growth-all}, it follows that \emph{any \eqref{eq:ssqg} satisfying our assumptions for $\gmm$ is nonlinearly {$H^{s} \to H^{s'}$} ill-posed for {$s' > 3+ {\frac{3}{2}} \beta_0$}.} Moreover, Corollaries~\ref{cor:norm-growth-ex} and \ref{cor:norm-growth-diss-ex} provide ranges of $(s, s')$ for which \eqref{eq:ssqg} and \eqref{eq:ssqg-diss}, respectively, are nonlinearly {$H^{s} \to H^{s'}$} ill-posed for some model cases of $\gmm$ and $\kpp \ups$.

\smallskip 

{When $\Omg = \bbT \times \bbR$ and $s = s'$ sufficiently large, we may furthermore exhibit an $H^{s}(\Omg)$ initial data corresponding to which no $L^{\infty} H^{s}$ solution to \eqref{eq:ssqg} exists on any time interval. }

\begin{maintheorem}[Nonexistence] \label{thm:nonexist}
	Let $\Omg = \bbT \times \bbR$. For any $s>\frac{3}{2}\beta_{0}+3$ and $\eps>0$, there exists $\tht_{0} \in H^{s}(\Omg)$ satisfying $\nrm{\tht_{0}}_{H^{s}} <\eps$ such that for any $\dlt>0$, there is no solution to \eqref{eq:ssqg} belonging to $L^\infty([0,\dlt];H^{s}(\Omg))$ with initial data $\tht_{0}$. 
\end{maintheorem}

\medskip

\noindent\textbf{Result {from} \cite{CJNO}: Nonlinear well(!)posedness of the logarithmically singular case with loss of regularity.}
Through Corollary~\ref{cor:norm-growth-ex} and Corollary~\ref{cor:nonlin-norm-infl-0}, we concluded nonlinear $H^{s}$-$H^{s'}$ illposedness of \eqref{eq:ssqg} for $s'$ strictly smaller than $s$ when, say, $\gmm(\xi) = \brk{\xi}^{\bt}$ for $\bt > 0$. However, this conclusion did \emph{not} apply to multipliers with slower growth, e.g., $\gmm(\xi) = \log (10 + \abs{\xi})$. This difference is no shortcoming of our approach. In the paper \cite{CJNO}, we actually obtain local well(!)posedness of \eqref{eq:ssqg} with $\gmm(\xi) = \log (10 + \abs{\xi})$ in Sobolev spaces with exponents that \emph{decrease in time}. As a result, this system is \emph{not} nonlinearly $H^{s}$-$H^{s'}$ ill-posed for any $s' < s$; in the opposite direction, Corollaries~\ref{cor:norm-growth-ex} and \ref{cor:nonlin-norm-infl-0} demonstrate that a decrease of the Sobolev exponent in time is inevitable.

\begin{maintheorem}[Wellposedness in the logarithmically singular system \cite{CJNO}] \label{thm:wp-log}
	 Consider the logarithmically singular SQG equation, possibly with dissipation: \begin{equation}\label{eq:sqg-log}
	 \begin{split}
	 \rd_t\tht + u\cdot\nb\tht+ \kappa\Upsilon(\tht)  = 0,\\
	 u = \nb^\perp \log(10+\Lmb)\tht.
	 \end{split}
	 \end{equation}
	 
	 In the inviscid case $(\kpp=0)$, for any $s_0>4$ and $\tht_0\in H^{s_0}$, there exists some $T = T(s_0,\nrm{\tht_0}_{H^{s_0}})>0$ such that there is a solution $\tht\in C([0,T];H^4)$ to \eqref{eq:sqg-log} with initial data $\tht_0$ satisfying \begin{equation*}
	 \begin{split}
	 \nrm{\tht(t,\cdot)}_{H^{s(t)}} \le C\nrm{\tht_0}_{H^{s_0}}
	 \end{split}
	 \end{equation*} for some continuous function $s(t)>4$ of $t$ with $s(0)=s_0$ in $t\in[0,T]$. The solution is unique in the class $C([0,T];H^4)$. 
	 
	 Furthermore, the dissipative system $(\kpp>0)$ is locally well-posed in $C([0,T];H^{s})$ for any $s>4$, as long as there exists some $\Xi_0>0$ such that \begin{equation*}
	 \begin{split}
	 \frac{\upsilon(|\xi|)}{\log(10+|\xi|)} \ge \frac{C_s}{\kappa}\nrm{\tht_0}_{H^s},\quad |\xi|>\Xi_0 
	 \end{split}
	 \end{equation*} for some $C_s>0$ depending only on $s$.
\end{maintheorem} 

\begin{remark}
	A similar wellposedness result can be proved for $\Gmm = \log^{\bt}(10+\Lmb)$ with any $\bt\le1$.
\end{remark}

{\begin{remark}
	The nonexistence theorem (Theorem~\ref{thm:nonexist}), when combined with Theorem~\ref{thm:wp-log}, shows that for logarithmic $\Gmm$, the local solution with $\tht_0\in H^s$ for some $s>4$ in general instantaneously escapes $H^s$ for $t>0$.
\end{remark}}


\subsection{Explicitly solvable toy model and discussion of difficulties}\label{subsec:toy}
Let us present a toy model \eqref{eq:sqg-lin-model-fourier} which clearly demonstrates degeneration of linear solutions. To arrive at the toy model, we may start from the nonlinear equation \begin{equation*}
\begin{split}
\rd_t\tht + \nb^\perp\Gmm(\tht)\cdot\nb\tht = 0, 
\end{split}
\end{equation*} and consider the linearization around the steady state $\bgtht(x_1,x_2) = - \frac{x_2^2}{2}$, under the \textit{formal} assumption that $\nabla^\perp\Gmm (x_2^2) \equiv 0$. The resulting equation for the perturbation, which is again denoted by $\tht$, is simply   \begin{equation}\label{eq:sqg-lin-model}
\begin{split}
\rd_t\tht + x_2 \rd_{x_1} \Gmm(\tht)= 0 . 
\end{split}
\end{equation} We may separate $x_1$-dependence under the ansatz \begin{equation*}
\begin{split}
\tht^{(\lmb_0)} (t,x_1,x_2) = e^{i\lmb_0 x_1} \varphi(t,x_2)
\end{split}
\end{equation*} for some $\lmb_0$. Assume for simplicity that the multiplier $\gmm$ for $\Gmm$ is radial. Denoting the dual variable of $x_2$ by $\xi$ and taking the Fourier transform, we have with $\Lmb(\xi) :=  \sqrt{\lmb_0^2+\xi^2}$ that  \begin{equation}\label{eq:sqg-lin-model-fourier}
\begin{split}
\lmb_0^{-1}\rd_t \widehat{\varphi} (\xi) + \gmm(\Lmb) \rd_\xi \widehat{\varphi} (\xi) =  - (\rd_\xi\gmm)(\Lmb) \frac{\xi}{\Lmb} \widehat{\varphi}(\xi). 
\end{split}
\end{equation}  This is simply a transport equation in $\xi$, which can be explicitly solved along the characteristics: we may define the trajectories $\xi(t;\xi_0)$ by \begin{equation}\label{eq:ODE-toy}
\left\{
\begin{aligned} 
\dot{\xi}(t;\xi_0) &= \lmb_0 \gmm( \Lmb( \xi(t;\xi_0) ) ), \\
\xi(0;\xi_0) & = \xi_0.
\end{aligned}
\right.
\end{equation} The solution can be written by \begin{equation}\label{eq:sol-toy}
\begin{split}
	\widehat{\varphi} (t,\xi(t;\xi_0)) = \frac{\gmm(\Lmb(\xi_0))}{\gmm(\Lmb(\xi(t;\xi_0)))} \widehat{\varphi}_0 (\xi_0) . 
\end{split}
\end{equation} The Jacobian of the flow map $\xi_0 \mapsto \xi(t;\xi_0)$ is given simply by \begin{equation*}
\begin{split}
	\frac{\gmm(\Lmb(\xi(t;\xi_0)))}{\gmm(\Lmb_0)}. 
\end{split}
\end{equation*} In the following, we shall take initial data $\widehat{\varphi}_0$ which is sharply concentrated near $\xi_0 \simeq \lmb_0$ and introduce simplifying notation $\Lmb(t) := \Lmb(\xi(t;\lmb_0))$, $\Lmb_0 := \sqrt{2}\lmb_0$. Then, we have that \begin{equation*}
\begin{split}
\nrm{\widehat{\varphi}(t)}_{L^2_\xi} \simeq \left( \frac{\gmm(\Lmb(t))}{\gmm(\Lmb_0)} \right)^{-\frac{1}{2}} \nrm{\widehat{\varphi}_0}_{L^2_\xi}. 
\end{split}
\end{equation*} This is consistent with propagation of  $\gmm^{\frac{1}{2}}\widehat{\varphi}$ in $L^2$. On the other hand, \begin{equation*}
\begin{split}
\nrm{|\xi|^s\widehat{\varphi}(t)}_{L^2_\xi} \simeq \frac{|\xi(t;\lmb_0)|^s}{|\lmb_0|^s}\left( \frac{\gmm(\Lmb(t))}{\gmm(\Lmb_0)} \right)^{-\frac{1}{2}} \nrm{|\xi|^s\widehat{\varphi}_0}_{L^2_\xi}. 
\end{split}
\end{equation*} Hence, for \eqref{eq:sqg-lin-model} to be ill-posed in $H^s$, it suffices to have for $|\lmb_0|\gg 1$ 
and $T = T(\lmb_0) \ll 1$
that  \begin{equation}\label{eq:illposed-Hs-condition}
\begin{split}
\frac{|\xi(T;\lmb_0)|^s}{|\lmb_0|^s}\left( \frac{\gmm(\Lmb(T;\lmb_0))}{\gmm(\Lmb_0(\lmb_0))} \right)^{-\frac{1}{2}} \gg 1 
\end{split}
\end{equation} or simply \begin{equation*}
\begin{split}
\frac{|\xi(T)|^{2s}}{\gmm(\xi(T))} \gg \frac{|\lmb_0|^{2s}}{\gmm(\lmb_0)}.
\end{split}
\end{equation*} 
Assuming that $\gmm$ is {increasing}, we have \begin{equation*}
\begin{split}
\xi(t;\lmb_0) \ge \lmb_0 + t\lmb_0 \gmm(\lmb_0).
\end{split}
\end{equation*} (This is expected to be sharp for small timescales.) Therefore, for some $T = T(\lmb_0)$ satisfying $ \frac{1}{\gmm(\lmb_0)} \ll T \ll 1$, we have $\xi(T) \gg \lmb_0$, and this will guarantee \eqref{eq:illposed-Hs-condition} for $s>0$ large.

Given some concrete symbol $\gmm$ (e.g. $\gmm(\Lmb)= \brk{\Lmb}^\beta$ for some $\beta>0$), one can see the range of $s, s'$ where the toy model \eqref{eq:sqg-lin-model} is $H^s$--$H^{s'}$ unstable in the sense of Definition \ref{def:HsHs'-illposed}, using the above formula for the solution in the Fourier variable. We leave the details of this computation for the interested reader. Let us demonstrate that in the logarithmic case (where $\gmm(\xi) \lesssim \log(\xi)$ for large $|\xi|$), $s = s'$ is forced. Indeed, for $s(T) = s-MT$ with some $M>0$ independent of $\lmb_0$, we have \begin{equation*}
\begin{split}
\frac{|\xi(T)|^{2s(T)}}{\gmm(\xi(T))} \lesssim \frac{|\lmb_0|^{2s}}{\gmm(\lmb_0)},
\end{split}
\end{equation*} which shows a losing estimate in the scale of time-dependent Sobolev spaces $H^{s(T)}$. It suggests that there could be a similar estimate in the nonlinear case; this is precisely the content of \cite{CJNO}. 

One may use the above model equation to understand the dissipative case as well. Then, \eqref{eq:sqg-lin-model} and \eqref{eq:sqg-lin-model-fourier} (after separating $x_1$-dependence) are simply replaced with \begin{equation}\label{eq:sqg-lin-model-diss}
	\begin{split}
		\rd_t\tht + x_2\rd_{x_1}\Gmm(\tht) = -\kpp\Upsilon(\tht)
	\end{split}
\end{equation} and \begin{equation}\label{eq:sqg-lin-model-fourier-diss}
\begin{split}
	\lmb_0^{-1}\rd_t \widehat{\varphi} (\xi) + \gmm(\Lmb) \rd_\xi \widehat{\varphi} (\xi) =  - (\rd_\xi\gmm)(\Lmb) \frac{\xi}{\Lmb} \widehat{\varphi}(\xi) - \kpp \upsilon(\Lmb) \widehat{\varphi}(\xi) .
\end{split}
\end{equation} (We are assuming that the symbol $\upsilon$ is radial.) Again, this equation is solvable along characteristics. Using the solution, one may see that when $\upsilon \gg \gmm$, there is well-posedness in $H^s$-spaces, while $\upsilon \ll \gmm$ still gives $H^s$ illposedness.

\medskip
\noindent {\bf Discussion of difficulties.}
We are now in a good position to explain the main difficulties in establishing Sobolev illposedness for the actual linear  {homogeneous equations associated with }\eqref{eq:lin-L2} and \eqref{eq:lin-L2-diss},  {as well as the nonlinear equations \eqref{eq:ssqg} and \eqref{eq:ssqg-diss}.} Comparing the inviscid linear case \eqref{eq:lin-L2} with \eqref{eq:sqg-lin-model}, there are two differences: (1)~first, the principal coefficient is not exactly linear, and (2)~second, a lower order term is present in \eqref{eq:lin-L2}.
With these two differences \textit{combined}, it becomes a challenging problem to just construct an approximate solution to \eqref{eq:lin-L2} which exhibits the same illposedness behavior with an explicit solution to \eqref{eq:sqg-lin-model}.  {Meanwhile, once a good approximate construction has been constructed, it can be upgraded to illposedness results for \eqref{eq:lin-L2}, and even the nonlinear equation \eqref{eq:ssqg}, using the robust testing (or duality) method introduced in \cite{JO1}.}

Henceforth, we focus on approximate solution construction. Regarding \textbf{difference~(1),} it has been known that having an exactly linear principal coefficient significantly simplifies the analysis (Craig--Goodman \cite{CrGm} is a good example). We have seen in the above that in such a case, taking the Fourier transform results in a transport term. Moreover, one can perform WKB analysis with a linear phase. In view of this, when the principal term is given by a differential operator with a linearly degenerate coefficient, it is natural to first apply a coordinate transform which makes the coefficient exactly linear near the degeneracy; this was the approach in our previous work \cite{JO1}. Such a coordinate change is not available in the case of pseudo-differential principal term. Next, regarding \textbf{difference~(2)}, note that the lower order term in \eqref{eq:lin-L2} is again given by a pseudo-differential operator, and its (generalized) order could become very close to that for the principal operator when $\Gmm$ is only \textit{slightly} singular (e.g. when $\gmm(\xi)=\log(10+|\xi|), \log(10+\log(10+|\xi|))$). In such cases, it is basically impossible to distinguish the lower order term from the principal term, and in the proof we indeed incorporate a part of the lower order term into the principal part. 

In the dissipative linear case \eqref{eq:lin-L2-diss},  {the main idea is to use the same approximate solution to \eqref{eq:lin-L2} and treat the dissipative term as an error}. Nevertheless, there is yet another serious difficulty: in general, there is no nontrivial steady solutions to \eqref{eq:ssqg-diss-shear} and we have to work with time-dependent shear flows. In general, the strength of the degeneracy of the principal coefficient $\Gmm\nb^\perp\bgtht$ {changes} with time\footnote{In principle, the location of the degeneracy would move in time as well; in the current work, we prohibit such behavior by imposing even symmetry both in the shear profile $\bgtht$ and the dissipative operator $\Upsilon$.}, whose effect cannot be regarded as a perturbation from the initial data $\Gmm\nb^\perp\bgtht_{0}$. Again, with respect to this difficulty, the most problematic case is when $\Gmm$ is slightly singular, because then the frequency growth is slower (hence most sensitive to time dependence of the background shear). 

To overcome these difficulties, we develop a fairly general framework for directly constructing degenerating wave packets for a linear pseudo-differential equation of the form $$\rd_{t} \phi + i p(t, x, D) \phi = 0,$$ where $p(t, x, D)$ takes into account not only the principal term but also key lower order terms in \eqref{eq:lin-L2}, and may possibly be time-dependent. While the formal derivation of the ansatz for the approximate solution is straightforward (see Section~\ref{sec:psdo}), the difficulty is to rigorously control the ansatz and the error in sufficiently long time scales within which significant degeneration occurs. Among others, two key ideas in this work that allow us to resolve this difficulty are: observations concerning the Hamilton--Jacobi equation for $\rd_{t} + i p$ and the associated transport equations that allow for controlling the ansatz in long enough time scales (see Section~\ref{sec:HJE}) and sharp estimates for oscillatory integrals appearing in the symbol for the error term (see Section~\ref{sec:conj-err}). We refer the reader to Section~\ref{subsec:ideas} for a more detailed discussion of the key ideas.

{\subsection{Interpretation of strong illposedness}\label{subsec:interpret}
	
	We would like to discuss the potential physical interpretations of strong illposedness, especially since we have demonstrated in the above that some generalized SQG equations with singular multipliers have physical origin. 
	
	Naïvely, strong illposedness for a PDE of physical origin (especially the nonexistence statement from Theorem \ref{thm:nonexist}) suggests that the model is physically invalid and should therefore be rejected. At this point, one could argue that the illposedness simply arises from choosing the wrong function spaces: the illposedness results we obtained in this work pertain to Sobolev spaces rather than, for example, analytic (or Gevrey) spaces. However, we shall proceed under the assumption that this is \textit{not} the case, which is supported by the rapid frequency growth inherent in the bicharacteristic ODE system. Actually, based on the fact that the frequency rate itself grows with the initial frequency, illposedness in analytic spaces for the Hall-MHD case was established in \cite{JO1}. Returning to the point where we should decide whether or not to reject the model, one should consider several possibilities (and answer the resulting questions): \begin{itemize}
		\item \textbf{Incorporating a higher order term}. Indeed, for both of Hall-MHD and AM-LQG equations we discussed in the above, it is natural to have some dissipative terms, which makes the equations (at least locally in time) well-posed. Then the natural question is: what is the implication of strong illposedness for the dissipative case? More concretely, what happens in the limit for the sequence of solutions when the dissipation coefficient vanishes? At this point, it is worth mentioning that dissipation is not the only physical mechanism for fixing illposedness; in the case of Hall-MHD, it was reported recently in \cite{CheBe} that by adding ``electron inertial effects'' to the Hall-MHD system (this is referred to extended magnetohydrodynamics, or XMHD for short, system in the physics community), wellposedness in Sobolev spaces is restored, without incorporating the dissipation term for the magnetic field. This ``inertial'' term is hyperbolic and not parabolic. Indeed, physicists have reported (\cite{Abdelhamid_2016,Miloshevich_2017}) that in magnetohydrodynamic turbulence (which roughly describes the behavior of solutions in the zero-dissipation limit), there are three frequency regimes: the smallest is governed by ideal MHD, the intermediate by the Hall current effect, and the largest by the inertial effect.
		\item \textbf{Wellposedness theory with nondegenerate data}. Without the addition of any other terms, the equations could be perfectly valid for certain ``regimes'': our linear and nonlinear illposedness results are crucially relying upon quadratic degeneracy of the background shear profile. It is possible that at least small perturbations of a linearly degenerate background is well-behaved (this is clear in the case of the linearized equation, at least). Indeed, in the case of Hall-MHD system, local wellposedness can be established near a nonzero constant magnetic field (\cite{JO2}), which turns out to be the original physical setup by Lighthill \cite{Light} who first derived the model. 
	\end{itemize}
	
	In conclusion, we emphasize that the illposedness results presented in this paper serve merely as a starting point for the mathematical study of generalized SQG equations with singular velocities, opening up several important and challenging problems. 
	
} 

{

\subsection{Further open problems}

We discuss a few technical open problems that naturally arise after this work. \begin{itemize}
	\item \textbf{Existence of weak solutions}. The paper \cite{CCCGW} proves the existence of global in time $L^{p}$ weak solutions for the case when $\Gmm = \Lmb^{\bt}$ with $\bt<0$. It seems that their existence proof fails as soon as $\Gmm$ becomes singular (in the sense that the multiplier diverges to infinity as $|\xi|\to\infty$). It would be interesting to see whether convex integration techniques can be adapted to give existence of solutions with singular multipliers. 
	\item \textbf{Gevrey illposedness}. While we expect the singular SQG equations to be ill-posed in analytic (and even in Gevrey) spaces, it seems quite difficult to obtain such an illposedness statement based on the norm growth estimates in the current paper. For this purpose, it might be necessary to construct degenerating wave packet solutions for a larger interval of time. Note that our previous work on the illposedness of the Hall-MHD system \cite{JO1} includes analytic/Gevrey illposedness. A similar problem, which could be also challenging, is to study whether the range of $(s,s')$ given in Corollary \ref{cor:norm-growth-ex} is sharp. To this end, one may consider the perturbation dynamics near other steady states.
	
	\item \textbf{Nonexistence in $\bbR^{2}$ and $\bbT^{2}$}. Note that our nonexistence statement (Theorem \ref{thm:nonexist}) requires the physical domain to be $\bbT\times\bbR$. When the domain is $\bbT^{2}$ (or more generally a compact two-dimensional manifold), there is a serious difficulty in superposing degenerating wave packets to upgrade norm inflation to nonexistence. On the other hand, when the domain is $\bbR^{2}$, the difficulty comes from spatially truncating the degenerating wave packets. 
	
	\item \textbf{Illposedness for 3D vector systems}. One may consider the \textit{active vector system} for the variable $B(t,\cdot):\bbR^3\to\bbR^3$ given by \begin{equation}  \label{eq:e-MHD-gen}
		\left\{
		\begin{aligned} 
			&\rd_t B + \nb \times( (\nb\times (-\lap)^{-\alp} B) \times B) = 0, \\
			&\nb\cdot B = 0, 
		\end{aligned}
		\right.
	\end{equation} for some $\alp\ge0$. This turns out to be a 3D generalization of SQG models, and interpolates between the 3D vorticity equations (when $\alp=1$) and the E-MHD system (when $\alp=0$) \cite{CJ,CCJ}. Local wellposedness as well as existence of some global solutions were obtained for $\alp \ge 1/2$ in \cite{CJ,CCJ}. It would be interesting to study local wellposedness of \eqref{eq:e-MHD-gen} for general values of $\alpha$. 
	
\end{itemize}

}

\subsection{Organization of the paper}\label{subsec:orga}

The rest of the paper is divided into six sections, which we briefly describe below. The proofs of the main results in each section is largely independent of those from other sections.

\begin{itemize}
	\item Section \ref{sec:ideas} contain preliminary computations on the linearized operators and a preview of the proofs of main results. 
	\item Section \ref{sec:psdo} begins with algebraic preliminaries regarding pseudo-differential calculus. The equations satisfied by the phase and amplitude functions are fixed in this section. Then, we provide expansion formulas for the pseudo-differential operators appearing after the conjugation by the phase function. In particular, an explicit representation formula for the remainder operator is derived. 
	\item Section \ref{sec:HJE} deals with the equations for the phase and amplitude functions chosen in Section \ref{sec:psdo}. In this section, we fix the choice of initial data for the phase and amplitude functions and derive sharp high order estimates for them. 
	\item In Section \ref{sec:conj-err}, we obtain operator bounds for symbols appeared in Section \ref{sec:psdo}. 
	\item In Section \ref{sec:wavepacket}, key estimates for the degenerating wave packet solutions are obtained, by applying sharp estimates from Section \ref{sec:HJE} to operator bounds from Section \ref{sec:conj-err}. 
	\item All the main theorems are proved in Section \ref{sec:proofs}, combining all the ingredients.
\end{itemize}

\subsection*{Acknowledgments}{ {We sincerely thank the anonymous referee for several helpful comments, which have been incorporated into the current version of the paper.} 
D.~Chae was supported by NRF grant No. 2021R1A2C1003234.
I.-J.~Jeong was supported  by NRF Korea grant No. 2022R1C1C1011051 and RS-2024-00406821. S.-J.~Oh was supported in part by the Samsung Science and Technology Foundation under Project Number SSTF-BA1702-02, a Sloan Research Fellowship and a National Science Foundation CAREER Grant under NSF-DMS-1945615.}

\section{Preliminaries and ideas of the proof}\label{sec:ideas}

In Section~\ref{subsec:prelim-sec2}, we derive the linearized operators and discuss its energy structure. We then proceed to define the notion of a $L^2$ solution, and derive the generalized energy identity. An outline of the proof together with some key ideas are given in Section~\ref{subsec:ideas}.

\subsection{Energy structure of the linearized operator} \label{subsec:prelim-sec2}
\subsubsection{Conjugated linearized operators}
Note that ${\bgtht} = f(x_2)$ is a (formal) solution {to \eqref{eq:ssqg}} for any (regular, decaying) profile $f$. Indeed,
	\begin{equation*}
	u \cdot \nb {\bgtht} = \rd_{x_2} \psi (x_2) \rd_{x_1} f(x_2) - \rd_{x_1} \psi(x_2)  \rd_{x_2} f(x_2) = 0.
	\end{equation*}
The direct linearization {of \eqref{eq:ssqg}} around $\bgtht$ is given as follows:
	\begin{align*}
	L_{\bgtht} \phi = \rd_{t} \phi - \nb^{\perp} \bgtht \cdot \nb \Gmm \phi + \nb^{\perp} \Gmm \bgtht \cdot \nb \phi = 0.
	\end{align*}
	Indeed, writing $\tht = \bgtht+ \phi$,
	\begin{align*}
	u &= \nb^{\perp} \Gmm (\bgtht + \phi), \\
	\rd_{t} (f + \phi) + u \cdot \nb (\bgtht + \phi)
	&= \rd_{t} \phi + \nb^{\perp} \Gmm \phi \cdot \nb \bgtht
	+ \nb^{\perp} \Gmm \bgtht \cdot \nb \phi + \nb^{\perp} \Gmm \phi \cdot \nb \phi \\
	&= \rd_{t} \phi - \nb \Gmm \phi \cdot \nb^{\perp} \bgtht
	+ \nb^{\perp} \Gmm \bgtht \cdot \nb \phi + \nb^{\perp} \Gmm \phi \cdot \nb \phi.
	\end{align*}
	Already we may observe that
	\begin{equation*}
	\tld{\calL}_{\bgtht} \phi = \rd_{t} \phi + \tld{P}_{\bgtht} \phi,
	\end{equation*}
	where the principal symbol of $\tld{P}_{\bgtht}$ is 
	\begin{equation*}
	{p_{\bgtht} :=} - i \nb^{\perp} \bgtht(x_1, x_2) \cdot \xi \Gmm(\xi) 
	\end{equation*}
	where $\xi = (\xi_{1}, \xi_{2})$, and we assumed that $\gmm(\xi) \to \infty$ as $\abs{\xi} \to \infty$. Note that $p_{\bgtht}$ is purely imaginary. However, as we will see soon, $L_{\bgtht}$ does not have a good energy structure; there is a problem with the sub-principal terms. For this reason, we will have to conjugate $L_{\bgtht}$ by $\Gmm^{\frac{1}{2}}$. 
	
As a motivation, let us first discuss the energy structure of $L_{\bgtht}$. Our computation will be formal.  We first write
	\begin{align*}
	\int L_{\bgtht} \phi \phi \, \ud x \ud y
	&= \frac{1}{2} \frac{\ud}{\ud t} \int \phi^{2} \, \ud x \ud y 
	- \int \left(\nb^{\perp} \bgtht \cdot \nb \Gmm \phi\right) \phi \, \ud x \ud y \\
	&\peq + \frac{1}{2} \int \nb^{\perp} \Gmm \bgtht \nb \phi^{2} \, \ud x \ud y
	\end{align*} 
	Integrating $\nb$ by parts, the last term vanishes since $\nb \cdot \nb^{\perp} = 0$. For the second term, we have the following chain of identities:
	\begin{align*}
	- \int \left(\nb^{\perp} \bgtht \cdot \nb \Gmm \phi\right) \phi \, \ud x \ud y
	& = - \frac{1}{2} \int \left(\nb^{\perp} \bgtht \cdot \nb \Gmm \phi\right) \phi \, \ud x \ud y
	+ \frac{1}{2} \int \phi \Gmm \left( \nb^{\perp} \bgtht \cdot \nb \phi\right) \, \ud x \ud y \\
	& = - \frac{1}{2} \int \phi [\nb^{\perp} \bgtht \cdot \nb, \Gmm] \phi \, \ud x \ud y.
	\end{align*}
	However, $[\nb^{\perp} \bgtht \cdot \nb, \Gmm]$ is symmetric and is not bounded in $L^{2}$, which is problematic. 
	
We wish to remove $[\nb^{\perp} \bgtht \cdot \nb, \Gmm]$, we need to conjugate the equation. Motivated by the desire to make the principal term exactly anti-symmetric (alternatively, via a formal, symbolic-calculus computation), we work with the variable $\varphi$ given by
	\begin{equation*}
	\phi = \Gmm^{-\frac{1}{2}} \varphi.
	\end{equation*}
	Accordingly, we define 
	\begin{equation*}
	\calL_{\bgtht} \varphi = \Gmm^{\frac{1}{2}} L_{\bgtht} (\Gmm^{-\frac{1}{2}} \varphi).
	\end{equation*}
	We compute
	\begin{align}\label{eq:lin-L2}
	\calL_{\bgtht} \varphi 
	= \rd_{t} \varphi 
	- \Gmm^{\frac{1}{2}} \nb^{\perp} \bgtht \cdot \nb \Gmm^{\frac{1}{2}} \varphi 
	+ \Gmm^{\frac{1}{2}} \nb^{\perp} \Gmm \bgtht \cdot \nb \Gmm^{-\frac{1}{2}} \varphi.
	\end{align}
	
We now investigate the energy structure of $\calL_{\bgtht}$; as before, we present only a formal computation. We begin with
	\begin{align*}
	\int \calL_{\bgtht} \varphi \varphi \, \ud x \ud y
	&= \frac{1}{2} \frac{\ud}{\ud t} \int \varphi^{2} \, \ud x \ud y
	- \int \left(\Gmm^{\frac{1}{2}} \nb^{\perp} \bgtht \cdot \nb \Gmm^{\frac{1}{2}} \varphi \right) \varphi \, \ud x \ud y \\
	&\peq
	+ \int \left(\Gmm^{\frac{1}{2}} \nb^{\perp} \Gmm \bgtht \cdot \nb \Gmm^{-\frac{1}{2}} \varphi\right) \varphi \, \ud x \ud y.
	\end{align*}
	By the symmetry of $\Gmm^{\frac{1}{2}}$ and anti-symmetry of $\nb^{\perp} \bgtht \cdot \nb$, the second term vanishes. However, the last term does not vanish (cf.~the computation for $L_{\bgtht} \phi$). For this term, we compute
	\begin{align*}
	&\int \left(\Gmm^{\frac{1}{2}} \nb^{\perp} \Gmm \bgtht \cdot \nb \Gmm^{-\frac{1}{2}} \varphi\right) \varphi \, \ud x \ud y \\
	&= \frac{1}{2} \int \left(\Gmm^{\frac{1}{2}} \nb^{\perp} \Gmm \bgtht \cdot \nb \Gmm^{-\frac{1}{2}} \varphi\right) \varphi \, \ud x \ud y
	- \frac{1}{2}\int \varphi \Gmm^{-\frac{1}{2}} \nb^{\perp} \Gmm \bgtht \cdot \nb \Gmm^{\frac{1}{2}}\varphi \, \ud x \ud y \\
	&= \frac{1}{2} \int \left(\nb^{\perp} \Gmm \bgtht \cdot \nb \varphi\right) \varphi \, \ud x \ud y
	+ \frac{1}{2} \int \left([\Gmm^{\frac{1}{2}}, \nb^{\perp} \Gmm \bgtht \cdot \nb] \Gmm^{-\frac{1}{2}} \varphi\right) \varphi \, \ud x \ud y \\
	&\peq
	- \frac{1}{2}\int \varphi \nb^{\perp} \Gmm \bgtht \cdot \nb \varphi \, \ud x \ud y 
	- \frac{1}{2}\int \varphi \Gmm^{-\frac{1}{2}} [\nb^{\perp} \Gmm \bgtht \cdot \nb, \Gmm^{\frac{1}{2}}] \varphi \, \ud x \ud y \\
	&= \frac{1}{2} \int \varphi \left([\Gmm^{\frac{1}{2}}, \nb^{\perp} \Gmm \bgtht \cdot \nb] \Gmm^{-\frac{1}{2}} + \Gmm^{-\frac{1}{2}} [\Gmm^{\frac{1}{2}}, \nb^{\perp} \Gmm \bgtht \cdot \nb]\right) \varphi \, \ud x \ud y.
	\end{align*}
	{An important observation is that the operator}
	\begin{equation}\label{eq:bdd-comm}
	[\Gmm^{\frac{1}{2}}, \nb^{\perp} \Gmm \bgtht \cdot \nb] \Gmm^{-\frac{1}{2}} + \Gmm^{-\frac{1}{2}} [\Gmm^{\frac{1}{2}}, \nb^{\perp} \Gmm \bgtht \cdot \nb]
	\end{equation}
	{is bounded on $L^{2}$; hence this term is acceptable. There are many ways to see this fact. For instance, one may utilize the symbolic calculus developed in Section~\ref{subsec:prelim} below, and observe that $[\Gmm^{\frac{1}{2}}, \nb^{\perp} \Gmm \bgtht \cdot \nb] \in Op(S(\gmm^{\frac{1}{2}}))$ and $\Gmm^{-\frac{1}{2}} \in Op(S(\gmm^{-\frac{1}{2}}))$, so \eqref{eq:bdd-comm} belongs to $Op(S(1))$ (i.e., is a quantization of a classical pseudodifferential operator of order $0$) and thus is bounded on $L^{2}$.} {For the convenience of the reader, we include an {independent} elementary proof of boundedness of this operator below in Proposition \ref{prop:bdd-comm}.}

Finally, we state the linearized operator for \eqref{eq:ssqg-diss}:
	\begin{equation} \label{eq:lin-L2-diss}
		\calL^{(\kpp)}_{\bgtht} \varphi 
		= \rd_{t} \varphi 
		- \Gmm^{\frac{1}{2}} \nb^{\perp} \bgtht \cdot \nb \Gmm^{\frac{1}{2}} \varphi 
		+ \Gmm^{\frac{1}{2}} \nb^{\perp} \Gmm \bgtht \cdot \nb \Gmm^{-\frac{1}{2}} \varphi + \kpp \Ups \varphi.
	\end{equation}
Note that $\calL^{(\kpp)}_{\bgtht} \Gmm^{\frac{1}{2}} \phi  = \Gmm^{\frac{1}{2}} L^{(\kpp)}_{\bgtht} \phi$.

{\begin{proposition}\label{prop:bdd-comm}
	Let $\mathring{\tht} \in H^{s}$ with $s$ sufficiently large depending on $\Gmm$. Then, we have \begin{equation*}
		\begin{split}
			\nrm{ [\Gmm^{\frac{1}{2}}, \nb^{\perp} \Gmm \bgtht \cdot \nb] \Gmm^{-\frac{1}{2}} \varphi }_{L^{2}} &\le  C(\Gmm,\mathring{\tht}) \nrm{\varphi}_{L^{2}} , \\
				\nrm{ \Gmm^{-\frac{1}{2}} [\Gmm^{\frac{1}{2}}, \nb^{\perp} \Gmm \bgtht \cdot \nb] \varphi }_{L^{2}} &\le  C(\Gmm,\mathring{\tht}) \nrm{\varphi}_{L^{2}}  
		\end{split}
	\end{equation*} for $\varphi \in L^{2}$. 
\end{proposition}

\begin{proof}
	We provide a proof just for the first inequality, since the argument for the other is similar. For simplicity, we set $F = \nb^{\perp} \Gmm \bgtht $. Then, we have that the Fourier transform of $[\Gmm^{\frac{1}{2}},F \cdot \nb] \Gmm^{-\frac{1}{2}} \varphi $ is given by \begin{equation*}
		\begin{split}
			I(\xi) := \int \gmm^{\frac12}(\xi) \hat{F}(\xi-\eta) i\eta \gmm^{-\frac12}(\eta) \hat{\varphi}(\eta) \, \ud\eta - \int \hat{F}(\xi-\eta) i\eta \hat{\varphi}(\eta)\,\ud\eta ,
		\end{split}
	\end{equation*} where $\hat{F}$ and $\hat{\varphi}$ are the Fourier transforms of $F$ and $\varphi$, respectively. 
	
	We proceed differently depending on the size of frequencies of $F$ and $\varphi$. First, when $|\xi-\eta| \ge |\eta|/10$, we simply take absolute values inside the integrals and use $|\xi|\lesssim |\xi-\eta|$, $ \gmm^{\frac12}(\xi) \lesssim_{\Xi_0} 1 +   \gmm^{\frac12}(\xi-\eta)$ (see \eqref{eq:slow-var}) and $\gmm^{-\frac12}(\eta) \lesssim 1$ to bound \begin{equation*}
		\begin{split}
			&\left|  \int_{ \{ |\xi-\eta| \ge |\eta|/10\}} \gmm^{\frac12}(\xi) \hat{F}(\xi-\eta) i\eta \gmm^{-\frac12}(\eta) \hat{\varphi}(\eta) \, \ud\eta - \int_{ \{ |\xi-\eta| \ge |\eta|/10\}} \hat{F}(\xi-\eta) i\eta \hat{\varphi}(\eta)\,\ud\eta    \right| \\
			&\quad \lesssim \left| \int_{ \{ |\xi-\eta| \ge |\eta|/10\}} (1+\gmm^{\frac12}(\xi-\eta) )|\xi-\eta||\hat{F}(\xi-\eta)| |\hat{\varphi}(\eta)| \, \ud\eta  \right|. 
		\end{split}
	\end{equation*} Next, when $|\xi-\eta| < |\eta|/10$, we combine the integrals to rewrite \begin{equation*}
	\begin{split}
		 \int_{ \{ |\xi-\eta| < |\eta|/10\}} \left(\gmm^{\frac12}(\xi)-\gmm^{\frac12}(\eta)\right) \hat{F}(\xi-\eta) i\eta \gmm^{-\frac12}(\eta) \hat{\varphi}(\eta) \, \ud\eta .
	\end{split}
	\end{equation*} We now take the absolute value inside the integral and use the mean value theorem with the condition $|\xi-\eta| < |\eta|/10$ to observe \begin{equation*}
	\begin{split}
		\left| \left(\gmm^{\frac12}(\xi)-\gmm^{\frac12}(\eta)\right) \eta \gmm^{-\frac12}(\eta) \right| \lesssim |\xi-\eta|.
	\end{split}
	\end{equation*} Therefore, combining two cases, we obtain altogether \begin{equation*}
	\begin{split}
		|I(\xi)| \lesssim  \left| \int  (1+\gmm^{\frac12}(\xi-\eta) )|\xi-\eta||\hat{F}(\xi-\eta)| |\hat{\varphi}(\eta)| \, \ud\eta  \right| 
	\end{split}
	\end{equation*} and applying Young's convolution inequality, we obtain  \begin{equation*}
	\begin{split}
		\nrm{I}_{L^{2}_{\xi}} \lesssim \nrm{ (1+\gmm^{\frac12}(\xi) )|\xi||\hat{F}(\xi)|  }_{L^{1}_{\xi}} \nrm{\hat{\varphi}}_{L^2_{\xi}}, 
	\end{split}
	\end{equation*} which finishes the proof by the Plancherel theorem.
\end{proof}

}

\subsubsection{Notion of a $L^2$ solution}
In the following, for simplicity, we consider \eqref{eq:lin-L2} as a special case of \eqref{eq:lin-L2-diss} obtained by taking $\kpp=0$.
Recall from above that a sufficiently smooth and decaying solution to  \eqref{eq:lin-L2-diss} satisfies \begin{equation*}
	\begin{split}
		\frac12 \frac{\ud}{\ud t} \nrm{ \varphi }_{L^2}^2 + \kpp \nrm{\Upsilon^{\frac12}\varphi}_{L^2}^2   + \frac12\brk{ \varphi, \left([\Gmm^{\frac{1}{2}}, \nb^{\perp} \Gmm \bgtht \cdot \nb] \Gmm^{-\frac{1}{2}} + \Gmm^{-\frac{1}{2}} [\Gmm^{\frac{1}{2}}, \nb^{\perp} \Gmm \bgtht \cdot \nb]\right)\varphi } = 0. 
	\end{split}
\end{equation*} Since \begin{equation*}
	\begin{split}
		\nrm{ \left([\Gmm^{\frac{1}{2}}, \nb^{\perp} \Gmm \bgtht \cdot \nb] \Gmm^{-\frac{1}{2}} + \Gmm^{-\frac{1}{2}} [\Gmm^{\frac{1}{2}}, \nb^{\perp} \Gmm \bgtht \cdot \nb]\right)\varphi  }_{L^2} \lesssim_{\Gmm,\bgtht} \nrm{\varphi}_{L^2}, 
	\end{split}
\end{equation*} we obtain from Gr\"onwall's inequality that \begin{equation}\label{eq:lin-energy-integrated}
	\begin{split}
		\nrm{\varphi(t)}_{L^{2}} \le \nrm{\varphi_0}_{L^2} \exp\left( C(\Gmm,\bgtht)t \right). 
	\end{split}
\end{equation} This motivates the following definition of a $L^2$-solution: 
\begin{definition}\label{def:L2-sol}
	Given some interval $I=[0,\tau]$, we say that $\varphi$ is a $L^2$-solution of  {$\calL^{(\kpp)}_{\bgtht} \varphi = 0$} if 
	\begin{itemize}
		\item $\varphi \in C_{w}(I;L^2) \cap L^2_{t}(I; \kpp^{-\frac12}\Upsilon^{-\frac12}L^2)$;
		\item $\varphi$ satisfies  {$\calL^{(\kpp)}_{\bgtht} \varphi = 0$} in the sense of distributions;
		\item $\varphi$ satisfies \eqref{eq:lin-energy-integrated}.
	\end{itemize} 
	We say that $\phi$ is an $\Gmm^{-\frac{1}{2}} L^{2}$-solution of $L^{(\kpp)}_{\bgtht} \phi = 0$ if $\varphi = \Gmm^{\frac{1}{2}} \phi$ is a $L^{2}$-solution of $\calL^{(\kpp)}_{\bgtht} \varphi= 0$. Moreover, we simply drop the requirement $\varphi\in L^2_{t}(I; \kpp^{-\frac12}\Upsilon^{-\frac12}L^2)$ in the inviscid case $\kpp=0$.
\end{definition}

Here, $C_{w}(I;L^2)$ is a subspace of $L^\infty_t(I;L^2)$ containing functions weakly continuous in time with values in $L^2$. The space $L^2_{t}(I; \kpp^{-\frac12}\Upsilon^{-\frac12}L^2)$ is defined by the norm $\int_{I} \nrm{\kpp^{\frac12}\Upsilon^{\frac12}\varphi(t)}_{L^2}^2 \,\ud t$. 

We have the following existence result: \begin{proposition}\label{prop:L2-sol-exist}
	Given any $\varphi_0 \in L^2$, there is at least one $L^2$ solution to  $\calL^{(\kpp)}_{\bgtht} \varphi = 0$  satisfying Definition \ref{def:L2-sol} for any $\kpp\ge0$.
\end{proposition}
We omit the proof, which is a simple application of the Aubin--Lions lemma; see \cite[Appendix A]{JO1} for details. 

\subsubsection{Generalized energy identity}

We now present the generalized energy identity, which is one of the main tools in the proof of linear and nonlinear illposedness. For the moment assume that $\varphi$ and $\psi$ are sufficiently smooth, decaying fast at infinity, and solve  {$\calL^{(\kpp)}_{\bgtht} \varphi = 0$} with errors $\err_\varphi$ and $\err_\psi$; that is, \begin{equation*}
	\begin{split}
		\calL^{(\kpp)}_{\bgtht} \varphi =\err_\varphi, \qquad \calL^{(\kpp)}_{\bgtht} \psi =\err_\psi . 
	\end{split}
\end{equation*} Then, we compute \begin{align*}
	\frac{\ud}{\ud t} \brk{\varphi,\psi} + 2\kpp\brk{ \Upsilon^{\frac12}\varphi,\Upsilon^{\frac12}\psi } 
	&= \brk{\err_\varphi,\psi} + \brk{\varphi,\err_\psi} + \brk{\Gmm^{\frac{1}{2}} \nb^{\perp} \bgtht \cdot \nb \Gmm^{\frac{1}{2}} \varphi  , \psi} + \brk{ \varphi  , \Gmm^{\frac{1}{2}} \nb^{\perp} \bgtht \cdot \nb \Gmm^{\frac{1}{2}}\psi} \\
	&\peq
	- \brk{\Gmm^{\frac{1}{2}} \nb^{\perp} \Gmm \bgtht \cdot \nb \Gmm^{-\frac{1}{2}} \varphi,\psi} - \brk{\varphi,\Gmm^{\frac{1}{2}} \nb^{\perp} \Gmm \bgtht \cdot \nb \Gmm^{-\frac{1}{2}}\psi} . 
\end{align*} Similarly as in the derivation of the energy identity, we have that the third and fourth terms on the right hand side cancel each other, using the anti-symmetry of $\nb^{\perp} \bgtht \cdot \nb$. Next, following the proof of the energy identity, the last two terms can be combined as follows: \begin{equation*}
	\begin{split}
		\brk{\Gmm^{\frac{1}{2}} \nb^{\perp} \Gmm \bgtht \cdot \nb \Gmm^{-\frac{1}{2}} \varphi,\psi} + \brk{\varphi,\Gmm^{\frac{1}{2}} \nb^{\perp} \Gmm \bgtht \cdot \nb \Gmm^{-\frac{1}{2}}\psi} = \brk{ \varphi, \left([\Gmm^{\frac{1}{2}}, \nb^{\perp} \Gmm \bgtht \cdot \nb] \Gmm^{-\frac{1}{2}} + \Gmm^{-\frac{1}{2}} [\Gmm^{\frac{1}{2}}, \nb^{\perp} \Gmm \bgtht \cdot \nb]\right)\psi}. 
	\end{split}
\end{equation*} Assuming for simplicity that $\err_\varphi=0$, we have the following generalized energy identity:
\begin{proposition}\label{prop:gei}
	Let $\varphi$ be a $L^2$ solution to  {$\calL^{(\kpp)}_{\bgtht} \varphi = 0$} on $I$ in the sense of Definition \ref{def:L2-sol}, and assume that $\psi$ satisfy $\calL^{(\kpp)}_{\bgtht} \psi =\err_\psi$ on $I$ with regularity \begin{equation*}
		\begin{split}
			\psi \in  C_{t}(I;L^2) \cap L^2_{t}(I; \kpp^{-\frac12}\Upsilon^{-\frac12}L^2) \cap L^2_{t}(I;H^1),\qquad \err_\psi \in L^1_{t}(I;L^2). 
		\end{split}
	\end{equation*} Then, we have \begin{equation}\label{eq:gei}
		\begin{split}
			\frac{\ud}{\ud t} \brk{\varphi,\psi} + 2\kpp\brk{ \Upsilon^{\frac12}\varphi,\Upsilon^{\frac12}\psi } 
			= \brk{\varphi,\err_\psi} +  \brk{ \varphi, \left([\Gmm^{\frac{1}{2}}, \nb^{\perp} \Gmm \bgtht \cdot \nb] \Gmm^{-\frac{1}{2}} + \Gmm^{-\frac{1}{2}} [\Gmm^{\frac{1}{2}}, \nb^{\perp} \Gmm \bgtht \cdot \nb]\right)\psi}
		\end{split}
	\end{equation} on $t\in I$. 
\end{proposition}
The proof follows from first mollifying $\varphi,\psi$ and repeating the computations above, which gives a generalized energy identity with some error terms arising from the mollification. Then it is not difficult to observe that the mollification errors vanish as the mollification parameter goes to zero, using the assumed regularity of $\varphi,\psi,\err_\psi$. We omit the straightforward details (cf. \cite[Proposition 2.3]{JO1}). 
Moreover, one may generalize the above proposition to the case when $\varphi$ is a $L^2$ solution with a $L^1_t(I;L^2)$ error, denoted by $\err_\varphi$. Then, we again have \eqref{eq:gei} with $\brk{\err_\varphi,\psi}$ added to the right hand side.

\subsection{Key ideas}\label{subsec:ideas}
\noindent \textit{Outline of the proof.} We give an overall picture of the illposedness proofs. Given \eqref{eq:lin-L2}, most of the work goes into construction of \emph{degenerating {wave packet}} solutions $\widetilde{\varphi}$ to \eqref{eq:lin-L2}. Basically, we would like to construct some approximate solutions which behave similarly with the solution of the toy model \eqref{eq:sol-toy}, given a quadratically degenerate shear steady state $\bgtht$. Some key properties that are required for the degenerating {wave packet}s solutions can be summarized as follows: \begin{itemize}
	\item Frequency localization: the initial data is sharply concentrated near some frequency $\lmb_{0}\gg1$; $\nrm{ \widetilde{\varphi}_0 }_{H^{s}} \sim \lmb_{0}^{s}$ for $s\ge0$.
	\item Error bound: for some interval $I=[0,{t_{0}}]$ with ${t_{0}}>0$, $\nrm{ \calL_{\bgtht}[\widetilde{\varphi}] }_{L^1_{t}(I;L^2)} \ll 1$. That is, $\widetilde{\varphi}$ is an approximate solution to \eqref{eq:lin-L2}. From the energy identity, it also follows that $\nrm{ \widetilde{\varphi} }_{L^\infty_t(I;L^2)} \lesssim \nrm{\widetilde{\varphi}_0}_{L^2} \lesssim 1$.
	\item Decay of negative Sobolev norms (degeneration estimate): we have a decomposition $\widetilde{\varphi} = \widetilde{\varphi}^{main} + \widetilde{\varphi}^{small}$ such that $\nrm{ \widetilde{\varphi}^{small}}_{L^\infty_t(I;L^2)}\ll 1$ and $\lmb_{0}^{s}\nrm{\widetilde{\varphi}^{main}(\tau)}_{H^{-s}} \ll 1$ for some $s>0$.\footnote{The decomposition $\widetilde{\varphi} = \widetilde{\varphi}^{main} + \widetilde{\varphi}^{small}$ is necessary since we would like to take advantage of the amplitude function compactly supported \textit{in space}. Therefore, unlike the toy model solution (which is compactly supported in the frequency side), there is some low-frequency part in $\widetilde{\varphi}$ which does not degenerate.}
	(The large parameter is $\lmb_{0}\gg1$.) 
\end{itemize}
Even in the case of dissipative linear equation  {$\calL^{(\kpp)}_{\bgtht} \varphi = 0$}, the above requirements remain unchanged, except that the profile $\bgtht$ becomes time-dependent (where $\rd_{t} f + \kpp \Ups f = 0$): we shall always incorporate the dissipation term $\kpp \Upsilon \tld{\varphi}$ as a part of the error term. Assuming for a moment that we are given $\widetilde{\varphi}$ satisfying the above, the rest of the illposedness proof follows a \emph{duality} or \emph{generalized energy argument} that originated from \cite{JO1} for Hall- and electron-MHD systems, which we now explain. For the moment we consider the linear illposedness result. To begin, simply take the initial data $\varphi_0 = \widetilde{\varphi}_0$ and apply the generalized energy identity \eqref{eq:gei} with $\psi=\widetilde{\varphi}$ and  $\varphi$ some $L^2$-solution to $\calL_{\bgtht} \varphi = 0$ on $I=[0,{t_{0}}]$ associated with $\varphi_0$. We then obtain \begin{equation*}
	\begin{split}
		\left|\frac{\ud}{\ud t} \brk{ \varphi, \widetilde{\varphi} } \right| \lesssim \nrm{\varphi(t)}_{L^2}(\nrm{\calL_{\bgtht}[\widetilde{\varphi}](t)}_{L^2}  + \nrm{\widetilde{\varphi}(t)}_{L^2}).
	\end{split}
\end{equation*} Using 
\begin{equation*}
	\begin{split}
		\nrm{\varphi(t)}_{L^2}\lesssim \nrm{\varphi_0}_{L^2},\qquad \nrm{\widetilde{\varphi}(t)}_{L^2}\lesssim \nrm{\widetilde{\varphi}_0}_{L^2},
	\end{split}
\end{equation*}
{which is valid for $0 \leq t \leq {t_{0}}$ with ${t_{0}} = O_{\bgtht}(1)$,} and then integrating in time, \begin{equation*}
\begin{split}
	\brk{ \varphi, \widetilde{\varphi}}(t) \ge \brk{ \varphi_0, \widetilde{\varphi}_0} - C\nrm{\varphi_0}_{L^2}(\nrm{\calL_{\bgtht}[\widetilde{\varphi}]}_{L^1([0,t];L^2)} + t\nrm{\widetilde{\varphi}_0}_{L^2}).
\end{split}
\end{equation*} Using the error bound, we can deduce for $t_0>0$ small that $\brk{ \varphi, \widetilde{\varphi}}(t_0)> \frac{9}{10} \brk{ \varphi_0, \widetilde{\varphi}_0}$ (say). Then, combining this estimate with \begin{equation*}
\begin{split}
	\brk{ \varphi, \widetilde{\varphi}}({t_{0}}) \le  \nrm{\varphi({t_{0}})}_{L^2} \nrm{\widetilde{\varphi}^{small}({t_{0}})}_{L^2} + \nrm{\varphi({t_{0}})}_{H^{s}} \nrm{\widetilde{\varphi}^{main}({t_{0}})}_{H^{-s}} 
\end{split}
\end{equation*} and $ \nrm{\varphi({t_{0}})}_{L^2} \nrm{\widetilde{\varphi}^{small}({t_{0}})}_{L^2} < \frac{1}{10}\brk{ \varphi_0, \widetilde{\varphi}_0}$ (say), we deduce that $\lmb_{0}^{-s}\nrm{\widetilde{\varphi}({t_{0}})}_{H^{s}}\gg1$. Here, the initial data $\varphi_0$ only needs to satisfy that $\brk{ \varphi_0, \widetilde{\varphi}_0} $ is comparable with $\nrm{ \varphi_0}_{L^2} \nrm{ \widetilde{\varphi}_0}_{L^2}$. 

To prove the nonlinear illposedness results, we assume towards a contradiction that a sufficiently smooth solution to the nonlinear equation exists, with initial data which is given by a small and smooth perturbation of some quadratically degenerate shear steady state $\bgtht$. Here, a key observation is that then the perturbation, after applying $\Gmm^{\frac12}$, is a solution of the linear operator with a nonlinear error, and that the nonlinear error in $L^2$ can be bounded in terms of the $L^2$-norm of the perturbation, under the assumption that the perturbation remains sufficiently smooth. This is responsible for the restriction on $s'$ in the nonlinear statements. 

In what follows, we explain a few key ideas that are involved in the construction of degenerating {wave packet}s. In \cite[Sections 1--2]{JO1},   a detailed introduction is given regarding the construction of degenerating {wave packet}s for the case of Hall- and electron-MHD systems. For this reason, here we emphasize on how the additional difficulties (over the Hall- and electron-MHD cases) are handled. 

\medskip

\noindent \textit{{Degenerating wave packet construction.}} There is a general recipe for construction of {wave packet} solutions, sometimes referred to as the WKB ansatz; one prepares the ansatz \begin{equation*}
	\begin{split}
		\widetilde{\varphi} = {\Re\left(\exp(i\Phi)a\right)},
	\end{split}
\end{equation*} and derive the equations that $\Phi$ and $a$ should satisfy, based on the given linear operator. The functions $\Phi$ and $a$ will be referred to as the \emph{phase} and the \emph{amplitude}, respectively, {and basic assumption in the analysis is that $a$ varies slower than $\Phi$ (i.e., $\rd_{t} a, \nb a$ are smaller than $\rd_{t} \Phi, \nb \Phi$ in a suitable way)}. In our problem, it is not only nontrivial to choose the correct evolution equations for the phase and amplitude but also to choose the initial data whose associated solution is well-behaved (in particular satisfying the required properties in the above) for a sufficiently long period of time. {When the linear operator is given by a differential operator, this process of extracting the equations for the phase and amplitude is rather straightforward. See classical works of Birkhoff \cite{Birkhoff} and Mizohata \cite{Mz} for the case of linear Schr\"odinger operator (possibly with lower order terms).\footnote{Furthermore, in the case of a differential operator, to deal with the variable coefficient of the principal term, one may work with a renormalized independent variable $x'$ which makes the coefficient into a constant. In turn, this allows us to propagate $\Phi(t,x')\sim \xi_{0}\cdot x'$, see below for the bicharacteristic ODE as well as the equation for $\Phi$. See the lecture notes of the second author \cite[Chap. 4]{Jeong} where this analysis is carried out in some detail for the case of AM-LQG.} }

Before we proceed to explain how such choices are made, let us give some heuristics for the evolution of {wave packet}s {(for a more detailed outline, see Section~\ref{subsec:outline} below).} Roughly, we would like to regard $\widetilde{\varphi}$ as a function which is well-localized both in the physical and Fourier variables, centered at some point $(X(t),\Xi(t))$ in the phase space. Writing our linearized operator as $\calL_{\bgtht} = \rd_t + ip_{\bgtht}$ modulo lower order terms and observing that $p_{\bgtht}$ is purely real, we expect that if $\widetilde{\varphi}$ at the initial time corresponds to $(X_0,\Xi_0)$, then for $t>0$, $\widetilde{\varphi}(t)$ corresponds to $(X(t),\Xi(t))$ which is given by the solution to the {\emph{Hamiltonian ODE system} associated with $p_{\bgtht}(t, x, \xi)$,}
\begin{equation} \label{eq:p-bgtht-bichar}
\left\{
\begin{aligned}
	\dot{X} &= \nb_{\xi}p_{\bgtht}(t,X,\Xi),\\
	\dot{\Xi} &= - \nb_{x} p_{\bgtht} (t,X,\Xi),
\end{aligned}
\right.
\end{equation} with initial data $(X_0,\Xi_0)$. We shall take $|\Xi_0|\gg1$ and $X_{0}$ sufficiently close to the degeneracy of $\bgtht$, and choose the sign of $\Xi_0$ in a way that $X(t)$ moves towards the degeneracy for $t>0$. 

To actually construct {wave packet}s following the above ODE trajectories, given $\calL_{\bgtht}$, we begin with writing \begin{equation}\label{eq:calL}
\begin{split}
	\calL_{\bgtht} = \rd_t + ip_{\bgtht} + s_{\bgtht} + r_{\bgtht},
\end{split}
\end{equation} where $s_{\bgtht}$ and $r_{\bgtht}$ are smoother than $ip_{\bgtht}$ by order at least one and two, respectively. Here, we take\begin{equation*}
\begin{split}
	p_{\bgtht} := -\nb^\perp\bgtht \cdot \xi\gmm(\xi) + \nb^\perp \Gmm\bgtht \cdot\xi
\end{split}
\end{equation*} and define the equation for $\Phi$ by \begin{equation*}
\begin{split}
	\rd_t \Phi + p_{\bgtht}(t,x,\nb\Phi) = 0. 
\end{split}
\end{equation*}
{This scalar first-order PDE is called the \emph{Hamilton--Jacobi equation} associated with the symbol $p_{\bgtht}(t, x, \xi)$. It is formally derived by collecting the ``worst terms'' in $\calL_{\bgtht} \widetilde{\varphi}$, i.e., those only involving $\rd_{t} \Phi$ or $\nb \Phi$, but \emph{not} $\rd_{t} a$ or $\nb a$. Indeed, this procedure can be directly implemented when $\calL_{\bgtht}$ is a differential operator, and in the present pseudodifferential case, we rely on symbolic calculus (in particular, Proposition~\ref{prop:conj-exp}) to extract such terms. We remark that the Hamilton--Jacobi equation provides a direct link between the wave packet construction and the Hamiltonian ODE \eqref{eq:p-bgtht-bichar} for $(X, \Xi)(t)$, in the sense that \eqref{eq:p-bgtht-bichar} is nothing but the bicharacteristic ODEs associated with the Hamilton--Jacobi equation (viewed as a scalar first-order PDE).}

{When $\calL_{\bgtht}$ is a differential operator, $i p_{\bgtht}$ can be simply chosen to be the principal part, since the remainder will automatically be of at least one order lower. However, in our case, we must exercise care in selecting these parts. Indeed, observe} that we have taken the main part of the lower order operator $\Gmm^{\frac12}\nb^\perp\bgtht\cdot\nb\Gmm^{-\frac12}$ into $p_{\bgtht}$. This is inevitable, as these two terms in $p_{\bgtht}$ are barely distinguishable when $\gmm$ is only slightly singular, and in general we do not have any control over the coefficient $\nb^\perp\Gmm\bgtht$. {As a result of our choice of $p_{\bgtht}$,} we are forced to design the initial data $\Phi_{0}$ which takes {the lower order term $\nb^\perp \Gmm\bgtht \cdot\xi$}  into account. Again, the goal is to propagate that $\nb\Phi(t,x) \simeq \Xi(t)$ for $x \simeq X(t)$ for a sufficiently long interval of time, within which degeneration occurs. {We will return to this point below, under ``{\it Estimate for the derivatives of the phase and amplitude}.''}

Returning to the expression \eqref{eq:calL}, note that the difference $\calL_{\bgtht} - (\rd_t + ip_{\bgtht})$ consists of a few commutators. The operator $s_{\bgtht}$ is defined by taking the principal terms of the commutators which are given by Poisson brackets, and then we simply write $r_{\bgtht}$ for the remainder {(see Proposition~\ref{prop:L-tht0-decomp-f} for the precise expressions)}.
{We then declare that $a$ solves the linear transport equation} \begin{equation*}
	\begin{split}
		\rd_t a + \nb_{\xi} p_{\bgtht}(x,\nb\Phi) \cdot \nb a + \left( \frac{1}{2}\nb^2_{\xi} p_{\bgtht}(x,\nb\Phi): \nb^2\Phi + s_{\bgtht}(x,\nb\Phi) \right) a = 0.
	\end{split}
\end{equation*} with the initial data $a_0$ taken to be a smooth bump function supported in a small neighborhood of $X_{0}$. {Heuristically, modulo the zeroth order term, this transport equation is derived by collecting the ``second worst terms'' in $\calL_{\bgtht} \widetilde{\varphi}$. Indeed, when $\calL_{\bgtht}$ is a differential operator, the principal part $\rd_{t} a + \nb_{\xi} p_{\bgtht}(x, \nb \Tht) \cdot \nb a$ consists precisely of those terms in $\calL_{\bgtht} \widetilde{\varphi}$ involving only one factor of $\rd_{t} a$ or $\nb a$; in the present pseudodifferential case, we rely on symbolic calculus (Proposition~\ref{prop:conj-exp}) to extract such terms. Observe that \eqref{eq:p-bgtht-bichar} is also the bicharacteristic ODEs for the this linear transport equation. Finally, we note note that the last (zeroth order) term in the preceding transport is chosen precisely so that $\nrm{a(t, \cdot)}_{L^{2}}$ is conserved; see Proposition~\ref{prop:trans-a} below.}

After solving the equations for $\Phi$ and $a$ to find $\widetilde{\varphi}={\Re(e^{i\bfPhi}a)}$, we then need to estimate Sobolev norms of both $\widetilde{\varphi}$ and the error $\calL_{\bgtht}(\widetilde{\varphi})$. To begin with, compared with $p_{\bgtht}$, the operator  $r_{\bgtht}$ is smoother by order 2, and therefore we expect $r_{\bgtht}(\widetilde{\varphi})$ to be small. However, in the error there is also a contribution from $ip_{\bgtht}$ and $s_{\bgtht}$, which occurs since $e^{i\Phi}$ is not completely localized in the frequency space. To see this contribution, we introduce the conjugation operator ${}^{(\Phi)}p_{\bgtht}(x,D)$ defined by \begin{equation*}
	\begin{split}
		e^{-i\Phi} p_{\bgtht}(x,D) (e^{i\Phi} a ) =: {}^{(\Phi)}p_{\bgtht}(x,D) a.
	\end{split}
\end{equation*} We can then formally expand {(more precisely, see Proposition~\ref{prop:conj-exp})} \ \begin{equation*}
\begin{split}
	{}^{(\Phi)}r_{p}:= {}^{(\Phi)}p_{\bgtht}(x,\xi) - \left(p_{\bgtht}(x,\nb\Phi) + \nb_{\xi}p_{\bgtht}(x,\nb\Phi) {\nb \Phi} -\frac{i}{2} \nb_{\xi}^2 p_{\bgtht}(x,\nb\Phi) \nb^2\Phi\right) . 
\end{split}
\end{equation*} Note that the last two terms are taken into account in the equation for $a$ above. The symbol for ${}^{(\Phi)}r_{p}$ can be expressed explicitly by an oscillatory integral, and we need a sharp bound for its operator norm into $L^{2}$. It turns out that both for the oscillatory integral bound and the degeneration estimate, {the key step is to find the \emph{inverse wave packet scale} $\mu(t)$ such that $\mu(t) \ll \lmb(t)$, $|\rd^{k+1}\Phi(t)| \aleq \mu^{k}(t)\lmb(t)$ (on the support of $a(t)$), and $|\rd^{k}a(t)|\aleq \mu^{k}(t)$, where $\lmb(t) \aeq |\rd\Phi(t)|  (\aeq \abs{\Xi(t)})$ is the overall \emph{wave packet frequency}.} Below, let us briefly explain how it is done. 

\medskip

\noindent \textit{Estimate for the derivatives of the phase and amplitude.} 
A basic difficulty in controlling the high derivatives of $\rd \Phi$ is that $\rd^{2} \Phi$ obeys a Ricatti-type ODE along characteristics, which tends to blow up in finite time (indeed, such a blow up corresponds to a focal point along a bicharacteristic). Therefore, some care is needed to control $\rd^{2} \Phi$ on a sufficiently long time interval, so as to see the effect of degeneration. {Equivalently, we need to carefully make a suitable choice of the initial data $\Phi(t=0)$ for $\Phi$.}

When the shear profile $\bgtht = f(x_2)$ is time-independent and degenerate at $x_2=0$, it turns out that the following solution given by \emph{separation of variables} avoids the basic difficulty and achieves the required bounds:\begin{equation*}
	\begin{split}
		\bfPhi(t,x) = E\lmb_0 + \lmb_{0}x_1 + \rd_{x_2}^{-1} \left( \gmm_{\lmb_{0}}^{-1}\left( \frac{E+\Gmm f'}{-f'} \right) \right),
	\end{split}
\end{equation*} where \begin{equation*}
\begin{split}
	E =  -f'(\overline{x})\gmm_{\lmb_{0}}((1-\tfrac{1}{2} \eps)\lmb_0) - \Gmm f'(\overline{x})
\end{split}
\end{equation*} with some $\overline{x}>0$ and $\eps>0$ sufficiently small, $\gmm_{\lmb_{0}}(\cdot) := \gmm(\lmb_0,\cdot)$,  and $\gmm_{\lmb_{0}}^{-1}$ is the inverse of $\gmm_{\lmb_{0}}$. As in the solution of the toy model above, we have separated out the dependence in $x_1$, which does not change in time due to the $x_1$-independence of the linear operator. While the formula looks somewhat complicated, {when $p_{\bgtht}$ is given by a differential operator}, this form corresponds to the linear ansatz in the renormalized coordinates {in which the coefficient of the principal symbol of $p_{\bgtht}$ is non-degenerate}, which seems to be a {natural} choice. This choice of initial data is especially important in the case of time-dependent background $\bgtht = f(t,x_{2})$. {As discussed before,} already in the case $k = 1$, the equation for $\rd_{x_2}^2\Phi$ becomes a Riccati-type ODE along characteristics, which could in principle grow much faster than $|\rd_{x}\Phi|^2 \aeq \lmb(t)^{2}$. On the other hand, one can check that our choice of $\Phi(t=0)$ makes the variable \begin{equation*}
\begin{split}
	h(t,X_2(t)) := \frac{\rd_{\xi_2} p_{\bgtht}(t, X_2(t), \Xi_2(t))}{\rd_{x_2} p_{\bgtht}(t, X_2(t), \Xi_2(t))} \rd_{x_2}^{2} \Phi (t, X_2(t)) + 1
\end{split}
\end{equation*} vanish at $t = 0$. (Here we are neglecting the dependence on $x_1$ and $\xi_1$.) Moreover, one can check that \emph{the quantity $h$ satisfies a remarkably simple equation}, which allows to propagate $|h(t)|\ll 1$ for {a suitably long time} and in turn gives a sharp bound for $\rd_{x_2}^{2} \Phi (t, X_2(t))$. {Moreover, we obtain a sharp control on the instantaneous expansion of characteristics
\begin{equation*}
	\left. \rd_{x_{2}} (\rd_{\xi_{2}} p(t, x_{2}, \rd_{x_{2}} \Phi(t, x_{2})))\right|_{t, x_{2}=X_{2}(t)} = - \left(\frac{\gmm_{\lmb_{0}}'(\Xi_{2}(t))}{\gmm_{\lmb_{0}}(\Xi_{2}(t))} - \frac{\gmm_{\lmb_{0}}''(\Xi_{2}(t))}{\gmm_{\lmb_{0}}'(\Xi_{2}(t))} \right)\dot{\Xi}_{2} + \cdots =: - A(t) + \cdots.
\end{equation*}
The wave packet scale $\mu^{-1}$ is then nothing but the product of the initial scale $\Dlt x_{0}$ and the integrated expansion factor, i.e.,
\begin{equation*}
	\mu^{-1} \aeq \Dlt x_{0} \exp\left(-\int_{0}^{t} A(t') \ud t' \right) \aeq \Dlt x_{0} \frac{\gmm_{\lmb_{0}}(\lmb_{0})}{\gmm_{\lmb_{0}}(\lmb(t))} \frac{\gmm_{\lmb_{0}}'(\lmb(t))}{\gmm_{\lmb_{0}}'(\lmb_{0})}
	\qquad \hbox{(up to a small power of $\lmb_{0}$).}
\end{equation*}
}
{The actual proof of the bound $\rd_{x_{2}}^{k+1} \Phi \aleq \mu(t)^{k} \lmb(t)$ proceeds by first propagating sharp bounds for $\rd_{x_2}^{k-1}h(t)$, and then converting it to that for $\rd_{x_2}^{k+1}\Phi$. It turns out to be essential to take advantage of the cancellations embedded in the transformation from $\rd_{x_{2}}^{2} \Phi$ to $h$, which are very difficult to see when one works directly with the equations for $\rd_{x_2}^{k+1}\Phi$.} Given the bounds for $\rd_{x_2}^{k+1}\Phi(t)$, it is relatively straightforward to obtain the bounds for $\rd_{x_2}^{k}a(t)$.

\section{Algebraic preliminaries for degenerate wave packet construction} \label{sec:psdo}
Our goal in Sections~\ref{sec:psdo}--\ref{sec:wavepacket} is to construct an approximate solution with initial frequency $O(\lmb_{0})$ to the equation $\calL_{\bgtht} \varphi = 0$ that is valid for times $[0, \frac{1}{1-\eps(\lmb_{0})} t_{f}(\tau_{M})]$, {where $\bgtht$, $M$ and $\tau_{M}$ obey the hypotheses of either Theorem~\ref{thm:norm-growth} (steady state case) or Theorem~\ref{thm:norm-growth-diss} ($f$ is time-dependent but even) and $\eps(\lmb_{0})$ is a small parameter that will be fixed in the construction (see Section~\ref{subsec:HJE-prelim}).} We look for an approximate solution of the form
\begin{equation*}
\tld{\varphi} = \Re \left( a(t, x) e^{i \bfPhi(t, x)} \right), 
\end{equation*}
where the \emph{amplitude} $a(t, x)$ is ``slowly varying'' compared to the real-valued \emph{phase} $\bfPhi$, such that $\calL_{\bgtht} \tld{\varphi}$ is sufficiently small. Anticipating the spatial degeneration property of $\tld{\varphi}$ (which comes with the frequency growth property that we want), we will refer to such an object $\tld{\varphi}$ as a \emph{degenerating wave packet} adapted to $\calL_{\bgtht}$.

In order to construct a degenerating wave packet $\tld{\varphi}$ and, more importantly, to bound the error $\calL_{\bgtht} \tld{\varphi}$, we need to develop tools for computing the action of the multiplier $\Gmm$ on $a e^{i \bfPhi}$. Section~\ref{subsec:prelim} is devoted to accomplishing this task. Then in Section~\ref{subsec:outline}, we use the algebraic identities derived in Section~\ref{subsec:prelim} to specify the construction of the phase $\bfPhi$ and the amplitude $a$ of a degenerate wave packet. The tasks of verifying the degeneration property of $\tld{\varphi}$ and bounding the error $\calL_{\bgtht} \tld{\varphi}$ are deferred to Section~\ref{sec:wavepacket}, after further necessary tools are developed (see also the end of Section~\ref{subsec:outline}).

\subsection{Some symbolic calculus}\label{subsec:prelim}
In this subsection, it will be convenient to generalize the set-up a bit and work with classical pseudo-differential operators with generalized order on $\bbR^{d}$. 

Let $m = m(\xi)$ be a smooth, even, positive, and slowly varying {(i.e., $\abs{\rd_{\xi}^{I} m(\xi)} \aleq_{I} \brk{\xi}^{-\abs{I}} m(\xi)$ for any multi-index $I$ and $\xi \in \bbR^{2}$)} symbol on $\bbR^{d}$. As introduced in Section~\ref{subsec:results}, $\gmm(\xi)$ is an example of such a symbol on $\bbR^{2}$ {and so is $\gmm^{-1}(\xi)$}. As in the case of Fourier multipliers, we say that \emph{a symbol $p = p(x, \xi)$ belongs to the class $S(m)$} if
\begin{equation} \label{eq:symb-m}
\abs{\rd_{x}^{J}  \rd_{\xi}^{I} p(x, \xi)} \aleq_{I, J} \brk{\xi}^{-\abs{I}} m(\xi) \quad \hbox{ for every pair of multi-indices $I$, $J$.}
\end{equation}
When $m = \brk{\xi}^{s}$, $S(m)$ coincides with the class of classical symbols of order $s$. Nonstandard examples of $m$ include $m(\xi) = \log^{s} (10 + \abs{\xi})$, $m(\xi) = \log^{s} (10 + \log (10 + \abs{\xi}))$ etc.

Given $p(x, \xi) \in S(m)$, we define its \emph{left quantization} to be the operator
\begin{equation} \label{eq:psdo-op}
{p(x, D_{1}, D_{2}) u = \int p(x, \xi_{1}, \xi_{2}) u(y) e^{i (\xi_{1} (x^{1} - y^{1}) + \xi_{2} (x^{2} - y^{2}))} \, \ud y^{1} \ud y^{2} \frac{\ud \xi_{1} \ud \xi_{2}}{(2 \pi)^{2}}.}
\end{equation}
{Accordingly, we use the notation $D_{j}$ for the quantization of $\xi_{j}$, or equivalently, the operator $i^{-1} \rd_{x^{j}}$. We will also often use the shorthand $D = (D_{1}, D_{2})$ and simply write $p(x, D) = p(x, D_{1}, D_{2})$.} We denote by $Op(S(m))$ the space of all left quantizations of symbols in $S(m)$. 

We begin by stating the basic symbolic calculus for composition of $Op(S(m))$-operators.
\begin{lemma} [Symbolic calculus] \label{lem:symb-calc}
	Let $p \in S(m_{p})$ and $q \in S(m_{q})$. Then the composition $p(x, D) q(x, D)$ belongs to $Op(S(m_{p} m_{q}))$, and for every $N \in \bbN$, its symbol $p \circ q(x, \xi)$ obeys
	\begin{equation} \label{eq:p-q-comp}
	p \circ q(x, \xi) - \sum_{\alp : \abs{\alp} \leq N-1} \frac{1}{\alp! i^{\abs{\alp}}} \rd_{\xi}^{\alp} p(x, \xi) \rd_{x}^{\alp} q(x, \xi) \in S(\brk{\xi}^{-N} m_{p} m_{q}).
	\end{equation}
	In particular,
	\begin{equation} \label{eq:p-q-comm}
	p \circ q(x, \xi) - q \circ p(x, \xi) - i^{-1} \{p, q\}(x, \xi) \in S(\brk{\xi}^{-2} m_{p}  m_{q}),
	\end{equation}
	where $\{p, q\} = \sum_{j} \left(\rd_{\xi_{j}} p \rd_{x_{j}} q - \rd_{x_{j}} p \rd_{\xi_{j}} q\right)$ is the Poisson bracket.
\end{lemma}
\begin{proof}
	The proof of this lemma is analogous to the standard case $m_{p}(\xi) = \brk{\xi}^{s_{p}}$, $m_{q}(\xi) = \brk{\xi}^{s_{q}}$, so we shall only sketch the main points.
	Formally, the symbol $p \circ q(x, \xi)$ for the composition $p(x, D) q(x, D)$ is given by
	\begin{equation*}
	p \circ q(x, \xi) = \iint p(x, \eta) q(y, \xi) e^{- i (x - y) \cdot (\xi - \eta)} \frac{\ud \eta}{(2 \pi)^{d}} \ud y.
	\end{equation*}
	Then $p \circ q \in S(m_{p} m_{q})$ and the expansion \eqref{eq:p-q-comp} then follows, as usual, by Taylor expansion of $p(x, \eta)$ around $\eta = \xi$, integration by parts in $y$ and \eqref{eq:symb-m}. Moreover, \eqref{eq:p-q-comm} follows from \eqref{eq:p-q-comp}. \qedhere
\end{proof}

The following conjugation result will be a starting point for our degenerating wave packet construction.
\begin{lemma}[Conjugation by $e^{i \bfPhi}$] \label{lem:conj}
	Let $p \in S(m)$ and $\bfPhi \in C^{\infty}$. Then we have
	\begin{align*}
	e^{-i \bfPhi} p(x, D) (a e^{i \bfPhi})
	= {}^{(\bfPhi)} p(x, D) a,
	\end{align*}
	where
	\begin{equation} \label{eq:p-Phi}
	{}^{(\bfPhi)} p(x, \xi)
	= \iint p(x, \eta) e^{i (\bfPhi(y) - \bfPhi(x) + (\eta- \xi) \cdot (x-y))} \, \frac{\ud \eta}{(2 \pi)^{d}} \, \ud y.
	\end{equation}
\end{lemma}
Observe carefully that we have not placed ${}^{(\bfPhi)} p(x, \xi)$ in any standard symbol class. In Propositions~\ref{prop:conj-exp} and \ref{prop:conj-err}, we will obtain an expansion and a bound for the operator ${}^{(\bfPhi)} p(x, \xi)$ that is adapted to the specific scenario we are interested in. 
\begin{proof}
	Indeed,
	\begin{align*}
	e^{-i \bfPhi} p(x, D) (a e^{i \bfPhi})
	&= \iint p(x, \xi) a(y) e^{i (\bfPhi(y) - \bfPhi(x) + \xi \cdot (x-y))} \, \frac{\ud \xi}{(2 \pi)^{d}}  \ud y \\
	&= \iiint p(x, \xi) \hat{a}(\eta) e^{i (\bfPhi(y) - \bfPhi(x) + \xi \cdot (x-y) + \eta \cdot y)} \, \frac{\ud \eta}{(2 \pi)^{d}} \frac{\ud \xi}{(2 \pi)^{d}}  \ud y \\
	&= \int \left(\iint p(x, \xi)  e^{i (\bfPhi(y) - \bfPhi(x) + (\xi-\eta) \cdot (x-y)}  \frac{\ud \xi}{(2 \pi)^{d}}  \ud y \right) \hat{a}(\eta) e^{i x \cdot \eta} \, \frac{\ud \eta}{(2 \pi)^{d}},
	\end{align*}
	so switching the variables $\xi$ and $\eta$, we obtain the desired claim.
\end{proof}

For our construction, we would like to expand ${}^{(\bfPhi)} p$ under the assumption that $\xi$ and the characteristic frequency of $\rd_{x} \bfPhi$ (bounded by $\mu$ in what follows) are smaller than the typical magnitude of $\rd_{x} \bfPhi$ (denoted by $\lmb$ in what follows). To begin with, the stationary set for the phase of the oscillatory integral on the RHS of \eqref{eq:p-Phi} is 
\begin{align*}
\rd_{y} \bfPhi(y) - (\eta - \xi) = 0, \quad
x - y = 0,
\end{align*}
or equivalently, $\eta = \xi + \rd_{y} \bfPhi(y)$ and $y = x$. Our assumption leads us to also expand in $\xi$ about $\xi = 0$. Following such a route, we are led to the following formulae for the expansion and the remainder.
\begin{proposition} [Formal expansion of ${}^{(\bfPhi)} p$]\label{prop:conj-exp}
	Let $p \in S(m)$ and $\bfPhi \in C^{\infty}$. The symbol ${}^{(\bfPhi)} p(x, \xi)$ admits expansions of the form
	\begin{equation} \label{eq:conj-exp}
	\begin{aligned}
	{}^{(\bfPhi)} p(x, \xi)
	&= p(x, \rd_{x} \bfPhi)
	+ {}^{(\bfPhi)} r_{p, -1}(x, \xi) \\
	&= p(x, \rd_{x} \bfPhi) + \sum_{j} \xi_{j} \rd_{\xi_{j}} p(x, \rd_{x} \bfPhi(x)) 
	- \frac{i}{2} \sum_{j, k} \rd_{\xi_{j}} \rd_{\xi_{k}} p(x, \rd_{x} \bfPhi(x)) \rd_{j} \rd_{k} \bfPhi(x) \\
	& \peq + {}^{(\bfPhi)} r_{p, -2}(x, \xi),
	\end{aligned}
	\end{equation}
	where
	\begin{equation} \label{eq:conj-rem-1} 
	\begin{aligned} 
	& {}^{(\bfPhi)} r_{p, -1}(x, \xi) \\ 
	&= \left( \int_{0}^{1} \iint\rd_{\xi_{j}} p(x, \sgm \eta + (1-\sgm) \rd_{y} \bfPhi(y))  e^{i (\bfPhi(y) - \bfPhi(x) + (\eta- \xi) \cdot (x-y))} \, \frac{\ud \eta}{(2 \pi)^{d}} \, \ud y \ud \sgm \right)   \xi_{j}   \\
	&\peq - i \int_{0}^{1} (1-\sgm) \iint \rd_{\xi_{j}} \rd_{\xi_{k}} p(x, \sgm \eta + (1-\sgm) \rd_{y} \bfPhi(y)) \rd_{j} \rd_{k} \bfPhi (y) \\
	&\phantom{\peq - i \int_{0}^{1} (1-\sgm) \iint} \times e^{i (\bfPhi(y) - \bfPhi(x) + (\eta- \xi) \cdot (x-y))} \, \frac{\ud \eta}{(2 \pi)^{d}} \, \ud y \ud \sgm, 
	\end{aligned}
	\end{equation}
	and
	\begin{equation} \label{eq:conj-rem-2} 
	\begin{aligned} 
	&{}^{(\bfPhi)} r_{p, -2}(x, \xi) \\
	&= \left( \int_{0}^{1} (1-\sgm) \iint \rd_{\xi_{j}}\rd_{\xi_{k}} p(x, \sgm \eta + (1-\sgm) \rd_{y} \bfPhi(y)) e^{i (\bfPhi(y) - \bfPhi(x) + (\eta- \xi) \cdot (x-y))} \, \frac{\ud \eta}{(2 \pi)^{d}} \, \ud y \ud \sgm \right) \xi_{j} \xi_{k}  \\
	&\peq - i \left( \int_{0}^{1} (1-\sgm)\left(\frac{3}{2} - \sgm\right) \iint \rd_{\xi_{j}} \rd_{\xi_{k}} \rd_{\xi_{\ell}} p(x, \sgm \eta + (1-\sgm) \rd_{y} \bfPhi(y)) \rd_{k} \rd_{\ell} \bfPhi(y) \right.  \\
	&\left. \phantom{\peq - i \int_{0}^{1} (1-\sgm)(\frac{3}{2} - \sgm) \iint}
	\times e^{i (\bfPhi(y) - \bfPhi(x) + (\eta- \xi) \cdot (x-y))} \, \frac{\ud \eta}{(2 \pi)^{d}} \, \ud y \ud \sgm \right) \xi_{j} \\
	&\peq - \frac{1}{2} \int_{0}^{1} (1-\sgm)^{2} \iint \rd_{\xi_{j}} \rd_{\xi_{k}} \rd_{\xi_{\ell}} p(x, \sgm \eta + (1-\sgm) \rd_{y} \bfPhi(y)) \rd_{j} \rd_{k} \rd_{\ell} \bfPhi(y) \\
	&\phantom{\peq - \frac{1}{2} \int_{0}^{1} (1-\sgm)^{2} \iint}
	\times e^{i (\bfPhi(y) - \bfPhi(x) + (\eta- \xi) \cdot (x-y))} \, \frac{\ud \eta}{(2 \pi)^{d}} \, \ud y \ud \sgm  \\
	&\peq - \frac{1}{2} \int_{0}^{1} (1-\sgm)^{2} \iint \rd_{\xi_{j}} \rd_{\xi_{k}} \rd_{\xi_{\ell}} \rd_{\xi_{m}} p(x, \sgm \eta + (1-\sgm) \rd_{y} \bfPhi(y)) \rd_{j} \rd_{m} \bfPhi(y) \rd_{k} \rd_{\ell} \bfPhi(y)  \\
	&\phantom{\peq - \frac{1}{2} \int_{0}^{1} (1-\sgm)^{2} \iint } 
	\times e^{i (\bfPhi(y) - \bfPhi(x) + (\eta- \xi) \cdot (x-y))} \, \frac{\ud \eta}{(2 \pi)^{d}} \, \ud y \ud \sgm,  
	\end{aligned}
	\end{equation}
	where the repeated indices in \eqref{eq:conj-rem-1} and \eqref{eq:conj-rem-2} are implicitly summed.
\end{proposition}
{The formal expansion of ${}^{(\bfPhi)} p$ is standard; see, for instance, \cite[Chapter~VII]{Tay}. The point of Proposition~\ref{prop:conj-exp} is the explicit formulae for the remainder symbols ${}^{(\bfPhi)} r_{p, -j}$, which we shall analyze in Section~\ref{sec:conj-err} with the specific estimates on $\rd_{x}^{k} \rd_{x} \bfPhi$ in our problem, to be proved in Section~\ref{sec:HJE}.}
\begin{proof}
	To ease the notation, in what follows, we work with the convention that the repeated indices are summed. Moreover, we introduce the $n$-th order remainder symbol
	\begin{equation} \label{eq:taylor-rem}
	r^{(n) j_{1} \ldots j_{n}}[p(x, \cdot)](\xi_{1}, \xi_{0}) = \int_{0}^{1} \rd_{\xi_{j_{1}}} \cdots \rd_{\xi_{j_{n}}} p(x, \sgm \xi_{1} + (1-\sgm) \xi_{0}) n (1-\sgm)^{n-1} \, \ud \sgm.
	\end{equation}
	Then the $(n-1)$-th Taylor expansion of $p(x, \xi)$ in $\xi$ about $\xi_{0}$ takes the form
	\begin{equation} \label{eq:p-taylor}
	\begin{split}
	p(x, \xi_{1}) & = \sum_{m=0}^{n} \frac{1}{m!} \rd_{\xi_{j_{1}}} \cdots \rd_{\xi_{j_{m}}} p(x, \xi_{0}) (\xi_1-\xi_0)^{j_1}\cdots (\xi_1-\xi_0)^{j_m} \\
	&\qquad + \frac{1}{n!} r^{(n) j_{1} \ldots j_{n}}[p(x, \cdot)](\xi_{1}, \xi_{0}) (\xi_{1} - \xi_{0})^{j_{1}} \cdots (\xi_{1} - \xi_{0})^{j_{n}}.
	\end{split}
	\end{equation}
	Also note the useful recursive identity (which is simply an integration by parts in $\sgm$)
	\begin{equation} \label{eq:r-recur}
	r^{(n) j_{1} \ldots j_{n}}[p(x, \cdot)](\xi_{1}, \xi_{0})
	= \rd_{\xi_{j_{1}}} \cdots \rd_{\xi_{j_{n}}} p(x, \xi_{0})
	+ \sum_{ j_{n+1}} \frac{1}{n+1} r^{(n+1) j_{1} \ldots j_{n} j_{n+1}}[p(x, \cdot)](\xi_{1}, \xi_{0}) (\xi_1-\xi_0)^{j_{n+1}}.
	\end{equation}
	
	In view of the stationary set of the phase and the expectation that $\abs{\xi} \ll \abs{\rd_{x} \bfPhi}$, we wish to expand $p(x, \eta)$ in $\eta$ about $\eta = \xi + \rd_{y} \bfPhi(y)$: 
	\begin{align*}
	{}^{(\bfPhi)} p(x, \xi)
	&=
	\iint p(x, \rd_{y} \bfPhi(y)) e^{i (\bfPhi(y) - \bfPhi(x) + (\eta- \xi) \cdot (x-y))} \, \frac{\ud \eta}{(2 \pi)^{d}} \, \ud y \\
	&\peq 
	+ \iint r^{(1) j}[p(x, \cdot)] (\eta, \rd_{y} \bfPhi(y)) (\eta_{j} - \rd_{j} \bfPhi(y)) e^{i (\bfPhi(y) - \bfPhi(x) + (\eta- \xi) \cdot (x-y))} \, \frac{\ud \eta}{(2 \pi)^{d}} \, \ud y.
	\end{align*}
	Note the identity
	\begin{align}
	&(\eta_{j} - \rd_{j} \bfPhi(y)) e^{i (\bfPhi(y) - \bfPhi(x) + (\eta- \xi) \cdot (x-y))} \notag \\
	&= (\xi_{j} + \eta_{j} - \xi_{j} - \rd_{j} \bfPhi(y)) e^{i (\bfPhi(y) - \bfPhi(x) + (\eta- \xi) \cdot (x-y))} \notag\\
	&= \xi_{j} e^{i (\bfPhi(y) - \bfPhi(x) + (\eta- \xi) \cdot (x-y))}
	+ i \rd_{y^{j}} e^{i (\bfPhi(y) - \bfPhi(x) + (\eta- \xi) \cdot (x-y))}. \label{eq:conj-stat-phase}
	\end{align}
	After integrating $\rd_{y^{j}}$ by parts, we obtain
	\begin{align*}
	{}^{(\bfPhi)} p(x, \xi) &= p(x, \rd_{x} \bfPhi)
	+ \xi_{j} \iint r^{(1) j}[p(x, \cdot)] (\eta, \rd_{y} \bfPhi(y)) e^{i (\bfPhi(y) - \bfPhi(x) + (\eta- \xi) \cdot (x-y))} \, \frac{\ud \eta}{(2 \pi)^{d}} \, \ud y \\
	&\peq + \iint r^{(1)j}[p(x, \cdot)] (\eta, \rd_{y} \bfPhi(y)) (\eta_{j} - \xi_{j} - \rd_{j} \bfPhi(y)) e^{i (\bfPhi(y) - \bfPhi(x) + (\eta- \xi) \cdot (x-y))} \, \frac{\ud \eta}{(2 \pi)^{d}} \, \ud y \\
	&= p(x, \rd_{x} \bfPhi)
	+ \xi_{j} \iint r^{(1) j}[p(x, \cdot)] (\eta, \rd_{y} \bfPhi(y)) e^{i (\bfPhi(y) - \bfPhi(x) + (\eta- \xi) \cdot (x-y))} \, \frac{\ud \eta}{(2 \pi)^{d}} \, \ud y \\
	&\peq + \iint r^{(1) j}[p(x, \cdot)] (\eta, \rd_{y} \bfPhi(y)) i \rd_{y^{j}} e^{i (\bfPhi(y) - \bfPhi(x) + (\eta- \xi) \cdot (x-y))} \, \frac{\ud \eta}{(2 \pi)^{d}} \, \ud y \\
	&= p(x, \rd_{x} \bfPhi)
	+ \xi_{j} \iint r^{(1) j}[p(x, \cdot)] (\eta, \rd_{y} \bfPhi(y)) e^{i (\bfPhi(y) - \bfPhi(x) + (\eta- \xi) \cdot (x-y))} \, \frac{\ud \eta}{(2 \pi)^{d}} \, \ud y \\
	&\peq - i \iint \left( \rd_{y^{j}} \left( r^{(1) j}[p(x, \cdot)] (\eta, \rd_{y} \bfPhi(y))\right) \right) e^{i (\bfPhi(y) - \bfPhi(x) + (\eta- \xi) \cdot (x-y))} \, \frac{\ud \eta}{(2 \pi)^{d}} \, \ud y,
	\end{align*}
	which already proves \eqref{eq:conj-rem-1}.
	
	We are interested in the next order terms, so we expand
	\begin{equation*}
	r^{(1) j}[p(x, \cdot)](\eta, \rd_{y} \bfPhi(y))
	= \rd_{\xi_{j}} p(x, \rd_{y} \bfPhi(y)) + \frac{1}{2} r^{(2) jk}[p(x, \cdot)](\eta, \rd_{y} \bfPhi(y)) (\eta_{k} - \rd_{k} \bfPhi(y)).
	\end{equation*}
	Thus,
	\begin{align*}
	{}^{(\bfPhi)} p(x, \xi)
	&= p(x, \rd_{x} \bfPhi)
	+ \xi_{j} (\rd_{\xi_{j}} p)(x, \rd_{x} \bfPhi(x)) 
	- \rst{i \rd_{y^{j}} \rd_{\xi_{j}} p(x, \rd_{y} \bfPhi(y))}_{y = x} \\
	&\peq + \frac{1}{2} \xi_{j} \iint r^{(2) jk}[p(x, \cdot)] (\eta, \rd_{y} \bfPhi(y)) (\eta_{k} - \rd_{k} \bfPhi(y)) e^{i (\bfPhi(y) - \bfPhi(x) + (\eta- \xi) \cdot (x-y))} \, \frac{\ud \eta}{(2 \pi)^{d}} \, \ud y \\
	&\peq - \frac{i}{2} \iint \left( \rd_{y^{j}} \left( r^{(2) jk}[p(x, \cdot)] (\eta, \rd_{y} \bfPhi(y)) (\eta_{k} - \rd_{k} \bfPhi(y)) \right) \right)  e^{i (\bfPhi(y) - \bfPhi(x) + (\eta- \xi) \cdot (x-y))} \, \frac{\ud \eta}{(2 \pi)^{d}} \, \ud y \\
	&= p(x, \rd_{x} \bfPhi)
	+ \xi_{j} (\rd_{\xi_{j}} p)(x, \rd_{x} \bfPhi(x)) 
	- i \rd_{\xi_{j}} \rd_{\xi_{k}} p(x, \rd_{x} \bfPhi(x)) \rd_{j} \rd_{k} \bfPhi(x) \\
	&\peq + \frac{1}{2} \xi_{j} \iint r^{(2) jk}[p(x, \cdot)] (\eta, \rd_{y} \bfPhi(y)) (\eta_{k} - \rd_{k} \bfPhi(y)) e^{i (\bfPhi(y) - \bfPhi(x) + (\eta- \xi) \cdot (x-y))} \, \frac{\ud \eta}{(2 \pi)^{d}} \, \ud y \\
	&\peq + \frac{i}{2} \iint r^{(2) jk}[p(x, \cdot)] (\eta, \rd_{y} \bfPhi(y)) \rd_{j} \rd_{k} \bfPhi(y) e^{i (\bfPhi(y) - \bfPhi(x) + (\eta- \xi) \cdot (x-y))} \, \frac{\ud \eta}{(2 \pi)^{d}} \, \ud y \\
	&\peq - \frac{i}{2} \iint \left( \rd_{y^{j}} \left( r^{(2) jk}[p(x, \cdot)] (\eta, \rd_{y} \bfPhi(y))\right) \right) (\eta_{k} - \rd_{k} \bfPhi(y)) e^{i (\bfPhi(y) - \bfPhi(x) + (\eta- \xi) \cdot (x-y))} \, \frac{\ud \eta}{(2 \pi)^{d}} \, \ud y.
	\end{align*}
	For the second to last term, we also expand
	\begin{equation*}
	r^{(2) jk}[p(x, \cdot)](\eta, \rd_{y} \bfPhi(y))
	= \rd_{\xi_{j}} \rd_{\xi_{k}} p(x, \rd_{y} \bfPhi(y)) + \frac{1}{3} r^{(3) jk\ell}[p(x, \cdot)](\eta, \rd_{y} \bfPhi(y)) (\eta_{\ell} - \rd_{\ell} \bfPhi(y)),
	\end{equation*}
	so that
	\begin{align*}
	{}^{(\bfPhi)} p(x, \xi)
	&= p(x, \rd_{x} \bfPhi)
	+ \xi_{j} (\rd_{\xi_{j}} p)(x, \rd_{x} \bfPhi(x)) 
	- \frac{i}{2} \rd_{\xi_{j}} \rd_{\xi_{k}} p(x, \rd_{x} \bfPhi(x)) \rd_{j} \rd_{k} \bfPhi(x) \\
	&\peq + \frac{1}{2} \xi_{j} \iint r^{(2) jk}[p(x, \cdot)] (\eta, \rd_{y} \bfPhi(y)) (\eta_{k} - \rd_{k} \bfPhi(y)) e^{i (\bfPhi(y) - \bfPhi(x) + (\eta- \xi) \cdot (x-y))} \, \frac{\ud \eta}{(2 \pi)^{d}} \, \ud y \\
	&\peq + \frac{i}{6} \iint r^{(3) jk\ell}[p(x, \cdot)] (\eta, \rd_{y} \bfPhi(y)) \rd_{j} \rd_{k} \bfPhi(y) (\eta_{\ell} - \rd_{\ell} \bfPhi(y))e^{i (\bfPhi(y) - \bfPhi(x) + (\eta- \xi) \cdot (x-y))} \, \frac{\ud \eta}{(2 \pi)^{d}} \, \ud y \\
	&\peq - \frac{i}{2} \iint \left( \rd_{y^{j}} \left( r^{(2) jk}[p(x, \cdot)] (\eta, \rd_{y} \bfPhi(y))\right) \right) (\eta_{k} - \rd_{k} \bfPhi(y)) e^{i (\bfPhi(y) - \bfPhi(x) + (\eta- \xi) \cdot (x-y))} \, \frac{\ud \eta}{(2 \pi)^{d}} \, \ud y.
	\end{align*}
	Using \eqref{eq:conj-stat-phase} and integrating by parts, we get
	\begin{align*}
	{}^{(\bfPhi)} p(x, \xi)
	&= p(x, \rd_{x} \bfPhi)
	+ \xi_{j} (\rd_{\xi_{j}} p)(x, \rd_{x} \bfPhi(x)) 
	- \frac{i}{2} \rd_{\xi_{j}} \rd_{\xi_{k}} p(x, \rd_{x} \bfPhi(x)) \rd_{j} \rd_{k} \bfPhi(x) \\
	&\peq + \frac{1}{2} \xi_{j} \xi_{k} \iint r^{(2) jk}[p(x, \cdot)] (\eta, \rd_{y} \bfPhi(y)) e^{i (\bfPhi(y) - \bfPhi(x) + (\eta- \xi) \cdot (x-y))} \, \frac{\ud \eta}{(2 \pi)^{d}} \, \ud y \\
	&\peq - \frac{i}{2} \xi_{j} \iint \left(\rd_{y^{k}} \left( r^{(2) jk}[p(x, \cdot)] (\eta, \rd_{y} \bfPhi(y)) \right) \right) e^{i (\bfPhi(y) - \bfPhi(x) + (\eta- \xi) \cdot (x-y))} \, \frac{\ud \eta}{(2 \pi)^{d}} \, \ud y \\
	&\peq + \frac{i}{6} \xi_{\ell} \iint r^{(3) jk\ell}[p(x, \cdot)] (\eta, \rd_{y} \bfPhi(y)) \rd_{j} \rd_{k} \bfPhi(y) e^{i (\bfPhi(y) - \bfPhi(x) + (\eta- \xi) \cdot (x-y))} \, \frac{\ud \eta}{(2 \pi)^{d}} \, \ud y \\
	&\peq + \frac{1}{6} \iint \left(\rd_{y^{\ell}} \left( r^{(3) jk\ell}[p(x, \cdot)] (\eta, \rd_{y} \bfPhi(y)) \rd_{j} \rd_{k} \bfPhi(y) \right)\right) e^{i (\bfPhi(y) - \bfPhi(x) + (\eta- \xi) \cdot (x-y))} \, \frac{\ud \eta}{(2 \pi)^{d}} \, \ud y \\
	&\peq - \frac{i}{2} \xi_{k} \iint \left( \rd_{y^{j}} \left( r^{(2) jk}[p(x, \cdot)] (\eta, \rd_{y} \bfPhi(y))\right) \right) e^{i (\bfPhi(y) - \bfPhi(x) + (\eta- \xi) \cdot (x-y))} \, \frac{\ud \eta}{(2 \pi)^{d}} \, \ud y \\
	&\peq - \frac{1}{2} \iint \left( \rd_{y^{j}} \rd_{y^{k}} \left( r^{(2) jk}[p(x, \cdot)] (\eta, \rd_{y} \bfPhi(y))\right) \right) e^{i (\bfPhi(y) - \bfPhi(x) + (\eta- \xi) \cdot (x-y))} \, \frac{\ud \eta}{(2 \pi)^{d}} \, \ud y.
	\end{align*}
	Hence, we obtain \eqref{eq:conj-exp} with ${}^{(\bfPhi)} r_{p}(x, \xi)$ consisting of all terms on the RHS except for the first three terms. Recalling \eqref{eq:taylor-rem},  \eqref{eq:conj-rem-2} then follows after a straightforward computation. 
	\qedhere
	
\end{proof}

\subsection{Specialization to $\calL_{\bgtht}$ and equations for the phase and the amplitude}\label{subsec:outline}
We now return to the problem at hand. {In view of the $x_{1}$-independence of $\bgtht$, we shall choose $\bfPhi$ and $a$ in the separated form
\begin{equation*}
	\bfPhi(t, x) = \lmb_{0} x_{1} + \Phi(t, x_{2}), \quad a = a(t, x_{2}),
\end{equation*}
where $\lmb_{0}$ is an integer.} Using standard symbolic calculus (Lemma~\ref{lem:symb-calc}), we first rewrite $\calL_{\bgtht}$ in a form that is more convenient to apply Lemma~\ref{lem:conj} and Proposition~\ref{prop:conj-exp}.

\begin{proposition} \label{prop:L-tht0-decomp-f}
	Let $\bgtht(t, x) = f(t, x_{2})$ and $\lmb_{0} \in \bbZ$. We have
	\begin{equation} \label{eq:L-tht0-decomp}
	\calL_{\bgtht} (e^{i \lmb_{0} x_{1}} \psi)
	= e^{i \lmb_{0} x_{1}} \left(\rd_{t} \psi + i p_{\bgtht, \lmb_{0}}(x_{2}, D_{2}) \psi + s_{\bgtht, \lmb_{0}}(x_{2}, D_{2}) \psi + r_{\bgtht, \lmb_{0}}(x_{2}, D_{2}) \psi\right),
	\end{equation}
	where
	\begin{align*}
	p_{\bgtht, \lmb_{0}}(x_{2}, \xi_{2}) &= \rd_{x_{2}} f(t, x_{2}) \lmb_{0} \gmm_{\lmb_{0}}(\xi_{2}) + (\rd_{x_{2}} \Gmm f)(t, x_{2}) \lmb_{0}, \\
	s_{\bgtht, \lmb_{0}}(x_{2}, \xi_{2}) &= - \frac{1}{2} \rd_{x_{2}}^{2} f(t, x_{2}) \rd_{\xi_{2}} \gmm_{\lmb_{0}}(\xi_{2}) \lmb_{0} - \frac{1}{2} (\rd_{x_{2}}^{2} \Gmm f)(t, x_{2}) \gmm_{\lmb_{0}}^{-1}(\xi_{2}) \rd_{\xi_{2}} \gmm_{\lmb_{0}}(\xi_{2}) \lmb_{0}, \\
	r_{\bgtht, \lmb_{0}}(x_{2}, \xi_{2}) &\in S\left( {\tfrac{\lmb_{0}}{1+\lmb_{0}^{2}+\xi_{2}^{2}} \gmm(\lmb_{0}, \xi_{2})}\right).
	\end{align*}
\end{proposition}
Observe that both $p_{\bgtht, \lmb_{0}}(x_{2}, \xi_{2})$ and $s_{\bgtht, \lmb_{0}}(x_{2}, \xi_{2})$ are real-valued and even.

\begin{proof}
	We begin by computing
	\begin{align*}
	\calL_{\bgtht} \varphi
	&= \rd_{t} \varphi
	- \Gmm^{\frac{1}{2}} \nb^{\perp} \bgtht \cdot \nb \Gmm^{\frac{1}{2}} \varphi 
	+ \Gmm^{\frac{1}{2}} \nb^{\perp} \Gmm \bgtht \cdot \nb \Gmm^{-\frac{1}{2}} \varphi \\
	&= \rd_{t} \varphi
	+ \Gmm^{\frac{1}{2}} \rd_{x_{2}} f(t, x_{2}) \rd_{x_{1}} \Gmm^{\frac{1}{2}} \varphi 
	- \Gmm^{\frac{1}{2}} \rd_{x_{2}} \Gmm f(t, x_{2}) \rd_{x_{1}} \Gmm^{-\frac{1}{2}} \varphi \\
	&=  \rd_{t} \varphi
	+ \rd_{x_{2}} f \rd_{x_{1}} \Gmm \varphi 
	- \rd_{x_{2}} \Gmm f \rd_{x_{1}} \varphi \\
	&\peq
	+ [\Gmm^{\frac{1}{2}}, \rd_{x_{2}} f] \rd_{x_{1}} \Gmm^{\frac{1}{2}} \varphi 
	- [\Gmm^{\frac{1}{2}}, \rd_{x_{2}} \Gmm f] \rd_{x_{1}} \Gmm^{-\frac{1}{2}} \varphi.
	\end{align*}
	For the last line, we factor out $\rd_{x_{1}}$ (whose symbol is $i \xi_{1}$) and use Lemma~\ref{lem:symb-calc} to write
	\begin{align*}
	&[\Gmm^{\frac{1}{2}}, \rd_{x_{2}} f] \Gmm^{\frac{1}{2}} 
	- \left( i^{-1} \{ \gmm^{\frac{1}{2}}, \rd_{x_{2}} f \} \gmm^{\frac{1}{2}}(\xi) \right) (x, D) \in Op(S(\brk{\xi}^{-2} m)), \\
	& [\Gmm^{\frac{1}{2}}, \rd_{x_{2}} \Gmm f] \Gmm^{-\frac{1}{2}}
	- \left( i^{-1} \{ \gmm^{\frac{1}{2}}, \rd_{x_{2}} f \} \gmm^{-\frac{1}{2}}(\xi) \right) (x, D) \in Op(S(\brk{\xi}^{-2})).
	\end{align*}
	Then using the identity
	\begin{equation*}
	\{\gmm^{\frac{1}{2}}, g \}
	= \frac{1}{2} \gmm^{-\frac{1}{2}} \rd_{\xi_{2}} \gmm (\xi) \rd_{x_{2}} g
	\end{equation*}
	which holds for any function $g = g(x_{2})$, it follows that
	\begin{equation} \label{eq:L-tht0-decomp-pre}
	\calL_{\bgtht} \varphi
	= \rd_{t} \varphi + i p_{\bgtht}(x_{2}, D_{1}, D_{2}) \varphi + s_{\bgtht}(x_{2}, D_{1}, D_{2}) \varphi + r_{\bgtht}(x_{2}, D_{1}, D_{2}) \varphi,
	\end{equation}
	where
	\begin{align*}
	p_{\bgtht}(x_{2}, \xi_{1}, \xi_{2}) &= \rd_{x_{2}} f(t, x_{2}) \xi_{1} \gmm(\xi_{1}, \xi_{2}) {-} (\rd_{x_{2}} \Gmm f)(t, x_{2}) \xi_{1}, \\
	s_{\bgtht}(x_{2}, \xi_{1}, \xi_{2}) &= {+} \frac{1}{2} \rd_{x_{2}}^{2} f(t, x_{2}) \rd_{\xi_{2}} \gmm(\xi_{1}, \xi_{2}) \xi_{1} - \frac{1}{2} (\rd_{x_{2}}^{2} \Gmm f)(t, x_{2}) \gmm^{-1}(\xi) \rd_{\xi_{2}} \gmm(\xi_{1}, \xi_{2}) \xi_{1}, \\
	r_{\bgtht}(x_{2}, \xi_{1} \xi_{2}) &\in S\left( {\tfrac{\lmb_{0}}{1+\lmb_{0}^{2}+\xi_{2}^{2}} \gmm(\lmb_{0}, \xi_{2})}\right).
	\end{align*}
	Next, we conjugate \eqref{eq:L-tht0-decomp-pre} by $e^{i \lmb_{0} x_{1}}$. For any symbol $p$, note that the effect of conjugating by the linear phase $\lmb_{0} x_{1}$ is simply the translation $\xi_{1} \mapsto \xi_{1} + \lmb_{0}$; i.e., ${}^{(\lmb_{0} x_{1})} p(x_{2}, \xi_{1}, \xi_{2}) = p(x_{2}, \lmb_{0} + \xi_{1}, \xi_{2})$, where the LHS is given in Lemma~\ref{lem:conj}. Moreover, it is clear from \eqref{eq:psdo-op} that for any function $\psi$ that is independent of $x_{1}$, we have $p(x_{2}, \lmb_{0} + D_{1}, D_{2}) \psi = p(x_{2}, \lmb_{0}, D_{2}) \psi$. Since $p_{\bgtht}(x_{2}, \lmb_{0}, \xi_{2})= p_{\bgtht, \lmb_{0}}(x_{2}, \xi_{2})$, $s_{\bgtht}(x_{2}, \lmb_{0}, \xi_{2})= s_{\bgtht, \lmb_{0}}(x_{2}, \xi_{2})$ and $r_{\bgtht}(x_{2}, \lmb_{0}, \xi_{2})= r_{\bgtht, \lmb_{0}}(x_{2}, \xi_{2})$, the proof of Proposition~\ref{prop:L-tht0-decomp-f} is complete. \qedhere
\end{proof}

With Proposition~\ref{prop:L-tht0-decomp-f} in hand, we are ready to specify the equations solved by the phase function $\bfPhi = \lmb_{0} x_{1} + \Phi(t, x_{2})$ and the amplitude function $a = a(t, x_{2})$, which would constitute the desired degenerating wave packet $\tld{\varphi} = \Re(e^{i \bfPhi} a) = \Re(e^{i (\lmb_{0} x_{1} + \Phi)} a)$. The scheme itself is simply the standard WKB procedure for the nonlocal operator $\calL_{\bgtht}$, but there are extra twists that arise in our scenario due to the degeneracy of $\bgtht$; see Remark~\ref{rem:wp-construct} below.

\smallskip
\noindent {\bf Equation for the phase $\Phi$.} We choose the phase function $\Phi = \Phi(t, x_{2})$ to be a solution to the Hamilton--Jacobi equation
\begin{equation} \label{eq:HJ-Phi}
\rd_{t} \Phi + p_{\bgtht, \lmb_{0}}(t, x_{2}, \rd_{x_{2}} \Phi) = 0,
\end{equation}
with initial data satisfying $\abs{\rd_{x_{2}} \Phi(0, x_{2})} \aeq \lmb_{0}$ (the precise choice of the initial data has to be well-adapted to our problem; see Section~\ref{sec:HJE} below).
We note that the $(X, \Xi)$-components of the characteristics for \eqref{eq:HJ-Phi} solve the Hamiltonian ODE associated with $p_{\bgtht, \lmb_{0}}(t, x_{2}, \xi_{2})$,
\begin{equation*}
\left\{
\begin{aligned}
\dot{X}_{2}(t) & = \rd_{\xi_{2}} p_{\bgtht, \lmb_{0}}(t, X_{2}(t), \Xi_{2}(t)), \\
\dot{\Xi}_{2}(t) & = - \rd_{x_{2}} p_{\bgtht, \lmb_{0}}(t, X_{2}(t), \Xi_{2}(t)).
\end{aligned}	
\right.
\end{equation*}

\smallskip
\noindent {\bf Equation for the amplitude $a$.} We choose the amplitude function $a = a(t, x_{2})$ to obey the transport equation
\begin{equation}  \label{eq:transport-a}
\tld{\nb}_{t} a := \rd_{t} a
+ \rd_{\xi_{2}} p_{\bgtht, \lmb_{0}}(x_{2}, \rd_{x_{2}} \Phi) \rd_{x_{2}} a
+ \left( \frac{1}{2} \rd_{\xi_{2}}^{2} p_{\bgtht, \lmb_{0}}(x_{2}, \rd_{x_{2}} \Phi) \rd_{x_{2}}^{2} \Phi 
+ s_{\bgtht, \lmb_{0}}(x_{2}, \rd_{x_{2}} \Phi)\right) a = 0,
\end{equation}
with smooth and compactly supported initial data. At this point, we note that if one takes the inner product of \eqref{eq:transport-a} with $a$ and integrate by parts, then the specific structure of $s_{\bgtht, \lmb_{0}}$ will show that the $L^{2}$ norm of $a$ stays bounded on an $O(1)$ timescale. {This is proved in Proposition \ref{prop:trans-a} below.}

\smallskip
\noindent {\bf Key computation for the error term.}  The reason for the choices of the equations \eqref{eq:HJ-Phi} and \eqref{eq:transport-a} is the following computation:
\begin{equation} \label{eq:error-ps}
\rd_{t} (a e^{i \Phi}) + i p_{\bgtht, \lmb_{0}}(a e^{i \Phi}) + s_{\bgtht, \lmb_{0}}(a e^{i \Phi})
= {}^{(\Phi)} r_{p_{\bgtht, \lmb_{0}}, -2} + {}^{(\Phi)} r_{s_{\bgtht, \lmb_{0}}, -1}.
\end{equation}
Indeed, by the equations for $\rd_t\Phi$ and $\rd_t a$: \begin{equation*}
\begin{split}
\rd_{t} (a e^{i \Phi}) + i p_{\bgtht, \lmb_{0}}(a e^{i \Phi}) + s_{\bgtht, \lmb_{0}}(a e^{i \Phi}) 
& = i\rd_t\Phi  ae^{i\Phi} + (\rd_t a) e^{i\Phi} + i p_{\bgtht, \lmb_{0}} (ae^{i\Phi}) + s_{\bgtht, \lmb_{0}} (ae^{i\Phi}) \\
& = i p_{\bgtht, \lmb_{0}}( ae^{i\Phi} )  -ip_{\bgtht, \lmb_{0}}(x,\rd_x\Phi)  ae^{i\Phi}  - \rd_{\xi_{2}} p_{\bgtht, \lmb_{0}}(x_{2}, \rd_{x_{2}} \Phi) \rd_{x_{2}} a e^{i\Phi} \\
&\quad - \frac{1}{2} \rd_{\xi_{2}}^{2} p_{\bgtht, \lmb_{0}}(x_{2}, \rd_{x_{2}} \Phi) \rd_{x_{2}}^{2} \Phi \, a e^{i\Phi}\\
&\quad + s_{\bgtht, \lmb_{0}} (ae^{i\Phi})  - s_{\bgtht, \lmb_{0}}(x, \rd_{x} \Phi) a e^{i\Phi}.
\end{split}
\end{equation*} Then, \eqref{eq:error-ps} follows from Proposition~\ref{prop:conj-exp}.

In conclusion, to justify that $\tld{\varphi} = \Re (a e^{i (\lmb_{0} x_{1} + \Phi)})$ is a good approximate solution, we need to use \eqref{eq:L-tht0-decomp} and \eqref{eq:error-ps}
to show that $\calL_{\bgtht} \tld{\varphi} = \Re (\calL_{\bgtht} (a e^{i (\lmb_{0} x_{1} + \Phi)}))$ is suitably small. As remarked at the beginning of this section, this goal shall be achieved in Section~\ref{sec:wavepacket}.

\begin{remark} \label{rem:wp-construct}
	In our case, due to the degeneracy of $p_{\tht}$, we expect $\rd_{x_{2}}^{1+k} \Phi$ and $\rd_{x_{2}}^{k} a$ to \emph{grow} rapidly in time for $k \geq 1$ (and indeed, this is precisely the behavior we wish to capture in order to show illposedness). For this reason, we cannot simply rely on standard pseudo-differential calculus to construct a wave packet that is valid up to a timescale where such a growth of higher derivatives (i.e., \emph{degeneration}) occurs. Nevertheless, as we shall see in Section~\ref{sec:HJE}, the growth rate of $\rd_{x_{2}}^{k} \Phi$ and $\rd_{x_{2}}^{k} a$ will be smaller than $\abs{\rd_{x_{2}} \Phi}^{k}$, which we may salvage via Proposition~\ref{prop:conj-err} in Section~\ref{sec:conj-err}.
\end{remark}

\section{The Hamilton--Jacobi equation and the associated transport operator} \label{sec:HJE}
 
This section is devoted to the choice and analysis of the phase and amplitude functions. The main results of this section are Proposition \ref{prop:Phi-derivatives} and \ref{prop:trans-a}, which provide sharp estimates on the derivatives of the phase and amplitude, respectively. 
 
Let us outline the structure of this section. After detailing the choice of parameters in Section \ref{subsec:HJE-prelim}, we consider the case of time-independent background state in Section \ref{subsec:HJE-s}. In this case, we take our phase function to be a solution of the Hamilton--Jacobi equation in separation of variables form. Thanks to the rather explicit form of the phase function, it is straightforward to obtain pointwise estimates for the derivatives of the phase function (Proposition \ref{prop:Phi-derivatives-time-indep}), and the resulting bounds serve as a guide for the time-dependent background case which is handled in Section \ref{subsec:HJE-gen}. In this latter case, we use the same initial data for the phase function, and estimate the solution by the method of characteristics and a bootstrapping argument. The key variable $h$ is introduced in this section in \eqref{eq:def-h}, which encodes some cancellations in the Hamilton--Jacobi equation and allows us to obtain sharp estimates on the second derivative of the phase function (Lemma \ref{lem:HJE-h} and Proposition \ref{prop:h}). Using the variable $h$, we prove Proposition \ref{prop:Phi-derivatives} in Section \ref{subsec:transport}, which gives the desired pointwise bounds on the derivatives of the phase function. Then, it is not very difficult to obtain estimates on the amplitude function, which is Proposition \ref{prop:trans-a} of Section \ref{subsec:transport2}. 

\subsection{Initial reductions and conventions} \label{subsec:HJE-prelim}
Throughout this section, we study the Hamilton--Jacobi equation \eqref{eq:HJ-Phi} with the following conventions: 
	\begin{itemize}
		\item Assume that $\dgx_{2} = 0$ and $\rd_{x_{2}}^{2} f(0, 0) < 0$.
		\item Write $x = x_{2}$ and $\xi = \xi_{2}$. Moreover, a prime ($'$) denotes $\rd_{x_{2}}$.
		\item Introduce the shorthand $\gmm_{\lmb_{0}}(\xi) = \gmm(\lmb_{0}, \xi)$.
	\end{itemize} 
With the above reductions, the Hamilton--Jacobi equation is simply given by 
\begin{equation} \label{eq:HJE}
	\rd_{t} \Phi + p(t, x, \rd_{x} \Phi) = 0 \qquad \hbox{ in } \bbR_{t} \times \bbR_{x}, 
\end{equation}
where 
\begin{equation} \label{eq:HJE-p}
p(t, x, \xi) = f'(t, x) \lmb_{0} \gmm_{\lmb_{0}} (\xi) + \Gmm f'(t, x) \lmb_{0}. 
\end{equation}
The key point of our analysis is that $f'(t, x)$ is assumed to vanish linearly at some point for each $t$. As a result, there exist bicharacteristics $(X(t), \Xi(t))$ associated with $p(t, x, \xi)$ that exhibits rapid growth of $\abs{\Xi}$. Our aim is to construct a solution $\Phi$ to \eqref{eq:HJE}, which we refer to as a \emph{phase function}, that exhibits such a behavior for a sufficiently long time.

\medskip
\noindent{\bf Parameters in the construction.}
In this section, we are given a symbol $\gmm$ and a function $f$ that satisfy the assumptions laid out in Section~\ref{subsec:results}. We are also given $\lmb_{0} \in \bbZ$, $M \geq 1$, $\dlt_{0} > 0$ and $\sgm_{0} \geq 0$ that satisfy \eqref{eq:gf-condition-1}--\eqref{eq:gf-condition-3}\footnote{To be absolutely precise, instead of \eqref{eq:gf-condition-2} we only need $\tau_{M} \leq 1$ in this section. The full condition is needed in Section~\ref{sec:wavepacket} below.}. When $f$ is time-dependent, we assume furthermore that $f$ is even and that we are given $\dlt_{1} > 0$ such that \eqref{eq:gf-condition-diss-2} holds. 

We {\bf fix a nonincreasing function $\eps(\lmb_{0})$ for $\lmb_{0} \geq \Xi_{0}$} such that, for $\lmb_{0} \geq \Lmb$,
\begin{equation} \label{eq:eps-choice} 
\max\set*{\frac{1}{\lmb_{0}^{\sgm_{0}+\frac{1}{6} \dlt_{0}}}, \, \frac{1}{\gmm_{\lmb_{0}}(\lmb_{0})^{1-\frac{1}{2} \dlt_{0}}}}
\leq \eps(\lmb_{0}) 
\leq \min\set*{\frac{1}{100}, \frac{1}{(10+2^{2\bt_{0}})\sup_{M' \in [1, M]} \frac{\gmm_{\lmb_{0}}(M' \lmb_{0})}{M'} \tau_{M'}}}.
\end{equation}
By the second nondissipative growth condition \eqref{eq:gf-condition-1}, we see that there always exists such a function, i.e., if $\eps(\lmb_{0})$ satisfies
\begin{equation*}
\max\set*{\frac{1}{\lmb_{0}^{\sgm_{0}+\frac{1}{6} \dlt_{0}}}, \, \frac{1}{\gmm_{\lmb_{0}}(\lmb_{0})^{1-\frac{1}{2} \dlt_{0}}}} 
\leq \eps(\lmb_{0}) 
\leq \frac{1}{2^{\bt_{0}+1}(2^{\bt_{0}}+3) \min\set{\gmm_{\lmb_{0}}(\lmb_{0})^{1-\dlt_{0}}, \lmb_{0}^{\sgm_{0}}}},
\end{equation*}
for $\lmb_{0} \geq \Xi_{0}$, then it would obey \eqref{eq:eps-choice}.

The primary role of $\eps(\lmb_{0})$ is as the relative physical localization scale of the wave packet: the support of $a(t=0, x)$ will be contained in an interval of size $\Dlt x_{0}$ that is comparable to $x_{0} \eps(\lmb_{0})$ up to a logarithmic power of $\lmb_{0}$, where $x_{0}$ is the initial location of the wave packet. For technical simplicity, we shall also choose (see \eqref{eq:x0-x1} and \eqref{eq:def-x1-x0} below)
\begin{equation*}
x_{0} = c_{x_{0}} \eps(\lmb_{0}), 
\end{equation*}
where {\bf we shall retain the freedom to choose $c_{x_{0}} > 0$ throughout this section} (it will only be fixed in Section~\ref{subsec:def-wavepacket}).

\begin{remark}
The restriction $\sgm_{0} \leq \frac{1}{3}(1 - 2\dlt_{0})$ in Theorems~\ref{thm:norm-growth} and \ref{thm:norm-growth-diss} arise from the uncertainty principle at $t = 0$ (which requires $\Dlt x_{0} \lmb_{0} \ageq 1$) and the wave packet error estimate; see Proposition~\ref{prop:wp-error} below. We note that the requirement $x_{0} \aeq \eps(\lmb_{0})$ can be weakened if $f'''(0, 0) = 0$, in which case the restriction $\sgm_{0} \leq \frac{1}{3}(1-2\dlt_{0})$ can be relaxed; see Remark~\ref{rem:x0-x1} below.
\end{remark}

In addition to $\dlt_{0}$ and $\dlt_{1}$, we are given $\dlt_{2}$ and $N_{0}^{-1}$, which will be chosen in Sections~\ref{sec:wavepacket} and \ref{sec:proofs} (in fact, $\dlt_{2} = \dlt_{0}^{10}$ and $N_{0} = 10^{4} \max\set{\dlt_{2}^{-1}(1+\bt_{0}), 1+\alp_{0}, {1+s, 1+s'}}$; see Proposition~\ref{prop:wp-error} and Section~\ref{subsec:lin-illposed-diss} below).
In the course of this section, {\bf we will choose additional small parameters in the following order} (each choice may also depend on $\gmm$ and $f$): 
\begin{equation*}
\dlt_{0}, \dlt_{1}, \dlt_{2}, N_{0}^{-1} \to \dlt_{3} \to \dlt_{4} \to \dlt_{5}.
\end{equation*}
In this section, {\bf we also require that}
\begin{equation*}
	\tau_{M} \leq \wpT, \qquad \lmb_{0} \geq \Lmb
\end{equation*}
where {\bf we shall retain the freedom to choose $0 < \wpT \leq 1$ and $\Lmb \geq 1$ throughout this section} ({$T$ may be fixed at the end of this section, while $\Lmb$ will be fixed in Section~\ref{subsec:deg}}). 

\subsection{The case of a shear steady state background} \label{subsec:HJE-s}

\noindent \textit{Choice of $\Phi$}. 
When $\bgtht = f(x_{2})$ is time-independent, i.e., $\rd_{t} f' = 0$, we may obtain a useful solution to the Hamilton--Jacobi equation by separation of variables: Consider the ansatz
\begin{equation*}
	\Phi(t, x) = E \lmb_{0} t + G(x).
\end{equation*}
Then $G(x)$ obeys the equation
\begin{equation} \label{eq:HJE-s}
	p(x, \rd_{x} G) = - E \lmb_{0}.
\end{equation}
By the positivity of $\xi \rd_{\xi} \gmm_{\lmb_{0}}$, $\gmm_{\lmb_{0}}(\cdot)$ is invertible on $(-\infty, - \lmb_{0}]$ or $[\lmb_{0}, \infty)$. {We are going to take $G(x)$ whose derivative is blowing up at the origin, and for $\rd_{x} \Phi = \rd_{x} G$ to grow (as $t$ increases) on characteristic curves, we need to choose $\rd_{x} G$ to be positive (indeed, this sign relation may be read off from the bicharacteristic equations).} In what follows, we denote by $\gmm_{\lmb_{0}}^{-1}$ the inverse of the restriction $\gmm_{\lmb_{0}} : [\lmb_{0}, \infty) \to (0, \infty)$. Hence, we obtain the explicit formula
\begin{equation} \label{eq:HJE-s-dPhi}
	\rd_{x} \Phi = \rd_{x} G := \gmm_{\lmb_{0}}^{-1} \left(\frac{E + \Gmm f'}{- f'}\right).
\end{equation} {Note that $\rd_{x} G \to \infty$ at the origin. }

To begin with, we proceed to show that there is a choice of $E$ in \eqref{eq:HJE-s-dPhi} so that 
\begin{equation}\label{eq:rd-Phi-bounds}
\lmb_{0} \leq \rd_{x} \Phi(0, x) \leq (1+\eps) \lmb_{0}
\end{equation} holds for $x_{0} < x < x_{1}$ and $\lmb_{0} \geq \Lmb$, provided that $x_{0}$ and $x_{1}$ are sufficiently small (with respect to $f'$, $\lmb_{0}$, and $\eps$) and $\Lmb$ is sufficiently large (depending on $\gmm$ and $\abs{f''(0)}^{-1} \Gmm f''$). Since $\gmm^{-1}_{\lmb_{0}}$ is increasing, \eqref{eq:rd-Phi-bounds} is equivalent to having \begin{equation}\label{eq:rd-Phi-bounds-2}
\begin{split}
\gmm_{\lmb_{0}}(\lmb_{0}) \le \frac{E + \Gmm f'(x)}{-f'(x)} \le \gmm_{\lmb_{0}}((1+\eps(\lmb_{0}))\lmb_0), \qquad x_{0}<x<x_{1}. 
\end{split}
\end{equation} We first take $x_{1} >0$ to be sufficiently small (recall that $f''(0)<0$), so that
\begin{equation} \label{eq:rd-Phi-bounds-x1-1}
	(1+\tfrac{\eps}{10}) f''(0) < f''(x) < (1-\tfrac{\eps}{10}) f''(0) \qquad \hbox{ for $x \in (0, x_{1})$}.
\end{equation}
In particular, we have $(1+\tfrac{\eps}{10}) f''(0) x < f'(x) < (1-\tfrac{\eps}{10}) f''(0) x$ for $x \in (0, x_{1})$. Then, we define 
\begin{equation*}
\begin{split}
E= -f'(x_1)\gmm_{\lmb_{0}}((1+\tfrac{1}{2} \eps)\lmb_0) - \Gmm f'(x_1).
\end{split}
\end{equation*} With this choice of $E$, \eqref{eq:rd-Phi-bounds-2} is trivially satisfied at $x=x_{1}$. We may now choose $x_0 <x_1 $ in a way that \eqref{eq:rd-Phi-bounds-2} is satisfied for all $x\in(x_{0},x_{1})$: we take \begin{equation}\label{eq:def-x0}
\begin{split}
\frac{x_{1}}{x_{0}} -1 = \frac{1}{10} \frac{\gmm_{\lmb_{0}}((1+\tfrac{3}{4}\eps)\lmb_0)-\gmm_{\lmb_{0}}((1+\tfrac{1}{2}\eps)\lmb_0)}{\gmm_{\lmb_{0}}((1+\tfrac{1}{2}\eps)\lmb_0)}. 
\end{split}
\end{equation} With these choices, we estimate \begin{equation*}
\begin{split}
\frac{f'(x_1)}{f'(x_0)} = 1 + \frac{\int_{x_0}^{x_1} f''(x)\,\ud x}{\int_{0}^{x_{0}} f''(x) \, \ud x} \le 1 + 10\left|1 - \frac{x_{1}}{x_{0}}\right| = \frac{\gmm_{\lmb_{0}}((1+\tfrac{3}{4}\eps)\lmb_0)}{\gmm_{\lmb_{0}}((1+\tfrac{1}{2}\eps)\lmb_0)}.
\end{split}
\end{equation*} where we take $\lmb_{0} \geq \Lmb$ larger if necessary to decrease $\eps$. Then, \begin{equation*}
\begin{split}
\frac{E + \Gmm f'(x)}{-f'(x)} & = \frac{-f'(x_1)\gmm_{\lmb_{0}}((1+\tfrac{1}{2} \eps)\lmb_0)-\Gmm f'(x_{1}) + \Gmm f'(x)}{-f'(x)} \\
&\le \frac{f'(x_1)}{f'(x_0)} \gmm_{\lmb_{0}}((1+\tfrac{1}{2}\eps)\lmb_0) + \frac{2}{f''(0)}\nrm{\Gmm f''}_{L^\infty(x_{0}, x_{1})}\frac{|x_{0}-x_{1}|}{x_{0}} \\
&\le \gmm_{\lmb_{0}}((1+\tfrac{3}{4}\eps)\lmb_0) + \frac{20 \nrm{\Gmm f''}_{L^{\infty}(x_{0}, x_{1})}}{\abs{f''(0)}} \frac{\gmm_{\lmb_{0}}((1+\tfrac{3}{4}\eps)\lmb_0)-\gmm_{\lmb_{0}}((1+\tfrac{1}{2}\eps)\lmb_0)}{\gmm_{\lmb_{0}}((1+\tfrac{1}{2}\eps)\lmb_0)}.
\end{split}
\end{equation*} 
To estimate the last line by $\gmm_{\lmb_{0}}((1+\eps)\lmb_{0})$, it suffices to verify 
\begin{equation} \label{eq:rd-Phi-bound-key}
\begin{split}
\frac{20 \nrm{\Gmm f''}_{L^{\infty}(x_{0}, x_{1})}}{\abs{f''(0)}} \frac{\gmm_{\lmb_{0}}((1+\tfrac{3}{4}\eps)\lmb_0)-\gmm_{\lmb_{0}}((1+\tfrac{1}{2}\eps)\lmb_0)}{\gmm_{\lmb_{0}}((1+\tfrac{1}{2}\eps)\lmb_0)} \le \gmm_{\lmb_{0}}((1+\eps)\lmb_0) - \gmm_{\lmb_{0}}((1+\tfrac{3}{4}\eps)\lmb_0).
\end{split}
\end{equation} From our assumptions on $\gmm$, this is guaranteed for $\lmb_0\ge\Lmb_0$ with sufficiently large $\Lmb_0$. Indeed, by the mean value theorem, there exist $\alp_{1}, \alp_{2} \in [0, 1]$ and $\alp_{3} \in [\tfrac{2+\alp_{1}}{4}, \tfrac{3+\alp_{2}}{4}]$ such that \begin{equation*}
\begin{split}
\frac{\gmm_{\lmb_{0}}((1+\tfrac{3}{4}\eps)\lmb_0)-\gmm_{\lmb_{0}}((1+\tfrac{1}{2} \eps)\lmb_0)}{\gmm_{\lmb_{0}}((1+\eps)\lmb_0) - \gmm_{\lmb_{0}}((1+\tfrac{3}{4} \eps)\lmb_0) } & = 1 +  \frac{\rd_{\xi}\gmm_{\lmb_{0}}((1+\tfrac{2+\alp_{1}}{4} \eps)\lmb_0)-\rd_{\xi}\gmm_{\lmb_{0}}((1+\tfrac{3+\alp_{2}}{4} \eps)\lmb_0)}{\rd_{\xi}\gmm_{\lmb_{0}}((1+\tfrac{3+\alp_{2}}{4} \eps)\lmb_0)}  \\
& \le 1 + \eps\lmb_{0} \frac{\abs{\rd_{\xi}^2\gmm_{\lmb_{0}}((1 + \alp_{3} \eps)\lmb_0)}}{\rd_{\xi}\gmm_{\lmb_{0}}((1+\tfrac{3+\alp_{2}}{4}\eps)\lmb_0)} \leq C,
\end{split}
\end{equation*} where we have used the ellipticity of $\xi \rd_{\xi}\gmm_{\lmb_0}$ for the last inequality. Since $\gmm_{\lmb_{0}} \to \infty$ as $\lmb_{0} \to \infty$, the desired bound \eqref{eq:rd-Phi-bound-key} follows by taking $\lmb_{0} \geq \Lmb$ large enough. We have verified the right inequality of \eqref{eq:rd-Phi-bounds-2}. The proof of the left inequality is only easier since $-f'(x)$ is increasing in $(x_{0},x_{1})$. We hence omit the proof.

We have shown that if $x_{1}$ and $x_{0}$ are chosen so that \eqref{eq:rd-Phi-bounds-x1-1} and \eqref{eq:def-x0} and $\Lmb$ is sufficiently large depending on $\gmm$ and $\abs{f''(0)}^{-1} \Gmm f''$, then $E$ can be chosen so that \eqref{eq:rd-Phi-bounds} holds. In what follows, {\bf we choose $x_{0}$, $x_{1}$ to be of the form }
\begin{equation}\label{eq:x0-x1}
\begin{split}
x_{0} = {c_{x_{0}} \eps(\lmb_{0})}, \quad {x_{1} = x_{0} + \Dlt x_{0}}, \quad \hbox{ where } {\frac{c c_{x_{0}}}{(\log \lmb_{0})^{2}} \eps(\lmb_{0})} \leq \left|\frac{\Dlt x_{0}}{x_{0}}\right| \leq c_{x_{0}} \eps(\lmb_{0}), 
\end{split}
\end{equation} 
where $c$ is a constant depending only on $\gmm$ and $f$, and {$c_{x_{1}} > 0$ obeys 
\begin{equation} \label{eq:def-x1}
	\nrm{f'''(x)}_{L^{\infty}(0, c_{x_{1}}\eps(\lmb_{0}))} c_{x_{0}} \leq \frac{1}{20} \abs{f''(0)}. 
\end{equation}
Clearly, such a choice of $x_{0}$ and $x_{1}$ ensures \eqref{eq:rd-Phi-bounds-x1-1}. In the following, there are several occasions in which we need to take $x_{0}$ and $x_{1}$ sufficiently small. To achieve this, {\bf we shall retain the freedom to shrink $c_{x_{0}} > 0$ until the end of this section}; note that this action only makes the LHS of \eqref{eq:def-x1} smaller.}
The second inequality in \eqref{eq:x0-x1} follows from \eqref{eq:def-x0}, mean value theorem, and the almost comparability of $\xi \rd_{\xi}\gmm_{\lmb_0}$ and $\gmm_{\lmb_{0}}$.

{
\begin{remark}\label{rem:x0-x1}
If $f'''$ vanishes at $0$, then we may choose $x_{0}$ be be larger. More precisely,
\begin{equation}\label{eq:x0-x1-nf}
\begin{split}
x_{0} = {c_{x_{0}} \eps(\lmb_{0})^{\frac{1}{n_{f}}}}
\end{split}
\end{equation} 
would work, where $n_{f}$ is the smallest natural number such that $f^{(n_{f}+2)}(0) \neq 0$.
\end{remark}
}

\medskip

\noindent \textit{Bicharacteristics associated with $\Phi$}. We consider the bicharacteristic ODE \begin{equation}\label{eq:ODE-bichar}
	\left\{
	\begin{aligned}
		\dot{X} &=  f'( X) \lmb_{0} \gmm_{\lmb_{0}}'(\Xi) , \\
		\dot{\Xi} &=   - f''( X) \lmb_{0} \gmm_{\lmb_{0}}(\Xi) + (\Gmm f'')( X) \lmb_{0} 
	\end{aligned}
	\right.
\end{equation} with initial data $(X_{0},\Xi_{0})$ satisfying $x_{0}<X_{0}<x_{1}$ and {$\lmb_{0} \leq \Xi_{0} = \rd\Phi(0,X_{0}) \leq (1+\eps) \lmb_{0}$}. 

\medskip

\noindent \textit{Control of $\Xi(t)$}.
 To obtain bounds on $\Xi(t)$, we introduce $\ollmb(t)$ and $\lmb(t)$, which are solutions to the ODEs  
\begin{equation}  \label{eq:ollmb-ODE}
\left\{
\begin{aligned} 
\dot{\ollmb}(t) &= -(1+\eps) f''( 0) \lmb_{0} \gmm_{\lmb_{0}}(\ollmb),\\
\ollmb(0) & = {(1+\eps)} \lmb_{0} ,
\end{aligned}
\right.
\end{equation}  and \begin{equation}  \label{eq:ullmb-ODE}
\left\{
\begin{aligned} 
\dot{\lmb}(t) &= -(1-\eps) f''( 0) \lmb_{0} \gmm_{\lmb_{0}}(\lmb),\\
\lmb(0) & = {\lmb_{0}} ,
\end{aligned}
\right.
\end{equation} respectively. 
We remind the reader that {$\eps = \eps(\lmb_{0})$ obeys \eqref{eq:eps-choice}}. We now claim that 
\begin{equation} \label{eq:Xi-lmb-compare}
	\lmb(t) \leq \Xi(t) \leq \ollmb(t).
\end{equation}

To begin with, observe that $\dot{X} = 0$ when $X = 0$, from which we see that no solution $(X(t), \Xi(t))$ with $X(0) > 0$ can traverse to the region $\set{X < 0}$. Note furthermore that in the region $\set{X > 0}$, we have $\dot{X} <0$. In conclusion, we see that $0 < X(t) < X(0) < x_{1}$ for any $t>0$.

Next, note that
\begin{equation*}
\begin{split}
-(1-\eps) f''(0) \lmb_{0} \gmm_{\lmb_{0}}(\lmb) 
\leq - f''( X) \lmb_{0} \gmm_{\lmb_{0}}(\Xi) + (\Gmm f'')( X) \lmb_{0} \leq -(1+\eps) f''(0) \lmb_{0} \gmm_{\lmb_{0}}(\ollmb)
\end{split}
\end{equation*} for any $X \in (0,x_{1})$ and $\lmb_{0} \leq \lmb \leq \Xi \leq \ollmb$. 
Indeed, by the first inequality of \eqref{eq:eps-choice}, we have $\eps \gmm_{\lmb_{0}}(\lmb) \to \infty$ as $\lmb_{0} \to \infty$. By taking  $\Lmb_0>0$ larger if necessary (in a way depending only on $\abs{f''(0)}^{-1} \nrm{\Gmm f''}_{L^{\infty}(0, x_{1})}$, $\gmm$ and $\eps$), we may guarantee that the preceding inequalities hold. Then, comparing the equations for $\dot{\Xi}$, $\dot{\lmb}$ and $\dot{\ollmb}$, we obtain \eqref{eq:Xi-lmb-compare}, as desired. 

We now prove the following lemma.
\begin{lemma} \label{lem:Xi-control} The following statements hold.
\begin{enumerate}
\item $\lmb(t) = M \lmb_{0}$ exactly at $t_{M} := \frac{1}{1-\eps} \abs{f''(0)}^{-1} \tau_{M}$, and $\lmb_{0} \le \lmb(t) \le M \lmb_{0}$ for $0 \le t \le t_{M}$.
\item As long as $\lmb(t) \leq M \lmb_{0}$, we have \begin{equation}\label{eq:ratio}
	\lmb(t) \leq \Xi(t) \leq 2 \lmb(t).
	\end{equation}  
\end{enumerate}
\end{lemma}

\begin{proof} It will be convenient to introduce $\overline{M}'(t)$ and $M'(t)$: $\overline{M}' := \ollmb/\lmb_0$ and $M' := \lmb/\lmb_0$. We have $\overline{M}'(t)\ge M'(t) \ge 1$. {we start with the identities \begin{equation} \label{eq:lmb-ollmb-t}
			\int_{\lmb_0 }^{M'\lmb_0} \frac{1}{\gmm_{\lmb_{0}}(\lmb)} \frac{\ud\lmb}{\lmb_{0}} = -f''(0)(1-\eps)t, \qquad
			\int_{(1+\eps) \lmb_0 }^{\overline{M}'\lmb_0} \frac{1}{\gmm_{\lmb_{0}}(\lmb)} \frac{\ud\lmb}{\lmb_{0}} = -f''(0)(1+\eps)t,
		\end{equation}
		which follow from \eqref{eq:ullmb-ODE} and \eqref{eq:ollmb-ODE}, respectively. Note that the LHS of the first identity is exactly $\tau_{M'}$; from this observation, the first statement easily follows. To prove \eqref{eq:ratio}, noting that $M'$ is strictly increasing in time, we shall split the proof into two time intervals, depending on whether $M'<2$ or $M'\ge2$. 

\medskip

\noindent \textbf{Case 1}. Assume that $M'(t)<2$. We are going to prove that $\overline{M}'(t) < 2M'(t)$ holds in this time interval by a bootstrap argument. Towards a contradiction, we may assume that there exists some $0<T<t_M$, such that $M'(t)<2, \overline{M}'(t) < 2M'(t)$ on $[0,T)$ and $\overline{M}'(T) = 2M'(T)$. 

We now restrict $t$ to $[0,T]$. Since $\lmb$ and $\ollmb$ are strictly increasing functions of time, we may consider $\ollmb = \ollmb(\lmb)$ (with some abuse of notation) and obtain \begin{equation*}
	\begin{split}
		\frac{\ud \ollmb}{\ud \lmb} = \frac{1+\eps}{1-\eps} \frac{\gmm( \lmb \cdot (\ollmb/\lmb) )}{\gmm(\lmb)} \le (1 + 3\eps) \left( \frac{\ollmb}{\lmb} \right)^{C}
	\end{split}
\end{equation*} where we may assume that $C>1$ since $2\lmb\ge\ollmb\ge\lmb$. Integrating in $\lmb$ and recalling that $\ollmb(\lmb_0)= (1+\eps)\lmb_0$, we obtain \begin{equation*}
\begin{split}
	\frac{1}{\ollmb^{C-1}} - \frac{1}{ ((1+\eps)\lmb_0)^{C-1}} \ge (1 + 3\eps) \left(  	\frac{1}{\lmb^{C-1}} -	\frac{1}{\lmb_0^{C-1}}       \right). 
\end{split}
\end{equation*} Applying this inequality at $t = T$, we have that $\ollmb = 2\lmb $ and obtain \begin{equation*}
\begin{split}
	\left(  (1 + 3\eps) - \frac{1}{2^{C-1}}  \right) \frac{1}{\lmb^{C-1}} \le \left( (1+3\eps)  - \frac{1}{(1+\eps)^{C-1}} \right) \frac{1}{\lmb_{0}^{C-1}} = O_{C}(\eps) \frac{1}{\lmb_0^{C-1}} . 
\end{split}
\end{equation*} and we can take $\Lmb_{0}$ larger so that $(1+3\eps) > \frac{1}{2^{C-1}}$. This gives $(\lmb/\lmb_0)^{C-1} \gg 1$, which is a contradiction to $M'(T) <2$. This guarantees \eqref{eq:ratio} in this case. 

\medskip

\noindent \textbf{Case 2}. When $M'(t)\ge 2$, we now combine both identities in \eqref{eq:lmb-ollmb-t} to obtain} \begin{equation*}
	\begin{split}
	\int_{\lmb_0 }^{(1+\eps) \lmb_0} \frac{1}{\gmm_{\lmb_{0}}(\lmb)} \frac{\ud\lmb}{\lmb_{0}} + \frac{2\eps}{1+\eps} \int_{(1+\eps) \lmb_0 }^{M'\lmb_0} \frac{1}{\gmm_{\lmb_{0}}(\lmb)} \frac{\ud\lmb}{\lmb_{0}} = \frac{1-\eps}{1+\eps} \int_{M'\lmb_0 }^{\overline{M}'\lmb_0} \frac{1}{\gmm_{\lmb_{0}}(\lmb)} \frac{\ud\lmb}{\lmb_{0}}. 
	\end{split}
	\end{equation*} Note that the LHS is bounded by \begin{equation*}
	\begin{split}
	&\int_{\lmb_0 }^{(1+\eps) \lmb_0} \frac{1}{\gmm_{\lmb_{0}}(\lmb)} \frac{\ud\lmb}{\lmb_{0}} + \frac{2\eps}{1+\eps} \int_{(1+\eps) \lmb_0 }^{M'\lmb_0} \frac{1}{\gmm_{\lmb_{0}}(\lmb)} \frac{\ud\lmb}{\lmb_{0}}  \\
	& \leq \frac{\eps}{\gmm_{\lmb_{0}}(\lmb_{0})} + 2 \eps \int_{\lmb_{0}}^{M' \lmb_{0}} \frac{1}{\gmm_{\lmb_{0}}(\lmb)} \frac{\ud \lmb}{\lmb_{0}} \\
	& \leq (2^{\bt_{0}}+1) \eps \int_{\lmb_{0}}^{2 \lmb_{0}} \frac{1}{\gmm_{\lmb_{0}}(\lmb)} \frac{\ud \lmb}{\lmb_{0}} 
	+ 2 \eps \int_{\lmb_{0}}^{M' \lmb_{0}} \frac{1}{\gmm_{\lmb_{0}}(\lmb)} \frac{\ud \lmb}{\lmb_{0}} \\
	& \leq (2^{\bt_{0}}+3) \eps \tau_{M'},
	\end{split}
	\end{equation*}
	where we used the quantitative slow-variance condition \eqref{eq:slow-var} for the second inequality.
	By the second inequality in \eqref{eq:eps-choice} and the assumption that $M' \leq M$, we have
	\begin{equation*}
	\begin{split}
	2^{\bt_{0}+1}(2^{\bt_{0}}+3) \eps \frac{\gmm_{\lmb_{0}}(M'\lmb_{0})}{M'} \tau_{M'} \le 1. 
	\end{split}
	\end{equation*} 
	As a result, \begin{equation*}
	\begin{split}
	\int_{M'\lmb_0 }^{\overline{M}'\lmb_0} \frac{1}{\gmm_{\lmb_{0}}(\lmb)} \frac{\ud\lmb}{\lmb_{0}} \leq
	2 (2^{\bt_{0}}+3) \eps \tau_{M'}
	 \leq 2^{-\bt_{0}} \frac{M'}{\gmm_{\lmb_{0}}(M' \lmb_{0})}
	 \leq \int_{M' \lmb_{0}}^{2M' \lmb_{0} } \frac{1}{\gmm_{\lmb_{0}}(\lmb)} \frac{\ud \lmb}{\lmb_{0}},
	\end{split}
	\end{equation*}
	where we used \eqref{eq:slow-var} for the last inequality. This implies $\overline{M}' \le 2M'$. Then, by \eqref{eq:Xi-lmb-compare}, \eqref{eq:ratio} follows. \qedhere\end{proof}
	
\medskip

\noindent \textit{{Control of $X(t)$}}.
{With Lemma~\ref{lem:Xi-control} for $\Xi(t)$}, we may obtain bounds on $X(t)$ using the conservation of the Hamiltonian:
\begin{equation*}
	p(X(t), \rd_{x} \Phi(X(t))) = p(X(0), \rd_{x} \Phi(X(0))).
\end{equation*} 
{This property holds since $(X(t), \rd_{x} \Phi(X(t)))$ solves the bicharacteristic ODEs \eqref{eq:ODE-bichar} (by the characteristic method, since $\Phi$ is a regular solution to the scalar first-order PDE \eqref{eq:HJE}), which is a Hamiltonian ODE with a time-independent Hamiltonian $p(x, \xi) = f'(x) \lmb_{0} \gmm_{\lmb_{0}}(\xi) + \Gmm f'(x) \lmb_{0}$.}
More specifically, we have 
\begin{equation*}
	f'(X(t))= \frac{f'(X(0)) \gmm_{\lmb_{0}}(\Xi(0)) + \Gmm f'(X(0)) - \Gmm f'(X(t))}{ \gmm_{\lmb_{0}}(\Xi(t))  }.
\end{equation*} 
Recall that $0 < X(t) < X(0)$. Therefore, we have
\begin{equation*}
\abs{\Gmm f'(X(0)) - \Gmm f'(X(t))} \leq \nrm{\Gmm f''}_{L^{\infty}(0, x_{1})} \abs{X(0) - X(t)} \leq \nrm{\Gmm f''}_{L^{\infty}(0, x_{1})} X(0). 
\end{equation*}
On the other hand, $\abs{f'(X(0))} \geq (1-\frac{\eps}{10}) \abs{f''(0)} X(0)$. Therefore, choosing $\Lmb$ sufficiently large depending on $\gmm$ and $\abs{f''(0)}^{-1} \Gmm f''$, we obtain  
\begin{equation}\label{eq:X-bounds}
\begin{split}
\frac{1}{2^{\bt_{0}+1}} \frac{X(0)\gmm_{\lmb_{0}}(\lmb_0)}{\gmm_{\lmb_{0}}(\lmb(t))}\le  X(t) \le   \frac{X(0)\gmm_{\lmb_{0}}(\lmb_0)}{\gmm_{\lmb_0}(\lmb(t) )} \quad \hbox{ for } \lmb_{0} \geq \Lmb.  
\end{split}
\end{equation} 
 
\medskip

\noindent \textit{{Control of $\rd_{x}^{2} \Phi(X(t))$}}. As we will see, a fundamental quantity for controlling the geometry of {nearby} characteristic curves is $\rd_{x}^{2} \Phi$. While this quantity may also be computed explicitly by differentiating \eqref{eq:HJE-s-dPhi}, it is more convenient to obtain an implicit formula by differentiating \eqref{eq:HJE}. Indeed, from \eqref{eq:HJE} we see that
\begin{equation*}
\rd_{x} p(x, \rd_{x} \Phi) + \rd_{\xi} p(x, \rd_{x} \Phi) \rd_{x}^{2} \Phi = 0,
\end{equation*}
so that
\begin{equation} \label{eq:HJE-s-ddPhi} 
\rd_{x}^{2} \Phi = - \frac{\rd_{x} p(x, \rd_{x} \Phi)}{\rd_{\xi} p(x, \rd_{x} \Phi)} = -\frac{ f''(x)\gmm_{\lmb_{0}}(\rd_x\Phi) + \Gmm f''(x) }{f'(x) \gmm_{\lmb_{0}}'(\rd_x\Phi)}.
\end{equation}
{By taking $\Lmb_0$ larger if necessary (in a way depending only on $\gmm$ and $\abs{f''(0)}^{-1} \Gmm f''$), we obtain the bound
\begin{equation*}
0 \leq - \rd_{x}^{2} \Phi(X(t)) \leq \frac{C}{x} \frac{\gmm_{\lmb_{0}}(\rd_{x} \Phi(X(t)))}{\rd_{\xi } \gmm_{\lmb_{0}}(\rd_{x} \Phi(X(t)))} 
\leq \frac{C}{x_{0}} \frac{(\gmm_{\lmb_{0}}(\Xi(t)))^{2}}{\gmm_{\lmb_{0}}(\lmb_{0}) \gmm_{\lmb_{0}}'(\Xi(t))}.
\end{equation*}

\begin{remark}
It is worth noting that $-\rd_{x}^{2} \Phi(X(t))$ remains finite as long as $(X(t), \Xi(t))$ exists. Geometrically, this reflects the fact that there are no focal points along each characteristic associated to our $\Phi$ constructed via separation of variables. This fact is not a-priori clear at the level of the ODE for $\rd_{x}^{2} \Phi(X(t))$, which is of Ricatti-type (see \eqref{eq:HJE-ricatti} below); hence, it should be seen as a benefit of our separation of variables approach. The above sharp bound for $-\rd_{x}^{2} \Phi(X(t))$ will serve as the basis for the sharp estimates for higher derivatives of $\Phi$ and the amplitude function $a$ (see \eqref{eq:transport-a}) along characteristics in Sections~\ref{subsec:transport}--\ref{subsec:transport2}.
\end{remark}

For the ensuing argument, we record separately the bound at the initial time $t =0$:
\begin{equation*}
0 \leq - \rd_{x}^{2} \Phi(x) \leq \frac{C}{x} \frac{\gmm_{\lmb_{0}}(\rd_{x} \Phi(x))}{\rd_{\xi } \gmm_{\lmb_{0}}(\rd_{x} \Phi(x))} 
\leq \frac{C}{x_{0}} \lmb_{0} (\log \lmb_{0})^{2}.
\end{equation*}

\medskip

\noindent \textit{Higher derivatives of $\Phi$ in $(x_{0}, x_{1})$}. 
We now consider higher derivatives of $\Phi$ in $(x_{0}, x_{1})$. For any $k \geq 2$, we claim that \begin{equation}\label{eq:Phi-higher}
\begin{split}
|\rd_x^{k}\Phi(x)|\le C_kx_0^{-k+1}\lmb_0(\log\lmb_0)^{2(k-1)}, \quad \hbox{ for } x_{0}<x<x_{1}.
\end{split}
\end{equation} 
Note that this corresponds to the bound for $\rd_{x}^{k} \Phi(X(0))$ at the initial time. We will obtain bounds for $\rd_{x}^{k} \Phi(X(t))$ with $t > 0$ later in Section~\ref{subsec:transport}, based on the bounds for $\rd_{x}^{2} \Phi(X(t))$ and \eqref{eq:Phi-higher}. }

We shall prove \eqref{eq:Phi-higher} with an induction in $k$; assuming it holds for $k = 2, \cdots, k_0+1$ for some $k_0$, we compute \begin{equation*}
\begin{split}
\rd_{x}^{k_0+2}\Phi & = \rd_{x}^{k_0} \left(  -\frac{ f''(x)\gmm_{\lmb_{0}}(\rd_x\Phi) + \Gmm f''(x) }{f'(x)\rd_{\xi}\gmm_{\lmb_{0}}(\rd_x\Phi)} \right) \\ 
& = -\sum_{\ell=0}^{k_0} C_{\ell,k_0} \rd_{x}^{\ell} \frac{1}{\rd_{\xi}\gmm_{\lmb_{0}}(\rd_{x}\Phi)} \left( \rd_{x}^{k_0-\ell} \left(\frac{\Gmm f''}{f' }\right) + \sum_{j=0}^{k_0-\ell} C_{j,k_0-\ell} \rd_{x}^{k_{0}-\ell-j} \left(\frac{f''}{f'} \right)\rd_{x}^{j} \gmm_{\lmb_{0}}(\rd_{x}\Phi)  \right),
\end{split}
\end{equation*} where $C_{\ell,k_0}$ and $C_{j,k_0-\ell}$ are combinatorial coefficients. To begin with, it is not difficult to see that \begin{equation*}
\begin{split}
\rd_{x}^{k_0-\ell} \left(\frac{\Gmm f''}{f' }\right) \le C_{k_0}x_{0}^{-k_0+\ell-1},\qquad \rd_{x}^{k_0-\ell-j} \left(\frac{f''}{f'} \right) \le C_{k_0} x_{0}^{-k_0+\ell+j-1}.
\end{split}
\end{equation*}
Next, we may expand $\rd_{x}^{j} \gmm_{\lmb_{0}}(\rd_{x}\Phi) $ using Fa\`{a} di Bruno's formula:  \begin{equation*}
\begin{split}
\rd_{x}^{j} \gmm_{\lmb_{0}}(\rd_{x}\Phi)  = \sum_{\mathfrak{a}: a_1+2a_2+\cdots+ja_j=j} \frac{j!}{a_1! 1!^{a_1} \cdots a_j!j!^{a_j}} (\rd^{a_1\cdots + a_j}_{\xi}\gmm_{\lmb_{0}} )(\rd_{x}\Phi) \prod_{b=1}^j (\rd_{x}^{1+b}\Phi)^{a_b} .
\end{split}
\end{equation*} In the above, the summation is over $j$-tuples $\mathfrak{a} = (a_1,\cdots,a_j)$ with non-negative integer entries satisfying $a_1+2a_2+\cdots+ja_j=j$. Using the induction assumption \eqref{eq:Phi-higher} and the ellipticity assumption for $\rd_{\xi} \gmm_{\lmb_{0}}$, we obtain that
\begin{equation*}
\begin{split}
|\rd_{x}^{j} \gmm_{\lmb_{0}}(\rd_{x}\Phi) | &\le x_0^{-j}(\log\lmb_0)^{2j} \sum_{\mathfrak{a}} C_{\mathfrak{a}} \lmb_0^{a_1\cdots + a_j} (\rd^{a_1\cdots + a_j}_{\xi}\gmm_{\lmb_{0}} )(\rd_{x}\Phi) \\
&\le C_{j}x_0^{-j}(\log\lmb_0)^{2j} \lmb_0 \rd_{\xi} \gmm_{\lmb_{0}}(\rd_{x}\Phi). 
\end{split}
\end{equation*} Similarly, we expand 
\begin{equation*}
\begin{split}
\rd_{x}^{\ell} \frac{1}{\rd_{\xi}\gmm_{\lmb_{0}}(\rd_{x}\Phi)} = \sum_{\mathfrak{a}: a_1+2a_2+\cdots+\ell a_\ell=\ell} \frac{\ell!}{a_1! 1!^{a_1} \cdots a_\ell!\ell!^{a_\ell}} (\rd^{a_1\cdots + a_\ell}_{\xi}\frac{1}{\rd_{\xi}\gmm_{\lmb_{0}}} )(\rd_{x}\Phi) \prod_{b=1}^\ell (\rd_{x}^{1+b}\Phi)^{a_b} .
\end{split}
\end{equation*} Again, using the ellipticity assumption on {$\xi\rd_{\xi}\gmm_{\lmb_{0}}$}, we see that \begin{equation*}
\begin{split}
\left|(\rd^{n}_{\xi}\frac{1}{\rd_{\xi}\gmm_{\lmb_{0}}} )(\rd_{x}\Phi)\right| \le C_{n}\lmb_0^{-n} \frac{1}{\rd_{\xi}\gmm_{\lmb_{0}}} (\rd_{x}\Phi). 
\end{split}
\end{equation*} This gives \begin{equation*}
\begin{split}
\left| \rd_{x}^{\ell} \frac{1}{\rd_{\xi}\gmm_{\lmb_{0}}(\rd_{x}\Phi)} \right|\le C_{\ell} x_0^{-\ell}(\log\lmb_0)^{2\ell}\frac{1}{\rd_{\xi}\gmm_{\lmb_{0}}}. 
\end{split}
\end{equation*} Collecting the bounds, we conclude that \begin{equation*}
\begin{split}
\left| \rd_{x}^{k_0+2} \Phi \right| \le C_{k_0} x_{0}^{-k_0-1}\lmb_0(\log\lmb_0)^{2(k_0+1)}
\end{split}
\end{equation*}  holds, which is just \eqref{eq:Phi-higher} with $k = k_{0}+2$. Therefore, we have arrived at the following
\begin{proposition}\label{prop:Phi-derivatives-time-indep}
	{Let $\rd_{x}\Phi$ be defined as in \eqref{eq:HJE-s-dPhi}. Then for any $k \geq 1$, $\rd_{x}^{k} \rd_{x} \Phi$} satisfy \begin{equation}\label{eq:rd-Phi-k-bound-time-indep}
	\begin{split}
	{\abs{\rd_{x}^{k} \rd_{x} \Phi(x)} \leq C_k x_{0}^{-k} (\log \lmb_{0})^{2k} \lmb_0 \qquad \hbox{ for } x_{0} < x < x_{1}.}
	\end{split}
	\end{equation}
\end{proposition}

\subsection{The case of a time-dependent background} \label{subsec:HJE-gen}
We now generalize our construction of the phase function to the case when $f'$ has time dependence. This time, we employ the method of characteristics to analyze \eqref{eq:HJE}. The explicit computation we performed in the time-independent case in Section~\ref{subsec:HJE-s} will serve as a very useful guide.

{\begin{remark}
A small modification of our scheme can handle $f'$ with a moving zero, by working with the new variable $\tilde{x} = x - \dgx(t)$ to fix the position of the zero. Then we need to add a term of the form $\dgx'(t) \rd_{x} \Phi$ in the Hamiltonian, which may be handled perturbatively (dominated by the first term in $p$) in the regime $M \ll \gmm(\lmb_{0})$. 
\end{remark}}
 
In this subsection, {\bf we assume that \eqref{eq:gf-condition-1}--\eqref{eq:gf-condition-3}, as well as \eqref{eq:gf-condition-diss-1}--\eqref{eq:gf-condition-diss-2}} (i.e., the hypotheses for Theorem~\ref{thm:norm-growth-diss}) {\bf hold}. 

\medskip

\noindent \textit{Choice of $\Phi(0, x)$.} 
Let $\Lmb$ be a large positive parameter to be fixed below. 
Introduce $\wpT$ so that
\begin{equation} \label{eq:wpT-choice}
	0 < \wpT \leq 1, \qquad \nrm{\rd_{t} f''(t', 0)}_{L^{\infty}_{t'}(0, \frac{100}{99} t_{f}(\wpT))} t_{f}(\wpT) \leq \frac{1}{10} \abs{f''(0, 0)},
\end{equation}
where we remind the reader that $t_{f}(\tau)$ is defined by $\tau = \int_{0}^{t_{f}(\tau)} -f''(t', 0) \, \ud t'$. We then introduce $c_{x_{0}} > 0$ such that
\begin{equation} \label{eq:cx0-choice}
	\sup_{t \in [0, \frac{100}{99} t_{f}(\wpT)]}\nrm{f'''(t, x)}_{L^{\infty}(0, 2 c_{x_{0}} \eps(\lmb_{0}))} c_{x_{1}} \leq \frac{1}{10} \abs{f''(t, 0)},
\end{equation}
where $\eps(\lmb_{0})$ is an nonincreasing function of $\lmb_{0} \in [\Xi_{0}, \infty)$ obeying \eqref{eq:eps-choice} for $\lmb_{0} \geq \Lmb$. We emphasize that {\bf we retain the freedom to shrink $\wpT, c_{x_{1}} > 0$ and enlarge $\Lmb$ until the end of this section.}

We look for the solution $\Phi$ to \eqref{eq:HJE} with the initial data
\begin{equation} \label{eq:HJE-id}
	\rd_{x} \Phi(0, x) = \gmm_{\lmb_{0}}^{-1} \left(\frac{ f'(0, x_{1} ) \gmm_{\lmb_{0}}((1+\tfrac{1}{2}\eps)\lmb_{0}) + \Gmm f'(0, x) - \Gmm f'(0,x_1)}{- f'(0, x)}\right)
\end{equation}
on $\set{t = 0} \times (x_{0}, x_{1})$, where
\begin{equation} \label{eq:def-x1-x0}
{x_{0} = c_{x_{0}} \eps(\lmb_{0}), \quad
x_{1} = x_{0} + \Dlt x_{0}, \quad
\frac{\Dlt x_{0}}{x_{0}} = \frac{1}{10}\left(\frac{\gmm_{\lmb_{0}}((1+\tfrac{3}{4} \eps(\lmb_{0})) \lmb_0)}{\gmm_{\lmb_{0}}((1+\tfrac{1}{2} \eps(\lmb_{0})) \lmb_0)}- 1\right),}
\end{equation}
and $c_{x_{0}}$ obeys \eqref{eq:cx0-choice}. Note that \eqref{eq:HJE-id}--\eqref{eq:def-x1-x0} are precisely \eqref{eq:HJE-s-dPhi}, \eqref{eq:x0-x1} with $f'$ frozen at $t = 0$. The relevance of this choice will be evident in the estimate for $\rd_{x}^{2} \Phi$ below. As before, observe that for $\Lmb$ sufficiently large,
\begin{equation*}
	\lmb_{0} \leq \rd_{x} \Phi(0, x) \leq (1+\eps) \lmb_{0} \quad \hbox{ on } \set{t = 0} \times (x_{0}, x_{1}) \hbox{ for } \lmb_{0} \geq \Lmb.
\end{equation*}

\medskip

\noindent \textit{Setup for the method of characteristics.} 
We now set up the method of characteristics for \eqref{eq:HJE}. Let $X(t)$ be a characteristic curve parametrized by $t$, and introduce $\Xi(t) = \rd_{x} \Phi(t, X(t))$. The bicharacteristic ODEs for $(X, \Xi)$ read as follows:
\begin{align}
	\dot{X} &= \rd_{\xi} p(t, X, \Xi) = f'(t, X) \lmb_{0} \gmm_{\lmb_{0}}'(\Xi) ,\label{eq:bichar-time-dependent-X} \\
	\dot{\Xi} &=  - \rd_{x} p(t, X, \Xi) = - f''(t, X) \lmb_{0} \gmm_{\lmb_{0}}(\Xi) + (\Gmm f'')(t, X) \lmb_{0}.\label{eq:bichar-time-dependent-Xi} 
\end{align}
Observe that these equations constitute the Hamiltonian ODEs corresponding to the time-dependent Hamiltonian $p(t, x, \xi)$. In the following, it will be always assumed that the initial data satisfy $x_0<X(0)<x_1$ and $\Lmb_0\le \lmb_0 \le \Xi(0) \leq (1+\eps) \lmb_{0}$. 

\medskip
\noindent \textit{Continuation criterion and set up for the bootstrap argument.}  
Before we continue, we briefly discuss how our bootstrap argument for constructing and estimating $\Phi$ is set up. Given $\Phi$, denote by $X(t; \ubr{x})$ the solution to $\dot{X} = \rd_{\xi} p(t, X, \rd_{x} \Phi(t, X))$. Observe that the method of characteristics guarantees the existence of $\Phi(t, x)$ for $X(t; x_{0}) < x < X(t; x_{1})$, initially for some short time interval $[0, t^{\ast})$. In what follows, we will prove $\nrm{\rd_{x}^{2} \Phi(t, x)}_{L^{\infty}(X(t; (x_{0}, x_{1})))}$ is uniformly bounded for $t < t^{\ast}$, provided that $t^{\ast} < \frac{100}{99} t_{f}(\tau_{M})$ (the time scale in Theorem~\ref{thm:norm-growth-diss}). Then $\ubr{x} \mapsto X(t^{\ast}; \ubr{x})$ is a bi-Lipschitz isomorphism, and $X(t), \Xi(t)$ and $\Phi(t, X(t))$ exist on a longer time interval; this allows us to set up a continuous induction (bootstrap) argument for constructing and estimating $\Phi$.

With such details in mind, in what follows, for the simplicity of exposition, {\bf we will simply assume the existence of $X(t), \Xi(t)$ and $\Phi(t, X(t))$ for $0 \leq t \leq \frac{100}{99} t_{f}(\tau_{M})$} and {\bf demonstrate how to derive bounds for $X(t), \Xi(t)$ and $\rd_{x}^{2} \Phi(t, X(t))$}.

\medskip

\noindent \textit{Control of $\Xi(t)$.}
To control $\Xi(t)$, we employ a comparison argument that is similar to the steady case. We introduce $\lmb(t)$ and $\ollmb(t)$, which are solutions to the ODEs
\begin{equation*}
\left\{
\begin{aligned}
	\dot{\lmb}(t) &= - (1-\eps) f''(t, 0) \lmb_{0} \gmm_{\lmb_{0}}(\lmb), \\
	\lmb(0) &= \lmb_{0},
\end{aligned}
\right.
\end{equation*}
and
\begin{equation*}
\left\{
\begin{aligned}
	\dot{\ollmb}(t) &= - (1+\eps) f''(t, 0) \lmb_{0} \gmm_{\lmb_{0}}(\ollmb), \\
	\ollmb(0) &= (1+\eps)\lmb_{0}.
\end{aligned}
\right.
\end{equation*}
Recall also that $t_{f}(\tau)$ is defined by the relation $\int_{0}^{t_{f}(\tau)} (-f'')(t, 0) \, \ud t = \tau$. For $0 \leq t \leq \frac{100}{99} t_{f}(\wpT)$, we claim that, for $\Lmb \leq \lmb \leq \Xi(t) \leq \ollmb$ and $0 < X(t) < x_{1}$,
\begin{align*}
-(1-\eps) f''(t, 0) \lmb_{0} \gmm_{\lmb_{0}}(\lmb)
&\leq - f''(t, X(t)) \lmb_{0} \gmm_{\lmb_{0}}(\Xi(t)) + (\Gmm f'')(t, X(t)) \lmb_{0} \\
&\leq -(1+\eps) f''(t, 0) \lmb_{0} \gmm_{\lmb_{0}}(\ollmb),
\end{align*}
provided that $\Lmb$ is sufficiently large depending on $\gmm$ and $\abs{f''(0, 0)} \Gmm f''$.
The proof is similar to the steady case using \eqref{eq:wpT-choice}, \eqref{eq:cx0-choice} and \eqref{eq:def-x1-x0}. The first inequality implies, in particular, that $\dot{\Xi} > 0$ for $\lmb_{0} \leq \Lmb$. Observing furthermore that $\dot{X} < 0$ but $X(t) > 0$, and comparing the ODEs for $\lmb(t)$, $\Xi(t)$ and $\ollmb(t)$, we arrive at
\begin{equation*}
	\lmb(t) \leq \Xi(t) \leq \ollmb(t) \qquad \hbox{ for } 0 \leq t \leq \frac{100}{99} t_{f}(\wpT).
\end{equation*}
The following analogue of Lemma~\ref{lem:Xi-control} holds:
\begin{lemma} \label{lem:Xi-control-time-dep}
Let $\lmb_{0}$, $M$ and $\tau_{M}$ obey \eqref{eq:gf-condition-1}--\eqref{eq:gf-condition-3} as well as $\tau_{M} \leq \wpT$. Then the following statements hold.
\begin{enumerate}
\item $\lmb(t) = M \lmb_{0}$ exactly at $t_{M} := \frac{1}{1-\eps} t_{f}(\tau_{M})$, and $\lmb_{0} \leq \lmb(t) \leq M \lmb_{0}$ for $0 \leq t \leq t_{M}$. 
\item As long as $\lmb(t) \leq M \lmb_{0}$, we have
\begin{equation}\label{eq:rd-Phi-bounds-time-dep}
	\lmb(t) \leq \Xi(t) =\rd_{x} \Phi(t, X(t)) \leq 2\lmb(t).
\end{equation}
\end{enumerate}
\end{lemma}
\begin{proof}
Using the identities
\begin{equation*}
	\int_{\lmb_{0}}^{M' \lmb_{0}} \frac{1}{\gmm_{\lmb_{0}}(\lmb)} \frac{\ud \lmb}{\lmb_{0}} = (1-\eps) t_{f}(t), \qquad
	\int_{\lmb_{0}}^{\br{M}' \lmb_{0}} \frac{1}{\gmm_{\lmb_{0}}(\lmb)} \frac{\ud \lmb}{\lmb_{0}} = (1+\eps) t_{f}(t)
\end{equation*}
in lieu of \eqref{eq:lmb-ollmb-t} (where $M' \lmb_{0} = \lmb(t)$ and $\br{M}' \lmb_{0} = \ollmb(t)$ as before), the proof of this lemma proceeds exactly as in that of Lemma~\ref{lem:Xi-control}. We omit the details.
\end{proof}

\medskip

\noindent \textit{Control of $X(t)$.} 
Next, we aim to obtain a bound for $X(t)$ that is comparable to \eqref{eq:X-bounds} in the steady case. Instead of a simple argument involving the Hamiltonian (which is exactly preserved in the steady case), here we need an analysis of the bicharacteristic ODEs.

Let $0 \leq t \leq \frac{1}{1-\eps} t_{f}(\tau_{M})$ with $\tau_{M} \leq \wpT$. Recall the soft fact that $0 < X(t) \leq X(0) < x_{1}$. Using the bicharacteristic ODEs, we may write
\begin{align*}
\frac{\dot{X}(t)}{X(t)} &= \frac{f'(t, X(t))}{X(t)} \lmb_{0} \gmm_{\lmb_{0}}'(\Xi(t)) \\
&= - \frac{\abs{f'(t, X(t))}}{\abs{f''(t, X(t))} X(t)} \frac{\gmm_{\lmb_{0}}'(\Xi(t))}{\gmm_{\lmb_{0}}(\Xi(t))} 
\left(\frac{\abs{f''(t, X(t))} \lmb_{0} \gmm_{\lmb_{0}}(\Xi(t))}{\abs{f''(t, X(t))} \lmb_{0} \gmm_{\lmb_{0}}(\Xi(t)) + \Gmm f''(t, X(t)) \lmb_{0} }\right) \dot{\Xi}(t) \\
&= - \frac{\abs{f'(t, X(t))}}{\abs{f''(t, X(t))} X(t)} \frac{\gmm_{\lmb_{0}}'(\Xi(t))}{\gmm_{\lmb_{0}}(\Xi(t))} 
\left(1 - \frac{\Gmm f''(t, X(t))}{\abs{f''(t, X(t))} \gmm_{\lmb_{0}}(\Xi(t)) + \Gmm f''(t, X(t))}\right) \dot{\Xi}(t).
\end{align*}
By \eqref{eq:wpT-choice} and \eqref{eq:cx0-choice}, for $\lmb_{0} \geq \Lmb$ with $\Lmb$ sufficiently large depending on $\gmm$ and $\abs{f''(0, 0)} \Gmm f''$ on $[0, \frac{100}{99} t_{f}(\wpT)] \times (0, x_{1})$, we have
\begin{equation*}
\abs*{\frac{\abs{f'(t, X(t))}}{\abs{f''(t, 0)} X(t)} - 1} \aleq C_{0} X(t), \qquad
\abs*{\frac{\Gmm f''(t, X(t))}{\abs{f''(t, X(t))} \gmm_{\lmb_{0}}(\Xi(t)) + \Gmm f''(t, X(t))} } \aleq
C_{1} \frac{1}{\gmm_{\lmb_{0}}(\Xi(t))},
\end{equation*}
where the implicit constants are absolute, $C_{0} = \abs{f''(0, 0)}^{-1} \nrm{f'''}_{L^{\infty}([0, \frac{100}{99} t_{f}(\wpT)] \times (0, x_{1}))}$ and $C_{1} = \abs{f''(0, 0)}^{-1} \nrm{\Gmm f''}_{L^{\infty}([0, \frac{100}{99} t_{f}(\wpT)] \times (0, x_{1}))}$. By integration in time, it follows that
\begin{align*}
	\log \frac{X(t)}{X(0)} - C C_{0}(X(0) - X(t)) &\leq \log \frac{\gmm_{\lmb_{0}}(\Xi(0))}{\gmm_{\lmb_{0}}(\Xi(t))} + C C_{1} \left(\frac{1}{\gmm_{\lmb_{0}}(\Xi(0))} - \frac{1}{\gmm_{\lmb_{0}}(\Xi(t))} \right), \\
	\log \frac{\gmm_{\lmb_{0}}(\Xi(0))}{\gmm_{\lmb_{0}}(\Xi(t))} - C C_{1} \left(\frac{1}{\gmm_{\lmb_{0}}(\Xi(0))} - \frac{1}{\gmm_{\lmb_{0}}(\Xi(t))} \right) &\leq \log \frac{X(t)}{X(0)} + C C_{0}(X(0) - X(t)),
\end{align*}
for some absolute constant $C > 0$. By the monotonicity properties of $\gmm_{\lmb_{0}}$, $X$ and $\Xi$, observe that all non-logarithmic terms may be made arbitrarily small by taking $c_{x_{0}} > 0$ small and $\Lmb$ large depending on $\gmm$,  $C_{0}$ and $C_{1}$. Finally, by the slow-variance of $\gmm_{\lmb_{0}}$ and Lemma~\ref{lem:Xi-control-time-dep}, we may replace $\gmm_{\lmb_{0}}(\Xi(0))$ and $\gmm_{\lmb_{0}}(\Xi(t))$ by $\gmm_{\lmb_{0}}(\lmb_{0})$ and $\gmm_{\lmb_{0}}(\lmb(t))$, respectively.

In conclusion, for $0 \leq t \leq \frac{1}{1-\eps} t_{f}(\tau_{M})$, $\tau_{M} \leq \wpT$ and $\lmb_{0} \geq \Lmb$, we obtain
\begin{equation}\label{eq:X-bounds-time-dep}
\begin{split}
\frac{1}{{2^{\bt_{0}+1}}} \frac{X(0)\gmm_{\lmb_{0}}(\lmb_0)}{\gmm_{\lmb_{0}}(\lmb(t))}\le  X(t) \le {2^{\bt_{0}+1}} \frac{X(0)\gmm_{\lmb_{0}}(\lmb_0)}{\gmm_{\lmb_0}(\lmb(t))}. 
\end{split}
\end{equation}
as long as $c_{x_{0}} > 0$ is sufficiently small and $\Lmb$ is sufficiently large depending on $\gmm$, $\abs{f''(0, 0)}^{-1} f'''$ and $\abs{f''(0, 0)}^{-1} \Gmm f''$.

\medskip

\noindent \textit{Control of $\rd_{x}^{2} \Phi(t, X(t))$.} Next, we turn to $\rd_{x}^{2} \Phi(t, X(t))$. Note that
\begin{equation*}
	\frac{\ud}{\ud t} \rd_{x}^{2} \Phi(t, X(t))
	= \rd_{t} \rd_{x}^{2} \Phi(t, X(t)) + \dot{X}(t) \rd_{x}^{3} \Phi(t, X(t)).
\end{equation*}
By the equation
\begin{equation*}
	\rd_{t} \rd_{x}^{2} \Phi (t, x) + \rd_{\xi} p(t, x, \rd_{x} \Phi) \rd_{x}^{3} \Phi(t, x) + \rd_{x}^{2} p(t, x, \rd_{x} \Phi) + 2 \rd_{x} \rd_{\xi} p(x, \rd_{x} \Phi) \rd_{x}^{2} \Phi(t, x) + \rd_{\xi}^{2} p(t, x, \rd_{x} \Phi) ( \rd_{x}^{2} \Phi )^{2} = 0,
\end{equation*}
it follows that
\begin{equation} \label{eq:HJE-ricatti}
\begin{aligned}
	\frac{\ud}{\ud t} \rd_{x}^{2} \Phi(t, X(t))
	&= - \rd_{x}^{2} p(t, X(t), \Xi(t)) - 2  \rd_{x} \rd_{\xi} p(t, X(t), \Xi(t)) \rd_{x}^{2} \Phi(t, X(t)) \\
	&\peq - \rd_{\xi}^{2} p(t, X(t), \Xi(t)) ( \rd_{x}^{2} \Phi (t, X(t)))^{2}.
\end{aligned}
\end{equation}
{Note that \eqref{eq:HJE-ricatti} is a Ricatti-type ODE. As is well-known, this ODE is prone to blowing up in finite time, and indeed this is the reflection of the possibility of formation of focal or conjugate points along bicharacteristics in the Hamilton--Jacobi formulation. While it is possible to play with the initial data for $\Phi$ and analyze \eqref{eq:HJE-ricatti} directly to obtain control on $\rd_{x}^{2} \Phi(t, X(t))$, it is not clear how to obtain control on a sufficiently long time interval needed for our purposes.}

Instead, motivated by \eqref{eq:HJE-s-ddPhi}, we perform the following inspired change of variables:
\begin{equation} \label{eq:def-h}
	h(t) := \frac{\rd_{\xi} p(t, X(t), \Xi(t))}{\rd_{x} p(t, X(t), \Xi(t))} \rd_{x}^{2} \Phi (t, X(t)) + 1.
\end{equation}
Our choice of the initial data \eqref{eq:HJE-id} is such that $h$ is initially zero. More precisely, by \eqref{eq:HJE-id}, we have
\begin{equation*}
\rd_{x} f'(0, 0) x_{1} \lmb_{0} \gmm_{\lmb_{0}}(\lmb_{0}) + p(0, x, \rd_{x} \Phi(0, x)) = 0,
\end{equation*}
hence by differentiating in $x$, it follows that
\begin{equation} \label{eq:HJE-h-id}
h(0) = 0 \qquad \hbox{ for } X(0) = x \in (x_{0}, x_{1}) \, \mbox{ and }\,  \Xi(0) = \rd_{x} \Phi(0, x).
\end{equation}
We shall show that the variable $h$ obeys the following remarkably nice evolution equation:
\begin{lemma} \label{lem:HJE-h}
	We have 
	\begin{equation} \label{eq:HJE-h}
	\dot{h}
	= s - (q + r + s) h + q h^{2},
	\end{equation}
	where
	\begin{equation} \label{eq:def-s-r-q}
s := 	- \frac{\rd_{t} \rd_{\xi} p}{\rd_{\xi} p}  
+\frac{\rd_{t} \rd_{x} p}{\rd_{x} p}, \qquad r := \rd_{x}^{2} p \frac{\rd_{\xi} p}{\rd_{x} p}, \qquad q := -\rd_{\xi}^{2} p \frac{\rd_{x} p}{\rd_{\xi} p},
\end{equation}
which are all evaluated at $(t, x, \xi) = (t, X(t), \Xi(t))$.
\end{lemma}

Note that $s = 0$ when $p$ is time-independent. By \eqref{eq:HJE-h-id}, it is then clear that when $p$ is time independent, $h$ remains zero; this is precisely the computation \eqref{eq:HJE-s-ddPhi} in the steady case! The advantage of \eqref{eq:HJE-h} over \eqref{eq:HJE-ricatti} is that we may now incorporate the time dependence of $p$ in a perturbative manner in the estimate for $\rd_{x}^{2} \Phi(t, X(t))$.

\begin{proof} To simplify the notation, we omit the dependence of various derivatives of $p$ on $(t, X(t), \Xi(t))$, as well as the dependence of $\rd_{x}^{2} \Phi$ on $(t, X(t))$. We begin by computing
\begin{align*}
	\frac{\ud}{\ud t} \frac{\rd_{\xi} p}{\rd_{x} p}
	&= \frac{\rd_{t} \rd_{\xi} p + \rd_{x} \rd_{\xi} p \dot{X} + \rd_{\xi}^{2} p \dot{\Xi}}{\rd_{x} p} - \frac{(\rd_{t} \rd_{x} p + \rd_{x}^{2} p \dot{X} + \rd_{x} \rd_{\xi} p \dot{\Xi}) \rd_{\xi} p}{(\rd_{x} p)^{2}} \\
	&= \left(\frac{\rd_{t} \rd_{\xi} p}{\rd_{\xi} p} + \rd_{x} \rd_{\xi} p - \rd_{\xi}^{2} p \frac{\rd_{x} p}{\rd_{\xi} p} - \frac{\rd_{t} \rd_{x} p}{\rd_{x} p} - \rd_{x}^{2} p \frac{\rd_{\xi} p}{\rd_{x} p} + \rd_{x} \rd_{\xi} p\right) \frac{\rd_{\xi} p}{\rd_{x} p} \\
	&= \left(2\rd_{x} \rd_{\xi} p - \rd_{\xi}^{2} p \frac{\rd_{x} p}{\rd_{\xi} p} - \rd_{x}^{2} p \frac{\rd_{\xi} p}{\rd_{x} p} + \frac{\rd_{t} \rd_{\xi} p}{\rd_{\xi} p}  - \frac{\rd_{t} \rd_{x} p}{\rd_{x} p} \right) \frac{\rd_{\xi} p}{\rd_{x} p}.
\end{align*}
Therefore,
\begin{align*}
	\frac{\ud}{\ud t} h
	&= \left(\frac{\ud}{\ud t} \frac{\rd_{\xi} p}{\rd_{x} p}\right) \rd_{x}^{2} \Phi(t, X(t))
	+ \frac{\rd_{\xi} p}{\rd_{x} p} \frac{\ud}{\ud t} \rd_{x}^{2} \Phi(t, X(t)) \\
	&= \left(2\rd_{x} \rd_{\xi} p - \rd_{\xi}^{2} p \frac{\rd_{x} p}{\rd_{\xi} p} - \rd_{x}^{2} p \frac{\rd_{\xi} p}{\rd_{x} p} + \frac{\rd_{t} \rd_{\xi} p}{\rd_{\xi} p}  - \frac{\rd_{t} \rd_{x} p}{\rd_{x} p} \right) \frac{\rd_{\xi} p}{\rd_{x} p} \rd_{x}^{2} \Phi(t, X(t)) \\
	&\peq
	- \rd_{x}^{2} p \frac{\rd_{\xi} p}{\rd_{x} p} 
	- 2 \rd_{x} \rd_{\xi} p \frac{\rd_{\xi} p}{\rd_{x} p} \rd_{x}^{2} \Phi(t, X(t))
	- \rd_{\xi}^{2} p \frac{\rd_{\xi} p}{\rd_{x} p} (\rd_{x}^{2} \Phi(t, X(t)))^{2}.
\end{align*}
Observe the cancellation of the term $2\rd_{x} \rd_{\xi} p \frac{\rd_{\xi} p}{\rd_{x} p} \rd_{x}^{2} \Phi(t, X(t))$. Writing $\frac{\rd_{\xi} p}{\rd_{x} p} \rd_{x}^{2} \Phi(t, X(t)) = h - 1$, we moreover observe that
\begin{align*}
	\frac{\ud}{\ud t} h
	&= 	\left(- \rd_{\xi}^{2} p \frac{\rd_{x} p}{\rd_{\xi} p} - \rd_{x}^{2} p \frac{\rd_{\xi} p}{\rd_{x} p} + \frac{\rd_{t} \rd_{\xi} p}{\rd_{\xi} p}  - \frac{\rd_{t} \rd_{x} p}{\rd_{x} p} \right) (h-1) \\
	&\peq
	- \rd_{x}^{2} p \frac{\rd_{\xi} p}{\rd_{x} p} 
	- \rd_{\xi}^{2} p \frac{\rd_{x} p}{\rd_{\xi} p} (h-1)^{2} \\
	&= \left(\rd_{\xi}^{2} p \frac{\rd_{x} p}{\rd_{\xi} p} 
	+ \rd_{x}^{2} p \frac{\rd_{\xi} p}{\rd_{x} p} 
	- \frac{\rd_{t} \rd_{\xi} p}{\rd_{\xi} p}  
	+\frac{\rd_{t} \rd_{x} p}{\rd_{x} p} \right) - \rd_{x}^{2} p \frac{\rd_{\xi} p}{\rd_{x} p} - \rd_{\xi}^{2} p \frac{\rd_{x} p}{\rd_{\xi} p} \\
	&\peq \left(- \rd_{\xi}^{2} p \frac{\rd_{x} p}{\rd_{\xi} p} - \rd_{x}^{2} p \frac{\rd_{\xi} p}{\rd_{x} p} + \frac{\rd_{t} \rd_{\xi} p}{\rd_{\xi} p}  - \frac{\rd_{t} \rd_{x} p}{\rd_{x} p} + 2 \rd_{\xi}^{2} p \frac{\rd_{x} p}{\rd_{\xi} p} \right)h 
	- \rd_{\xi}^{2} p \frac{\rd_{x} p}{\rd_{\xi} p} h^{2},
\end{align*}
from which \eqref{eq:HJE-h} follows.\end{proof} 

We now analyze \eqref{eq:HJE-h} to obtain a uniform control on $h$ under the additional assumption \eqref{eq:gf-condition-diss-2} compared to the steady case. Then, by \eqref{eq:def-h}, $\rd_{x}^{2} \Phi(t, X(t))$ would enjoy similar estimates as in the steady case.

\begin{proposition}\label{prop:h}
Let $\lmb_{0}$, $M$ and $\tau_{M}$ satisfy \eqref{eq:gf-condition-1}--\eqref{eq:gf-condition-3}, $\tau_{M} \leq \wpT \leq 1$ as well as \eqref{eq:gf-condition-diss-2}. Then for any $\dlt_{5} > 0$, by taking $c_{x_{0}}$ smaller and $\Lmb$ larger depending on $\dlt_{5}$, $\gmm$ and $f$ (for the precise dependence, see $C_{0}$, $C_{0}'$ and $C_{1}$ in the proof), we have
\begin{equation} \label{eq:h-small}
	\abs{h(t)} \leq \dlt_{5} \qquad \hbox{ for } 0 \leq t \leq \frac{1}{1-\eps} t_{f}(\tau_{M}) \hbox{ and } \lmb_{0} \geq \Lmb.
\end{equation}
\end{proposition}
\begin{proof} 
In this proof, all implicit constants are absolute unless otherwise stated.
Consider the ODE \eqref{eq:HJE-h}. Note that $q$ can be alternatively written as
\begin{equation} \label{eq:HJE-h-q}
q
= \frac{\rd_{\xi}^{2} \gmm_{\lmb_{0}}(\Xi(t))}{\rd_{\xi} \gmm_{\lmb_{0}}(\Xi(t))} (-\rd_{x} p)(t, X(t), \Xi(t))
=  \left(\rd_{\xi} \log \gmm_{\lmb_{0}}'\right) (\Xi(t)) \dot{\Xi}(t).
\end{equation} We begin by expanding the terms in \eqref{eq:HJE-h} that involve $\rd_{t}$:
\begin{align*}
-s=\frac{\rd_{t} \rd_{\xi} p}{\rd_{\xi} p} - \frac{\rd_{t} \rd_{x} p}{\rd_{x} p} 
&= \frac{\rd_{t} f'}{f'} - \frac{\rd_{t} \rd_{x} f' \gmm_{\lmb_{0}}(\xi) + \rd_{t} \Gmm f''}{f'' \gmm_{\lmb_{0}}(\xi) + \Gmm f''} \\
&= \frac{\rd_{t} f'}{f'} - \frac{\rd_{t} f''}{f''}
+ \frac{\rd_{t} f''}{f''} \frac{ \Gmm f''}{f'' (\gmm_{\lmb_{0}}(\xi) + \Gmm f'')} + \frac{\rd_{t} \Gmm f''}{f'' (\gmm_{\lmb_{0}}(\xi) + \Gmm f'')}.
\end{align*}
Observe that, in the first two terms, the terms of order $O(1)$ cancel and we are left with $O(x)$. On the other hand, the remaining two terms are bounded by $O(\frac{1}{\gmm_{\lmb_{0}}(\xi)})$. Combined with Lemma~\ref{lem:Xi-control-time-dep} and \eqref{eq:X-bounds-time-dep}, we obtain (for $\lmb_{0} \geq \Lmb$ sufficiently large)
\begin{align}
	\abs{s} &\aleq \abs{f''(0, 0)} \left( C_{1} X + \frac{C_{0}}{\gmm_{\lmb_{0}}(\Xi(t))} \right) 
	\aleq \abs{f''(0, 0)} \frac{\gmm_{\lmb_{0}}(\lmb_{0})}{\gmm_{\lmb_{0}}(\lmb(t))} \left(C_{1} c_{x_{0}} \eps(\lmb_{0}) + \frac{C_{0}}{\gmm_{\lmb_{0}}(\lmb_{0})}\right) \notag \\
	&\aleq \frac{\gmm_{\lmb_{0}}(\lmb_{0})}{\gmm_{\lmb_{0}}(\lmb(t))^{2}} \frac{\dot{\lmb}(t)}{\lmb_{0}} \left(C_{1} c_{x_{0}} \eps(\lmb_{0}) + \frac{C_{0}}{\gmm_{\lmb_{0}}(\lmb_{0})}\right), \label{eq:HJE-h-s}
\end{align}
where 
\begin{align*}
C_{0} &= \abs{f''(0, 0)}^{-2} \nrm{\rd_{t} \Gmm f''}_{L^{\infty}([0, \tfrac{1}{1-\eps} t_{f}(\tau_{M})] \times (0, 2 c_{x_{0}} \eps))} + \abs{f''(0, 0)}^{-1} \nrm{\Gmm f''}_{L^{\infty}([0, \tfrac{1}{1-\eps} t_{f}(\tau_{M})] \times (0, 2 c_{x_{0}} \eps))}, \\
C_{1} &= \abs{f''(0, 0)}^{-2} \nrm{\rd_{t} f'''}_{L^{\infty}([0, \tfrac{1}{1-\eps} t_{f}(\tau_{M})] \times (0, 2 c_{x_{0}} \eps))} + \abs{f''(0, 0)}^{-1} \nrm{f'''}_{L^{\infty}([0, \tfrac{1}{1-\eps} t_{f}(\tau_{M})] \times (0, 2 c_{x_{0}} \eps))}.
\end{align*}
Next, for the term $r$, we again use Lemma~\ref{lem:Xi-control-time-dep} and \eqref{eq:X-bounds-time-dep} to estimate
\begin{equation} \label{eq:HJE-h-r}
	\abs{r} \aleq C_{1}' X(t) \frac{\gmm_{\lmb_{0}}'(\lmb(t))}{\gmm_{\lmb_{0}}(\lmb(t))} \dot{\lmb(t)}
	\aleq C_{1}' c_{x_{0}} \eps(\lmb_{0}) \frac{\gmm_{\lmb_{0}}'(\lmb(t)) \gmm_{\lmb_{0}}(\lmb_{0}) }{\gmm_{\lmb_{0}}(\lmb(t))^{2}} \dot{\lmb}(t),
\end{equation}
where
\begin{equation*}
	C_{1}' = C_{1} + \gmm_{\lmb_{0}}(\lmb_{0})^{-1} \abs{f''(0, 0)}^{-1} \nrm{\Gmm f'''}_{L^{\infty}([0, \tfrac{1}{1-\eps} t_{f}(\tau_{M})] \times (0, 2 c_{x_{0}} \eps))}.
\end{equation*}

We are now ready to set up a bootstrap argument (continuous induction in time) to prove \eqref{eq:h-small}. Assume, without any loss of generality, that $\dlt_{5} < \dlt_{1}$. Initially, recall that $h(0) = 0$. As a bootstrap assumption, assume that $\abs{h(t)} < \dlt_{1}$ on some time interval $[0, t^{\ast}]$. By the method of integrating factor, \eqref{eq:HJE-h-q} and the bootstrap assumption, we estimate, for any $t \in [0, t^{\ast}]$,
\begin{align*}
	\abs{h(t)}
	&\leq \abs*{\int_{0}^{t} \exp\left(\int_{t'}^{t} (q + r + s - q h)(t'') \, \ud t'' \right) s(t') \ud t'} \\
	&\leq \int_{0}^{t} \exp\left(\int_{t'}^{t} (q + \abs{r} + \abs{s} + \abs{q} \dlt_{1} )(t'') \, \ud t'' \right) \abs{s(t')} \ud t' \\
	&\leq \int_{0}^{t} \max\set*{\frac{\gmm_{\lmb_{0}}'(\lmb(t'))}{\gmm_{\lmb_{0}}'(M'(t) \lmb_{0})}, \frac{\gmm_{\lmb_{0}}'(M'(t) \lmb_{0})}{\gmm_{\lmb_{0}}'(\lmb(t'))}}^{\dlt_{1}} \frac{\gmm_{\lmb_{0}}'(\lmb(t'))}{\gmm_{\lmb_{0}}'(M'(t) \lmb_{0})} \abs{s(t')} \, \ud t' \\
	& \pleq \times \exp \left(\int_{0}^{t} (\abs{s} + \abs{r})(t'') \, \ud t'' \right).
\end{align*}
We first claim that the last term is uniformly bounded in $t$. Indeed, by \eqref{eq:HJE-h-s}, monotonicity of $\gmm_{\lmb_{0}}$, Lemma~\ref{lem:Xi-control-time-dep} (by which $M'(t) \leq M$) and $\tau_{M} \leq \wpT \leq 1$,
\begin{align}
	\int_{0}^{t} \abs{s}(t') \, \ud t'
	&\aleq \int_{\lmb_{0}}^{M'(t) \lmb_{0}} \frac{\gmm_{\lmb_{0}}(\lmb_{0})}{\gmm_{\lmb_{0}}(\lmb(t))^{2}} \frac{\ud \lmb}{\lmb_{0}} \left(C_{1} c_{x_{0}} \eps(\lmb_{0}) + \frac{C_{0}}{\gmm_{\lmb_{0}}(\lmb_{0})}\right) \notag \\
	&\leq \int_{\lmb_{0}}^{M'(t) \lmb_{0}} \frac{1}{\gmm_{\lmb_{0}}(\lmb(t))} \frac{\ud \lmb}{\lmb_{0}} \left(C_{1} c_{x_{0}} \eps(\lmb_{0}) + \frac{C_{0}}{\gmm_{\lmb_{0}}(\lmb_{0})}\right) \notag \\
	&\aleq \left(C_{1} c_{x_{0}} \eps(\lmb_{0}) + \frac{C_{0}}{\gmm_{\lmb_{0}}(\lmb_{0})}\right). \label{eq:HJE-h-int-s}
\end{align}
Next, by \eqref{eq:HJE-h-r} and $M'(t) \leq M$
\begin{align}
	\int_{0}^{t} \abs{r}(t') \, \ud t'
	\aleq C_{1}' c_{x_{0}} \eps(\lmb_{0}) \int_{\lmb_{0}}^{M'(t) \lmb_{0}} \frac{\gmm_{\lmb_{0}}(\lmb_{0}) \gmm_{\lmb_{0}}'(\lmb)}{\gmm_{\lmb_{0}}(\lmb)^{2}} \, \ud \lmb 
	\leq C_{1}' c_{x_{0}} \eps(\lmb_{0}). \label{eq:HJE-h-int-r}
\end{align}
In particular, by taking $c_{x_{0}} > 0$ smaller and $\Lmb$ larger depending on $\gmm$, $C_{0}$, $C_{1}$ and $C_{1}'$, we may ensure that the RHSs of \eqref{eq:HJE-h-int-s} and \eqref{eq:HJE-h-int-r} are smaller than, say, $1$. Thus,
\begin{align*}
	\abs{h(t)}
	&\aleq \int_{\lmb_{0}}^{M'(t) \lmb_{0}} \max\set*{\frac{\gmm_{\lmb_{0}}'(\lmb)}{\gmm_{\lmb_{0}}'(M'(t) \lmb_{0})}, \frac{\gmm_{\lmb_{0}}'(M'(t) \lmb_{0})}{\gmm_{\lmb_{0}}'(\lmb)}}^{\dlt_{1}} \frac{\gmm_{\lmb_{0}}'(\lmb)}{\gmm_{\lmb_{0}}'(M'(t) \lmb_{0})} \frac{\gmm_{\lmb_{0}}(\lmb_{0})}{\gmm_{\lmb_{0}}(\lmb)^{2}} \frac{\ud \lmb}{\lmb_{0}} \\
	&\pleq \times \left(C_{1} c_{x_{0}} \eps(\lmb_{0}) + \frac{C_{0}}{\gmm_{\lmb_{0}}(\lmb_{0})} \right).
\end{align*}
We split the $\lmb$-integral as follows. Define $E := \set{\lmb \in (0, M'(t) \lmb_{0}) : \gmm_{\lmb_{0}}'(M'(t) \lmb_{0})^{-1} \gmm_{\lmb_{0}}'(\lmb) \leq 1}$. As in the proof of \eqref{eq:HJE-h-int-s}, we have
\begin{align*}
&\int_{(0, M'(t) \lmb_{0}) \cap E} \max\set*{\frac{\gmm_{\lmb_{0}}'(\lmb)}{\gmm_{\lmb_{0}}'(M'(t) \lmb_{0})}, \frac{\gmm_{\lmb_{0}}'(M'(t) \lmb_{0})}{\gmm_{\lmb_{0}}'(\lmb)}}^{\dlt_{1}} \frac{\gmm_{\lmb_{0}}'(\lmb)}{\gmm_{\lmb_{0}}'(M'(t) \lmb_{0})} \frac{\gmm_{\lmb_{0}}(\lmb_{0})}{\gmm_{\lmb_{0}}(\lmb)^{2}} \frac{\ud \lmb}{\lmb_{0}} \\
&\leq \int_{0}^{M'(t) \lmb} \frac{\gmm_{\lmb_{0}}(\lmb_{0})}{\gmm_{\lmb_{0}}(\lmb)^{2}} \frac{\ud \lmb}{\lmb_{0}} 
\leq \tau_{M} \leq \wpT \leq 1.
\end{align*}
On the other hand, the $\lmb$-integral on $(0, M'(t) \lmb_{0}) \setminus E$ is bounded from above by the LHS of \eqref{eq:gf-condition-diss-2}. In conclusion,
\begin{equation*}
\begin{aligned}
	\abs{h(t)} &\aleq \left(C_{1} c_{x_{0}} \eps(\lmb_{0}) + \frac{C_{0}}{\gmm_{\lmb_{0}}(\lmb_{0})} \right).
\end{aligned}\end{equation*}
Taking $c_{x_{0}} > 0$ smaller and $\Lmb$ larger depending on $\dlt_{5}$, $\gmm$, $C_{0}$ and $C_{1}$, we obtain $\abs{h(t)} \leq \dlt_{5}$, which improves the bootstrap assumption. \qedhere
\end{proof}

\subsection{Estimates for the phase along characteristic curves} \label{subsec:transport}

We pick up from the end of either Section~\ref{subsec:HJE-s} and \ref{subsec:HJE-gen}; that is, {\bf we assume that the hypotheses, and therefore the conclusions, of one of these sections hold.} Our goal in this subsection is to obtain sharp bounds for the derivatives of $\Phi$, which are essential for sharp estimates for the amplitude function in Section~\ref{subsec:transport2} below. Our key idea is again to consider a renormalization of the form \eqref{eq:def-h}.

To this end, we consider the transport operator 
\begin{equation*}
	\nb_{t} := \rd_{t} + \rd_{\xi} p (t, x, \rd_{x} \Phi) \rd_{x}.
\end{equation*}
Observe that the characteristics for this operator are precisely $X(t)$ in Sections~\ref{subsec:HJE-s} and \ref{subsec:HJE-gen}. Moreover, the solutions to $\nb_{t} \phi = 0$ obey a-priori $L^{\infty}$-bounds. We introduce the commutator notation
\begin{equation}\label{eq:A-def}
	[\nb_{t}, \rd_{x}] = A \rd_{x}, \quad \hbox{ or equivalently, } A = - \rd_{x} \left( \rd_{\xi} p (t, x, \rd_{x} \Phi(t, x))\right).
\end{equation}
We generalize the characteristic-wise definition \eqref{eq:def-h} of $h$ by
\begin{equation*}
	h(t, x) = \frac{\rd_{\xi} p(t, x, \rd_{x} \Phi(t, x))}{\rd_{x} p(t, x, \rd_{x} \Phi(t,x))} \rd_{x}^{2} \Phi(t, x) + 1.
\end{equation*}
The transport equations for $h$ and its the derivatives then follow from Lemma~\ref{lem:HJE-h}: \begin{equation}\label{eq:HJE-h-der}
\begin{split}
\nb_{t} \rd_{x}^{k} h &= k A \rd_{x}^{k} h - (q + r + s) \rd_{x}^{k} h + 2 q h \rd_{x}^{k} h + \rd_{x}^{k}s \\
&\peq + \sum_{\ell=1}^{k-1} C_{k,\ell}^1 \rd_{x}^{\ell}A \rd_{x}^{k-\ell}h 
 + \sum_{\ell=1}^{k} C_{k,\ell}^{2} \rd_{x}^{\ell}\left( q + r + s \right) \rd_{x}^{k-\ell} h  \\
&\peq + \sum_{\ell=1}^{k} \sum_{m=0}^{k-\ell} C^{3}_{k,\ell,m} \rd_{x}^{\ell} q \rd_{x}^{m}h \rd_{x}^{k-\ell-m} h .
\end{split}
\end{equation} 
In the above, $q$, $r$ and $s$ are defined as in \eqref{eq:def-s-r-q} but evaluated at $(t, x, \xi) = (t, x, \rd_{x} \Phi(t, x))$, and $C_{k,\ell}^{1}$, $C_{k,\ell}^{2}$ and $C_{k,\ell,m}^{3}$ are some combinatorial coefficients.

Before we continue, we note the following key commutator computation:
\begin{equation} \label{eq:comm}
\begin{aligned}{}
	[\nb_{t}, \rd_{x}] =  A \rd_{x} 
	&= - \left( \rd_{x} \rd_{\xi} p (x, \rd_{x} \Phi)
	- \rd_{\xi}^{2} p(x, \rd_{x} \Phi) \frac{\rd_{x} p(x, \rd_{x} \Phi)}{\rd_{\xi} p (x, \rd_{x} \Phi)} (1-h) \right)  \rd_{x} \\
	&= - \left( \frac{\rd_{\xi} \rd_{x} p (x, \rd_{x} \Phi)}{\rd_{x} p(x, \rd_{x} \Phi)}
	- \frac{\rd_{\xi}^{2} p(x, \rd_{x} \Phi)}{\rd_{\xi} p(x, \rd_{x} \Phi)} (1-h) \right) \rd_{x} p(x, \rd_{x} \Phi)  \rd_{x} \\
	&=  \left( \frac{\gmm'_{\lmb_{0}}(\Xi)}{\gmm_{\lmb_{0}}(\Xi)}
	- \frac{\gmm''_{\lmb_{0}}(\Xi)}{\gmm'_{\lmb_{0}}(\Xi)} + \frac{O(\dlt_{5} + \gmm_{\lmb_{0}}(\lmb_{0})^{-1})}{\Xi} \right) \dot{\Xi}  \rd_{x},
\end{aligned}
\end{equation}
where $O(\cdot)$ refers to a term whose absolute value is bounded by $C (\cdot)$ with $C$ depending on $\gmm$ and $\abs{f''(0, 0)}^{-1} \Gmm f''$. In what follows, {\bf we shall take $\Lmb_0$ larger so that $\gmm_{\lmb_{0}}(\lmb_{0})^{-1} <\dlt_{5}$}. 

\begin{proposition}\label{prop:Phi-derivatives}
Let $N_{0}$ be a positive integer. There exist {$\dlt_{3} \ll \frac{\dlt_{2}}{N_{0}}$} and $\dlt_{5} \ll \dlt_{0} \dlt_{3}$ such that, the following holds. If $\abs{h(t, X(t))} \leq \dlt_{5}$ for $0 \leq t  \leq \frac{1}{1-\eps} t_{f}(\tau_{M})$, where $\tau_{M} \leq 1$, then for $k = 1, \ldots, N_{0}$,
\begin{align}
	\abs{\rd_{x}^{k} \rd_{x} \Phi(t,X(t))} &\leq \mu(t)^{k-1}(-\rd_{x}^{2}\Phi(t,X(t))), \label{eq:rd-Phi-k-rd-Phi-2} \\
	\abs{\rd_{x}^{k} \rd_{x} \Phi(t, X(t))} &\leq \mu(t)^{k} \lmb(t), \label{eq:rd-Phi-k-bound}
\end{align}
where 
\begin{equation} \label{eq:wp-scale}
	\mu(t) = \frac{\lmb_{0}^{2 \dlt_{3} N_{0}}}{x_{0} \eps(\lmb_{0})} \frac{\gmm_{\lmb_{0}}(\lmb(t))}{\gmm_{\lmb_{0}}(\lmb_{0})} \frac{\gmm_{\lmb_{0}}'(\lmb_{0})}{\gmm_{\lmb_{0}}'(\lmb(t))}.
\end{equation}
\end{proposition}

\begin{remark}
	We remark that the requirement $\abs{h(t, X(t))} \leq \dlt_{5}$ is vacuous in the steady case (Section~\ref{subsec:HJE-s}), whereas in the time-dependent case it may be ensured by taking $\Lmb$ larger and $c_{x_{0}}$ smaller depending on $\dlt_{5}$ (Proposition~\ref{prop:h} in Section~\ref{subsec:HJE-gen}).
\end{remark}

\begin{proof}
	In the proof, we fix a bicharacteristic curve $(X(t),\Xi(t))$ with $x_0< X(0)<x_{1}$ and $\Xi(0)= \rd_{x}\Phi(0, X(0))$. 
	
	\medskip
	\noindent \textbf{An integration factor.}
	For simplicity, we set \begin{equation}\label{eq:def-I}
		\begin{split}
			I(t) := \int_{0}^{t} \left( \frac{\rd_{\xi}\rd_{x}p}{\rd_{x}p}  - \frac{\rd_{\xi}^2p}{\rd_{\xi}p} (1 - h)\right)(\tau,X(\tau),\Xi(\tau)) (-\rd_{x}p)(\tau,X(\tau),\Xi(\tau)) \,\ud \tau.
		\end{split}
	\end{equation} Recalling the computations following \eqref{eq:comm}, we have the bounds \begin{equation}\label{eq:I-bounds}
	\begin{split}
	 \frac{\gmm_{\lmb_{0}}(\Xi)}{\gmm_{\lmb_{0}}(\lmb_{0})} \frac{\gmm_{\lmb_{0}}'(\lmb_{0})}{\gmm_{\lmb_{0}}'(\Xi)} \left(\frac{\Xi}{\lmb_{0}}\right)^{ - \dlt_{4}} \le \exp(I(t)) \le \frac{\gmm_{\lmb_{0}}(\Xi)}{\gmm_{\lmb_{0}}(\lmb_{0})} \frac{\gmm_{\lmb_{0}}'(\lmb_{0})}{\gmm_{\lmb_{0}}'(\Xi)} \left(\frac{\Xi}{\lmb_{0}}\right)^{ \dlt_{4}}
	\end{split}
	\end{equation} where {\bf $\dlt_{4} = C \dlt_{5}$ with some constant $C>0$ depending on $\gmm$ and $\abs{f''(0, 0)}^{-1} \Gmm f''$}. 
	We shall establish the claimed bounds by propagating a sharp estimate based on the quantity $I(t)$.

	\medskip 
	
	\noindent \textbf{Induction base case.} 
	The following bound for $h$, which could be weaker than \eqref{eq:h-small}, is more convenient as an induction hypothesis:
\begin{equation}\label{eq:HJE-h-induction-0}
\begin{split}
	\abs{h(t, X(t))} \leq \lmb_{0}^{\frac{\dlt_{3}}{4}} \frac{1}{\gmm_{\lmb_{0}}'(\Xi(t))} \frac{1}{\lmb_{0}} .
\end{split}
\end{equation} 
To prove this, we introduce $h^{(0)}(t) = \lmb_{0} \gmm_{\lmb_{0}}'(\Xi(t)) h(t, X(t))$, which incorporates an integrating factor for the term $q h$. Then
\begin{equation*}
	\frac{\ud}{\ud t} h^{(0)} + (r+s - 2q h) h^{(0)}
	= \lmb_{0} \gmm_{\lmb_{0}}'(\Xi(t)) s.
\end{equation*}
By the ellipticity assumption for $\xi \rd_{\xi} \gmm_{\lmb_{0}}$ and $\dot{\Xi} = - \rd_{x} p$, we have
\begin{equation} \label{eq:qh-dlt5}
	2 \abs{qh} \leq C \dlt_{5} \frac{\dot{\Xi}}{\Xi},
\end{equation}
where $C$ depends on $\gmm$. Using \eqref{eq:HJE-h-s}, \eqref{eq:HJE-h-int-s} and \eqref{eq:HJE-h-int-r}, we obtain
\begin{align*}
	\abs{h^{(0)}(t)}
	&\leq \int_{0}^{t} \exp \left(\int_{t'}^{t} \abs{r} + \abs{s} + 2 \abs{q h} \, \ud t'' \right) \lmb_{0} \gmm_{\lmb_{0}}'(\Xi) \abs{s}\, \ud t' \\
	&\aleq \left(\frac{\Xi(t)}{\lmb_{0}}\right)^{C \dlt_{5}} \int_{0}^{t} \frac{\gmm_{\lmb_{0}}'(\Xi) \gmm_{\lmb_{0}}(\lmb_{0})}{\gmm_{\lmb_{0}}(\Xi)^{2}} \dot{\Xi} \, \ud t'.
\end{align*}
The first factor is bounded by $\lmb_{0}^{\frac{C \dlt_{5}}{\dlt_{0}}}$, where as the $t'$-integral is evaluated and bounded as
\begin{equation} \label{eq:HJE-h-int-s-re}
\int_{0}^{t} \frac{\gmm_{\lmb_{0}}'(\Xi) \gmm_{\lmb_{0}}(\lmb_{0})}{\gmm_{\lmb_{0}}(\Xi)^{2}} \dot{\Xi} \, \ud t'
= \int_{\lmb_{0}}^{\Xi(t)} \frac{\gmm_{\lmb_{0}}'(\Xi) \gmm_{\lmb_{0}}(\lmb_{0})}{\gmm_{\lmb_{0}}(\Xi)^{2}} \, \ud \Xi 
= 1 - \frac{\gmm_{\lmb_{0}}(\lmb_{0})}{\gmm_{\lmb_{0}}(\Xi(t))} \leq 1.
\end{equation}
Hence, by taking $\dlt_{5} \ll \dlt_{0} \dlt_{3}$ and returning to $h$, \eqref{eq:HJE-h-induction-0} follows.

	\medskip 
	
	\noindent \textbf{Induction hypothesis.} 
	We now turn to the case $k_{0} \geq 1$. When $k_{0} > 1$, we assume the following in addition to $\abs{h} \leq \dlt_{5}$: For $k = 1, \ldots, k_{0}-1$,
	\begin{align}
	|\rd^{k}_{x}h(t, X(t))| &\leq \lmb_{0}^{\dlt_{3}k^2} x_{0}^{-k} {\eps(\lmb_{0})^{-k}}\exp(kI(t)) \frac{1}{\gmm_{\lmb_{0}}'(\Xi(t))} \frac{1}{\lmb_{0}},	\label{eq:HJE-h-induction}\\
	|\rd^{k}_{x} \rd_{x}^{2} \Phi (t, X(t))| &\leq \lmb_{0}^{\dlt_{3} ((k+1)^2 -1)} x_{0}^{-k} {\eps(\lmb_{0})^{-k}}\exp(kI(t)) (-\rd_{x}^{2} \Phi(t, X(t))). \label{eq:HJE-Phi-induction}
	\end{align}

	To handle the quadratic term in $h$, it is easier to work with the following simplification of \eqref{eq:HJE-h-induction}. Note that
	\begin{equation} \label{eq:HJE-h-simple-0}
\frac{1}{\lmb_{0} \gmm_{\lmb_{0}}'(\Xi)}
\aleq \frac{\Xi}{\lmb_{0} \gmm_{\lmb_{0}}(\Xi)} (\log \Xi)^{2} 
\aleq \max \set*{\frac{1}{\gmm_{\lmb_{0}}(\lmb_{0})}, \tau_{M}} \dlt_{0}^{-2} (\log \lmb_{0})^{2}
\aleq \dlt_{0}^{-2} (\log \lmb_{0})^{2}.
\end{equation}
In the second inequality, we used the fact that $\frac{\Xi}{\lmb_{0} \gmm_{\lmb_{0}}(\Xi)}$ is clearly bounded by $\frac{1}{\gmm_{\lmb_{0}}(\lmb_{0})}$ for $\Xi \leq 4$, and by $\tau_{M}$ for $\Xi \geq 4$ thanks to $\lmb \leq \Xi \leq 2\lmb$ and $\lmb \leq M \lmb_{0}$. As a consequence of $|h| \leq \dlt_{5}$ for $k = 0$ and \eqref{eq:HJE-h-induction} and \eqref{eq:HJE-h-simple-0} for $k \geq 1$, we obtain, for $k = 0, \ldots, k_{0} -1$, 
	\begin{equation}\label{eq:HJE-h-simple}
	\abs{x_{0}^{k} {\eps(\lmb_{0})^{k}} \exp(-k I(t)) \rd^{k}_{x}h(t, X(t))} \aleq_{\dlt_{0}} \lmb_{0}^{\dlt_{3}k^2} (\log \lmb_{0})^{2}.
	\end{equation} 

	\medskip 
	
	\noindent \textbf{Induction argument.} Our goal now is to prove \eqref{eq:HJE-h-induction} and \eqref{eq:HJE-Phi-induction} for $k = k_{0}$. 
	
	To use \eqref{eq:HJE-h-der} to establish \eqref{eq:HJE-h-induction} for $k = k_{0}$, we work with new variables that incorporate integrating factors for cancelling the large coefficient $k A - q$. For $k \geq 1$, define 
\begin{align*}
	h^{(k)}(t) &= x_{0}^{k} {\eps(\lmb_{0})^{k}} \exp(- k I(t)) \lmb_{0} \gmm_{\lmb_{0}}'(\Xi(t)) \rd_{x}^{k}h(t, X(t)),& 
	A^{(k)}(t) &= x_{0}^{k} {\eps(\lmb_{0})^{k}} \exp(- k I(t)) \rd_{x}^{k} A(t, X(t)), \\
	s^{(k)}(t) &= x_{0}^{k} {\eps(\lmb_{0})^{k}} \exp(- k I(t)) \rd_{x}^{k} s(t, X(t)), &
	r^{(k)}(t) &= x_{0}^{k} {\eps(\lmb_{0})^{k}} \exp(- k I(t)) \rd_{x}^{k} r(t, X(t)), \\
	q^{(k)}(t) &= x_{0}^{k} {\eps(\lmb_{0})^{k}} \exp(- k I(t)) \rd_{x}^{k} q(t, X(t)). & &
\end{align*}
Observe that \eqref{eq:HJE-h-induction-0} and \eqref{eq:HJE-h-induction} are equivalent to
\begin{equation} \label{eq:h-(k)-induction}
	\abs{h^{(0)}} \leq \lmb_{0}^{\frac{\dlt_{3}}{4}}, \quad
	\abs{h^{(k)}} \leq \lmb_{0}^{\dlt_{3} k^{2}} \quad \hbox{ for } 1\leq k < k_{0}.
\end{equation}
	Moreover, evaluating along a characteristic curve, we note that \eqref{eq:HJE-h-der} can be written as 
	\begin{equation} \label{eq:HJE-h-der2}
\begin{aligned}
	\frac{\ud}{\ud t} h^{(k)} + (r+s - 2q h) h^{(k)}
	&= \lmb_{0} \gmm_{\lmb_{0}}'(\Xi(t)) s^{(k)}
	+ \sum_{\ell=1}^{k-1} C^{1}_{k, \ell} A^{(\ell)} h^{(k-\ell)}  \\
	&\peq + \sum_{\ell=1}^{k} C^{2}_{k, \ell} \left( q^{(\ell)} + r^{(\ell)} + s^{(\ell)}\right) h^{(k-\ell)} \\
	&\peq + \sum_{\ell=1}^{k} \sum_{m=0}^{k-\ell} C^{3}_{k, \ell, m} \frac{q^{(\ell)}}{\lmb_{0} \gmm_{\lmb_{0}}'(\Xi) } h^{(m)} h^{(k-\ell-m)} \\
	&\peq + \sum_{m=1}^{k-1} C^{3}_{k, 0, m} \frac{q}{\lmb_{0} \gmm_{\lmb_{0}}'(\Xi) } h^{(m)} h^{(k-m)}.
\end{aligned}
\end{equation}

We shall prove the following estimates for the coefficients of \eqref{eq:HJE-h-der2}: for $1 \leq k \leq k_{0}$,  
\begin{align} 
	\abs{s^{(k)}(t)} &\leq \lmb_{0}^{\dlt_{3} (k^2-\frac{1}{2})} \frac{\gmm_{\lmb_{0}}(\lmb_{0})}{\gmm_{\lmb_{0}}(\Xi(t))^{2} \lmb_{0}} \dot{\Xi}(t) \left(C_{k, 1} x_{0} + \frac{C_{k, 0}}{\gmm_{\lmb_{0}}(\lmb_{0})}\right), \label{eq:h-coeff-1} \\
	\abs{r^{(k)}(t)} &\leq C_{k ,0} C_{k, 1} \lmb_{0}^{\dlt_{3} (k^2-\frac{1}{2})} \frac{\gmm_{\lmb_{0}}'(\Xi(t))\gmm_{\lmb_{0}}(\lmb_{0})}{\gmm_{\lmb_{0}}(\Xi(t))^{2}} \dot{\Xi}(t) x_{0}, \label{eq:h-coeff-2} \\
	\abs{q^{(k)}(t)} &\leq C_{k, 0} \lmb_{0}^{\dlt_{3} (k^2-\frac{1}{2})} \frac{\dot{\Xi}(t)}{\Xi(t)}, \label{eq:h-coeff-3}
\end{align}
and for $k > 1$,
\begin{align}
	\abs{A^{(k-1)}(t)} &\leq C_{k-1, 0} \lmb_{0}^{\dlt_{3} ((k-1)^2-\frac{1}{2})} \frac{\dot{\Xi}(t)}{\Xi(t)}. \label{eq:h-coeff-4}
\end{align}
Here, $C_{k, 1}$ depends on $\abs{f''(0,0)}^{-1} (x \rd_{x})^{k'} f'''$, $\abs{f''(0,0)}^{-1} (x \rd_{x})^{k'} \Gmm f'''$ and $\abs{f''(0,0)}^{-2} (x \rd_{x})^{k'} \rd_{t} f'''$ for $0 \leq k' \leq k$; and $C_{k, 0}$ depends on $\abs{f''(0, 0)}^{-1} (x \rd_{x})^{k'} f''$, $\abs{f''(0, 0)}^{-2} (x \rd_{x})^{k'} \rd_{t} f''$, $\abs{f''(0, 0)}^{-1} (x \rd_{x})^{k'} \Gmm f''$ and $\abs{f''(0, 0)}^{-2} (x \rd_{x})^{k'} \rd_{t} \Gmm f''$ for $0 \leq k' \leq k$. Without loss of generality, we may assume $C_{k', j} \leq C_{k, j}$ for $k' \leq k$ and $j=0, 1$.

Assuming \eqref{eq:h-coeff-1}--\eqref{eq:h-coeff-4}, we can improve the induction hypothesis \eqref{eq:HJE-h-induction} for $h$ as follows. Using \eqref{eq:HJE-h-int-s}, \eqref{eq:HJE-h-int-r}, \eqref{eq:qh-dlt5} and \eqref{eq:HJE-h-der2}, we have
\begin{align*}
\abs{h^{(k_{0})}} &\leq \int_{0}^{t} \exp\left(\int_{t'}^{t} (\abs{r} + \abs{s} + \abs{q} \dlt_{5} \, \ud t'' \right) \abs{\hbox{(RHS of \eqref{eq:HJE-h-der2})}} \, \ud t' 
\aleq \left(\frac{\Xi(t)}{\lmb_{0}}\right)^{C \dlt_{5}}  \int_{0}^{t} \abs{\hbox{(RHS of \eqref{eq:HJE-h-der2})}} \, \ud t'.
\end{align*}
By \eqref{eq:h-(k)-induction} for $k < k_{0}$, \eqref{eq:h-coeff-1}--\eqref{eq:h-coeff-4}, as well as \eqref{eq:HJE-h-simple-0}, we obatin
\begin{align*}
\abs{\hbox{(RHS of \eqref{eq:HJE-h-der2})}}
&\aleq_{C_{k_{0}, 0}, C_{k_{0}, 1}, \dlt_{0}^{-1}} \lmb_{0}^{\dlt_{3} (k_{0}^{2}-\frac{1}{2})}  \left(\frac{\gmm_{\lmb_{0}}(\lmb_{0}) \gmm_{\lmb_{0}}'(\Xi)}{\gmm_{\lmb_{0}}(\Xi)^{2}}
+ \lmb_{0}^{\frac{1}{4}} (\log \lmb_{0})^{2} \frac{1}{\Xi}\right) \dot{\Xi} \\
&\phantom{\aleq_{C_{k_{0}, 0}, C_{k_{0}, 1}, \dlt_{0}^{-1}}}
+ \sum_{m=1}^{k_{0}-1}\lmb_{0}^{\dlt_{3} (m^{2}+(k_{0}-m)^{2})} (\log \lmb_{0})^{2} \frac{\dot{\Xi}}{\Xi},
\end{align*}
where the last term arises from the contribution of the last sum in \eqref{eq:HJE-h-der2} (this sum is vacuous when $k_{0} = 1$). Using $m^{2} + (k_{0} - m)^{2} \leq k_{0}^{2} - 2$ for $1 \leq m \leq k_{0}-1$, we obtain
\begin{align*}
\abs{h^{(k_{0})}} 
&\aleq_{C_{k_{0}, 0}, C_{k_{0}, 1}, \dlt_{0}^{-1}}
\lmb_{0}^{\dlt_{3}(k_{0}^{2}-\frac{1}{4})} (\log \lmb_{0})^{2} \left(\frac{\Xi(t)}{\lmb_{0}}\right)^{C \dlt_{5}}\left( 1 - \frac{\gmm_{\lmb_{0}}(\lmb_{0})}{\gmm_{\lmb_{0}}(\Xi(t))}+ \log \frac{\Xi(t)}{\lmb_{0}} \right).
\end{align*}
Taking $\dlt_{5} \ll \dlt_{0} \dlt_{3}$ and requiring $\lmb_{0} \geq \Lmb$, where $\Lmb$ is sufficiently large, we obtain $\abs{h^{(k_{0})}} \leq \lmb_{0}^{\dlt_{3} k_{0}^{2}}$, which is equivalent to \eqref{eq:HJE-h-induction} for $k = k_{0}$, as desired.

Next, we improve \eqref{eq:HJE-Phi-induction}. We begin with the formula
\begin{equation} \label{eq:rd-k0+2-phi}
	\rd_{x}^{k_{0}} \rd_{x}^{2} \Phi(t, x)
	= \left(\frac{\rd_{x} p}{\rd_{\xi} p} \right) \rd_{x}^{k_{0}} h
	+ \rd_{x}^{k_{0}} \left(\frac{\rd_{x} p}{\rd_{\xi} p} \right) (h-1)
	+ \sum_{\ell=1}^{k_{0}-1} \rd_{x}^{\ell} \left(\frac{\rd_{x} p}{\rd_{\xi} p} \right) \rd_{x}^{k-\ell} h,
\end{equation}
where $\rd_{x} p$ and $\rd_{\xi} p$ are evaluated at $(t, x, \xi) = (t, x, \rd_{x} \Phi(t, x))$. We will prove that, for $0 \leq k \leq k_{0}$,
\begin{equation} \label{eq:rd-k-Phi-coeff}
	\abs*{x_{0}^{k} {\eps(\lmb_{0})^{k}} \exp(-kI(t)) \rd_{x}^{k} \left(\frac{\rd_{x} p}{\rd_{\xi} p}\right)(t, X(t), \Xi(t))} 
	\leq C_{k, 0} \lmb_{0}^{\dlt_{3} k^2} (-\rd_{x}^{2} \Phi),
\end{equation}
where $C_{k, 0}$ depends on $\abs{f''(0, 0)}^{-1} (x \rd_{x})^{k'} \Gmm f''$ and $\abs{f''(0, 0)}^{-2} (x \rd_{x})^{k'} \rd_{t} \Gmm f''$ for $0 \leq k' \leq k$. Without loss of generality, we may assume that these constants are the same as those in \eqref{eq:h-coeff-1}--\eqref{eq:h-coeff-4}. In the proof of \eqref{eq:rd-k-Phi-coeff}, it is important that $\rd_{x}^{k} (\frac{\rd_{x}p}{\rd_{\xi}p} )$ involves only $\rd_{x}^{k'}\Phi$ with $k' \leq k + 1 < k_{0}+2$, so that we may apply the induction hypothesis \eqref{eq:HJE-Phi-induction}. We postpone the details until later.

Assuming \eqref{eq:rd-k-Phi-coeff}, and also using $\abs{h} \leq \dlt_{5}$, \eqref{eq:HJE-h-simple} for $k \leq k_{0}$ (the case $k = k_{0}$ has just been established) and \eqref{eq:rd-k+2-Phi-Xi} for $k \leq k_{0}-1$, we may estimate the RHS of \eqref{eq:rd-k0+2-phi} by
\begin{equation*}
	\abs{\rd_{x}^{k_{0}} \rd_{x}^{2} \Phi(t, X(t)} \aleq_{\dlt_{0}, C_{k_{0}, 0}} \lmb_{0}^{\dlt_{3} k_{0}^{2}} (\log \lmb_{0})^{2} x_{0}^{-k_{0}} {\eps(\lmb_{0})^{-k_{0}}} \exp(k_{0} I) (-\rd_{x}^{2} \Phi),
\end{equation*}
where used the simple inequality $(k_{0}-\ell)^{2} + \ell^{2} \leq k_{0}^{2}$. Since the exponent $\dlt_{3}((k_{0}+1)^{2} - 1)$ on $\lmb_{0}$ in \eqref{eq:HJE-Phi-induction} is strictly greater than $\dlt_{3} k_{0}^{2}$, taking $\lmb_{0}$ large enough depending on $\dlt_{0}$, $\dlt_{3}$ and $C_{k_{0}, 0}$, we obtain \eqref{eq:HJE-Phi-induction} for $k = k_{0}$. 

\medskip 
	
	\noindent \textbf{Proof of \eqref{eq:h-coeff-1}--\eqref{eq:h-coeff-4}, \eqref{eq:rd-k-Phi-coeff}.} 
	In the proof, we fix some $k\ge 1$. We may assume that the bounds \eqref{eq:HJE-h-induction} and \eqref{eq:HJE-Phi-induction} are available for any $\ell$ satisfying $\ell < k$. 

\medskip
\noindent \textit{Some preliminary computations.}
We claim that
\begin{align} 
	\abs{\rd_{x}^{2} \Phi(t, X(t))} &\leq \frac{C \lmb_{0}^{\frac{\dlt_{3}}{2}}}{x_{0} {\eps_{0}(\lmb_{0})}} \exp(I(t)) \Xi(t), \label{eq:rd-2-Phi-Xi} \\
	\frac{1}{X(t)} &\leq \frac{C \lmb_{0}^{\frac{\dlt_{3}}{2}}}{x_{0} {\eps(\lmb_{0})}} \exp(I(t)), \label{eq:X-1-mu}
\end{align}
where $C$ depends on $\gmm$. Note that \eqref{eq:rd-2-Phi-Xi} combined with \eqref{eq:HJE-Phi-induction} would lead to, for any $\ell \leq k-1$,
\begin{equation} \label{eq:rd-k+2-Phi-Xi}
	\abs{\rd^{\ell+1}_{x} \rd_{x} \Phi (t, X(t))} \leq \lmb_{0}^{\dlt_{3} ((\ell+1)^2-\frac{1}{2})} x_{0}^{-(\ell+1)} {\eps(\lmb_{0})^{-(\ell+1)}}\exp((\ell+1)I(t)) \Xi(t). 
\end{equation}
Ignoring the powers of $\lmb_{0}^{\dlt_{3}}$, \eqref{eq:rd-k+2-Phi-Xi} tells us that every derivative of $\rd_{x} \Phi$ loses $x_{0}^{-1} \eps(\lmb_{0})^{-1} \exp(I)$, which is essentially $\mu$ in \eqref{eq:wp-scale}; \eqref{eq:X-1-mu} says that the loss $x^{-1}$ is more favorable than $x_{0}^{-1} \eps(\lmb_{0})^{-1} \exp(I)$ along characteristics.

Essential to the proofs of both \eqref{eq:rd-2-Phi-Xi} and \eqref{eq:X-1-mu} is the following consequence of \eqref{eq:eps-choice}, which is connected to the assumption \eqref{eq:gf-condition-1}: If $2 \lmb_{0} \leq \lmb(t) \leq M \lmb_{0}$, we have
\begin{equation} \label{eq:scale-comp-key}
	\frac{\eps_{0}(\lmb_{0})}{\lmb_{0} \gmm_{\lmb_{0}}'(\lmb_{0})} 
	\leq C \frac{(\log \lmb_{0})^{2} \eps_{0}(\lmb_{0})}{\gmm(\lmb_{0})} 
	\leq C \eps_{0}(\lmb_{0}) (\log \lmb_{0})^{2} \tau_{M'} 
	\leq C(\log \lmb_{0})^{2} \frac{M'}{\gmm_{\lmb_{0}}(M' \lmb_{0})},
\end{equation}
where $M' = \lmb_{0}^{-1} \lmb(t)$. 

To prove \eqref{eq:rd-2-Phi-Xi}, we first use \eqref{eq:def-h}, $\abs{h} \leq \dlt_{5}$, \eqref{eq:X-bounds} or \eqref{eq:X-bounds-time-dep}, and \eqref{eq:I-bounds} to estimate, for $\lmb_{0} \geq \Lmb$ sufficiently large depending on $\abs{f''(0, 0)}^{-1} \Gmm f''$,
\begin{align*}
	\frac{x_{0} {\eps_{0}(\lmb_{0})} \abs{\rd_{x}^{2} \Phi(t, X(t))}}{\Xi(t) \exp(I(t))} 
	&\leq \frac{C {\eps_{0}(\lmb_{0})}}{\Xi(t) \exp(I(t))} \frac{\rd_{x} p(t, X(t), \Xi(t))}{\rd_{\xi} p(t, X(t), \Xi(t))} \\
	&\leq C {\eps_{0} (\lmb_{0})}\left(\frac{\lmb_{0}}{\Xi(t)}\right)^{1-\dlt_{4}}\frac{\gmm_{\lmb_{0}}(\Xi(t))}{\lmb_{0} \gmm_{\lmb_{0}}'(\lmb_{0})}.
\end{align*}
For $\Xi(t) \leq 4 \lmb_{0}$, \eqref{eq:rd-2-Phi-Xi} is obvious, and for $\Xi(t) \geq 4 \lmb_{0}$, we apply \eqref{eq:scale-comp-key} and $\lmb(t) \leq \Xi(t) \leq 2 \lmb(t)$ (from Lemma~\ref{lem:Xi-control} or \ref{lem:Xi-control-time-dep}). 

Next, to prove \eqref{eq:X-1-mu}, we use \eqref{eq:X-bounds} or \eqref{eq:X-bounds-time-dep} and \eqref{eq:I-bounds} to estimate
\begin{align*}
	\frac{x_{0} {\eps(\lmb_{0})}}{X(t) \exp(I(t))}
	&\aleq {\eps_{0}(\lmb_{0})} \left(\frac{\Xi(t)}{\lmb_{0}}\right)^{\dlt_{4}} \frac{\gmm_{\lmb_{0}}'(\Xi(t))}{\gmm_{\lmb_{0}}'(\lmb_{0})}.
\end{align*}
Again, for $\Xi(t) \leq 4 \lmb_{0}$, \eqref{eq:X-1-mu} is obvious. For $\Xi(t) \geq 4 \lmb_{0}$, we apply \eqref{eq:scale-comp-key} and $\lmb(t) \leq \Xi(t) \leq 2 \lmb(t)$ (from Lemma~\ref{lem:Xi-control} or \ref{lem:Xi-control-time-dep}) and estimate
\begin{align*}
{\eps_{0}(\lmb_{0})}\left(\frac{\Xi(t)}{\lmb_{0}}\right)^{\dlt_{4}} \frac{\gmm_{\lmb_{0}}'(\Xi(t))}{\gmm_{\lmb_{0}}'(\lmb_{0})}
\leq C \left(\frac{\Xi(t)}{\lmb_{0}}\right)^{\dlt_{4}} (\log \lmb_{0})^{2} \frac{\Xi(t) \gmm_{\lmb_{0}}'(\Xi(t))}{\gmm_{\lmb_{0}}(\Xi(t))} 
\leq C \lmb_{0}^{\frac{\dlt_{4}}{\dlt_{0}}} (\log \lmb_{0})^{2},
\end{align*}
which is acceptable.

\medskip

\noindent \textit{Derivatives of $q$, $r$, $s$ and $A$.}
We start with some preparations. By Fa\`a di Bruno's formula,
\begin{equation*}
	\rd_{x}^{\ell} \gmm_{\lmb_{0}} (\rd_{x} \Phi) = \sum_{\mathfrak{a} : a_{1} + 2 a_{2} + \cdots + \ell a_{\ell} = \ell} \frac{\ell!}{a_{1}! 1!^{a_{1}} \cdots a_{\ell}! \ell!^{a_{\ell}}} (\rd_{\xi}^{a_{1} + \cdots + a_{\ell}} \gmm_{\lmb_{0}})(\rd_{x} \Phi) \prod_{b=1}^{\ell} (\rd_{x}^{1+b} \Phi)^{a_{b}}.
\end{equation*}
Observe that we see at most $\ell$ derivatives falling on $\rd_{x} \Phi$, to which \eqref{eq:rd-k+2-Phi-Xi} applies as long as $\ell \leq k$. Using also the ellipticity assumption for $\gmm_{\lmb_{0}}$ and $\sum_{b=1}^{\ell} (b a_{b})^{2} \leq \ell^{2}$, we obtain, for any $0 \leq \ell \leq k$,
\begin{equation} \label{eq:gmm-Xi-der}
	\abs*{\left.\rd_{x}^{\ell} (\gmm_{\lmb_{0}}(\rd_{x} \Phi))\right|_{(t, x) = (t, X(t))}}
	\aleq_{\gmm, \ell} \lmb_{0}^{\dlt_{3}(\ell^{2}-\frac{1}{2})} x_{0}^{-\ell} {\eps(\lmb_{0})^{-\ell}} \exp(\ell I)  \gmm_{\lmb_{0}}(\Xi) . 
\end{equation}
Similarly, but using instead the ellipticity assumption for $\xi \rd_{\xi} \gmm_{\lmb_{0}}$, we also obtain, for any $0 \leq \ell \leq k$,
\begin{align}
	\abs*{\left.\rd_{x}^{\ell} (\gmm_{\lmb_{0}}'(\rd_{x} \Phi))\right|_{(t, x) = (t, X(t))}}
	&\aleq_{\gmm, \ell} \lmb_{0}^{\dlt_{3}(\ell^{2} - \frac{1}{2})} x_{0}^{-\ell} {\eps(\lmb_{0})^{-\ell}} \exp(\ell I) \gmm_{\lmb_{0}}'(\Xi). \label{eq:gmm'-Xi-der}
\end{align}
Next, we introduce the notation $\tilde{f} = (-f''(0, 0))^{-1} f$. In what follows, we write $C_{\ell, 1}$ (resp.~$C_{\ell, 0}$) for a constant, that may vary from line to line, that depends on $(x \rd_{x})^{k'} \tilde{f}'''$, $(x \rd_{x})^{k'} \Gmm \tilde{f}'''$ and $(-f''(0,0))^{-1} (x \rd_{x})^{k'} \rd_{t} \tilde{f}'''$ (resp.~$(x \rd_{x})^{k'} \tilde{f}''$, $(-f''(0, 0))^{-1} (x \rd_{x})^{k'} \rd_{t} \tilde{f}''$, $(x \rd_{x})^{k'} \Gmm \tilde{f}''$ and \\ $(-f''(0, 0))^{-1} (x \rd_{x})^{k'} \rd_{t} \Gmm \tilde{f}''$) for $0 \leq k' \leq \ell$. Then thanks to \eqref{eq:X-1-mu}, we easily have, for any $\ell \geq 0$,
\begin{equation} \label{eq:tilde-f''} 
\begin{aligned}
	& \abs*{\tilde{f}^{(\ell+2)}(t, X(t)) }
	+ \abs*{\frac{\rd_{t} \tilde{f}^{(\ell+2)}(t, X(t))}{-f''(0, 0)}} 
	+ \abs*{\Gmm \tilde{f}^{(\ell+2)}(t, X(t))}
	+ \abs*{\frac{\rd_{t} \Gmm \tilde{f}^{(\ell+2)}(t, X(t))}{-f''(0, 0)}} \\
	&\leq  C_{\ell, 0}\lmb_{0}^{\dlt_{3} \frac{\ell}{2}} x_{0}^{-\ell} {\eps(\lmb_{0})^{-\ell}} \exp(\ell I), 
\end{aligned}
\end{equation}
\begin{equation} \label{eq:tilde-f'''}
	\abs*{\tilde{f}^{(\ell+3)}(t, X(t)) }
	+ \abs*{\Gmm \tilde{f}^{(\ell+3)}(t, X(t))}
	+ \abs*{\frac{\rd_{t} \tilde{f}^{(\ell+3)}(t, X(t))}{-f''(0, 0)}}
	\leq  C_{\ell, 1} \lmb_{0}^{\dlt_{3} \frac{\ell}{2}} x_{0}^{-\ell} {\eps(\lmb_{0})^{-\ell}} \exp(\ell I). 
\end{equation}

We are now ready to prove \eqref{eq:h-coeff-1}--\eqref{eq:h-coeff-4}. We start with \eqref{eq:h-coeff-3} for $q$, which may be written as
\begin{equation*}
	q(t, x) = \frac{\gmm_{\lmb_{0}}''(\rd_{x} \Phi)}{\gmm_{\lmb_{0}}'(\rd_{x} \Phi)} (-f'')(0, 0) \lmb_{0} \gmm_{\lmb_{0}}(\rd_{x} \Phi) \left(\tilde{f}''(t, x) + \gmm_{\lmb_{0}}(\rd_{x} \Phi)^{-1} \Gmm \tilde{f}''(t, x) \right).
\end{equation*}
By \eqref{eq:gmm'-Xi-der} and the ellipticity property of $\xi \rd_{\xi} \gmm_{\lmb_{0}}$, we have, for any $0 \leq \ell \leq k$,
\begin{equation*}
\abs*{x_{0}^{\ell} {\eps(\lmb_{0})^{\ell}} \exp(-\ell I) \left.\rd_{x}^{\ell} \frac{\gmm_{\lmb_{0}}''(\rd_{x} \Phi)}{\gmm_{\lmb_{0}}'(\rd_{x} \Phi)}\right|_{(t, x)=(t, X(t))}} \aleq_{\gmm, \ell} \lmb_{0}^{\dlt_{3}(\ell^{2} - \frac{1}{2})}  \frac{1}{\Xi}.
\end{equation*}
By \eqref{eq:gmm-Xi-der} and $\dot{\Xi} = - \rd_{x} p$, we also see that
\begin{equation*}
\abs*{x_{0}^{\ell} {\eps(\lmb_{0})^{\ell}} \exp(-\ell I) \left.\rd_{x}^{\ell} (-f''(0, 0)) \lmb_{0} \gmm_{\lmb_{0}}(\rd_{x} \Phi) \right|_{(t, x)=(t, X(t))}} \aleq_{\gmm, \ell} \lmb_{0}^{\dlt_{3}(\ell^{2} - \frac{1}{2})}  \dot{\Xi}.
\end{equation*}
Derivatives of the last factor are easily bounded using \eqref{eq:gmm-Xi-der} and \eqref{eq:tilde-f''}. Putting together these bounds, \eqref{eq:h-coeff-3} follows.

Next, we prove \eqref{eq:h-coeff-4} for $A$, which we write as
\begin{align*}
A(t, x) &= - \rd_{\xi} \rd_{x} p (t, x, \rd_{x} \Phi) - q(t, x) (1-h(t,x)) \\
&= - f''(0, 0) \lmb_{0} \gmm_{\lmb_{0}}'(\rd_{x} \Phi) \tilde{f}''(t, x) - q(t, x) (1-h(t,x)).
\end{align*}
We now apply $\rd_{x}^{k-1}$ to the above expression. The contribution of $q(1-h)$ is already acceptable thanks to the previous bound for derivatives of $q$, $\abs{h} \leq \dlt_{5}$ and \eqref{eq:HJE-h-simple} for $k-1$ (indeed, note that we need to control only up to $k-1$ derivatives of $h$). By \eqref{eq:gmm'-Xi-der}, ellipticity of $\gmm_{\lmb_{0}}$ and $\dot{\Xi} = - \rd_{x} p$, we may also bound the contribution of the first term by $C_{k-1, 0} \lmb_{0}^{\dlt_{3}((k-1)^{2} -\frac{1}{2})} \Xi^{-1} \dot{\Xi}$, which is acceptable.

To prove \eqref{eq:h-coeff-2} for $r$, we write
\begin{align*}
r(t, x) &= \frac{\rd_{x}^{2} p}{\rd_{x} p} \rd_{\xi} p(t, x, \rd_{x} \Phi) 
	= (-f''(0, 0)) x \gmm_{\lmb_{0}}'(\rd_{x} \Phi) \frac{\tilde{f}'''(t, x) + \gmm_{\lmb_{0}}(\rd_{x} \Phi)^{-1}\Gmm \tilde{f}'''(t, x)}{\tilde{f}''(t, x) + \gmm_{\lmb_{0}}(\rd_{x} \Phi)^{-1}\Gmm \tilde{f}''(t, x)} \frac{\tilde{f}'(t, x)}{x},
\end{align*}
and apply $\rd_{x}^{k}$ to the above expression. By Taylor expansion and \eqref{eq:tilde-f''}, we have, for any $\ell \geq 0$,
\begin{equation*}
\abs*{x_{0}^{\ell} {\eps(\lmb_{0})^{\ell}} \exp(- \ell I) \left. \rd_{x}^{\ell} \left(x^{-1} \tilde{f}'(t, x)\right) \right|_{(t, x)=(t, X(t)}} \aleq_{\ell} C_{\ell, 0} \lmb_{0}^{\dlt_{3} \frac{\ell}{2}}.
\end{equation*}
On the other hand, using \eqref{eq:X-bounds} or \eqref{eq:X-bounds-time-dep}, \eqref{eq:X-1-mu}, \eqref{eq:gmm-Xi-der}, \eqref{eq:gmm'-Xi-der}, \eqref{eq:tilde-f''} and \eqref{eq:tilde-f'''}, it is straightforward to bound each factor appropriately and establish \eqref{eq:h-coeff-2}.

Finally, \eqref{eq:h-coeff-1} for $s$ is proved by starting from the explicit form of $s$ given in the proof of Proposition~\ref{prop:h}, then applying Taylor expansion for $(f')^{-1}\rd_{t} f' - (f'')^{-1} \rd_{t} f''$, \eqref{eq:gmm-Xi-der}, \eqref{eq:tilde-f''} and \eqref{eq:tilde-f'''}, as well as \eqref{eq:X-bounds} or \eqref{eq:X-bounds-time-dep} and $\dot{\Xi} = - \rd_{x} p$.

\medskip
\noindent \textit{Proof of \eqref{eq:rd-k-Phi-coeff}.}
We begin by writing
\begin{align*}
	\frac{\rd_{x} p}{\rd_{\xi} p} = \frac{\tilde{f}''(t, x) + \gmm_{\lmb_{0}}(\rd_{x} \Phi)^{-1} \Gmm \tilde{f}''(t, x)}{x^{-1} \tilde{f}'(t, x)} \frac{\gmm_{\lmb_{0}}(\rd_{x} \Phi)}{x \gmm_{\lmb_{0}}'(\rd_{x} \Phi)},
\end{align*}
and applying $\rd_{x}^{k}$ to the above expression. By \eqref{eq:X-1-mu}, \eqref{eq:gmm-Xi-der}, \eqref{eq:gmm'-Xi-der} and \eqref{eq:tilde-f''} (as well as a Taylor-expansion argument for $x^{-1} f'$ as before), we have
\begin{align*}
	\abs*{x_{0}^{k} {\eps(\lmb_{0})^{k}} \exp(-k I) \rd_{x}^{k} \left(\frac{\rd_{x} p}{\rd_{\xi} p}\right)(t, X(t), \Xi(t))} 
	\leq C_{k, 0} \lmb_{0}^{\dlt_{3} (k^{2} - \frac{1}{2})} \frac{\gmm_{\lmb_{0}}(\rd_{x} \Phi)}{x \gmm_{\lmb_{0}}'(\rd_{x} \Phi)}.
\end{align*}
Then by the relation $\rd_{x}^{2} \Phi = - \frac{\rd_{x} p}{\rd_{\xi} p} (1-h)$ and $\abs{h} \leq \dlt_{5}$ and the above expression for $\frac{\rd_{x} p}{\rd_{\xi} p}$, \eqref{eq:rd-k-Phi-coeff} follows. \qedhere
\end{proof} 

\subsection{Estimates for the amplitude} \label{subsec:transport2} 
 
In this section, we obtain $L^2$-estimates for the amplitude function $a$, {\bf working under the same hypotheses as in Section~\ref{subsec:transport}.} First, we rewrite the amplitude equation \eqref{eq:transport-a}, incorporating the time dependence of $f$ and using the notations and conventions in this section: \begin{equation}\label{eq:transport-a-simplified}
\begin{split}
\rd_t a + \rd_\xi p(t,x,\rd_{x}\Phi) \rd_{x} a + \left( \frac{1}{2} \rd_{\xi}^2 p(t,x,\rd_{x}\Phi) \rd^2_{x}\Phi + s(t,x,\rd_{x}\Phi) \right) a = 0. 
\end{split}
\end{equation} Moreover, we recall that \begin{equation}\label{eq:s-simplified}
\begin{split}
s(t,x,\xi) & =  \frac{1}{2} f''(t,x) \rd_{\xi}\gmm_{\lmb_{0}}(\xi) \lmb_{0} - \frac{1}{2} \Gmm f'' (t,x) \frac{\rd_{\xi}\gmm_{\lmb_{0}}(\xi)}{\gmm_{\lmb_{0}}(\xi)} \lmb_0 \\
& = \frac{1}{2}\rd_{x}\rd_{\xi} p (t,x,\xi)  - \frac{1}{2} \Gmm f''(t,x) \frac{\rd_{\xi}\gmm_{\lmb_{0}}(\xi)}{\gmm_{\lmb_{0}}(\xi)} \lmb_0. 
\end{split}
\end{equation} It will be convenient to define \begin{equation*}
\tld{\nb}_{t} = \rd_{t} + \rd_{\xi} p (x, \rd_{x} \Phi) \rd_{x} + \frac{1}{2} \left(\rd_{x} \rd_{\xi} p(t, x, \rd_{x} \Phi) + \rd_{\xi}^{2} p(t, x, \rd_{x} \Phi) \rd_{x}^{2} \Phi \right)
\end{equation*} so that solutions to $\tld{\nb}_{t} a = 0$ obey a-priori $L^{2}$-bounds. Then, \eqref{eq:transport-a-simplified} is simply given by \begin{equation*}
\begin{split}
\tld{\nb}_{t} a = \frac{1}{2} \Gmm f'' \frac{ \lmb_0 \rd_{\xi}\gmm_{\lmb_{0}}(\rd_{x}\Phi)}{\gmm_{\lmb_{0}}(\rd_{x}\Phi)} a.
\end{split}
\end{equation*} Commutation with $\rd_{x}$ again gives rise to a similar factor as before: recalling \eqref{eq:A-def}, we compute that \begin{equation*}
\begin{split}
[\rd_{x}, \tld{\nb}_{t} ] &= [\rd_{x}, \nb_{t}] + \frac12 [\rd_{x}, \rd_{x}\rd_{\xi} p(x,\rd_{x}\Phi) + \rd_{\xi}^2p(x,\rd_{x}\Phi)\rd_{x}^2\Phi ] {= - A \rd_{x} - \frac{1}{2} \rd_{x} A}.
\end{split}
\end{equation*}
For simplicity, we also set \begin{equation*}
\begin{split}
H(t,x):=   - \frac{1}{2} \Gmm f''(t,x) \frac{\gmm_{\lmb_{0}}'(\rd_{x}\Phi(t,x))}{\gmm_{\lmb_{0}}(\rd_{x}\Phi(t,x))} \lmb_0
\end{split}
\end{equation*} so that the equations for $a$, $\rd_{x}a$, and $\rd_{x}^2 a$ read 
\begin{equation*}
\begin{split}
\tld{\nb}_{t} a = Ha ,
\end{split}
\end{equation*}\begin{equation*}
\begin{split}
\tld{\nb}_{t} \rd_{x} a = [\nb_{t},\rd_{x}] a + {\frac{1}{2} \rd_{x} A} a + \rd_{x}(Ha),
\end{split}
\end{equation*} and \begin{equation*}
\begin{split}
\tld{\nb}_{t} \rd_{x}^{2} a =  2[\nb_{t},\rd_{x}] \rd_{x}a + {\frac{3}{2} \rd_{x} A}\rd_{x} a + {\frac{1}{2}   \rd_{x}((\rd_{x} A)a)} + \rd_{x}^2(Ha).
\end{split}
\end{equation*}
\begin{proposition} \label{prop:trans-a} Assume that $a_{0}(x)$ is sufficiently smooth and supported in $(x_{0},x_{1})$. For the solution of \eqref{eq:transport-a-simplified} with initial data $a(t=0)=a_{0}$, we have, for $0 \leq k \leq N_{0}$,
	\begin{equation*}
	\nrm{\mu^{-k} \rd_{x}^{k} a}_{L^{2}} \aleq \sum_{\ell=0}^{k} \nrm{({\mu_{0}^{-1}} \rd_{x})^{\ell} a_{0}}_{L^{2}} \qquad \hbox{ for } t \leq \tau 
	\end{equation*}
	{where
	\begin{equation} \label{eq:wp-scale-0}
	\mu_{0} = \frac{(\log \lmb_{0})^{2}}{x_{0} {\eps(\lmb_{0})}},
\end{equation}
	and the implicit constant depends on $k$ and $\gmm$, as well as $\abs{f''(0, 0)}^{-1} (x \rd_{x})^{k'} f''$,  $\abs{f''(0, 0)}^{-1} (x \rd_{x})^{k'} \Gmm f''$ for $0 \leq k' \leq k$.
}
 
\end{proposition}

\begin{proof}
	In what follows, the symbol $\aleq_{k}$ signifies an implicit constant that may depend on $k$ and $\gmm$, as well as $\abs{f''(0, 0)}^{-1} (x \rd_{x})^{k'} f''$,  $\abs{f''(0, 0)}^{-1} (x \rd_{x})^{k'} \Gmm f''$ for $0 \leq k' \leq k$. 

To begin with, we observe the following form of the equation for $\rd_{x}^{k}a$: \begin{equation*}
	\begin{split}
	\tld{\nb}_{t} \rd_{x}^{k} a = k[\nb_{t},\rd_{x}]\rd_{x}^{k-1} a + \tld{R}_{k}= kA\rd_{x}^{k } a + \tld{R}_{k}, 
	\end{split}
	\end{equation*} where $\tld{R}_{k}$ is a linear combination of the terms \begin{equation*}
	\begin{split}
	{\rd_{x}^{\ell} A\rd_{x}^{k-\ell}a, \quad \rd_{x}^{\ell} H \rd_{x}^{k-\ell}a }
	\end{split}
	\end{equation*} for $0\le \ell\le k$, with the exception that the term with $\ell=0$  does not appear in the case of {$A$}. It will be convenient to first solve the following transport equation:\begin{equation*}
	\begin{split}
	\nb_{t} I = A, 
	\end{split}
	\end{equation*} with the initial data $I(t=0) = 0$. The solution, which is simply the integral in time of $A$ along each characteristic curve, is nothing but $I$ defined in \eqref{eq:def-I}. Since $\tld{\nb}_{t} - \nb_{t}$ is simply a multiplication operator, we have that \begin{equation}\label{eq:amp-conj}
	\begin{split}
	\tld{\nb}_{t} \left( \exp(-k I(t,x)) \rd_{x}^{k} a(t,x) \right) =  \exp(-k I(t,x)) \tld{R}_{k}. 
	\end{split}
	\end{equation}
	
	We now set up the induction hypothesis: For $0 \leq k \leq k_{0}$, we shall require \begin{equation}\label{eq:amp-induction}
	\begin{split}
	\sup_{ 0\le t \le \tau}\nrm{\exp(-kI(t,x))\rd_{x}^{k} a(t,x)}_{L^2}  \lesssim_{k} x_{0}^{-k} {\eps(\lmb_{0})^{-k}} \lmb_0^{\dlt_{3} k^2} \sum_{0\le\ell\le k} \nrm{ (x_{0} {\eps(\lmb_{0})} \rd_{x})^{k}a_{0}  }_{L^2}. 
	\end{split}
	\end{equation} 	
	We first need to check the above for $k = 0$. {Recall the definition of $H$. By Lemma~\ref{lem:Xi-control} or \ref{lem:Xi-control-time-dep} and the definition of $\dot{\lmb}$, we have the obvious bound
	\begin{equation*}
		\abs{H} \aleq \frac{\gmm_{\lmb_{0}}'(\lmb)}{\gmm_{\lmb_{0}}(\lmb)^{2}} \dot{\lmb},
	\end{equation*}
	Then from the equation for $a$, we have, by Gr\"onwall's inequality,}
	\begin{equation*}
	\begin{split}
	\sup_{ 0\le t \le \tau} \nrm{a(t)}_{L^2} \lesssim \nrm{a_0}_{L^2}. 
	\end{split}
	\end{equation*} Let us now estimate $\exp(-k I(t,x)) \rd_{x}^{k} a(t,x) $ in $L^2$ for some $k>0$, using \eqref{eq:amp-induction} for all $\ell<k$.  Recalling the bound \eqref{eq:h-coeff-4} for $\rd_{x}^{k}A$ along characteristics, we  obtain the uniform bound $(k>0)$ \begin{equation*}
	\begin{split}
	|\rd_{x}^{k}A| \lesssim_{k} x_0^{-k} {\eps(\lmb_{0})^{-k}} \frac{\dot{\Xi}}{\Xi} \lmb_{0}^{\dlt_{3}( k^2 - \frac{1}{2})} \exp(kI ) \lesssim_{k} x_0^{-k} {\eps(\lmb_{0})^{-k}} \lmb_{0}^{\dlt_{3}( k^2 - \frac{1}{2})} \exp(kI )  \frac{\dot{\lmb}}{\lmb} 
	\end{split}
	\end{equation*} using simply that $ \frac{\dot{\Xi}}{\Xi} \le C\frac{\dot{\lmb}}{\lmb} $. This allows us to bound \begin{equation*}
	\begin{split}
	\nrm{  \exp(-kI) \rd_{x}^{\ell} A\rd_{x}^{k-\ell}a }_{L^2} \lesssim_{k} x_0^{-\ell} {\eps(\lmb_{0})^{-\ell}} \lmb_{0}^{\dlt_{3}( \ell^2 -\frac{1}{4})}   \frac{\dot{\lmb}}{\lmb} \nrm{  \exp(-(k-\ell)I) \rd_{x}^{k-\ell}a }_{L^2} . 
	\end{split}
	\end{equation*} 
	{Finally,} we need to estimate derivatives of $H$. We claim that for $k>0$, \begin{equation*}
	\begin{split}
|	\rd_{x}^{k} H | \lesssim_{k} x_{0}^{-k} {\eps(\lmb_{0})^{-k}} \exp(kI) \lmb_{0}^{\dlt_{3}(k^2-\frac{1}{4})} {\frac{\gmm_{\lmb_{0}}'(\lmb)}{\gmm_{\lmb_{0}}(\lmb)^{2}} \dot{\lmb}}. 
	\end{split}
	\end{equation*} We omit the proof, which can be done along the lines of the proof of   \eqref{eq:h-coeff-1}--\eqref{eq:h-coeff-4}. Applying these bounds to \eqref{eq:amp-conj} and taking the $L^2$ inner product with $\exp(-k I(t,x)) \rd_{x}^{k} a(t,x)$, we obtain that\begin{equation*}
	\begin{split}
	\frac{\ud}{\ud t} \nrm{ \exp(-k I(t,x)) \rd_{x}^{k} a(t,x) }_{L^2} 
	&\lesssim_{k} \sum_{\ell={1}}^{k} x_0^{-\ell} {\eps(\lmb_{0})^{-\ell}}\lmb_{0}^{\dlt_{3}( \ell^2-\frac{1}{4}) } {\left(  \frac{\dot{\lmb}}{\lmb}  + \frac{\gmm_{\lmb_{0}}'(\lmb)\dot{\lmb}}{\gmm_{\lmb_{0}}(\lmb)^{2}} \right)} \nrm{  \exp(-(k-\ell)I) \rd_{x}^{k-\ell} a }_{L^2} \\
	&\paleq + {\frac{\gmm_{\lmb_{0}}'(\lmb) \dot{\lmb}}{\gmm_{\lmb_{0}}(\lmb)^{2}} } \nrm{ \exp(-k I(t,x)) \rd_{x}^{k} a(t,x) }_{L^2} \\
	&\lesssim_{k} x_0^{-k} {\eps(\lmb_{0})^{-k}}\lmb_{0}^{\dlt_{3}( k^2-\frac{1}{4}) } {\left(  \frac{\dot{\lmb}}{\lmb}  + \frac{\gmm_{\lmb_{0}}'(\lmb) \dot{\lmb}}{\gmm_{\lmb_{0}}(\lmb)^{2}} \right)} \sum_{\ell=0}^{{k-1}} \nrm{ (x_{0} {\eps(\lmb_{0})}\rd_{x})^{\ell }a_{0}  }_{L^2}  \\
	&\paleq + {\frac{\gmm_{\lmb_{0}}'(\lmb) \dot{\lmb}}{\gmm_{\lmb_{0}}(\lmb)^{2}}} \nrm{ \exp(-k I(t,x)) \rd_{x}^{k} a(t,x) }_{L^{2}}.
	\end{split}
	\end{equation*} Applying Gr\"onwall's inequality finishes the proof of \eqref{eq:amp-induction}. Recalling the definitions of $\mu$ and $\mu_{0}$ from \eqref{eq:wp-scale} and \eqref{eq:wp-scale-0}, respectively, this completes the proof of the proposition (by taking $\Lmb$ larger if necessary). \qedhere  
\end{proof}

\subsection{Cutoff and extension of the phase function}\label{subsec:cutoff-extend}

In this section, we shall extend $\Phi$ constructed in the previous subsections globally in space. For this purpose, we take points $x_{0}' = \frac{2x_{0}+x_{1}}{3}$ and $x_{1}' = \frac{x_{0}+2x_{1}}{3}$ so that $x_0 < x_{0}' < x_{1}' < x_{1}$ and let $\chi_{(x_{0}, x_{1})}$ be a smooth cutoff supported in $(x_{0}, x_{1})$ that equals $1$ on $(x_{0}', x_{1}')$. Then the support of $1 - \chi_{(x_{0}, x_{1})}$ has two components; we denote by $\chi_{(-\infty, x_{0})}$ and $\chi_{(x_{1}, \infty)}$ the smooth cutoffs supported in the left and the right components, respectively, such that $1 - \chi_{(x_{0}, x_{1})} = \chi_{(-\infty, x_{0})} +  \chi_{(x_{1}, \infty)}$. For $\ubr{x} \in (x_{0}, x_{1})$, we define
\begin{equation}
\begin{aligned}
	\rd_{x} \Phi^{global}(t, X(t; \ubr{x})) &= \chi_{(x_{0}, x_{1})}(\ubr{x}) \rd_{x} \Phi(t, X(t; \ubr{x})) \\
	&\peq + \chi_{(-\infty, x_{0})}(\ubr{x}) \rd_{x} \Phi(t, X(t; x_{0}))
	+ \chi_{(x_{1}, \infty)}(\ubr{x}) \rd_{x} \Phi(t, X(t; x_{1})),
\end{aligned}
\end{equation}
where $X(t; \ubr{x})$ is the characteristic curve solving $\dot{X} = \rd_{\xi} p(t, X, \rd_{x} \Phi(t, X))$ with $X(0; \ubr{x}) = \ubr{x}$. We furthermore normalize
\begin{equation*}
	\Phi^{global}(t, X(t; \tfrac{x'_{0}+x'_{1}}{2})) = \Phi(t, X(t; \tfrac{x'_{0}+x'_{1}}{2})).
\end{equation*}
Finally, for $x$ outside of the image of $X(t; (x_{0}, x_{1}))$, we extend $\rd_{x} \Phi^{global}(t, x)$ (smoothly) by constants. 

By definition,
\begin{equation*}
	\Phi^{global} = \Phi(t, x) \qquad \hbox{ for } X(t, x_{0}') \leq x \leq X(t, x_{1}').
\end{equation*}
To continue, for each $t$, denote by $\ubr{x}(t, x)$ the inverse of the map $\ubr{x} \mapsto X(t; \ubr{x})$. Since
\begin{equation*}
	\nb_{t} \ubr{x}(t, x) = 0, \qquad \ubr{x}(0, x) = x,
\end{equation*}
by an argument similar to Section~\ref{subsec:transport} using \eqref{eq:h-coeff-4}, we have
\begin{equation*}
\abs{\rd_{x}^{k} \ubr{x}(t, X(t; \ubr{x}))} \aleq_{k} \mu^{k}.
\end{equation*}
It follows that
{\begin{equation*}
	\lmb \leq \rd_{x} \Phi^{global} \leq 2 \lmb, \qquad
	\abs{\rd_{x}^{k} \rd_{x} \Phi^{global}} \aleq_{k} \mu^{k} \lmb \qquad \hbox{ for } 1 \leq k \leq N.
\end{equation*} In Sections~\ref{sec:wavepacket} and \ref{sec:proofs}, {\bf we shall write $\Phi$ for $\Phi^{global}$}. This abuse of notation is minor, since our wave packet would be of the form $\Re(a e^{i \Phi})$, where $a(t, \cdot)$ is supported in $(X(t, x_{0}'), X(t, x_{1}'))$, on which $\Phi^{global} = \Phi$. }

\section{Oscillatory integrals} \label{sec:conj-err}
The purpose of this section is to formulate and prove a result concerning the $L^{2}$-bound for the operators arising in the remainder in Proposition~\ref{prop:conj-exp}. {Logically, this section is \emph{self-contained} so {\bf the symbols in this section are detached from their meanings in the previous sections.} On the other hand, the results of this section will be applied in Section~\ref{sec:wavepacket} to objects that are suggested by the notation here.}

To begin with, observe that each term in the remainder symbols ${}^{(\Phi)} r_{p, -1}(x, \xi)$ and ${}^{(\Phi)} r_{p, -2}(x, \xi)$ in \eqref{eq:conj-rem-1}--\eqref{eq:conj-rem-2} is of the form $\int_{0}^{1} {}^{(\Phi)} q(x, \xi; \sgm) \, \ud \sgm$, where 
\begin{equation} \label{eq:conj-q-sgm}
{}^{(\Phi)} q(x, \xi; \sgm) := \iint q(x, \sgm \eta + (1-\sgm) \rd_{y} \Phi(y)) e^{i (\Phi(y) - \Phi(x) + (\eta - \xi) \cdot (x-y))} \, \ud y \frac{\ud \eta}{(2 \pi)^{2}}
\end{equation}
for some symbol $q$. Hence, it is expedient to formulate a $L^{2}$-bound result for an operator of the form ${}^{(\Phi)} q(x, D; \sgm)$ under suitable assumptions on $q$. 

Our main result is as follows:
\begin{proposition} \label{prop:conj-err}
	Let $\lmb_{0} \in \bbN$, $\lmb \geq \lmb_{0}$, $\mu > 0$, $\lmb_{q} > 0$ and $N_{0} \in \bbN$. Assume that the phase function $\Phi : \bbR_{x_{2}} \to \bbR$ and the symbol $q(x_{2}, \xi_{2}) : \bbR_{x_{2}} \times \bbR_{\xi_{2}} \to \bbC$ satisfy, for any integer $0 \leq k \leq N_{0}$,
	\begin{align} 
	\abs{\rd_{x_{2}}^{k} \rd_{x_{2}} \Phi(x)} &\leq A_{\Phi, k} \mu^{k} \lmb, \label{eq:conj-err-Phi} \\
	\abs{\rd_{\xi_{2}}^{k} q(x_{2}, \xi_{2})} &\leq A_{q, k} \lmb_{q}^{-k} g(x_{2}) M(\xi_{2}), \label{eq:conj-err-q}
	\end{align}
	respectively, where $A_{\Phi, j}, A_{q, j} > 0$ are increasing. Furthermore, assume that $\lmb_{q}, \lmb, \mu$ satisfy
	\begin{equation} \label{eq:conj-err-lmb-mu}
	\frac{\mu}{\lmb_{q}} \leq 1, \quad \frac{\lmb \mu}{\lmb_{q}^{2}} \leq 1.
	\end{equation}
	Then for any smooth $a = a(x_{2})$, $0 \leq \sgm \leq 1$, {$X_{0} \in \bbR$, $s \in \bbR$} and $\max\set{10^{4}, 100 \abs{s}} \leq N \leq N_{0}$, we have the bound
	\begin{equation} \label{eq:conj-err}
\begin{aligned}
	&\nrm*{\left({}^{(\Phi)} q(x_{2}, D_{2}; \sgm) a\right) {\brk{\mu (x_{2} - X_{0})}^{s}}}_{L^{2}}  \\
	&\aleq_{s, N, A_{\Phi, N}, A_{q, N}} \sum_{\ell = 0}^{N} \nrm*{\left(\br{g}_{<\mu^{-1}}(x_{2}) \br{\calM}_{<\mu^{-1}} (x_{2}; \sgm) (\mu^{-1} \rd_{x_{2}})^{\ell} a(x_{2})\right) {\brk{\mu (x_{2}-X_{0})}^{s}}}_{L^{2}},
\end{aligned}	\end{equation}
	where
	\begin{align*}
	\br{g}_{<\mu^{-1}}(x_{2})
	&= \sup_{y_{2}}  \frac{g(y_{2})}{\brk{\mu (x_{2} - y_{2})}^{\frac{N}{100}}}, \\
	\br{\calM}_{<\mu^{-1}}(x_{2}; \sgm)
	&= \sup_{y_{2}} \int  \frac{\calM_{<\lmb_{q}^{-1}, <\lmb_{q}}(y_{2}, \xi_{2}; \sgm)}{\brk{\mu (x_{2}-y_{2})}^{\frac{N}{100}}} \frac{\mu^{-1}}{\brk{\mu^{-1} \xi_{2}}^{\frac{N}{100}}}
	\, \ud \xi_{2},
	\end{align*}
	and
	\begin{align*}
	& \calM_{<\lmb_{q}^{-1}, <\lmb_{q}}(x_{2}, \xi_{2}; \sgm)  \\
	&= \iint M(\sgm (\zt + \xi_{2} + \rd_{x_{2}} \Phi(x) - \rd_{x_{2}} \Phi(x_{1}, x_{2} + z) - \tfrac{1}{2} \rd_{x_{2}}^{2} \Phi(x) z) + \rd_{x_{2}} \Phi(x_{1}, x_{2}+z)) \\
	&\phantom{= \iint} \times \frac{\lmb_{q}}{\brk{\lmb_{q}z}^{\frac{N}{100}}} \frac{\lmb_{q}^{-1}}{\brk{\lmb_{q}^{-1} \zt}^{\frac{N}{100}}} \, \ud \zt \ud z.
	\end{align*}
\end{proposition} 
\begin{remark}
	Note that the RHS of \eqref{eq:conj-err} involves a spatial weight $\br{g}_{<\mu^{-1}}(x_{2}) \br{\calM}_{<\mu^{-1}} (x_{2}; \sgm)$; this feature is important for exploiting the degeneracy of the principal symbol $p_{\bgtht}$ in the proof of the error bound in Section~\ref{subsec:wp-error}. We also point out that in \eqref{eq:conj-err}, we are allowed to lose as many derivatives of $a$ (with weights in $\mu^{-1}$) as we wish; this feature simplifies the proof (see Remark~\ref{rem:conj-err-symb-x}). 
\end{remark}

\subsection{Oscillatory integral bounds for the symbol} \label{subsec:conj-err-symb}
In this section, we aim to derive key pointwise bounds for the symbol ${}^{(\Phi)} q(x, \xi; \sgm)$ and its $\xi$-derivatives. 

The main goal of this subsection is to prove the following pointwise bound for the symbol ${}^{(\Phi)} q(x_{2}, \xi_{2}; \sgm)$:
\begin{lemma} \label{lem:conj-err-symb}
	For any $x_{1} \in \bbR$, $\sgm \in [0, 1]$, $\ell, k \in \bbN_{0}$ and $C_{0} \in \bbN$ such that $\ell + 10C_{0} \leq N_{0}$, we have
	\begin{equation} \label{eq:osc-Q-key}
	\begin{aligned}
	\abs{\rd_{\xi_{2}}^{\ell} {}^{(\Phi)} q(x_{2}, \xi_{2}; \sgm)} &\aleq_{\ell, A_{\Phi, 10 C_{0}}, A_{q, \ell + 10 C_{0}}} \lmb_{q}^{-\ell} 
	g(x_{2}) \calM_{<\lmb_{q}^{-1}, <\lmb_{q}}(x_{2}, \xi_{2}; \sgm) \max\set*{1, \mu \lmb_{q}^{-1}, \frac{\lmb}{\lmb_{q}} \left( \mu \lmb_{q}^{-1}  + \mu^{3} \lmb_{q}^{-3} \right)}^{C_{0}+2},
	\end{aligned}
	\end{equation}
	where
	\begin{align*}
	&\calM_{<\lmb_{q}^{-1}, <\lmb_{q}}(x_{2}, \xi_{2}; \sgm)  \\
	&= \iint M(\sgm (\zt + \xi_{2} + \rd_{x_{2}} \Phi(x_{2}) - \rd_{x_{2}} \Phi(x_{2}+z) - \tfrac{1}{2} \rd_{x_{2}}^{2} \Phi(x_{2}) z) + \rd_{x_{2}} \Phi(x_{2}+z)) \\
	&\phantom{=\iint} \times \frac{\lmb_{q}}{\brk{\lmb_{q} z}^{C_{0}}} \frac{\lmb_{q}^{-1}}{\brk{\lmb_{q}^{-1} \zt}^{C_{0}}} \, \ud \zt \ud z.
	\end{align*}
\end{lemma}

\begin{proof}
	We begin by laying out some notational simplifications to be used in the proof. To simplify the notation, we shall write $x = x_{2}$, $\xi = \xi_{2}$, $q(\xi) = q(x, \xi)$ and $Q(x, \xi) = {}^{(\Phi)} q(x, \xi; \sgm)$. Then
	\begin{equation*}
	Q(x, \xi) = \iint q(\sgm \eta + (1-\sgm) \rd_{x} \Phi(y)) e^{i (\Phi(y) - \Phi(x) + (\eta - \xi)(x - y))} \, \ud y \frac{\ud \eta}{2 \pi}.
	\end{equation*}
	Moreover, the assumption \eqref{eq:conj-err-q} becomes
	\begin{equation} \label{eq:q-order}
	\abs{\rd_{\xi}^{n} q(\xi)} \aleq_{n, A_{q, n}} \lmb_{q}^{-n} g(x) M(\xi).
	\end{equation}
	Finally, by the second assumption on the phase, we have $\abs{\rd_{x}^{n} \rd_{x} \Phi(x)} \aleq_{n, A_{\Phi, n}} \mu^{n} \lmb$. 
	
	In what follows, we first give a detailed proof of the case $\ell = 0$, then indicate the necessary modifications for higher $\ell$'s.
	
	The phase function $\psi(x, \xi; y, \eta)$, defined by
	\begin{align*}
	\psi(x, \xi; y, \eta) = \Phi(y) - \Phi(x) + (\eta - \xi) (x-y),
	\end{align*}
	is stationary at $y = x$ and $\eta = \xi + \rd_{x} \Phi(x)$. By Taylor expansion, we may write
	\begin{align*}
	\psi(x, \xi; y, \eta) & = \rd_{x} \Phi(x) (y-x) + \tfrac{1}{2} \rd_{x}^{2} \Phi(x) (y-x)^{2} + (\eta - \xi) (x-y) + \frac{1}{3!} r^{(3)}[\Phi](y, x) (y-x)^{3} \\
	& =  (x-y) (\eta - \xi - \rd_{x} \Phi(x) + \tfrac{1}{2} \rd_{x}^{2} \Phi(x) (x-y)) + \frac{1}{3!} r^{(3)}[\Phi](y, x) (y-x)^{3}.
	\end{align*}
	We make the change of variables $(y, \eta) \mapsto (z, \zt)$, where
	\begin{align*}
	z &= y-x, \\
	\zt &= \eta - \xi - \rd_{x} \Phi(x) + \tfrac{1}{2} \rd_{x}^{2} \Phi(x) (x-y),
	\end{align*}
	so that
	\begin{equation} \label{eq:Q-reduced}
	\begin{aligned}
	Q(x, \xi) &= \iint \tld{q}(x, \xi; z, \zt) e^{i r^{(3)}[\Phi](x+z, x) \frac{z^{3}}{3!}} e^{- i \zt z} \, \ud \zt \ud z,
	\end{aligned}\end{equation}
	where
	\begin{equation*}
	\tld{q}(x, \xi; z, \zt) = \frac{1}{2\pi} q\left( \sgm (\zt + \xi + \rd_{x} \Phi(x) - \rd_{x} \Phi(x+z) - \tfrac{1}{2} \rd_{x}^{2} \Phi(x) z) + \rd_{x} \Phi(x+z) \right).
	\end{equation*}
	Now observe that
	\begin{equation} \label{eq:Q-ampl-reg}
	\begin{aligned}
	&\abs{\rd_{\zt}^{m} \rd_{z}^{n} \tld{q}(x, \xi; z, \zt)} \\
	& \aleq_{m, n, A_{\Phi, n}, A_{q, m+n}} \lmb_{q}^{-m} \max\set*{\frac{\lmb \mu}{\lmb_{q}}, \mu}^{n} \\
	& \phantom{\aleq_{m, n}} \times g(x) M(\sgm (\zt + \xi + \rd_{x} \Phi(x) - \rd_{x} \Phi(x+z) - \tfrac{1}{2} \rd_{x}^{2} \Phi(x) z) + \rd_{x} \Phi(x+z)).
	\end{aligned}\end{equation}
	
	We shall view \eqref{eq:Q-reduced} as an oscillatory integral with the simple phase $\zt z$. Accordingly, we introduce the dyadic decomposition
	\begin{equation*}
	Q_{Z, H}(x, \xi) = \iint \chi_{H}(\zt) \chi_{Z}(z) \tld{q}(x, \xi; z, \zt) e^{i r^{(3)}[\Phi](x+z, x) \frac{z^{3}}{3!}} e^{- i \zt z} \, \ud \zt \ud z,
	\end{equation*}
	as well as $Q_{<Z, H}(x, \xi)$, $Q_{Z, <H}(x, \xi)$ and $Q_{< Z, < H}(x, \xi)$, which are similarly defined. We also introduce the shorthands
	\begin{equation*}
	M_{Z, H} = \int_{z \in (\frac{1}{2} Z, 2 Z)} \int_{\zt \in (\frac{1}{2} H, 2 H)} M(\sgm (\zt + \xi + \rd_{x} \Phi(x) - \rd_{x} \Phi(x+z) - \tfrac{1}{2} \rd_{x}^{2} \Phi(x) z) + \rd_{x} \Phi(x+z)) \, \ud \zt \ud z ,
	\end{equation*}
	as well as $M_{< H, Z}$, $M_{H, < Z}$ and $M_{< H,  <Z}$, which are similarly defined.
	
	The core localized integral bounds are as follows: For dyadic numbers $Z \ageq \lmb_{q}^{-1}$ and $H \ageq \lmb_{q}$,
	\begin{align}
	\abs{Q_{Z, H}(x, \xi)}
	&\aleq_{m, n, A_{\Phi, n+2}, A_{q, m+n}} g(x) M_{Z, H} \left(Z^{-1} \lmb_{q}^{-1} \right)^{m} \label{eq:osc-Q-HZ} \\
	&\phantom{\aleq_{m, n, A_{\Phi, n+2}, A_{q, m+n}}} \times \left( H^{-1} \max\set*{Z^{-1}, \frac{\lmb \mu}{\lmb_{q}}, \mu, \lmb \mu^{2} Z^{2} + \lmb \mu^{3} Z^{3}} \right)^{n}, \notag \\
	\abs{Q_{< \lmb_{q}^{-1}, H}(x, \xi)} 
	&\aleq_{n, A_{\Phi, n+2}, A_{q, n}} g(x) M_{< \lmb_{q}^{-1}, H}  \left( H^{-1} \max\set*{\lmb_{q}, \frac{\lmb \mu}{\lmb_{q}}, \mu, \frac{\lmb}{\lmb_{q}} \frac{\mu^{2}}{\lmb_{q}}  + \frac{\lmb}{\lmb_{q}} \frac{\mu^{3}}{\lmb_{q}^{2}}} \right)^{n}, \label{eq:osc-Q-Hlow} \\
	\abs{Q_{Z, < \lmb_{q}}(x, \xi)}
	& \aleq_{m, A_{q, m}} g(x) M_{Z, < \lmb_{q}}  \left(Z^{-1} \lmb_{q}^{-1} \right)^{m}, \label{eq:osc-Q-lowZ}\\
	\abs{Q_{< \lmb_{q}^{-1}, < \lmb_{q}}(x, \xi)}
	& \aleq g(x) M_{< \lmb_{q}^{-1}, < \lmb_{q}}. \label{eq:osc-Q-lowlow}
	\end{align}
	Indeed, \eqref{eq:osc-Q-lowlow} is trivial by the definition of $M_{<\lmb_{q}^{-1}, <\lmb_{q}}$ and the fact that the volume of the support of $\chi_{<\lmb_{q}}(\zt) \chi_{<\lmb_{q}^{-1}}(z)$ is $O(1)$. To prove \eqref{eq:osc-Q-lowZ}, we simply use $e^{-i \zt z} = i z^{-1} \rd_{\zt} e^{- i \zt z}$ and integration by parts and estimate
	\begin{align*}
	\abs{Q_{Z, <\lmb_{q}}}
	&= \abs*{\iint \chi_{<\lmb_{q}}(\zt) \chi_{Z}(z) \tld{q}(x, \xi; z, \zt) e^{i r^{(3)}[\Phi](x+z, x) \frac{z^{3}}{3!}} z^{-n} \rd_{\zt}^{n} e^{- i \zt z} \, \ud \zt \ud z} \\
	&=\abs*{\iint \rd_{\zt}^{m} \left( \chi_{<\lmb_{q}}(\zt) \tld{q}(x, \xi; z, \zt)\right) \chi_{Z}(z)  e^{i r^{(3)}[\Phi](x+z, x) \frac{z^{3}}{3!}} z^{-m}  e^{- i \zt z} \, \ud \zt \ud z} \\
	&\aleq_{m, A_{q, m}} g(x) M_{Z, <\lmb_{q}}  \, Z^{-m} \lmb_{q}^{-m},
	\end{align*}
	where $Z \lmb_{q}$ is the volume of the support of $\chi_{<\lmb_{q}}(\zt) \chi_{Z}(z)$.
	Next, to prove \eqref{eq:osc-Q-Hlow}, we use $e^{-i \zt z} = i \zt^{-1} \rd_{z} e^{- i \zt z}$ and integration by parts to bound
	\begin{align*}
	\abs*{Q_{<\lmb_{q}^{-1}, H}}
	&= \abs*{\iint \chi_{H}(\zt) \rd_{z}^{n} \left( \chi_{<\lmb_{q}^{-1}}(z) \tld{q}(x, \xi; z, \zt) e^{i r^{(3)}[\Phi](x+z, x) \frac{z^{3}}{3!}} \right) \zt^{-n} e^{- i \zt z} \, \ud \zt \ud z}.
	\end{align*}
	Then the desired estimate \eqref{eq:osc-Q-Hlow} follows, via the chain rule and the Leibniz rule, from \eqref{eq:Q-ampl-reg} and
	\begin{equation} \label{eq:r3Phi-reg}
	\abs*{\rd_{z}^{n'} \left(r^{(3)}[\Phi](x+z, x)\right)} \aleq_{A_{\Phi, n'+2}} \lmb \mu^{n'+2}. 
	\end{equation}
	Finally, to prove \eqref{eq:osc-Q-HZ}, we use both $e^{-i \zt z} = i z^{-1} \rd_{\zt} e^{- i \zt z}$ and $e^{-i \zt z} = i \zt^{-1} \rd_{z} e^{- i \zt z}$ and integration by parts to estimate
	\begin{align*}
	\abs*{Q_{Z, H}}
	&= \abs*{\iint \rd_{\zt}^{m} \left( \chi_{H}(\zt) \tld{q}(x, \xi; z, \zt) \right) \chi_{Z}(z) e^{i r^{(3)}[\Phi](x+z, x) \frac{z^{3}}{3!}} z^{-m} e^{- i \zt z} \, \ud \zt \ud z} \\
	&= \abs*{\iint \rd_{z}^{n} \left( \rd_{\zt}^{m} \left( \chi_{H}(\zt) \tld{q}(x, \xi; z, \zt) \right) \chi_{Z}(z) z^{-m} e^{i r^{(3)}[\Phi](x+z, x) \frac{z^{3}}{3!}} \right) \zt^{-n} e^{- i \zt z} \, \ud \zt \ud z}.
	\end{align*}
	From the last line, the desired estimate \eqref{eq:osc-Q-HZ} follows, via the chain rule and the Leibniz rule, from \eqref{eq:Q-ampl-reg} and \eqref{eq:r3Phi-reg}.
	
	We are now ready to prove \eqref{eq:osc-Q-key}. In what follows, we omit the dependence of implicit constants on $A_{\Phi, 10 C_{0}}$ and $A_{q, 10 C_{0}}$.  Consider \eqref{eq:osc-Q-HZ} with $m = C_{0}+2$ and $n = 3m + C_{0}+2$, and sum over $Z \ageq \lmb_{q}^{-1}$, $H \ageq \lmb_{q}$. Then
	\begin{align*}
	& \sum_{Z \ageq \lmb_{q}^{-1}, \, H \ageq \lmb_{q}} \abs{Q_{Z, H}(x, \xi)}  \\
	& \aleq g(x) \sum_{Z \ageq \lmb_{q}^{-1}, \, H \ageq \lmb_{q}}  M_{Z, H} \cdot H^{-m} \max\set*{Z^{-1}, \frac{\lmb \mu}{\lmb_{q}}, \mu, \lmb \mu^{2} Z^{2} + \lmb \mu^{3} Z^{3}}^{m}
	Z^{-n} \lmb_{q}^{-n} \\
	& \aleq g(x) \sum_{H \ageq \lmb_{q}, } \left( \sum_{Z \ageq \lmb_{q}^{-1}} M_{Z, H}  (Z \lmb_{q})^{-C_{0}}\right) \, H^{-m} \max\set*{\lmb_{q}, \frac{\lmb \mu}{\lmb_{q}}, \mu, \lmb \mu^{2} \lmb_{q}^{-2} + \lmb \mu^{3} \lmb_{q}^{-3}}^{m} \\	
	& \aleq g(x) \left( \sum_{Z \ageq \lmb_{q}^{-1}, \, H \ageq \lmb_{q}} M_{Z, H}  (Z \lmb_{q})^{-C_{0}} (H \lmb_{q}^{-1})^{-C_{0}}\right) \max\set*{1, \mu \lmb_{q}^{-1}, \frac{\lmb}{\lmb_{q}} \left(\mu \lmb_{q}^{-1} + \mu^{3} \lmb_{q}^{-3}\right)}^{C_{0}+2} \\
	& \aleq g(x) \calM_{<\lmb_{q}^{-1}, <\lmb_{q}} \max\set*{1, \mu \lmb_{q}^{-1}, \frac{\lmb}{\lmb_{q}} \left(\mu \lmb_{q}^{-1} + \mu^{3} \lmb_{q}^{-3}\right)}^{C_{0}+2}.
	\end{align*}
	Next, invoking \eqref{eq:osc-Q-Hlow} with $m = C_{0} + 2$ and summing over $H \ageq \lmb_{q}$, we obtain
	\begin{align*}
	& \sum_{H \ageq \lmb_{q}} \abs{Q_{< \lmb_{q}^{-1}, H}(x, \xi)} \\
	& \aleq g(x) \sum_{H \ageq \lmb_{q}} M_{< \lmb_{q}^{-1}, H}   \cdot H^{-m} \max\set*{\lmb_{q}, \frac{\lmb \mu}{\lmb_{q}}, \mu, \frac{\lmb}{\lmb_{q}} \frac{\mu^{2}}{\lmb_{q}}  + \frac{\lmb}{\lmb_{q}} \frac{\mu^{3}}{\lmb_{q}^{2}}}^{m} \\
	& \aleq g(x) \left( \sum_{H \ageq \lmb_{q}} M_{< \lmb_{q}^{-1}, H}   (H \lmb_{q}^{-1})^{-C_{0}} \right) \max\set*{1, \mu \lmb_{q}^{-1}, \frac{\lmb}{\lmb_{q}}  \left( \mu \lmb_{q}^{-1}  + \mu^{3} \lmb_{q}^{-3} \right)}^{C_{0}+2} \\
	& \aleq g(x) \calM_{<\lmb_{q}^{-1}, < \lmb_{q}} \max\set*{1, \mu \lmb_{q}^{-1}, \frac{\lmb}{\lmb_{q}} \left(\mu \lmb_{q}^{-1} + \mu^{3} \lmb_{q}^{-3}\right)}^{C_{0}+2}.
	\end{align*}
	From \eqref{eq:osc-Q-lowZ} with $n = C_{0} + 2$ and summing over $Z \ageq \lmb_{q}^{-1}$, we obtain
	\begin{align*}
	\sum_{Z \ageq \lmb_{q}^{-1}} \abs{Q_{Z, < \lmb_{q}}(x, \xi)}
	& \aleq g(x) \sum_{Z \ageq \lmb_{q}^{-1}} M_{Z, < \lmb_{q}}  \left(Z^{-1} \lmb_{q}^{-1} \right)^{n} \\
	& \aleq g(x) \left( \sum_{Z \ageq \lmb_{q}^{-1}} M_{Z, < \lmb_{q}}  (Z \lmb_{q})^{-C_{0}} \right) \\
	& \aleq g(x) \calM_{< \lmb_{q}^{-1}, <\lmb_{q}}.
	\end{align*}
	Combining the preceding three bounds with \eqref{eq:osc-Q-lowlow}, we obtain \eqref{eq:osc-Q-key} with $\ell = 0$.
	
	To handle the case $\ell > 0$, note in \eqref{eq:Q-reduced} that the $\xi$-dependence of $Q(x, \xi)$ arises entirely from that of $\tld{q}(x, \xi; z, \zt)$ in the integrand. By \eqref{eq:q-order} and the definition of $\tld{q}(x, \xi; z, \zt)$, observe that
	\begin{equation*}
	\abs{\rd_{\xi}^{\ell} \rd_{\zt}^{m} \rd_{z}^{n} \tld{q}(x, \xi; z, \zt)}
	\aleq_{\ell, m, n, A_{\Phi, n}, A_{q, \ell+m+n}} \lmb_{q}^{-\ell} \times (\hbox{RHS of \eqref{eq:Q-ampl-reg}}).
	\end{equation*}
	Using the preceding bound in place of \eqref{eq:Q-ampl-reg} and repeating the remainder of the $\ell = 0$ proof, we obtain \eqref{eq:osc-Q-key} with higher $\ell$'s.
	\qedhere
\end{proof}

\begin{remark} \label{rem:conj-err-symb-x}
	Under additional assumptions on $x$-derivatives of $q(x_{1}, x, \lmb_{0}, \xi)$, the computation as above also yields bounds on $\rd_{x}^{k} Q$. For instance, if the symbol $q$ is independent of $x$, then $\abs{\rd_{x}^{k} \rd_{\xi}^{\ell} Q} \aleq \left(\frac{\lmb}{\lmb_{q}}\mu \max\set{1, \mu^{2} \lmb_{q}^{-2}}\right)^{k} \times (\hbox{RHS of \eqref{eq:osc-Q-key} with $g \equiv 1$})$. In this case, by the Calder\'on--Vaillancourt theorem, $Q(x, D)$ is $L^{2}$-bounded with norm $O(\calM_{< \lmb_{q}, < \lmb_{q}^{-1}})$ as long as $\frac{\lmb \mu}{\lmb_{q}^{2}} = O(1)$ and $\lmb_{q}^{-1} \mu = O(1)$. However, since we are able to exploit higher derivatives bounds for $a$ in \eqref{eq:conj-err} (which originate from Section~\ref{sec:HJE}), such spatial derivative bounds for $Q$ will not be necessary; see Section~\ref{subsec:conj-err-kernel} below.
\end{remark}

\subsection{Kernel bounds and proof of Proposition~\ref{prop:conj-err}} \label{subsec:conj-err-kernel}
To complete the proof of Proposition~\ref{prop:conj-err}, it remains to translate the symbol bound \eqref{eq:osc-Q-key} to an operator bound for ${}^{(\Phi)} q(x_{2}, D_{2}; \sgm)$. As we are allowed to use high derivatives of $a$ on the RHS of \eqref{eq:conj-err}, it suffices to utilize the following simple kernel bound: 

\begin{lemma} \label{lem:psdo-kernel}
	Let $n \in \bbN_{0}$ and $p(x, \xi) \in C^{\infty}(\bbR \times \bbR)$ satisfy, for any $0 \leq \ell \leq n$,
	\begin{equation} \label{eq:psdo-kernel-hyp}
	\abs{\rd_{\xi}^{\ell} p(x, \xi)} \leq A_{p, \ell} \lmb_{p}^{-\ell} G(x, \xi),
	\end{equation}
	where $A_{p, \ell}$ are increasing.
	Let $\nu \in 2^{\bbZ}$ and $k \in \bbN_{0}$. Then the kernel $K_{p (-\lap)^{-k} P_{\nu}}(x, y)$ of the operator $p(x, D) (-\lap)^{-k} P_{\nu}(D) u$
	obeys, for any $x \neq y$, the pointwise bound
	\begin{equation} \label{eq:psdo-kernel}
	\abs{K_{p (-\lap)^{-k} P_{\nu}}(x, y)} \aleq_{n, A_{p, n}, k} \frac{\nu^{-2k} \int_{\frac{1}{4} \nu < \abs{\xi} < 4 \nu} G(x, \xi) \, \ud \xi}{\brk{\min \set{\nu, \lmb_{p}} (x-y)}^{n}}.
	\end{equation}
	Moreover, the kernel $K_{p P_{< \nu}}(x, y)$ of the operator $p(x, D) P_{< \nu}(D) u$ obeys the same bound as $K_{p P_{\nu}}$.
\end{lemma}
\begin{proof}
	By definition,
	\begin{equation*}
	p(x, D) P_{\nu}(D) u = \int p(x, \xi) P_{\nu}(\xi) \hat{u}(\xi) e^{i x \cdot \xi}\, \frac{\ud \xi}{ 2\pi } = \int K_{p P_{\nu}}(x, y) u(y)\, \ud y,
	\end{equation*}
	so formally,
	\begin{equation*}
	K_{p (-\lap)^{-2k} P_{\nu}}(x, y) = \int p(x, \xi) \abs{\xi}^{-2k} P_{\nu}(\xi) e^{i (x-y) \cdot \xi} \, \frac{\ud \xi}{ 2\pi }.
	\end{equation*}
	Thanks to the compact support property of $P_{\nu}(\xi)$, this formal computation is readily justified. Moreover, by the identity $e^{i (x-y) \cdot \xi} = (i (x_{j}-y_{j}))^{-1}\rd_{\xi_{j}} e^{i (x-y) \cdot \xi}$, integration by parts and using the bounds \eqref{eq:psdo-kernel-hyp} and $\abs{\rd_{\xi}^{\ell} P_{\nu}(\xi)} \aleq_{\ell} \nu^{-\ell}$, \eqref{eq:psdo-kernel} follows. The case of $K_{p P_{< \nu}}$ is entirely analogous. \qedhere
\end{proof}

We are now ready to complete the proof of Proposition~\ref{prop:conj-err}. 
\begin{proof}
	We introduce the shorthand $Q(x_{2}, \xi_{2}) = {}^{(\Phi)} q(x_{2}, D_{2}; \sgm)$. We begin by splitting
	\begin{align*}
	Q(x_{2}, D_{2}) a
	= Q(x_{2}, D_{2}) P_{<\mu}(D_{2}) a
	+ \sum_{\nu \in 2^{\bbZ}, \nu > \mu} {}^{(\Phi)} q(x_{2}, D_{2}; \sgm) P_{\nu}(D_{2}) a.
	\end{align*}
	Consider the summand $Q(x_{2}, D_{2}) P_{\nu}(D_{2}) a$ with $\nu > \mu$. By Lemma~\ref{lem:conj-err-symb} with $C_{0}= \frac{N}{100}$ and \eqref{eq:conj-err-lmb-mu}, $Q(x_{2}, \xi_{2})$ obeys \eqref{eq:psdo-kernel-hyp} with $G(x, \xi) = g(x_{2}) \calM_{<\lmb_{q}^{-1}, <\lmb_{q}}(x_{2}, \xi_{2}; \sgm)$. By Lemma~\ref{lem:psdo-kernel} with $n = N$, the kernel of $Q(x_{2}, D_{2}) P_{\nu}(D_{2})$ obeys the bound
	\begin{align*}
	\abs{K_{Q (-\rd_{x_{2}}^{2})^{-k} P_{\nu}}(x_{2}, y_{2})} 
	& \aleq \frac{\nu^{-2k} g(x_{2}) \int_{\abs{\xi_{2}} \aeq \nu} \calM_{<\lmb_{q}^{-1}, < \lmb_{q}}(x_{2}, \xi_{2}; \sgm) \ud \xi_{2}}{\brk{\min\set{\nu, \lmb_{q}}(x_{2}-y_{2})}^{N}}.
	\end{align*}
	Note that, for $\nu \geq \mu$,
	\begin{align*}
	\frac{g(x_{2})}{\brk{\nu (y_{2}-x_{2})}^{\frac{N}{100}}}
	& \leq \frac{g(x_{2})}{\brk{\mu (y_{2}-x_{2})}^{\frac{N}{100}}} 
	\aleq \sup_{z_{2}}  \frac{g(z_{2})}{\brk{\mu (y_{2}-z_{2})}^{C_{1}}} = \br{g}_{<\mu^{-1}}(y_{2}), \\
	\frac{\int_{\abs{\xi_{2}} \aeq \nu} \calM_{<\lmb_{q}^{-1}, < \lmb_{q}}(x_{2}, \xi_{2}; \sgm) \ud \xi_{2}}{\brk{\nu(y_{2}-x_{2})}^{\frac{N}{100}}}
	& \aleq \frac{\int_{\abs{\xi_{2}} \aeq \nu} \calM_{<\lmb_{q}^{-1}, < \lmb_{q}}(x_{2}, \xi_{2}; \sgm) \frac{\mu^{-1}}{\brk{\mu^{-1} \xi_{2}}^{\frac{N}{100}}} \ud \xi_{2}}{\brk{\mu(y_{2}-x_{2})}^{\frac{N}{100}}} \mu \left(\frac{\nu}{\mu}\right)^{\frac{N}{100}} \\
	& \leq \br{\calM}_{<\mu^{-1}}(x_{2}; \sgm) \mu \left(\frac{\nu}{\mu}\right)^{\frac{N}{100}}, \\
	{\frac{\brk{\mu (x_{2}-X_{0})}^{s}}{\brk{\nu(y_{2} - x_{2})}^{\abs{s}}}} 
	& {\leq \frac{\brk{\mu (x_{2}-X_{0})}^{s}}{\brk{\mu(y_{2} - x_{2})}^{\abs{s}}}
	\aleq \brk{\mu (y_{2}-X_{0})}^{s}.}
	\end{align*}
	Therefore,
	\begin{align*}
	& \abs{K_{Q (-\rd_{x_{2}}^{2})^{-k} P_{\nu}}(x_{2}, y_{2})} {\brk{\mu (x_{2}-X_{0})}^{s}}\\
	& \aleq \nu^{-2k}  \left(\frac{\mu}{\min\set{\nu, \lmb_{q}}}\right) \left(\frac{\nu}{\mu}\right)^{\frac{N}{100}} \br{g}_{<\mu^{-1}}(y_{2}) \br{\calM}_{<\mu^{-1}}(y_{2}; \sgm) {\brk{\mu (y_{2}-X_{0})}^{s}} \frac{ \min\set{\nu, \lmb_{q}}}{\brk{\min\set{\nu, \lmb_{q}}(x_{2}-y_{2})}^{\frac{N}{2}}} ,
	\end{align*}
	Since $\frac{N}{2} > 1$, the last factor defines a kernel that is $L^{2}$-bounded. Hence,
	\begin{align*}
	&\nrm{Q(x_{2}, D_{2}) P_{\nu}(D_{2}) a  {\brk{\mu (x_{2}-X_{0})}^{s}}}_{L^{2}} \\
	& \aleq \nu^{-2k} \left(\frac{\mu}{\min\set{\nu, \lmb_{q}}}\right) \left(\frac{\nu}{\mu}\right)^{\frac{N}{100}} \nrm{\br{g}_{<\mu^{-1}}(y_{2}) \br{\calM}_{<\mu^{-1}}(y_{2}; \sgm)  \rd_{x_{2}}^{2k} a(y_{2})  {\brk{\mu (y_{2}-X_{0})}^{s}}}_{L^{2}_{y_{2}}}
	\end{align*}
	Choosing $k = \lfloor \frac{N}{2} \rfloor$ and summing over all dyadic numbers $\nu > \mu$ (here, we use $N \gg 100$), we obtain
	\begin{align*}
	&\sum_{\nu \in 2^{\bbZ}, \, \nu > \mu} \nrm{Q(x_{2}, D_{2}) P_{\nu}(D_{2}) a {\brk{\mu (x_{2}-X_{0})}^{s}}}_{L^{2}} \\
	&\aleq \nrm{\br{g}_{<\mu^{-1}}(y_{2}) \br{\calM}_{<\mu^{-1}}(y_{2}; \sgm) (\mu^{-1} \rd_{x_{2}})^{2 \lfloor \frac{N}{2} \rfloor} a(y_{2}) \brk{\mu (y_{2}-X_{0})}^{s}}_{L^{2}_{y_{2}}},
	\end{align*}
	which is acceptable. Moreover, the following can be proved in an entirely analogous manner:
	\begin{align*}
	\nrm{Q(x_{2}, D_{2}) P_{< \mu}(D_{2}) a  {\brk{\mu (x_{2}-X_{0})}^{s}}}_{L^{2}}
	\aleq \nrm{\br{g}_{<\mu^{-1}}(y_{2}) \br{\calM}_{<\mu^{-1}}(y_{2}; \sgm) a(y_{2})  {\brk{\mu (y_{2}-X_{0})}^{s}}}_{L^{2}_{y_{2}}}.
	\end{align*}
	Combining the last two bounds, we obtain \eqref{eq:conj-err}. \qedhere
\end{proof}

\section{Construction of degenerating wave packets}\label{sec:wavepacket}
In this section, we finally put together all tools developed so far to construct degenerating wave packets for $\calL_{\bgtht}$.

\subsection{Specification of the construction} \label{subsec:def-wavepacket}
As in Section~\ref{sec:HJE}, we are given: a symbol $\gmm$ and a function $f$ that satisfy the assumptions in Section~\ref{subsec:results}; $\lmb_{0} \in \bbN$, $M > 1$ and $\dlt_{0} > 0$ that satisfy \eqref{eq:gf-condition-1}--\eqref{eq:gf-condition-3}. When $f$ is time-dependent, we are given $\dlt_{1} > 0$ such that \eqref{eq:gf-condition-diss-2} hold, as well as that $f$ is even. To apply the results of Section~\ref{sec:HJE}, we need to further specify $\dlt_{2} > 0$ and $N_{0}$; $\dlt_{2}$ will be specified below in Proposition~\ref{prop:wp-error}, while {\bf $N_{0}$ will be chosen in Section~\ref{sec:proofs}} (see Proposition~\ref{prop:wp-error} for the condition that $N_{0}$ has to satisfy).

We apply the results of Section~\ref{sec:HJE} with the above parameters, and obtain a global phase function $\Phi = \Phi(t, x_{2})$ on $[0, \frac{1}{1-\eps} t_{f}(\tau_{M})] \times \bbR$ as in Section~\ref{subsec:cutoff-extend}, which agrees with the solution to the Hamilton--Jacobi equation \eqref{eq:HJE} constructed in Sections~\ref{subsec:HJE-s}--\ref{subsec:HJE-gen} in the region $X(t, x'_{0}) \leq x \leq X(t, x'_{1})$. We inherit the parameters $\dlt_{3}$, $\dlt_{4}$, $\dlt_{5}$ and $T$ fixed in Section~\ref{sec:HJE} (see Section~\ref{subsec:HJE-prelim}). The parameters $c_{x_{0}}$ and $\Lmb$, which were not fixed in Section~\ref{sec:HJE}, will be chosen in this section; we shall fix $c_{x_{0}}$ in this subsection, and $\Lmb \geq 1$ in Section~\ref{subsec:deg}. 

Given the global phase function $\Phi$ from Section~\ref{sec:HJE}, the degenerating wave packet takes the form
\begin{equation} \label{eq:wavepacket}
\tld{\varphi}(t, x_{1}, x_{2}) = \Re\left(e^{i \lmb_{0} x_{1}} \tld{\psi} \right), \quad \hbox{ where } \tld{\psi}= a(t, x_{2}) e^{i \Phi(t, x_{2})},
\end{equation}
and $a$ solves \eqref{eq:transport-a}. When $\Omg = \bbT \times \bbR$, {\bf we fix a choice of $c_{x_{0}}$} so that it obeys all the requirements in Section~\ref{sec:HJE}, then we take as the initial data for $a$ the following:
\begin{equation} \label{eq:def-a0}
	a_{0}(x_{2}) = \frac{1}{(x'_1 - x'_{0})^{\frac{1}{2}}} \chi\left( \frac{x_{2} - x'_{0}}{x'_{1}-x'_{0}} \right), 
	\end{equation} where $\chi(\cdot)$ is a smooth nonnegative function supported in $[0,1]$, so that $\supp a_{0} \subseteq (x_{0}, x_{1})$ and $\nrm{a_{0}}_{L^{2}} = 1$. Then $\supp a(t, \cdot) \subseteq (X(t; x'_{0}), X(t; x'_{1}))$, so that \emph{$\Phi(t, x_{2})$ solves \eqref{eq:HJ-Phi} in the support of $\tld{\varphi}$}; hence, the computation in Section~\ref{sec:psdo} is applicable. Moreover, recalling \eqref{eq:x0-x1} or \eqref{eq:def-x1-x0} and \eqref{eq:wp-scale-0}, as well as the fact that $x'_{0} \aeq x_{0}$ and $x'_{1} - x'_{0} \aeq \Dlt x_{0}$, we have, for any $k \geq 0$,
	\begin{equation*}
	\nrm{(\mu_{0}^{-1} \rd_{x})^{k} a_{0}}_{L^{2}} \aleq_{k} 1,
\end{equation*}
so that Proposition~\ref{prop:trans-a} is useful. Moreover, $\nrm{\tld{\varphi}(0, x_{1}, x_{2})}_{L^{2}_{x_{1}, x_{2}}(\Omg)} = 1$, and for $m_{0} > 0 $ sufficiently large,
\begin{equation*}
	\nrm{P_{m_{0}^{-1} \lmb_{0} < \cdot < m_{0} \lmb_{0}} (D_{2}) (e^{i \Phi(0, x_{2})} a_{0})}_{L^{2}_{x_{2}}} \ageq 1,
\end{equation*}
where the implicit constant is independent of $\lmb_{0}$.

When $\Omg = \bbT^{2}$, we simply periodize \eqref{eq:wavepacket} and set
\begin{equation} \label{eq:wavepacket-period}
	\tld{\varphi}(t, x_{1}, x_{2}) = \sum_{n \in \bbZ} {}^{(\bbT \times \bbR)}\tld{\varphi}(t, x_{1}, x_{2} - n).
\end{equation}
where ${}^{(\bbT \times \bbR)}\tld{\varphi}$ is the wave packet in the case $\Omg = \bbT \times \bbR$. {\bf Choosing $c_{x_{0}} > 0$ to be sufficiently small}, we may ensure that \emph{each summand (in particular, $a_{0}$) is supported in a fundamental domain}.

\subsection{Degeneration and initial estimates for the wave packet}\label{subsec:deg} 
The purpose of this subsection is to obtain sharp bounds on the $H^{s}$-norm of $\tld{\varphi}$ for a suitable range of $s$. In what follows, the following convention is in effect: If $\nrm{\cdot}_{X_{x_{2}}}$ is a norm for functions of $x_{2}$, then {\bf $\nrm{\cdot}_{X_{x_{2}}(\Omg)}$ denotes either $\nrm{\cdot}_{X_{x_{2}}(\bbR)}$ or $\nrm{\cdot}_{X_{x_{2}}(\bbT)}$ depending on whether $\Omg = \bbT \times \bbR$ or $\bbT^{2}$, respectively}. We also introduce, for any $k \in \bbN_{0}$ and $L > 0$, 
\begin{equation*}
	\nrm{a}_{H_{(L)}^{k}(\Omg)}^{2} = \sum_{k' = 0}^{k} \nrm{(L \rd_{x_{2}})^{k'} a}_{L^{2}_{x_{2}}(\Omg)}^{2}.
\end{equation*}

\begin{proposition}[Degeneration and initial estimates]\label{prop:deg} 
Let $\Omg = \bbT \times \bbR$ or $\bbT^{2}$, and let $\tld{\varphi}$ be given as in Section~\ref{subsec:def-wavepacket}. For $\lmb_{0} \geq \Lmb$ sufficiently large (depending on $\gmm$ and the parameters fixed in Section~\ref{sec:HJE}), the following statements hold.
\begin{enumerate}
\item There exists a decomposition \begin{equation*}
	\widetilde{\varphi}= \widetilde{\varphi}^{main} + \widetilde{\varphi}^{small}
	\end{equation*} such that for any $s, \sgm \in \bbR$ and $N \in \bbN$ such that $\abs{s} + \bt_{0} \abs{\sgm} \leq N-1$ and $0 \leq t \leq \frac{1}{1-\eps} t_{f}(\tau_{M})$,
	\begin{align} 
	\nrm{\Gmm^{\sgm} \widetilde{\varphi}^{main}(t, x_{1}, x_{2})}_{H^s_{x_{2}}(\Omg)} &\aleq_{\gmm, f, N} \gmm_{\lmb_{0}}(\lmb)^{\sgm} \lmb^s \nrm{ a_{0}}_{H_{(\mu_{0}^{-1})}^{N}(\Omg)}, \label{eq:deg-main} \\
	\nrm{\widetilde{\varphi}^{small}(t, x_{1}, x_{2})}_{L^{2}_{x_{2}}(\Omg)} &\aleq_{\gmm, f} \lmb^{-1} \mu \nrm{ a_{0}}_{H_{(\mu_{0}^{-1})}^{1}(\Omg)}. \label{eq:deg-small}
	\end{align}
\item At $t = 0$, we additionally have, for any $m_{0} > 0$ (and taking $\Lmb$ larger if necessary),
	\begin{align}
	\nrm{\Gmm^{\sgm} \widetilde{\varphi}(0, x_{1}, x_{2})}_{H^{s}_{x_{2}}(\Omg)} &\aleq_{\gmm, f, N} \gmm_{\lmb_{0}}(\lmb_{0})^{\sgm} \lmb_{0}^{s} \nrm{a_{0}}_{H_{(\mu_{0}^{-1})}^{N}(\Omg)},	\label{eq:id-upper} \\
	\nrm{\Gmm^{\sgm} \widetilde{\varphi}(0, x_{1}, x_{2})}_{H^{s}_{x_{2}}(\Omg)} &\ageq_{\gmm, f, m_{0}} \gmm_{\lmb_{0}}(\lmb_{0})^{\sgm} \lmb_{0}^{s} \nrm{P_{m_{0}^{-1}\lmb_{0} < \cdot < m_{0} \lmb_{0}}(D_{2})(e^{i \Phi(0, x_{2})} a_{0}(x_{2}))}_{L^{2}(\Omg)},	\label{eq:id-lower} 
	\end{align}
\end{enumerate}

	Moreover, analogous conclusions hold for $\tld{\psi}$.
\end{proposition}

The core ingredient of the proof of Proposition~\ref{prop:deg} is the following simple lemma (on $\bbR$):
\begin{lemma} \label{lem:wp-dyadic}
	Assume $n \in \bbN_{0}$, $\lmb \leq \abs{\rd_{x_{2}} \Phi(t, x_{2})} \leq \br{\lmb}$ and $\abs{\rd_{x_{2}}^{\ell} \rd_{x_{2}} \Phi} \leq A_{\Phi, \ell} \mu^{\ell} \br{\lmb}$ on $\bbR$ for $1 \leq \ell \leq n$, where $A_{\Phi, \ell} > 0$ is increasing. Then for any $\nu \geq 1$ and function $a = a(x_{2})$,
	\begin{align} 
	\nrm{P_{\nu}(D_{2})(e^{i \Phi} a)}_{L^{2}(\bbR)} &\aleq_{n, A_{\Phi, n}} \min \set*{\left(\frac{\nu + ( \br\lmb / \lmb) \mu}{\lmb}\right)^{n}, \left(\frac{\br{\lmb} + \mu}{\nu}\right)^{n} } \sum_{\ell=0}^{n} \nrm{(\mu^{-1} \rd_{x_{2}})^{\ell} a}_{L^{2}(\bbR)}, \label{eq:wp-dyadic} \\
	\nrm{P_{<\nu}(D_{2})(e^{i \Phi} a)}_{L^{2}(\bbR)} &\aleq_{n, A_{\Phi, n}} \left(\frac{\nu + ( \br\lmb / \lmb)  \mu}{\lmb}\right)^{n} \sum_{\ell=0}^{n} \nrm{(\mu^{-1} \rd_{x_{2}})^{\ell} a}_{L^{2}(\bbR)}. \label{eq:wp-dyadic-low}
	\end{align}
\end{lemma}
\begin{proof}
To prove \eqref{eq:wp-dyadic} for $n = 1$, we write
\begin{align*}
	P_{\nu}(D_{2})\left(a e^{i \Phi}\right) &= P_{\nu}(D_{2})\left( \tfrac{a}{i \rd_{x_{2}} \Phi} \rd_{x_{2}} e^{i \Phi} \right) 
	= - P_{\nu}(D_{2}) D_{2} \left(\tfrac{a}{\rd_{x_{2}} \Phi} e^{i \Phi} \right) - P_{\nu}\left(\rd_{x_{2}}(\tfrac{a}{i \rd_{x_{2}}\Phi}) e^{i \Phi} \right) , \hbox{ or } \\
	P_{\nu}(D_{2})\left(a e^{i \Phi}\right) &= \frac{P_{\nu}(D_{2}) D_{2}}{\abs{D_{2}}^{2}} i \rd_{x_{2}} \left(a e^{i \Phi}\right),
\end{align*}
and use the hypotheses, as well as the bounds $\nrm{P_{\nu}(D_{2}) D_{2}}_{L^{2} \to L^{2}} \aleq \nu$ and $\nrm{\tfrac{P_{\nu}(D_{2})D_{2}}{\abs{D_{2}}^{2}}}_{L^{2} \to L^{2}} \aleq \nu^{-1}$. Higher $n$'s are handled by iteration. Finally, \eqref{eq:wp-dyadic-low} is proved using an analogous identity as the first one for $P_{< \nu}$. \qedhere
\end{proof}

\begin{proof}[Proof of Proposition~\ref{prop:deg}]
To begin with, observe that, thanks to the support conditions for $\tld{\varphi}$ and $a_{0}$, we may assume that $\Omg = \bbT \times \bbR$ without any loss of generality. Note the identity
\begin{equation} \label{eq:deg-varphi-psi}
	\brk{D}^{s} \Gmm^{\sgm} \tld{\varphi}  = \brk{D}^{s} \Gmm^{\sgm} \Re(e^{i \lmb_{0} x_{1}} \tld{\psi}) = \tfrac{1}{2}(e^{i \lmb_{0} x_{1}} \brk{D_{2}}_{\lmb_{0}}^{s} \gmm_{\lmb_{0}}(D_{2})^{\sgm} \tld{\psi} + e^{-i \lmb_{0} x_{1}} \brk{D_{2}}_{-\lmb_{0}}^{s} \gmm_{-\lmb_{0}}(D_{2})^{\sgm} \br{\tld{\psi}}),
\end{equation}
where $\brk{\xi_{2}}_{\lmb_{0}}^{s} = (\lmb_{0}^{2} + \xi_{2}^{2})^{\frac{s}{2}}$. In view of the symmetries $\brk{\xi_{2}}_{-\lmb_{0}} = \brk{\xi_{2}}_{\lmb_{0}}$ and $\gmm_{-\lmb_{0}}(\xi_{2}) = \gmm_{\lmb_{0}}(-\xi_{2})$, the two terms on the RHS are always treated similarly. In what follows, we focus on $\brk{D_{2}}_{\lmb_{0}}^{s} \gmm_{\lmb_{0}}(D_{2})^{\sgm} \tld{\psi}$, except in the proof of \eqref{eq:id-lower}.
To simplify the notation, we shall abbreviate $P_{\nu} = P_{\nu}(D_{2})$, $L^{2} = L^{2}_{x_{2}}$ and $H^{s}= H^{s}_{x_{2}}$ in this proof. We shall also often suppress the dependence of constants on $\gmm$ and $f$.

Note that $\rd_{x_{2}} \Phi$ obeys the hypothesis of Lemma~\ref{lem:wp-dyadic} with $\br{\lmb} \aleq \lmb$. Applying Lemma~\ref{lem:wp-dyadic} for any $0 \leq n \leq N$, then using Proposition~\ref{prop:trans-a} to bound $\nrm{a}_{H_{(\mu^{-1})}^{n}} \aleq \nrm{a_{0}}_{H_{(\mu_{0}^{-1})}^{n}}$, we obtain
	\begin{align}
\nrm{P_{\nu}(D_{2})\tld{\psi}}_{L^{2}}  
&\aleq_{N} \min \set*{\left(\frac{\nu + \mu}{\lmb}\right)^{n}, \left(\frac{\lmb + \mu}{\nu}\right)^{n}} \nrm{a_{0}}_{H^{n}_{(\mu_{0}^{-1})}}, \label{eq:wp-dyadic-a0} \\
\nrm{P_{< \nu}(D_{2})\tld{\psi}}_{L^{2}}  
&\aleq_{N} \left(\frac{\nu + \mu}{\lmb}\right)^{n} \nrm{a_{0}}_{H^{n}_{(\mu_{0}^{-1})}}. \label{eq:wp-dyadic-low-a0}
\end{align}
From \eqref{eq:wp-scale}, note moreover that
\begin{equation}  \label{eq:mu-precise}
	\frac{\mu}{\lmb} \aleq \lmb_{0}^{-1+2 \dlt_{3} N_{0}} (\log \lmb_{0})^{-2} \eps(\lmb_{0})^{-1} x_{0}^{-1} = \lmb_{0}^{-1+2 \dlt_{3} N_{0}} (\log \lmb_{0})^{-2} \eps(\lmb_{0})^{-2} c_{x_{0}}^{-1},
\end{equation}
where the implicit constant depends only on $\gmm$. Recall \eqref{eq:eps-choice} and that $\dlt_{3} \ll \frac{\dlt_{2}}{N_{0}} \ll \frac{\dlt_{0}}{N_{0}}$ (see Proposition~\ref{prop:Phi-derivatives}). Hence, by taking $\lmb_{0} \geq \Lmb$ sufficiently large, we may ensure that
\begin{equation} \label{eq:mu-small}
	\frac{\mu}{\lmb} \leq \lmb_{0}^{-\frac{1}{3}-\frac{1}{10} \dlt_{0}}.
\end{equation}
We also observe that, in view of Assumption~1 for $\Gmm$, we have, for any $b \in L^{2}(\bbR)$ and $\nu > 0$,
\begin{align} 
	\nrm{\brk{D_{2}}_{\lmb_{0}} \gmm_{\lmb_{0}}(D_{2}) P_{\nu} b}_{L^{2}}
	&\aleq_{\gmm, s, \sgm} (\nu + \lmb_{0})^{s} \gmm_{\lmb_{0}}(\nu)^{\sgm} \nrm{P_{\nu} b}_{L^{2}}, \label{eq:multiplier-dyadic} \\
	\nrm{\brk{D_{2}}_{\lmb_{0}} \gmm_{\lmb_{0}}(D_{2}) P_{< \lmb_{0}} b}_{L^{2}}
	&\aleq_{\gmm, s, \sgm} \lmb_{0}^{s} \gmm_{\lmb_{0}}(\lmb_{0})^{\sgm} \nrm{P_{< \lmb_{0}} b}_{L^{2}}. \label{eq:multiplier-dyadic-low}
\end{align}
Now we are ready to begin the proof of the proposition in earnest. To prove the first statement, we begin by defining 
\begin{equation*}
\tld{\psi}^{main} = \sum_{\nu \geq \mu} P_{\nu} \tld{\psi}, \qquad \tld{\psi}^{small} = \tld{\psi} - \tld{\psi}^{main},
\end{equation*}
and, accordingly, $\tld{\varphi}^{main} = \Re(e^{i \lmb_{0} x_{1}} \tld{\psi}^{main})$ and $\tld{\varphi}^{small}= \tld{\varphi} - \tld{\varphi}^{main}$.
Applying \eqref{eq:wp-dyadic-a0} with $n = N$, as well as \eqref{eq:multiplier-dyadic} and \eqref{eq:slow-var}, we obtain
\begin{align*}
	\gmm_{\lmb_{0}}(\lmb)^{-\sgm} \lmb^{-s} \nrm{\brk{D_{2}}_{\lmb_{0}} \gmm_{\lmb_{0}}(D_{2}) \sum_{\nu \geq \mu} P_{\nu} \tld{\psi}}_{L^{2}}
	&\aleq_{N} \sum_{\mu \leq \nu < \max \set{\mu, \lmb_{0}}} \left( \frac{\lmb}{\lmb_{0}} \right)^{\bt_{0} \abs{\sgm}} \left( \frac{\lmb_{0}}{\lmb} \right)^{s} \left(\frac{\nu}{\lmb}\right)^{N} \nrm{a_{0}}_{H^{N}_{(\mu_{0}^{-1})}} \\
	&\peq 
	+\sum_{\max\set{\mu, \lmb_{0}} \leq \nu < \lmb} \left( \frac{\lmb}{\nu} \right)^{\bt_{0} \abs{\sgm}} \left( \frac{\nu}{\lmb} \right)^{s} \left(\frac{\nu}{\lmb}\right)^{N} \nrm{a_{0}}_{H^{N}_{(\mu_{0}^{-1})}} \\
	&\peq 
	+ \sum_{\nu > \lmb} \left( \frac{\nu}{\lmb} \right)^{\bt_{0} \abs{\sgm}} \left( \frac{\nu}{\lmb} \right)^{s} \left(\frac{\lmb}{\nu}\right)^{N} \nrm{a_{0}}_{H^{N}_{(\mu_{0}^{-1})}}.
\end{align*}
By the assumed lower bound on $N$, the RHS may be bounded by $\nrm{a_{0}}_{H^{N}_{(\mu_{0}^{-1})}}$ (up to a constant), which implies \eqref{eq:deg-main}. On the other hand, for $\tld{\psi}^{small} = P_{< \mu} \tld{\psi}$, \eqref{eq:deg-small} follows quickly from \eqref{eq:wp-dyadic-low-a0} with $n = 1$ and Proposition~\ref{prop:trans-a}.

Next, we turn to the second statement. Note that $(\lmb, \mu, a) = (\lmb_{0}, \mu_{0}, a_{0})$ at $t= 0$. In order to establish \eqref{eq:id-upper}, in view of \eqref{eq:deg-main}, it suffices to prove
\begin{equation*}
	\nrm{\brk{D_{2}}_{\lmb_{0}} \Gmm_{\lmb_{0}}^{\sgm}(D_{2}) \tld{\psi}^{small}(t=0)}_{L^{2}} \aleq \gmm_{\lmb_{0}}(\lmb_{0})^{\sgm} \lmb_{0}^{s} \nrm{a_{0}}_{H_{(\mu_{0}^{-1})}^{1}}.
\end{equation*}
Since $\tld{\psi}^{small}(t=0) = P_{< \mu_{0}} \tld{\psi}_{0}$ and $\mu_{0} < \lmb_{0}$, the desired bound follows from \eqref{eq:wp-dyadic-low-a0} and \eqref{eq:multiplier-dyadic-low}.
%
To prove \eqref{eq:id-lower}, observe that, by \eqref{eq:deg-varphi-psi} and orthogonality in the frequency space (first in $\xi_{1}$ and then in $\xi_{2}$), we have
\begin{align*}
\nrm{\brk{D}^{s} \Gmm^{\sgm} \tld{\varphi}_{0}}_{L^{2}_{x_{1}, x_{2}}} \geq \tfrac{1}{2}\nrm{e^{i \lmb_{0} x_{1}} \brk{D_{2}}_{\lmb_{0}}^{s} \gmm_{\lmb_{0}}(D_{2})^{\sgm} \tld{\psi}_{0}}_{L^{2}_{x_{1}, x_{2}}}
\ageq \nrm{\brk{D_{2}}_{\lmb_{0}}^{s} \gmm_{\lmb_{0}}(D_{2})^{\sgm} P_{m_{0}^{-1} \lmb_{0} < \cdot < m_{0} \lmb_{0}}\tld{\psi}_{0}}_{L^{2}}.
\end{align*}
Hence, it suffices to prove 
\begin{align*}
\nrm{P_{m_{0}^{-1} \lmb_{0} < \cdot < m_{0} \lmb_{0}}\tld{\psi}_{0}}_{L^{2}} \aleq_{m_{0}} \lmb_{0}^{-s} \gmm_{\lmb_{0}}(\lmb_{0})^{-\sgm} \nrm{\brk{D_{2}}_{\lmb_{0}}^{s} \gmm_{\lmb_{0}}(D_{2})^{\sgm} P_{m_{0}^{-1} \lmb_{0} < \cdot < m_{0} \lmb_{0}}\tld{\psi}_{0}}_{L^{2}},
\end{align*}
which follows from \eqref{eq:multiplier-dyadic}. This completes the proof. \qedhere
\end{proof}

{\bf We may now fix our choice of $\Lmb$} so that it obeys all the requirements that appeared so far.

\subsection{Error estimate}\label{subsec:wp-error}
The main result of this subsection is the following:
\begin{proposition}[Error estimate] \label{prop:wp-error}
	Let $\tld{\varphi}$ be constructed as in Section~\ref{subsec:def-wavepacket}, with $\dlt_{2} = \dlt_{0}^{10}$ and $N_{0} \geq \frac{10^{4}}{\dlt_{2}}(1+\bt_{0})$. Then there exists $\dlt_{6} > 0$ such that
	\begin{equation} \label{eq:wp-error}
	\int_{0}^{\frac{1}{1-\eps}t_{f}(\tau_{M})}\nrm{\calL_{\bgtht} \tld{\varphi}}_{L^{2}_{x_{1}, x_{2}}(\Omg)} \, \ud t \aleq \lmb_{0}^{-\dlt_{6}} \nrm{a_{0}}_{H^{N_{0}}_{(\mu_{0}^{-1})}(\Omg)}.
	\end{equation}
\end{proposition}

\begin{proof}
	By translation in $x_{2}$ and rescaling $t$, we assume, without loss of generality, that $\dgx_{2} = 0$ and $\rd_{x_{2}}^{2} f(0, 0) = -1$. We let all implicit constants in this proof to depend on $\gmm$ and $f$. Note that
	\begin{align*}
	\calL_{\bgtht} \tld{\varphi} &= \Re \left(\calL_{\bgtht} (e^{i \lmb_{0} x_{1}} \tld{\psi})\right) \\
	&= \Re \left(e^{i \lmb_{0} x_{1}} \left( i p_{\bgtht, \lmb_{0}}(x_{2}, D_{2}) + s_{\bgtht, \lmb_{0}}(x_{2}, D_{2}) + r_{\bgtht, \lmb_{0}}(x_{2}, D_{2}) \right) \tld{\psi} \right) .
	\end{align*}
	Hence, in what follows, we focus on estimating the $L^{1}_{t} L^{2}_{x_{2}}$-norm of
	\begin{equation*}
	\left( i p_{\bgtht, \lmb_{0}}(x_{2}, D_{2}) + s_{\bgtht, \lmb_{0}}(x_{2}, D_{2}) + r_{\bgtht, \lmb_{0}}(x_{2}, D_{2}) \right) \tld{\psi}.
	\end{equation*}
	
	We first treat the case $\Omg = \bbT \times \bbR$. 
	
	\medskip
	\noindent {\bf Step~1: Contribution of $i p_{\bgtht, \lmb_{0}} + s_{\bgtht, \lmb_{0}}$, low frequency input.} 
	By Assumption~1 on $\gmm$, it follows that 
	\begin{align*}
	\abs*{\rd_{x_{2}}^{k} \rd_{\xi_{2}}^{\ell} \left( \left(i p_{\bgtht, \lmb_{0}}(x_{2}, \xi_{2}) + s_{\bgtht, \lmb_{0}}(x_{2}, \xi_{2}) \right)P_{<\lmb^{1-\dlt_{2}}}(\xi_{2})\right)}
	\aleq_{k, \ell} \brk{\xi_{2}}^{-\ell} \lmb_{0} \gmm(\lmb_{0}, \lmb^{1-\dlt_{2}}).
	\end{align*}
	By the standard $L^{2}$-boundedness of a pseudo-differential operator with a classical symbol and {Lemma~\ref{lem:wp-dyadic}}, we have 
	\begin{align*}
	\nrm*{\left(i p_{\bgtht, \lmb_{0}}(x_{2}, D_{2}) + s_{\bgtht, \lmb_{0}}(x_{2}, D_{2}) \right)P_{<\lmb^{1-\dlt_{2}}}(D_{2}) \tld{\psi}}_{L^{2}}
	\aleq \lmb_{0} (\max \set{\lmb_{0}, \lmb^{1-\dlt_{2}}})^{\bt_{0}} \left(\frac{\mu + \lmb^{1-\dlt_{2}}}{\lmb} \right)^{N} \nrm{a}_{H_{(\mu^{-1})}^{N}}.
	\end{align*}
	By \eqref{eq:mu-small}, we see that $\frac{\mu+\lmb^{1-\dlt_{2}}}{\lmb} \leq \max\set{\lmb_{0}^{-\frac{1}{3}-\frac{1}{10} \dlt_{0}}, \lmb^{-\dlt_{2}}}$. Choosing $N = \dlt_{2}^{-1} (\bt_{0} + 2)$, which bounded by $N_{0}$, then applying Proposition~\ref{prop:trans-a}, the RHS is $O(\lmb_{0}^{-1} \nrm{a_{0}}_{H^{N_{0}}_{(\mu_{0}^{-1})}})$, which is sufficient.
	
	\medskip
	\noindent {\bf Step~2: Decomposition of $i p_{\bgtht, \lmb_{0}} + s_{\bgtht, \lmb_{0}}$}
	Thanks to our pointwise bound $\lmb \leq \rd_{x_{2}} \Phi \leq 2 \lmb$, observe that
	\begin{align*}
	p_{\bgtht, \lmb_{0}} (x_{2}, \rd_{x_{2}} \Phi) P_{>\lmb^{1-\dlt_{2}}}(\rd_{x_{2}} \Phi) &= p_{\bgtht, \lmb_{0}} (x_{2}, \rd_{x_{2}} \Phi), \\
	s_{\bgtht, \lmb_{0}} (x_{2}, \rd_{x_{2}} \Phi) P_{>\lmb^{1-\dlt_{2}}}(\rd_{x_{2}} \Phi) &= s_{\bgtht, \lmb_{0}} (x_{2}, \rd_{x_{2}} \Phi),
	\end{align*}
	so that by definition,
	\begin{equation*}
	{}^{(\Phi)} r_{p_{\bgtht, \lmb_{0}}, -2} = {}^{(\Phi)} r_{p_{\bgtht, \lmb_{0}} P_{>\lmb^{1-\dlt_{2}}}, -2}, \quad
	{}^{(\Phi)} r_{s_{\bgtht, \lmb_{0}}, -1} = {}^{(\Phi)} r_{s_{\bgtht, \lmb_{0}} P_{>\lmb^{1-\dlt_{2}}}, -1}.
	\end{equation*}
	To proceed, we furthermore split $p_{\bgtht, \lmb_{0}}$ into $p_{\bgtht, \lmb_{0}}^{main} + p_{\bgtht, \lmb_{0}}^{rem}$
	\begin{align*}
	p_{\bgtht, \lmb_{0}}^{main}(x_{2}, \xi_{2}) &= \rd_{x_{2}} f(t, x_{2}) \lmb_{0} \gmm_{\lmb_{0}}(\xi_{2}), \\
	p_{\bgtht, \lmb_{0}}^{rem}(x_{2}, \xi_{2}) &= \Gmm \rd_{x_{2}} f(t, x_{2}) \lmb_{0}.
	\end{align*}
	By the linear dependence of ${}^{(\Phi)} r_{p, -k}$ on $p$, the remaining task reduces to estimating
	\begin{align*}
	\left(i {}^{(\Phi)} r_{p_{\bgtht, \lmb_{0}}^{main} P_{>\lmb^{1-\dlt_{2}}}, -2}  + i {}^{(\Phi)} r_{p_{\bgtht, \lmb_{0}}^{rem} P_{>\lmb^{1-\dlt_{2}}}, -2}
	+ {}^{(\Phi)} r_{s_{\bgtht, \lmb_{0}} P_{>\lmb^{1-\dlt_{2}}}, -1}\right)(x_{2}, D_{2})\tld{\psi}.
	\end{align*}
	
	\medskip
	\noindent {\bf Step~3: Contribution of $i p_{\bgtht, \lmb_{0}}^{main}$.}
	In this step, we estimate the contribution of $p_{\bgtht, \lmb_{0}}^{main} P_{>\lmb^{1-\dlt_{2}}}$. Observe that this symbol obeys
	\begin{equation*}
	\rd_{\xi_{2}}^{\ell} \left( p_{\bgtht, \lmb_{0}}^{main}(x_{2}, \xi_{2}) P_{>\lmb^{1-\dlt_{2}}}(\xi_{2}) \right) \aleq_{\ell} \lmb^{-\ell(1-\dlt_{2})}\abs{\rd_{x_{2}} f(x_{2})} \lmb_{0} m(\lmb_{0}, \xi_{2}).
	\end{equation*}
	As a result, each term in \eqref{eq:conj-rem-2} for $p = p_{\bgtht, \lmb_{0}}^{main}(x_{2}, \xi_{2}) P_{>\lmb^{1-\dlt_{2}}}(\xi_{2})$ is of the form (after rewriting $\xi_2= \mu \cdot (\xi_2/\mu)$) \begin{equation*}
		\begin{split}
		 \int_{0}^{1} {}^{(\Phi)} q(x_{2}, \xi_{2}; \sgm) \, \ud \sgm  \left( \frac{\xi_{2}}{\mu} \right)^{\ell} , \qquad \ell = 0, 1, 2, 
		\end{split}
	\end{equation*} where $q$ obeys the hypothesis of Proposition~\ref{prop:conj-err} with
	\begin{equation*}
	\lmb_{q} = \lmb^{1-\dlt_{2}}, \quad 
	g(x_{2}) = \abs{f'(x_{2})}, \quad
	M(\xi_{2}) = \lmb_{0} \gmm_{\lmb_{0}}(\xi_{2}) \left(\frac{\mu^{2}}{{\lmb^{2(1-\dlt_{2})}}} + \frac{\mu^{2} \lmb}{{\lmb^{3(1-\dlt_{2})}}} + \frac{\mu^{2} \lmb^{2}}{{\lmb^{4(1-\dlt_{2})}}}\right).
	\end{equation*}
	In order to apply Proposition~\ref{prop:conj-err}, note that, since $\frac{\mu}{\lmb_{q}} \aleq \lmb_{0}^{-\frac{1}{3}-\frac{1}{20} \dlt_{0}}$ by \eqref{eq:mu-small}, $f'$ is smooth and bounded and $\gmm$ obeys \eqref{eq:slow-var}, we have
\begin{equation*}
\br{g}_{<\mu^{-1}} \aleq x_{2} + \mu^{-1}, \qquad
\br{\calM}_{<\mu^{-1}} \aleq M(\lmb),
\end{equation*}
	where we choose $N = 1000 (1+\bt_{0})$. {We remark that the second inequality may be readily verified using the fact that $M(\xi_{2})$ is increasing and slowly varying (i.e., Assumptions 1 and 2 for $\Gmm$).} Putting together the above observations with Propositions~\ref{prop:conj-exp} and \ref{prop:conj-err} (with $s = 0$), we arrive at
	\begin{align*}
	\nrm{{}^{(\Phi)} r_{p_{\bgtht, \lmb_{0}}^{main} P_{>\lmb^{1-\dlt_{2}}}, -2} \tld{\psi}}_{L^{2}}
	&\aleq (X(t; x_{0}) + \mu^{-1}) \lmb_{0} \gmm_{\lmb_{0}}(\lmb) \left(\frac{\mu^{2}}{{\lmb^{2(1-\dlt_{2})}}} + \frac{\mu^{2} \lmb}{{\lmb^{3(1-\dlt_{2})}}} + \frac{\mu^{2} \lmb^{2}}{{\lmb^{4(1-\dlt_{2})}}}\right) \nrm{a}_{H_{(\mu^{-1})}^{N}}\\
	&\aleq (X(t; x_{0}) + \mu^{-1}) \lmb_{0} \gmm_{\lmb_{0}}(\lmb) \left(\frac{\mu}{\lmb}\right)^{2} \lmb_{0}^{\frac{10 \dlt_{2}}{\dlt_{0}}} \nrm{a_{0}}_{H_{(\mu^{-1})}^{N_{0}}}.
	\end{align*}
	We need to estimate the integral of the preceding expression on $t \in [0, \frac{1}{1-\eps} t_{f}(\tau_{M})]$. By \eqref{eq:X-1-mu} from the proof of Proposition~\ref{prop:Phi-derivatives}, it follows that $\mu^{-1} \aleq X(t; x_{0})$; hence it suffices to estimate the contribution of $X(t; x_{0})$. Using \eqref{eq:eps-choice}, \eqref{eq:X-bounds} or \eqref{eq:X-bounds-time-dep} and \eqref{eq:mu-precise}, we have
	\begin{align*}
	&\int_{0}^{\frac{1}{1-\eps} t_{f}(\tau_{M})} X(t; x_{0}) \lmb_{0} \gmm_{\lmb_{0}}(\lmb) \left(\frac{\mu}{\lmb}\right)^{2} \lmb_{0}^{\frac{10 \dlt_{2}}{\dlt_{0}}} \nrm{a_{0}}_{H^{N_{0}}_{(\mu_{0}^{-1})}} \, \ud t \\
	&\aleq \tau_{M} \lmb_{0} \gmm_{\lmb_{0}}(\lmb_{0})  \lmb_{0}^{-2+4 \dlt_{3} N_{0}+ \frac{10 \dlt_{2}}{\dlt_{0}}} (\log \lmb_{0})^{-4} \eps(\lmb_{0})^{-3} c_{x_{0}}^{-1} \nrm{a_{0}}_{H^{N_{0}}_{(\mu_{0}^{-1})}}  \\
	&\aleq \tau_{M} \frac{\gmm_{\lmb_{0}}(\lmb_{0})}{\lmb_{0}^{1-\dlt_{0}-3\sgm_{0}}}  \lmb_{0}^{-\dlt_{0}+4 \dlt_{3} N_{0}+ \frac{10 \dlt_{2}}{\dlt_{0}} + \frac{1}{2} \dlt_{0}} (\log \lmb_{0})^{-4} c_{x_{0}}^{-1} \nrm{a_{0}}_{H^{N_{0}}_{(\mu_{0}^{-1})}}  \\
	\end{align*}
	which is tightly acceptable thanks to \eqref{eq:gf-condition-2}. 
	
	\medskip
	\noindent {\bf Step~4: Contribution of $i p_{\bgtht, \lmb_{0}}$, remainder.}
	Now, we estimate the contribution of $i p_{\bgtht, \lmb_{0}}^{rem} P_{>\lmb^{1-\dlt_{2}}}$. Each term in \eqref{eq:conj-rem-2} for $p = p_{\bgtht, \lmb_{0}}^{rem}(x_{2}, \xi_{2}) P_{>\lmb^{1-\dlt_{2}}}(\xi_{2})$ is of the form $\int_{0}^{\sgm} {}^{(\Phi)} q(x_{2}, \xi_{2}; \sgm) \, \ud \sgm$, where $q$ obeys the hypothesis of Proposition~\ref{prop:conj-err} with
	\begin{equation*}
	\lmb_{q} = \lmb^{1-\dlt_{2}}, \quad 
	g(x_{2}) = \nrm{\Gmm \rd_{x_{2}} f}_{L^{\infty}}, \quad
	M(\xi_{2}) = \lmb_{0} \left(\frac{\mu^{2}}{{\lmb^{2(1-\dlt_{2})}}} + \frac{\mu^{2} \lmb}{{\lmb^{3(1-\dlt_{2})}}} + \frac{\mu^{2} \lmb^{2}}{{\lmb^{4(1-\dlt_{2})}}}\right).
	\end{equation*}
	Since $g(x_{2})$ and $M(\xi_{2})$ are constant, Proposition~\ref{prop:conj-err} (with $s = 0$) immediately implies
	\begin{align*}
	\nrm{{}^{(\Phi)} r_{p_{\bgtht, \lmb_{0}}^{rem} P_{>\lmb^{1-\dlt_{2}}}, -2} \tld{\psi}}_{L^{2}}
	&\aleq \nrm{\Gmm \rd_{x_{2}} f}_{L^{\infty}} \lmb_{0} \left(\frac{\mu^{2}}{{\lmb^{2(1-\dlt_{2})}}} + \frac{\mu^{2} \lmb}{{\lmb^{3(1-\dlt_{2})}}} + \frac{\mu^{2} \lmb^{2}}{{\lmb^{4(1-\dlt_{2})}}}\right) \nrm{a}_{H_{\mu^{-1}}^{10000}}\\
	&\aleq \nrm{\Gmm \rd_{x_{2}} f}_{L^{\infty}} \lmb_{0}  \frac{\lmb_{0}^{10 \dlt_{2} + N_{0} \dlt_{3}}}{\lmb_{0}^{2}} \nrm{a_{0}}_{H_{\mu^{-1}}^{N_{0}}},
	\end{align*}
	which is acceptable.
	
	\medskip
	\noindent {\bf Step~5: Contribution of $s_{\bgtht, \lmb_{0}}$.}
	In this step, we handle the contribution of $s_{\bgtht, \lmb_{0}} P_{>\lmb^{1-\dlt_{2}}}$. Each term in \eqref{eq:conj-rem-1} for $p = s_{\bgtht, \lmb_{0}}^{rem}(x_{2}, \xi_{2}) P_{>\lmb^{1-\dlt_{2}}}(\xi_{2})$ is of the form $\int_{0}^{\sgm} {}^{(\Phi)} q(x_{2}, \xi_{2}; \sgm) \, \ud \sgm$, where $q$ obeys the hypothesis of Proposition~\ref{prop:conj-err} with
	\begin{equation*}
	\lmb_{q} = \lmb^{1-\dlt_{2}}, \quad 
	g(x_{2}) = g:= \nrm{\rd_{x_{2}}^{2} f}_{L^{\infty}} + \nrm{\Gmm \rd_{x_{2}}^{2} f}_{L^{\infty}}, \quad
	M(\xi_{2}) = \lmb_{0} \frac{\xi_{2}}{\abs{\xi_{2}}} \gmm_{\lmb_{0}}'(\xi_{2})\left(\frac{\mu}{{\lmb^{1-\dlt_{2}}}} + \frac{\mu \lmb}{{\lmb^{2(1-\dlt_{2})}}}\right).
	\end{equation*}
	In order to apply Proposition~\ref{prop:conj-err}, note that, by \eqref{eq:slow-var} and the almost comparability of $\gmm$ and $\xi_{2} \rd_{\xi_{2}} \gmm$, we have $\br{\calM}_{<\mu^{-1}} \aleq M(\lmb)$ if we choose $N = 1000(1+\bt_{0})$. On the other hand, $g$ is constant, so $\br{g}_{<\mu^{-1}} = g$. Putting together the above observations with Propositions~\ref{prop:conj-exp} and \ref{prop:conj-err} (with $s = 0$), we arrive at
	\begin{align*}
	\nrm{{}^{(\Phi)} r_{s_{\bgtht, \lmb_{0}} P_{>\lmb^{1-\dlt_{2}}}, -1} \tld{\psi}}_{L^{2}}
	&\aleq g \lmb_{0} \gmm_{\lmb_{0}}'(\lmb)\left(\frac{\mu}{{\lmb^{1-\dlt_{2}}}} + \frac{\mu \lmb}{{\lmb^{2(1-\dlt_{2})}}}\right) \nrm{a}_{H_{\mu^{-1}}^{N}}\\
	&\aleq g \lmb_{0} \gmm_{\lmb_{0}}'(\lmb) \left(\frac{\mu}{\lmb}\right) \lmb_{0}^{\frac{10 \dlt_{2}}{\dlt_{0}}} \nrm{a_{0}}_{H_{\mu^{-1}}^{N_{0}}}.
	\end{align*}
	To estimate the integral of the preceding expression on $t \in [0, \frac{1}{1-\eps} t_{f}(\tau_{M})]$, we use \eqref{eq:eps-choice} and \eqref{eq:mu-small} to estimate
	\begin{align*}
	\int_{0}^{\frac{1}{1-\eps} t_{f}(\tau_{M})} \lmb_{0} \gmm_{\lmb_{0}}'(\lmb) \left(\frac{\mu}{\lmb}\right) \lmb_{0}^{\frac{10 \dlt_{2}}{\dlt_{0}}} \, \ud t
	&\aleq \left(\frac{\mu}{\lmb}\right) \lmb_{0}^{\frac{10 \dlt_{2}}{\dlt_{0}}} \int_{\lmb_{0}}^{M \lmb_{0}} \frac{\gmm_{\lmb_{0}}'(\lmb)}{\gmm_{\lmb_{0}}(\lmb)} \, \ud \lmb \\
	&\aleq \lmb_{0}^{-\frac{1}{3}-\frac{1}{10} \dlt_{0}+\frac{10 \dlt_{2}}{\dlt_{0}}} \bt_{0} \log M \aleq \dlt_{0}^{-1} \bt_{0} \lmb_{0}^{-\frac{1}{3}-\frac{1}{10} \dlt_{0} +\frac{10 \dlt_{2}}{\dlt_{0}}} \log \lmb_{0},
	\end{align*}
	which is acceptable.
	
	\medskip
	\noindent {\bf Step~6: Contribution of $r_{\bgtht, \lmb_{0}}$.}
	We begin by splitting
	\begin{align*}
	\nrm{r_{\bgtht, \lmb_{0}}(x_{2}, D_{2}) \tld{\psi}}_{L^{2}}
	&\aleq \nrm{r_{\bgtht, \lmb_{0}}(x_{2}, D_{2}) P_{< \max \set{\mu, \lmb_{0}}}(D_{2})\tld{\psi}}_{L^{2}} \\
	&\peq + \sum_{\nu > \max\set{\mu, \lmb_{0}}} \nrm{r_{\bgtht, \lmb_{0}}(x_{2}, D_{2}) P_{\nu}(D_{2})\tld{\psi}}_{L^{2}}.
	\end{align*}
	For the first term, note that $\frac{\lmb_{0}}{\gmm(\lmb_{0}, \max\set{\mu, \lmb_{0}})} r_{\bgtht, \lmb_{0}}(x_{2}, \xi_{2}) P_{< \max \set{\mu, \lmb_{0}}}(\xi_{2})$ is a classical symbol of order $0$. By the $L^{2}$-boundedness of its quantization, as well as Lemma~\ref{lem:wp-dyadic}, we obtain the bound
	\begin{align*}
	\nrm{r_{\bgtht, \lmb_{0}}(x_{2}, D_{2}) P_{< \max\set{\mu, \lmb_{0}}}(D_{2})\tld{\psi}}_{L^{2}}
	& \aleq \frac{1}{\lmb_{0}} \gmm_{\lmb_{0}}(\max\set{\mu, \lmb_{0}}) \left(\frac{\max\set{\mu, \lmb_{0}}}{\lmb}\right)^{N} \nrm{a}_{H_{\mu^{-1}}^{N}},
	\end{align*}
	which is acceptable if we take $N = \dlt_{2}^{-1} (\bt_{0} + 2)$, which is bounded by $N_{0}$. 
	
	For each summand in the second term, note that $\frac{\nu^{2}}{\lmb_{0} \gmm_{\lmb_{0}}(\nu)} r_{\bgtht, \lmb_{0}}(x_{2}, \xi_{2}) P_{\nu}(\xi_{2})$ is a classical symbol of order $0$. As before, we therefore have
	\begin{align*}
	\nrm{r_{\bgtht, \lmb_{0}}(x_{2}, D_{2}) P_{\nu}(D_{2})\tld{\psi}}_{L^{2}}
	&\aleq \frac{\lmb_{0}}{\nu^{2}} \gmm_{\lmb_{0}}(\nu) \min \set*{\left(\frac{\nu}{\lmb}\right)^{N}, \left(\frac{\lmb}{\nu}\right)^{N}} \nrm{a}_{H_{\mu^{-1}}^{N}}.
	\end{align*}
	By the slow variance assumption on $\gmm$, we have $\gmm_{\lmb_{0}}(\nu) \aleq \gmm_{\lmb_{0}}(\lmb) \max \set{\left(\frac{\nu}{\lmb}\right)^{\bt_{0}}, \left(\frac{\lmb}{\nu}\right)^{\bt_{0}}}$.
	Hence, applying the preceding bound with $N = 1000(1+\bt_{0})$, we obtain the summed bound
	\begin{align*}
	\sum_{\nu > \mu} \nrm{r_{\bgtht, \lmb_{0}}(x_{2}, D_{2}) P_{\nu}(D_{2})\tld{\psi}}_{L^{2}}
	\aleq \lmb_{0}^{\dlt_{2}} \frac{\lmb_{0}}{\lmb^{2}} \gmm_{\lmb_{0}}(\lmb) \nrm{a_{0}}_{H_{\mu^{-1}}^{N}}.
	\end{align*}
	Then
	\begin{align*}
	\int_{0}^{\frac{1}{1-\eps} t_{f}(\tau_{M})} \sum_{\nu > \mu} \nrm{r_{\bgtht, \lmb_{0}}(x_{2}, D_{2}) P_{\nu}(D_{2})\tld{\psi}}_{L^{2}} \, \ud t
	& \aleq \lmb_{0}^{\dlt_{2}} \int_{0}^{\frac{1}{1-\eps} t_{f}(\tau_{M})} \frac{\lmb_{0}}{\lmb^{2}} \gmm_{\lmb_{0}}(\lmb) \, \ud t \\
	& \aleq \lmb_{0}^{\dlt_{2}} \int_{\lmb_{0}}^{M \lmb_{0}} \frac{1}{\lmb^{2}} \ud \lmb 
	\aleq \lmb_{0}^{-1+\dlt_{2}},
	\end{align*}
	which is acceptable. 
	 
	Finally, we consider the case $\Omg = \bbT^{2}$, in which case, by the $x_{2}$-periodicity of $\calL_{\bgtht}$,
	\begin{equation*}
	\calL_{\bgtht} \tld{\varphi} (x_{1}, x_{2})= \sum_{k \in \bbZ} (\calL_{\bgtht} ({}^{(\bbT \times \bbR)}\tld{\varphi}))(x_{1}, x_{2} - k).
\end{equation*}
	We repeat the above procedure, but now use the boundedness on $L^{2}(\brk{x_{2}}^{4} \ud x_{2})$ for classical pseudo-differential operators (for Step~1), and Proposition~\ref{prop:conj-err} with $X_{0} = X(t; x_{0})$ and $s = 2$ for $\Phi$-conjugated remainders. Then we obtain {the unit-scale-localized bounds}
	\begin{equation} \label{eq:wp-error-loc}
\int_{0}^{\frac{1}{1-\eps} t_{f}(\tau_{M})} \nrm{\calL_{\bgtht} ({}^{(\bbT \times \bbR)}\tld{\varphi})}_{L^{2}(k, k+1)} \, \ud t\aleq \min\set*{1, \frac{1}{k^{2}}} \lmb_{0}^{-\dlt_{6}}.
\end{equation}
Summing up in $k$, the desired conclusion in the periodic case follows. \qedhere
\end{proof}

\section{Proof of illposedness}\label{sec:proofs}
We now establish the illposedness theorems stated in Section~\ref{subsec:results}.

\subsection{Linear illposedness}\label{subsec:lin-illposed}

In this section, we prove Theorem \ref{thm:norm-growth}. As it is assumed in the statement of Theorem \ref{thm:norm-growth}, we are given a quadratically degenerate function $f$ together with parameters $\lmb_{0} \in \bbN$, $M>1$, $0<\dlt_0<\frac{1}{100}$, $0 \leq \sgm_{0} \leq \frac{1}{3} (1- 2\dlt_{0})$, so that the conditions \eqref{eq:gf-condition-1}--\eqref{eq:gf-condition-3} hold.  We also set $c_{1} = \frac{1}{3} + \frac{1}{10} \dlt_{0}$ (cf.~\eqref{eq:mu-small}).  Without loss of generality, we may assume that $\mathring{x}_{2} = 0$ and $f''(0) < 0$. We introduce parameters $\Lmb_{0} \geq 1$ and $0 < T_{0} \leq 1$, to be determined in the course of the proof, and require $\lmb_{0} \geq \Lmb_{0}$ and $\tau_{M} \leq T_{0}$. We also fix the regularity exponents $s, s' \in \bbR$.

We apply the construction of a degenerating wave packet in Sections~\ref{sec:psdo}--\ref{sec:wavepacket} with $N_{0} = 10000 \max\set{1+\bt_{0}, 1+\abs{s}, 1+\abs{s'}}$ and the above parameters; we take $\Lmb_{0} \geq \Lmb$ and $T_{0} \leq T$. As a result, we obtain\footnote{Observe from Section~\ref{subsec:def-wavepacket} that while the constant $\mu$ and bounds depend on the choice of $N_{0}$, the wave packet $\varphi$ itself is independent of $N_{0}$.} a degenerating wave packet $\tld{\varphi} = \Re\left(a e^{i \bfPhi}\right)$ satisfying $\nrm{\tld{\varphi}_{0}}_{L^{2}(\Omg)} = 1$, where $\tld{\varphi}_{0} =\left. \tld{\varphi} \right|_{t=0}$. Introducing the shorthand 
\begin{equation*}
\nnrm{\tld{\varphi}_{0}} := \nrm{a_{0}}_{H_{(\mu_{0}^{-1})}^{N_{0}}}, 
\end{equation*}
note that $\nnrm{\tld{\varphi}_{0}} \aeq 1$. Moreover, in view of Proposition~\ref{prop:deg}.(2) and the choice of $m_{0}$ in Section~\ref{subsec:def-wavepacket}, we have, for any $s'', \sgm \in \bbR$,
\begin{equation} \label{eq:lin-illposed-tldvarphi0}
\lmb_{0}^{-s''} \gmm_{\lmb_{0}}(\lmb_{0})^{-\sgm} \nrm{\Gmm^{\sgm} \tld{\varphi}_{0}}_{H^{s''}} \aeq_{s'', \sgm} 1.
\end{equation}
We introduce $\phi_{0}$ and $\varphi_{0}$ defined by the relations
\begin{equation} \label{eq:lin-illposed-phi0}
	\phi_{0}(x) = \Gmm^{-\frac{1}{2}} \varphi_{0}(x), \quad \varphi_{0}(x)= \tld{\varphi}_{0}(x).
\end{equation}

We now proceed to prove \eqref{eq:norm-growth}. 
Let $\phi$ be an $\Gmm^{-\frac{1}{2}} L^{2}$-solution to $L_{\bgtht} \phi = 0$ on $\left[0, \frac{100}{99} t_{f}(\tau_{M}) \right]$ with $\left. \phi\right|_{t = 0} = \phi_{0}$. Then $\varphi = \Gmm^{\frac{1}{2}} \phi$ is a $L^{2}$-solution to $\calL_{\bgtht} \varphi = 0$ on the same interval with $\left. \varphi \right|_{t = 0} = \Gmm^{\frac{1}{2}} \phi_{0} = \tld{\varphi}_{0}(x)$. With the simplified notation  $t^{\ast} = \frac{1}{1-\eps} t_{f}(\tau_{M})$, note that $t^{\ast} \leq \frac{100}{99} t_{f}(\tau_{M})$ by \eqref{eq:eps-choice}. In view of \eqref{eq:lin-illposed-tldvarphi0} and \eqref{eq:lin-illposed-phi0}, it suffices to establish 
\begin{equation*}
	\begin{split}
		\sup_{t \in [0, t^{\ast}] } \nrm{\Gmm^{-\frac{1}{2}} \varphi(t,\cdot)}_{H^{s'}} \ge C_{s'} \frac{1}{\gmm_{\lmb_{0}}(M \lmb_{0})^{\frac{1}{2}}} M^{s'} \lmb_{0}^{s'} 
	\end{split}
\end{equation*} for any $s'>0$.

Recall  that $\tld{\varphi}  = \Re (e^{i\bfPhi}a)$   is well-defined on $\left[0, t^{\ast} \right]$ and belongs to $C^\infty_{c}(\Omg)$. Therefore, the generalized energy inequality \begin{equation*}
	\begin{split}
		\left| \frac{\ud}{\ud t} \brk{\varphi,\widetilde{\varphi}} - \brk{\calL_{\bgtht}\varphi,\widetilde{\varphi}} - \brk{\varphi, \calL_{\bgtht}\widetilde{\varphi}} \right| \lesssim \nrm{\varphi}_{L^2} \nrm{\widetilde{\varphi}}_{L^2} 
	\end{split}
\end{equation*} can be justified on the time interval $\left[0, t^{\ast}  \right]$. Therefore, we immediately obtain the inequality  \begin{equation*}
	\begin{split}
		\left| \frac{\ud}{\ud t} \brk{\varphi,\widetilde{\varphi}} \right| \lesssim ( \nrm{\widetilde{\varphi}}_{L^2} + \nrm{\calL_{\bgtht}\widetilde{\varphi}}_{L^2} ) \nrm{\varphi}_{L^2} .
	\end{split}
\end{equation*} Integrating in time and  applying the error estimate \eqref{eq:wp-error}   gives\begin{equation*}
	\begin{split}
		\left| \brk{\varphi,\widetilde{\varphi}}(t ) - \brk{\varphi,\widetilde{\varphi}}(0)\right|& \lesssim \int_{0}^{t } ( \nrm{\widetilde{\varphi}}_{L^2} + \nrm{\calL_{\bgtht}\widetilde{\varphi}}_{L^2} ) \nrm{\varphi}_{L^2} \, \ud t' \\
		& \lesssim (t \nrm{\widetilde{\varphi}_0}_{L^2}  + \lmb_{0}^{-\dlt_{6}} \nnrm{\widetilde{\varphi}_0}) \nrm{\varphi_0}_{L^2} 
	\end{split}
\end{equation*} for any $0< t \le t^{\ast}$. In the last inequality we have used $\nrm{\widetilde{\varphi}}_{L^2} \lesssim \nrm{\widetilde{\varphi}_0}_{L^2}$ and $\nrm{\varphi}_{L^\infty([0,t^{\ast} ];L^2)} \lesssim \nrm{\varphi_0}_{L^2}$. Note that in the above inequalities, the implicit constants depend only on $f,\gmm$. Therefore, taking $\Lmb_0$ larger and the constant {$T_0$} smaller if necessary (the latter in a way depending only on $f,\gmm$), we may guarantee that \begin{equation*}
	\begin{split}
		\left| \brk{\varphi,\widetilde{\varphi}}(t^{\ast}) - \brk{\varphi,\widetilde{\varphi}}(0)\right| \le \frac{1}{4} \nrm{\varphi_0}_{L^2} \nrm{\widetilde{\varphi}_0}_{L^2},
	\end{split}
\end{equation*}	which gives, recalling $\varphi|_{t=0} = \widetilde{\varphi}|_{t=0} = \tld{\varphi}_{0}$, \begin{equation*}
	\begin{split}
		\brk{\varphi,\widetilde{\varphi}}(t^{\ast}) > \frac{1}{4} \nrm{\varphi_0}_{L^2} \nrm{\widetilde{\varphi}_0}_{L^2}. 
	\end{split}
\end{equation*} Furthermore, from $\widetilde{\varphi} = \widetilde{\varphi}^{main} + \widetilde{\varphi}^{small}$ and the estimate for $\widetilde{\varphi}^{small}$ in  \eqref{eq:deg-small} and \eqref{eq:mu-small}, \begin{equation*}
	\begin{split}
		\brk{ \varphi, \widetilde{\varphi}^{main} }(t^{\ast}) &\ge \frac{1}{4}\nrm{\varphi_0}_{L^2} \nrm{\widetilde{\varphi}_0}_{L^2} -  \nrm{\varphi_0}_{L^2}\nrm{\widetilde{\varphi}^{small}(t^{\ast})}_{L^2}  \\
		&\ge \frac{1}{4}\nrm{\varphi_0}_{L^2} \nrm{\widetilde{\varphi}_0}_{L^2} - C\lmb_0^{-c_1}\nrm{\varphi_0}_{L^2}\nnrm{\widetilde{\varphi}_0} \ge \frac{1}{8}\nrm{\varphi_0}_{L^2} \nrm{\widetilde{\varphi}_0}_{L^2} ,
	\end{split}
\end{equation*} by again taking $\Lmb_0$ larger if necessary. Combining this inequality with $\nrm{\widetilde{\varphi}_0}_{L^2} = 1$, \begin{equation*}
	\begin{split}
		\brk{ \varphi, \widetilde{\varphi}^{main} }(t^{\ast}) \le \nrm{\Gmm^{-\frac{1}{2}} \varphi(t^{\ast})}_{H^{s'}}\nrm{\Gmm^{\frac{1}{2}} \widetilde{\varphi}^{main}(t^{\ast})}_{H^{-s'}},
	\end{split}
\end{equation*} and \eqref{eq:deg-main}, we conclude that \begin{equation*}
	\begin{split}
		\nrm{\phi(t^{\ast})}_{H^{s'}}= \nrm{\Gmm^{-\frac{1}{2}} \varphi(t^{\ast})}_{H^{s'}} \ge C_{s'} \frac{1}{\gmm_{\lmb_{0}}(\lmb(t^{\ast}))^{\frac{1}{2}}} \lmb(t^{\ast})^{s'} \nrm{\varphi_0}_{L^2} 
		\geq C_{s'} \frac{1}{\gmm_{\lmb_{0}}(M \lmb_{0})^{\frac{1}{2}}} M^{s'}\lmb_0^{s'} .
	\end{split}
\end{equation*} In the last inequality, we have used \eqref{eq:norm-growth-ini} and that $\lmb(t^{\ast}) = M\lmb_0$ (see Lemma~\ref{lem:Xi-control}). This completes the proof of Theorem \ref{thm:norm-growth}. \hfill \qedsymbol

\subsection{Linear illposedness, dissipative case}\label{subsec:lin-illposed-diss}
In this section, we proceed to prove Theorem \ref{thm:norm-growth-diss}. Let $f_{0}$ be a smooth function satisfying the assumptions in the statement of Theorem \ref{thm:norm-growth-diss}. Without loss of generality, we further assume that $f''_{0}(0) < 0$. Recall that $f(t,x_{2})$ for $t\ge0$ was defined to be the solution of $\rd_{t} f + \kappa \Upsilon f = 0$. From the assumption that the symbol of $\Upsilon$ is even, we have that $f(t,\cdot)$ is even as well. In particular, $f'(t,0) = 0$ for all $t\ge0$. We are also given parameters $\lmb_{0} \in \bbN$, $M > 1$, $0 < \dlt_{1} \leq \dlt_{0} < \frac{1}{100}$ and $0 \leq \sgm_{0} \leq \frac{1}{3} (1- 2 \dlt_{0})$, so that the conditions \eqref{eq:gf-condition-1}--\eqref{eq:gf-condition-3} as well as \eqref{eq:gf-condition-diss-1}--\eqref{eq:gf-condition-diss-2} hold. {We also set $c_{1} = \frac{1}{3} + \frac{1}{10} \dlt_{0}$ (cf.~\eqref{eq:mu-small}).} Let $\Lmb_{1} \geq 1$ and $0 < T_{1} \leq 1$ be parameters to be determined below, and we also require that $\lmb_{0} \geq \Lmb_{1}$ and $\tau_{M} \leq T_{1}$. Fix also the regularity exponents $s, s' \in \bbR$.

We apply the construction in Sections~\ref{sec:psdo}--\ref{sec:wavepacket} with $N_{0} = 10^{4} \max\set{1+\bt_{0}, 1+\alp_{0}, 1+\abs{s}, 1+\abs{s'}}$ and the above parameters; we take $\Lmb_{1} \leq \Lmb$ and $T_{1} \leq T$. As in the proof of Theorem~\ref{thm:norm-growth}, we obtain a $L^{2}$-normalized degenerating wave packet $\tld{\varphi} = \Re(a e^{i \bfPhi})$; we set $\phi_{0} = \Gmm^{-\frac{1}{2}} \tld{\varphi}_{0}$, where $\tld{\varphi}_{0} = \left. \tld{\varphi}\right|_{t=0}$.

Let $\phi$ be a $\Gmm^{-\frac{1}{2}} L^2$-solution to $L_{\bgtht}^{(\kappa)} \phi = 0$ on $\left[0, \frac{100}{99} t_{f}(\tau_{M}) \right]$ with $\phi|_{t=0} = \phi_{0}$, so that $\varphi := \Gmm^{\frac{1}{2}} \phi$ is a $L^{2}$-solution to $\calL_{\bgtht}^{(\kpp)} \varphi = 0$ with $\left. \varphi \right|_{t=0} = \varphi_{0}$. We also introduce $t^{\ast} = \frac{1}{1-\eps} t_{f}(\tau_{M})$, which obeys $t^{\ast} \leq \frac{100}{99} t_{f}(\tau_{M})$. As before, we apply the generalized energy inequality to $\varphi$ and $\widetilde{\varphi}$ and obtain \begin{equation*}
	\begin{split}
		\left| \frac{\ud}{\ud t} \brk{\varphi,\widetilde{\varphi}} - \brk{\calL^{(\kpp)}_{\bgtht}\varphi,\widetilde{\varphi}} - \brk{\varphi, \calL_{\bgtht}^{(\kpp)}\widetilde{\varphi}} \right| \lesssim \nrm{\varphi}_{L^2} \nrm{\widetilde{\varphi}}_{L^2} ,
	\end{split}
\end{equation*} which is valid on the time interval $\left[0, t^{\ast}  \right]$. This time, we obtain the inequality  \begin{equation}\label{eq:gei-linear-diss}
	\begin{split}
		\left| \frac{\ud}{\ud t} \brk{\varphi,\widetilde{\varphi}} \right| \lesssim ( \nrm{\widetilde{\varphi}}_{L^2} + \nrm{\calL_{\bgtht}\widetilde{\varphi}}_{L^2} + \kpp \nrm{ \Upsilon \widetilde{\varphi} }_{L^2} ) \nrm{\varphi}_{L^2} .
	\end{split}
\end{equation} 
Since $N_{0} \geq 10^{4} (1+\alp_{0})$, arguing as in the proof of Proposition~\ref{prop:deg}, we have
\begin{equation*}
	\kpp \nrm{\Upsilon \tld{\varphi}}_{L^{2}} \aleq \kpp \ups(\lmb_{0}, \lmb) \nnrm{\tld{\varphi}_{0}}, 
\end{equation*}
where $\nnrm{\tld{\varphi}_{0}}$ is defined as in the proof of Theorem~\ref{thm:norm-growth}.
Applying the previous bound to \eqref{eq:gei-linear-diss} and proceeding as in the inviscid case gives \begin{equation*}
	\begin{split}
		\left| \brk{\varphi,\widetilde{\varphi}}(t ) - \brk{\varphi,\widetilde{\varphi}}(0)\right| 
		& \lesssim \left(t \nrm{\widetilde{\varphi}_0}_{L^2}  + \left( \lmb_{0}^{-\dlt_{6}} +  \int_0^{t} |\upsilon(\lmb_{0},\lmb(t'))|   \,\ud t' \right) \nnrm{\widetilde{\varphi}_0}  \right) \nrm{\varphi_0}_{L^2} 
	\end{split}
\end{equation*} for any $0< t \le t^{\ast}$. Taking $t=t^{\ast}$ and using $\dot{\lmb} = (1-\eps) \abs{f''(0, 0)} \lmb_{0} \gmm(\lmb_{0}, \lmb)$ with a change of variables, we have that \begin{equation*}
	\begin{split}
		\int_0^{t^{\ast}} |\upsilon(\lmb_{0},\lmb(t))|   \,\ud t \aleq \int_{\lmb_0}^{M\lmb_{0}} \frac{\upsilon(\lmb_0,\lmb)}{\gmm(\lmb_0,\lmb)} \frac{\ud \lmb}{\lmb_{0}} .
	\end{split}
\end{equation*} Under the assumption \eqref{eq:gf-condition-diss-1}, taking $\Lmb_1$ larger if necessary, we can conclude that \begin{equation*}
	\begin{split}
		\brk{\varphi,\widetilde{\varphi}}(t^{\ast})  \ge \frac{1}{16} \nrm{\varphi_0}_{L^2} \nrm{\tld{\varphi}_{0}}_{L^{2}}. 
	\end{split}
\end{equation*} The rest of the proof is identical to the inviscid case. \hfill \qedsymbol

\subsection{Nonlinear illposedness}\label{subsec:nonlin-illposed}

In this section, we will establish Theorem \ref{thm:nonlin-norm-infl}. We only consider the inviscid case, as the dissipative case can be treated similarly. 

We assume the same hypothesis and conventions as in the proof of Theorem~\ref{thm:norm-growth} in the beginning of Section~\ref{subsec:lin-illposed}. Given $\eps>0$ and $\mathring{\tht} = f(x_{2})$ in the statement of the Theorem \ref{thm:nonlin-norm-infl}, we consider the sequence \begin{equation*}
	\begin{split}
	 {\vtht_0^{(\lmb_{0})}} := c \eps  {\gmm(\lmb_{0}, \lmb_{0})^{\frac{1}{2}} \lmb_{0}^{-s} \Gmm^{-\frac{1}{2}} \tld{\varphi}_{0}} \in C^{\infty}_{c}(\Omg), 
	\end{split}
\end{equation*} where $\tld{\varphi}_{0} = \Re( e^{i\bfPhi} a )|_{t=0}$ is the $L^{2}$-normalized degenerating wave packet at the initial time with frequency $\lmb_{0}$ from the proof of Theorem \ref{thm:norm-growth}. Here, $c >0$ is an absolute constant inserted to guarantee that \begin{equation*}
	\begin{split}
		\nrm{  {\vtht_0^{(\lmb_{0})}}  }_{H^s} \le \eps 
	\end{split}
\end{equation*} uniformly for $\lmb_0 \ge \Lmb_{0}$ (the same $\Lmb_{0}$ from Theorem \ref{thm:norm-growth}), as it is required in the statement of the Theorem \ref{thm:nonlin-norm-infl}. We then consider the sequence initial data \begin{equation*}
	\begin{split}
		\tht^{(\lmb_0)}|_{t=0} = \mathring{\tht} + {\vtht_0^{(\lmb_{0})}}
	\end{split}
\end{equation*} with $\lmb_0\ge \Lmb_0$ for \eqref{eq:ssqg}. Recalling the parameters $\dlt, A> 0$ given in the statement, we may assume towards a contradiction that for any sufficiently large $\lmb_0$, there exists a solution $\tht^{(\lmb_0)}$ to \eqref{eq:ssqg} with initial data  $\tht^{(\lmb_0)}|_{t=0}$ satisfying $\tht^{(\lmb_0)} - \mathring{\tht} \in L^\infty([0,\dlt];H^{s'})$ and  \begin{equation*}
	\begin{split}
		\sup_{t\in[0,\dlt]} \nrm{ (\tht^{(\lmb_0)} - \mathring{\tht})(t,\cdot) }_{ H^{s'}} \le A. 
	\end{split}
\end{equation*} From now on, we shall fix some large $\lmb_0 \ge \Lmb_{0}$ and omit writing out the dependence of $\lmb_0$ on the solution. On the time interval $[0,\dlt]$, we simply define \begin{equation*}
	\begin{split}
		\varphi^\star(t,\cdot) = \Gmm^{\frac{1}{2}}(\tht^{(\lmb_0)} - \mathring{\tht})(t,\cdot).
	\end{split}
\end{equation*} Since $\tht^{(\lmb_0)}$ and $ \mathring{\tht}$ are solutions to \eqref{eq:ssqg}, we see that the equation for $\varphi^\star$ is given by \begin{equation*}
	\begin{split}
		\calL_{\bgtht}\varphi^\star = -\Gmm^{\frac{1}{2}} \left( \nb^\perp \Gmm^{\frac{1}{2}} \varphi^\star \cdot \nb \Gmm^{-\frac{1}{2}} \varphi^\star \right) =: \calN[\varphi^\star]. 
	\end{split}
\end{equation*} We claim that 
\begin{align} 
\brk{\varphi^{\star}_{0}, \tld{\varphi}_{0}} &\aeq \nrm{\varphi^{\star}_{0}}_{L^{2}}, \label{eq:nonlinearity-initial} \\
\nrm{ \calN[\varphi^\star] }_{L^2} &\le C(1+A) \nrm{ \varphi^\star }_{L^2}. \label{eq:nonlinearity-L2}
\end{align}
Bound \eqref{eq:nonlinearity-initial} follows from Proposition~\ref{prop:deg}.(2).
To prove \eqref{eq:nonlinearity-L2}, we use $s' > \frac{3}{2}\beta_{0} + 3$. We take some $0<\varepsilon<s'-(3+\beta_0)$ and estimate \begin{equation*}
	\begin{split}
		\nrm{ \calN[\varphi^\star] }_{L^2} & \le C\nrm{  \nb^\perp \Gmm^{\frac{1}{2}} \varphi^\star \cdot \nb \Gmm^{-\frac{1}{2}} \varphi^\star }_{H^{\frac{\beta_{0}}{2}}} \\
		& \le C(\nrm{  \nb^\perp \Gmm^{\frac{1}{2}} \varphi^\star   }_{H^{\frac{\beta_{0}}{2}}} \nrm{  \nb \Gmm^{-\frac{1}{2}} \varphi^\star}_{L^\infty} + \nrm{  \nb^\perp \Gmm^{\frac{1}{2}} \varphi^\star   }_{L^{\infty}} \nrm{  \nb \Gmm^{-\frac{1}{2}} \varphi^\star}_{H^{\frac{\beta_{0}}{2}}}  )\\
		& \le C_{\varepsilon} \nrm{\varphi^\star}_{L^2}^{1- \frac{\beta_0+1}{s'} } \nrm{\varphi^\star}_{H^{s'}}^{ \frac{\beta_0+1}{s'} } \nrm{\varphi^\star}_{L^2}^{1- \frac{2+\varepsilon}{s'} } \nrm{\varphi^\star}_{H^{s'}}^{ \frac{2+\varepsilon}{s'} } \\
		&\qquad + C_{\varepsilon}\nrm{\varphi^\star}_{L^2}^{1- \frac{\frac{\beta_0}{2}+1}{s'} } \nrm{\varphi^\star}_{H^{s'}}^{ \frac{\frac{\beta_0}{2}+1}{s'} } \nrm{\varphi^\star}_{L^2}^{1- \frac{2+\varepsilon+\frac{\beta_0}{2}}{s'} } \nrm{\varphi^\star}_{H^{s'}}^{ \frac{2+\varepsilon+\frac{\beta_0}{2}}{s'} } \\
		&\le C_{\varepsilon}(1+A)\nrm{\varphi^\star}_{L^2}.
	\end{split}
\end{equation*} In this chain of estimates, we used $\nrm{\Gmm^{\frac{1}{2}} h}_{H^{s''}} \aleq_{s''} \nrm{h}_{H^{s''+\frac{\bt}{2}}}$ and $\nrm{\Gmm^{-\frac{1}{2}} h}_{H^{s''}} \aleq_{s''} \nrm{h}_{H^{s''}}$, which follow the assumptions on $\Gmm$ and Littlewood--Paley decomposition, as well as the Sobolev product estimate and interpolation estimate.  This verifies the claim. 

Using the energy structure of $\calL_{\bgtht}$, as well as the cancellation $\brk{\calN[\varphi^{\star}], \varphi^{\star}} = 0$, we obtain
\begin{equation*}
	\begin{split}
		\left| \frac{\ud}{\ud t} \brk{\varphi^\star,\varphi^\star} \right| \lesssim   {\nrm{\varphi^\star}_{L^2}^{2}}.
	\end{split}
\end{equation*} Hence, by Gronwall's inquality, \begin{equation}\label{eq:nonlin-energy}
	\begin{split}
		\nrm{\varphi^\star(t)}_{L^2} \le \nrm{\varphi^\star_0}_{L^2} \exp( {C t})
	\end{split}
\end{equation} for $0\le t \le \dlt$ with $C>0$ depending only on $f$ and $\Gmm$. Using the generalized energy inequality together with \eqref{eq:nonlinearity-L2} and \eqref{eq:nonlin-energy}, 
\begin{equation}\label{eq:gei-nonlin}
	\begin{split}
		\left| \frac{\ud}{\ud t} \brk{\varphi^\star,\widetilde{\varphi}} \right|& \lesssim ( \nrm{\widetilde{\varphi}}_{L^2} + \nrm{\calL_{\bgtht}\widetilde{\varphi}}_{L^2} )( \nrm{\varphi^\star}_{L^2} + \nrm{ \calN[\varphi^\star] }_{L^2} ) \\
		&\lesssim (1+A)\exp( {C t}) ( \nrm{\widetilde{\varphi}}_{L^2} + \nrm{\calL_{\bgtht}\widetilde{\varphi}}_{L^2} ) \nrm{\varphi^\star_0}_{L^2}. 
	\end{split}
\end{equation} Taking $t \le {\dlt}$ and integrating in time, we obtain similarly as in the proof of linear illposedness that 
\begin{equation}\label{eq:gei-nonlin-pre}
	\begin{split}
		\brk{\varphi^\star, \widetilde{\varphi}}(t) \ge \brk{\varphi^{\star}_{0}, \tld{\varphi}_{0}}  - C(1+A)  \left( t \nrm{\widetilde{\varphi}_0}_{L^2} + \lmb_0^{-\dlt_{6}} \nnrm{\widetilde{\varphi}_0} \right)\nrm{\varphi^\star_0}_{L^2}  .
	\end{split}
\end{equation}
At this point, from the definition of $H^{s}$-$H^{s'}$ instability (Definition \ref{def:HsHs'-illposed}), we have a sequence $(M_{(n)}, \lmb_{0 (n)})$ satisfying the growth conditions \eqref{eq:gf-condition-1}--\eqref{eq:gf-condition-3} (for some $\sgm_{0(n)}$), $\tau_{M_{(n)}}\to 0$, and $\frac{\gmm(\lmb_{0 (n)}, \lmb_{0 (n)})^{\frac{1}{2}}}{\gmm(\lmb_{0 (n)}, M_{(n)} \lmb_{0 (n)})^{\frac{1}{2}}}M_{(n)}^{s'} \lmb_{0 (n)}^{s'-s} \to \infty$. Therefore, by taking $n$ to be sufficiently large, we have \begin{equation*}
	\begin{split}
	t^{\ast}_{(n)} := 	\frac{1}{1-\eps} t_{f}(\tau_{M_{(n)}}) < \min\{ \dlt, \tfrac{1}{10 C(1+A)} \}, \quad C(1+A)\lmb_{0 (n)}^{-\dlt_{6}} \nnrm{\widetilde{\varphi}_0}   < \frac{1}{10}.
	\end{split}
\end{equation*} Here, $C$ is the constant from \eqref{eq:gei-nonlin-pre} which is independent of $\lmb_0$. From \eqref{eq:gei-nonlin-pre}, we now deduce that \begin{equation*}
	\begin{split}
		\brk{\varphi^\star, \widetilde{\varphi}}(t^{\ast}_{(n)}) \ge \frac12 \nrm{\varphi^{\star}_0}_{L^2}.
	\end{split}
\end{equation*}  By Proposition~\ref{prop:deg}, Cauchy--Schwartz and the $H^{s'}$-$H^{-s'}$ duality, we may decompose \begin{equation*}
	\begin{split}
		\brk{\varphi^\star, \widetilde{\varphi}}(t^{\ast}_{(n)})  \le \nrm{\Gmm^{-\frac12}  \varphi^\star (t^{\ast}_{(n)}) }_{H^s}   \nrm{\Gmm^{\frac12} \widetilde{\varphi}^{main}(t^{\ast}_{(n)})}_{H^{-s}  } + \nrm{ \varphi^\star (t^{\ast}_{(n)}) }_{L^{2}}   \nrm{\widetilde{\varphi}^{small}(t^{\ast}_{(n)})}_{L^{2}} .
	\end{split}
\end{equation*} where \begin{equation*}
\begin{split}
\nrm{ \varphi^\star (t^{\ast}_{(n)}) }_{L^{2}}   \nrm{\widetilde{\varphi}^{small}(t^{\ast}_{(n)})}_{L^{2}} <\frac{1}{10} \nrm{\varphi^\star_0}_{L^2}.
\end{split}
\end{equation*}  By \eqref{eq:deg-main} in Proposition~\ref{prop:deg}, we obtain \begin{equation}\label{eq:norm-growth-nonlin}
	\begin{split}
		\nrm{ \Gmm^{-\frac{1}{2}}\varphi^\star (t^{\ast}_{(n)}) }_{H^{s'}} & \gtrsim \frac{  \nrm{\varphi^\star_0}_{L^2} }{ \nrm{\Gmm^{\frac{1}{2}} \widetilde{\varphi}^{main}(t^{\ast}_{(n)})}_{H^{-{s'}} } }  
		\gtrsim \eps \frac{\gmm(\lmb_{0 (n)}, \lmb_{0 (n)})^{\frac{1}{2}}}{\gmm(\lmb_{0 (n)}, M_{(n)} \lmb_{0 (n)})^{\frac{1}{2}}} M_{(n)}^{s'} \lmb_{0 (n)}^{s'-s}.
	\end{split}
\end{equation} 
By $H^{s}$-$H^{s'}$ instability (see Definition \ref{def:HsHs'-illposed}), the RHS becomes arbitrary large as $n\to\infty$. On the other hand, using the definition of $\varphi^\star$, we obtain \begin{equation*}
	\begin{split}
		A \ge \nrm{\tht(t^{\ast}_{(n)})}_{H^{s'}} = \nrm{\Gmm^{-\frac{1}{2}} \varphi^\star(t^{\ast}_{(n)}) + \mathring{\tht} }_{H^{s'}}\ge \nrm{\Gmm^{-\frac{1}{2}} \varphi^\star(t^{\ast}_{(n)})}_{H^{s'}} -  \nrm{\mathring{\tht} }_{H^{s'}}
	\end{split}
\end{equation*} which is a contradiction with \eqref{eq:norm-growth-nonlin} as we take $n\to\infty$. The proof is now complete. 
\hfill \qedsymbol

\subsection{Nonexistence}\label{subsec:nonexist}

The goal of this section is to prove Theorem \ref{thm:nonexist}. We proceed in several steps. 

\medskip

\noindent \textit{1. Choice of the initial data}. To begin with, we fix some $f_{0}(x_{2})$ satisfying the following properties: \begin{itemize}
	\item $f_{0}$ is $C^\infty$ smooth and supported in $[-2,2]$;
	\item $f_{0}(x_{2})=-\frac12 x_{2}^2$ for $|x_{2}|\le 1$. 
\end{itemize} 
For $\lmb_0 \ge \Lmb_{0}$, let us denote by $\widetilde{\varphi}^{(\lmb_0)}(t)$ the degenerate wave packet solution adapted to $f_{0}$ with frequency $\lmb_0$, normalized in $L^2$. Furthermore, we assume that the wave packet at the initial time is supported in $(\frac12,1]$. For convenience, we recall a few essential properties of $\widetilde{\varphi}^{(\lmb_0)}(t)$. Define  $M(\lmb_0) :=  \min\set*{\gmm(\lmb_{0}, \lmb_{0})^{1-\dlt_{0}}, \lmb_{0}^{\dlt_{0}}, \, \gmm(\lmb_{0}, \lmb_{0})^{\frac{1-2\dlt_{0}}{\bt_{0}}}, \lmb_{0}^{\frac{1-5\dlt_{0}}{3\bt_{0}}}} $ and $\tau(\lmb_0) := \int_{\lmb_0}^{M(\lmb_0)\lmb_0} \frac{1}{\gmm(\lmb_0,\lmb)} \frac{\ud\lmb}{\lmb_0}$. \begin{itemize}
\item Degeneration estimate: there is a decomposition $\widetilde{\varphi}^{(\lmb_0)} = \widetilde{\varphi}^{(\lmb_0), main} + \widetilde{\varphi}^{(\lmb_0), small}$ such that for $t\in[0,\tau(\lmb_0)]$,\begin{equation*}
\begin{split}
\nrm{ \Gmm^{\frac{1}{2}} \widetilde{\varphi}^{(\lmb_0), main}(t) }_{H^{-s}} \le C_{-s} \gmm(\lmb_{0}, \lmb(t))^{\frac{1}{2}} \lmb(t)^{-s} , \qquad \nrm{\widetilde{\varphi}^{(\lmb_0), small} (t)}_{L^2} \le C\lmb(t)^{-c_{1}}.
\end{split}
\end{equation*}
\item Error estimate: $\nrm{ \calL_{f_{0}} \widetilde{\varphi}^{(\lmb_0)} }_{L^1([0,\frac{100}{99} t_{f_{0}}(\tau_{M})];L^2)} \le C\lmb_{0}^{-\dlt_{6}}$ and $\nrm{ \calL_{f_{0}} \widetilde{\varphi}^{(\lmb_0)} }_{L^1([0,\frac{100}{99} t_{f_{0}}(\tau_{M})];L^2[k, k+1])} \le C \min\set{1, k^{-2}} \lmb_{0}^{-\dlt_{6}}$ (see \eqref{eq:wp-error} and \eqref{eq:wp-error-loc}).
\end{itemize} By $\lmb(t)$, we mean the solution of \eqref{eq:ullmb-ODE} with $f = f_{0}$, which verifies $\lmb(\tau(\lmb_0)) \ge M(\lmb_0)\lmb_0$. In the above, it is important that the positive constants $C$, $c_1$ and $\dlt_{6}$ do not depend on $\lmb_0$.

We now take the following shear steady state: \begin{equation*}
\begin{split}
	\bgtht(x_{2}) = \sum_{k=k_{0}}^{\infty} f_{k} := \sum_{k=k_{0}}^{\infty} a_{k} f_{0}(x_{2} - y_{k} )
\end{split}
\end{equation*} where $\{ a_{k} \}_{k\ge 1}$ can be any square summable sequence; {\bf we fix it to be $a_{k}=k^{-2}$} for simplicity. On the other hand, $\{y_{k}\}_{k\ge 1}$ is a strictly increasing sequence to be determined; $y_{1}=1$ and we shall take $y_{k}$ sufficiently large with respect to $y_{k-1}$ for $k\ge 2$. By taking $k_{0}\ge 1$ large, it is guaranteed that $\nrm{\bgtht}_{H^{s}}<\frac{\eps}{2}$ for any given $\eps>0$ and $s$. 

Next, we consider the perturbation \begin{equation*}
	\begin{split}
		 {\vtht_0} =  {\Gmm^{-\frac{1}{2}} \varphi^{\star}_{0}} := \sum_{k=k_{0}}^{\infty} a_{k}\gmm(\lmb_{0, k},\lmb_{0, k})^{\frac12} \lmb_{0,k}^{-s} \Gmm^{-\frac12}  {\tld{\varphi}_{0,k}},
	\end{split}
\end{equation*} where 
\begin{equation*}
 {\tld{\varphi}_{0, k} := \widetilde{\varphi}^{(\lmb_{0, k})}(t=0, x_{2}-y_k),}
\end{equation*}
and $\lmb_{0,k} \ge \Lmb_{0}$ is a strictly increasing sequence to be determined below.  Observe that (using a simple rescaling in time) $\widetilde{\varphi}^{(\lmb_{0, k})}(a_{k}t, x_{2}-y_k)$ is simply the wave packet adapted to the rescaled and translated profile $f_{k}$ with frequency $\lmb_{0, k}$. By taking larger $k_{0}$ if necessary, we can guarantee that $\nrm{\vtht_0}_{H^{s}}<\frac{\eps}{2}$. 

We then take the initial data \begin{equation*}
	\begin{split}
		\tht_{0} = \bgtht_{0} + \vtht_0
	\end{split}
\end{equation*} for \eqref{eq:ssqg}, which satisfies $\nrm{\tht_0}_{H^{s}}<\eps$. Towards a contradiction, we assume that there exist $\dlt>0$ and a solution $\tht \in L^\infty([0,\dlt];H^{s})$ to \eqref{eq:ssqg} with $\tht(t=0)=\tht_0$. We denote \begin{equation*}
\begin{split}
A = \sup_{t \in [0,\dlt]} \nrm{\tht(t)}_{H^{s}} 
\end{split}
\end{equation*} and define on $t\in[0,\dlt]$ \begin{equation*}
\begin{split}
	\vtht(t) := \tht(t) - \bgtht, \qquad \varphi^\star(t) := \Gmm^{\frac12}\vtht(t).
\end{split}
\end{equation*}  

In what follows, we shall often suppress the dependence of implicit constants on $\Gmm$, $f_{0}$ and $s$. We remark that, logically, the sequence $\set{\lmb_{0, k}}$ shall be fixed first, and then shall $\set{y_{0, k}}$ be fixed -- the last part crucially uses the unboundedness of $\Omg = \bbT \times \bbR$. 

\medskip

\noindent \textit{2. Localization of the energy identity}. We first introduce some cutoff functions. Let $\chi\ge0$ be a smooth function supported on $[-1,1]$ and satisfies $\chi=1$ on $[-\frac12,\frac12]$. Assuming that $y_{k-1}$ is given for some $k\ge2$, we take $y_{k} \ge 8y_{k-1}$ and $\chi_{k}(x_2) := \chi( 2y_{k}^{-1} (x_{2}-y_{k}))$. It is then guaranteed that the support of $\chi_{k}$ is disjoint from each other. 

Recall that the equation for $\varphi^\star$ is given by \begin{equation*}
	\begin{split}
		\calL_{\bgtht}\varphi^\star + \Gmm^{\frac{1}{2}}( \nb^\perp \Gmm^{\frac{1}{2}} \varphi^\star \cdot \nb \Gmm^{-\frac{1}{2}} \varphi^\star) = 0.
	\end{split}
\end{equation*} 
As in Section~\ref{subsec:nonlin-illposed}, testing the equation against $\varphi^{\star}$, we immediately obtain 
\begin{equation*}
	\frac{\ud}{\ud t} \nrm{\varphi^{\star}}_{L^{2}}^{2} \aleq \nrm{\varphi^{\star}}_{L^{2}}^{2}
\end{equation*}
which implies, by Gronwall's inequality,
\begin{equation} \label{eq:nonexist-en}
	\nrm{\varphi^{\star}}_{L^{2}} \aleq \nrm{\varphi^{\star}_{0}}_{L^{2}} \exp(C t) .
\end{equation}
Moreover, multiplying the equation by $\chi_k$ and testing against $\chi_k\varphi^\star$, we have from \begin{equation*}
\begin{split}
&	\brk{	- \chi_k\Gmm^{\frac{1}{2}} \nb^{\perp} \bgtht \cdot \nb \Gmm^{\frac{1}{2}} \varphi^\star 
	+ \chi_k\Gmm^{\frac{1}{2}} \nb^{\perp} \Gmm \bgtht \cdot \nb \Gmm^{-\frac{1}{2}} \varphi^\star, \chi_k\varphi^\star} \\
 & \qquad =\brk{	- \Gmm^{\frac{1}{2}} \nb^{\perp} \bgtht \cdot \nb \Gmm^{\frac{1}{2}}( \chi_k\varphi^\star )
	+  \Gmm^{\frac{1}{2}} \nb^{\perp} \Gmm \bgtht \cdot \nb \Gmm^{-\frac{1}{2}} (\chi_k\varphi^\star), \chi_k\varphi^\star} \\
&\qquad\quad  + \brk{	- \Gmm^{\frac{1}{2}} \nb^{\perp} \bgtht \cdot [\chi_k,\nb \Gmm^{\frac{1}{2}}] \varphi^\star 
	+ \Gmm^{\frac{1}{2}} \nb^{\perp} \Gmm \bgtht \cdot[ \chi_k, \nb \Gmm^{-\frac{1}{2}}] \varphi^\star, \chi_k\varphi^\star} \\
	& \qquad\quad + \brk{	- [\chi_k, \Gmm^{\frac{1}{2}}] \nb^{\perp} \bgtht \cdot \nb \Gmm^{\frac{1}{2}} \varphi^\star 
	+ [\chi_k, \Gmm^{\frac{1}{2}}] \nb^{\perp} \Gmm \bgtht \cdot \nb \Gmm^{-\frac{1}{2}} \varphi^\star, \chi_k\varphi^\star},
\end{split}
\end{equation*} and  \begin{equation*}
\begin{split}
	\brk{ \chi_{k} \Gmm^{\frac{1}{2}}( \nb^\perp \Gmm^{\frac{1}{2}} \varphi^\star \cdot \nb \Gmm^{-\frac{1}{2}} \varphi^\star),  \chi_k\varphi^\star  } & = \brk{ [\chi_{k} ,\Gmm^{\frac{1}{2}}]( \nb^\perp \Gmm^{\frac{1}{2}} \varphi^\star \cdot \nb \Gmm^{-\frac{1}{2}} \varphi^\star),  \chi_k\varphi^\star  } \\
	&\qquad + \brk{ [\chi_{k},\nb^\perp \Gmm^{\frac{1}{2}}] \varphi^\star \cdot \nb \Gmm^{-\frac{1}{2}} \varphi^\star,  \Gmm^{\frac{1}{2}}(\chi_k\varphi^\star)  } 
\end{split}
\end{equation*} that \begin{equation*}
\begin{split}
	\frac{\ud}{\ud t} \nrm{ \chi_k\varphi^\star}_{L^2}^2 \lesssim_{ \bgtht,\Gmm}   \nrm{ \chi_k\varphi^\star}_{L^2}^2  +  {A} \nrm{\chi_{k}'}_{L^\infty}  \nrm{ \chi_k\varphi^\star}_{L^2}. 
\end{split}
\end{equation*} Here, it is important that the implicit constant depends only on $\bgtht$ and $\Gmm$ and is independent of $k$.  {We remark that we have used the commutator bounds $\nrm{[g, \Gmm^{\frac{1}{2}}] h}_{H^{s'}} \aleq_{s', \Gmm} \nrm{g'}_{L^{\infty}} \nrm{h}_{H^{s'+\bt_{0}-1}}$ and $\nrm{[g, \Gmm^{-\frac{1}{2}}] h}_{H^{s'}} \aleq_{s', \Gmm} \nrm{g'}_{L^{\infty}} \nrm{h}_{H^{s'}}$, which may be established using the assumptions on $\Gmm$, Littlewood--Paley decomposition and writing out the commutator using the integral kernel of $\Gmm^{\pm \frac{1}{2}} P_{\nu}$; we omit the standard details.} Regarding $\bgtht$, this estimate requires $\nrm{\Gmm\nb\bgtht}_{L^\infty}<\infty$, and here it suffices to have that $s>\beta_{0}+2$. 

Requiring that $y_{k}$ satisfies
\begin{equation*}
y_{k}^{-1} \aleq a_{k} \gmm(\lmb_{0, k},\lmb_{0,k})^{\frac{1}{2}} \lmb_{0,k}^{-s}
\end{equation*}
for suitably chosen implicit constant independent of $k$, we may ensure that
\begin{equation*}
\nrm{\chi_{k}'}_{L^\infty} \le  \nrm{ \chi_k\varphi^\star_{0}}_{L^2}.
\end{equation*}
 {By Gronwall's inequality,} we conclude the following \textit{localized energy estimate}: \begin{equation}\label{eq:lee}
\begin{split}
	\sup_{t \in [0,\dlt]}\nrm{ \chi_k\varphi^\star(t)}_{L^2} \lesssim \nrm{ \chi_k\varphi^\star_{0}}_{L^2}  {\exp(C (1+A)t)}. 
\end{split}
\end{equation}

\medskip

\noindent \textit{3. Localization of the generalized energy identity}. We denote \begin{equation*}
	\begin{split}
		\widetilde{\varphi}_{k}(t,x_2):= \widetilde{\varphi}^{(\lmb_{0,k})}(a_{k}t,x_2-y_k).
	\end{split}
\end{equation*} From the properties of $\widetilde{\varphi}^{(\lmb_{0})}$ summarized in the above, we have that with \begin{equation*}
\begin{split}
M_{k} := M(\lmb_{0,k}), \qquad \tau_{k} := \tau(\lmb_{0,k})
\end{split}
\end{equation*} \begin{itemize}
\item for $t\in[0,a_{k}^{-1}\tau_{k}]$, $\widetilde{\varphi}_{k}(t)$ is supported on $[y_{k},y_{k}+1]$ and satisfies $\nrm{\widetilde{\varphi}(t)}_{L^\infty_{t}([0,a_{k}^{-1}\tau_{k}];L^2  )} \le C$;
\item there is a decomposition $\widetilde{\varphi}_{k} = \widetilde{\varphi}^{main}_{k} + \widetilde{\varphi}^{small}_{k}$ such that for $t\in[0,a_{k}^{-1}\tau_{k}]$, \begin{equation*}
\begin{split}
\nrm{ \widetilde{\varphi}^{main}_{k}(t) }_{H^{-s}} \le C_{-s}\lmb_{k}^{-s}(a_{k}t) , \qquad \nrm{\widetilde{\varphi}^{small}_{k} (t)}_{L^2} \le C\lmb_{k}^{-c_0}(a_{k}t);
\end{split}
\end{equation*}
\item we have $\nrm{ \calL_{\bgtht} \widetilde{\varphi}_{k} }_{L^1([0,a_{k}^{-1}\tau_{k}];L^2)} \le C\lmb_{0,k}^{-c_0}$  {and $\nrm{ \calL_{\bgtht} \widetilde{\varphi}_{k} }_{L^1([0,a_{k}^{-1}\tau_{k}];L^2[\ell, \ell+1])} \le C \min\set{1, \ell^{-2}} \lmb_{0,k}^{-c_0}$.}
\end{itemize} Here, $\lmb_{k}(t)$ is the solution to \eqref{eq:ullmb-ODE} with $\lmb_{0}=\lmb_{0,k}$ and $f = \bgtht$. The constants $C, C_{-s}$ do not depend on $k$; in particular, the $L^{1}_{t}L^{2}$ estimates follow directly from scaling property of $\calL_{\bgtht}$ \eqref{eq:lin-L2}. 
In particular, on $t\in [0, \min\{ \dlt, a_{k}^{-1}\tau_{k} \} ]$, we have $\chi_{k}^2	\widetilde{\varphi}_{k}(t)=\chi_{k}\widetilde{\varphi}_{k}(t)=	\widetilde{\varphi}_{k}(t)$. Before we proceed, note \begin{equation*}
\begin{split}
\tau_{k}  {\le \frac{M_{k}}{\gmm(\lmb_{0, k}, \lmb_{0, k})} \leq} \frac{1}{\gmm^{\dlt_0}(\lmb_{0,k},\lmb_{0,k})} 
\end{split}
\end{equation*} and therefore by taking $\lmb_{0,k}$ sufficiently large, we may ensure that $a_{k}^{-1}\tau_{k} \le \dlt$. Choosing $\lmb_{0, k}$ larger if necessary, we may also ensure that $a_{k} M_{k}^{s - \bt_{0}} \to \infty$ as $k \to \infty$.

Denoting $\err_{k} := \calL_{\bgtht}\widetilde{\varphi}_{k}$, we deduce from applying \eqref{eq:gei} with $\varphi=\varphi^\star$ and $\psi=\widetilde{\varphi}_{k}$ that  \begin{equation}\label{eq:gei2}
\begin{split}
	\frac{\ud}{\ud t} \brk{\chi_{k}\varphi^\star,\widetilde{\varphi}_{k}} 
	&= \brk{\varphi^\star,\err_{k}} +  \brk{ \varphi^\star, \left([\Gmm^{\frac{1}{2}}, \nb^{\perp} \Gmm \bgtht \cdot \nb] \Gmm^{-\frac{1}{2}} + \Gmm^{-\frac{1}{2}} [\Gmm^{\frac{1}{2}}, \nb^{\perp} \Gmm \bgtht \cdot \nb]\right)(\chi_{k}\widetilde{\varphi}_k)} \\
	&\qquad -\brk{ \Gmm^{\frac{1}{2}}( \nb^\perp \Gmm^{\frac{1}{2}} \varphi^\star \cdot \nb \Gmm^{-\frac{1}{2}} \varphi^\star), \chi_{k}^2 \widetilde{\varphi}_k}.
\end{split}
\end{equation} We write \begin{equation*}
\begin{split}
	\brk{\varphi^\star,\err_{k}}=\brk{\chi_{k}\varphi^\star,\err_{k}}+\brk{\varphi^\star,(1-\chi_{k})\err_{k}}
\end{split}
\end{equation*} and  {apply Cauchy--Schwartz.} Furthermore, modulo several commutators  {involving $\chi_{k}$}, the other two terms on the right hand side of \eqref{eq:gei2} can be written as \begin{equation*}
\begin{split}
	\brk{\chi_{k} \varphi^\star, \left([\Gmm^{\frac{1}{2}}, \nb^{\perp} \Gmm \bgtht \cdot \nb] \Gmm^{-\frac{1}{2}} + \Gmm^{-\frac{1}{2}} [\Gmm^{\frac{1}{2}}, \nb^{\perp} \Gmm \bgtht \cdot \nb]\right)\widetilde{\varphi}_k} -\brk{ \Gmm^{\frac{1}{2}}( \nb^\perp \Gmm^{\frac{1}{2}} (\chi_{k}\varphi^\star) \cdot \nb \Gmm^{-\frac{1}{2}} (\chi_{k}\varphi^\star)),   \widetilde{\varphi}_k}.
\end{split}
\end{equation*} Now, it is important that we have \begin{equation*}
\begin{split}
	\nrm{ \Gmm^{\frac{1}{2}}( \nb^\perp \Gmm^{\frac{1}{2}} (\chi_{k}\varphi^\star) \cdot \nb \Gmm^{-\frac{1}{2}} (\chi_{k}\varphi^\star)) }_{L^2} \lesssim (1+A)\nrm{\chi_{k}\varphi^\star}_{L^2}. 
\end{split}
\end{equation*}
Observe that the terms involving commutators can be bounded in absolute value by \begin{equation*}
\begin{split}
	\lesssim (1+A^{2})\nrm{\chi_k'}_{L^\infty} \nrm{ \widetilde{\varphi}_k }_{L^2}\lesssim(1+ A^{2})\nrm{\chi_{k}\varphi^\star_0}_{L^2} \nrm{ \widetilde{\varphi}_k }_{L^2} .
\end{split}
\end{equation*}
Therefore, we have arrived at the following localized and generalized energy estimate: \begin{equation}\label{eq:lgee}
\begin{split}
	\left|  \frac{\ud}{\ud t} \brk{\chi_{k}\varphi^\star,\widetilde{\varphi}_{k}}  \right| &\lesssim (1+A^2) ( \nrm{ \widetilde{\varphi}_k}_{L^2} + \nrm{\err_k}_{L^2} )(\nrm{\chi_{k}\varphi^\star}_{L^2}+\nrm{\chi_{k}\varphi^\star_0}_{L^2}) {+ \nrm{\varphi^{\star}}_{L^{2}} \nrm{(1-\chi_{k}) \err_{k}}_{L^{2}}.} 
\end{split}
\end{equation}

\medskip

\noindent \textit{4. Conclusion}. Applying \eqref{eq:nonexist-en}, \eqref{eq:gei2} to \eqref{eq:lgee} and integrating in time, we obtain that \begin{equation*}
	\begin{split}
		\brk{\chi_{k}\varphi^\star,\widetilde{\varphi}_{k}} (t) &\ge \brk{\chi_{k}\varphi^\star,\widetilde{\varphi}_{k}} (0) \\
		&\peq - C (1+A^2)  {\exp(C (1+A) t)} \nrm{\chi_{k}\varphi^\star_0}_{L^2} ( t\nrm{ \widetilde{\varphi}_{0,k}}_{L^2} + \int_0^t \nrm{\err_k(\tau)}_{L^2}\,\ud\tau  ) \\
		&\peq  {- C \exp(C t) \nrm{\varphi^{\star}_{0}}_{L^{2}} \int_0^t \nrm{(1-\chi_{k}) \err_k(\tau)}_{L^2}\,\ud\tau.}
	\end{split}
\end{equation*} 
 {By Proposition~\ref{prop:deg}.(2) and our construction, observe that 
 \begin{equation*}
\brk{\chi_{k}\varphi^\star,\widetilde{\varphi}_{k}} (0) \aeq \nrm{\chi_{k}\varphi^\star_{0}}_{L^{2}} \aleq a_{k} \gmm(\lmb_{0, k}, \lmb_{0, k})^{\frac{1}{2}} \lmb_{0, k}^{-s}. 
\end{equation*}
We fix $t = a_{k}^{-1} \tau_{k}$. Using the localized error bound, if we take $y_{k}$ large enough depending on $\lmb_{0, k}$, we may ensure that
\begin{equation*}
\nrm{\varphi^{\star}_{0}}_{L^{2}} \int_0^t \nrm{(1-\chi_{k}) \err_k(\tau)}_{L^2}\,\ud\tau
\leq \lmb_{0, k}^{-c_{1}} \brk{\chi_{k}\varphi^\star,\widetilde{\varphi}_{k}} (0).
\end{equation*}
For $k$ large enough so that $\tau_{k} \ll (1+A^{2})^{-1}$ and $\lmb_{0, k}^{-c_{1}} \ll (1+A^{2})^{-1}$, we obtain
}
\begin{equation*}
\begin{split}
	\brk{\chi_{k}\varphi^\star,\widetilde{\varphi}_{k}} (a_{k}^{-1}\tau_{k}) \ge \frac12\brk{\chi_{k}\varphi^\star,\widetilde{\varphi}_{k}} (0).
\end{split}
\end{equation*} On the other hand, we may bound \begin{equation*}
\begin{split}
	\brk{\chi_{k}\varphi^\star,\widetilde{\varphi}_{k}} (a_{k}^{-1}\tau_{k}) \le \nrm{\Gmm^{-\frac12} \chi_k\varphi^\star (a_{k}^{-1}\tau_{k}) }_{H^{s}}\nrm{\Gmm^{\frac12} \widetilde{\varphi}^{main}_k(a_{k}^{-1}\tau_{k})}_{H^{-{s}}} +  \nrm{ \chi_k\varphi^\star (a_{k}^{-1}\tau_{k}) }_{L^{2}}\nrm{\widetilde{\varphi}^{small}_k(a_{k}^{-1}\tau_{k})}_{L^{2}} .
\end{split}
\end{equation*} By taking $k$ larger if necessary, we may guarantee that \begin{equation*}
\begin{split}
 \nrm{ \chi_k\varphi^\star (a_{k}^{-1}\tau_{k}) }_{L^{2}}\nrm{\widetilde{\varphi}^{small}_k(a_{k}^{-1}\tau_{k})}_{L^{2}} < \frac14\brk{\chi_{k}\varphi^\star,\widetilde{\varphi}_{k}} (0).
\end{split}
\end{equation*} This gives 
\begin{equation*}
	\begin{split}
		\nrm{\Gmm^{-\frac12}  \chi_k\varphi^\star (a_{k}^{-1}\tau_{k}) }_{H^{s}} & \gtrsim \frac{  \nrm{\chi_{k} \varphi^\star_0}_{L^2} \nrm{\widetilde{\varphi}_{k,0}}_{L^2}  }{ \nrm{\Gmm^{\frac12} \widetilde{\varphi}_{k}^{main}(a_{k}^{-1}\tau_{k})}_{H^{-{s}}} } \gtrsim  {a_{k}} \frac{\gmm(\lmb_{0,k})^{\frac{1}{2}}}{\lmb_{0,k}^s} \frac{ \lmb^{s}_{k}(\tau_{k})}{\gmm(\lmb_{k}(\tau_{k}))^{\frac{1}{2}}} \\
		&\gtrsim  {a_{k}} \frac{\gmm(\lmb_{0,k})^{\frac{1}{2}}}{\gmm(\lmb_{k}(\tau_{k}))^{\frac{1}{2}}} M_{k}^{s} \gtrsim  {a_{k}} M_{k}^{s-\beta_{0}}. 
	\end{split}
\end{equation*} For $k$ large, $\nrm{ [\Gmm^{-\frac12},\chi_{k}] \varphi^\star }_{H^{s}} \lesssim (1+A)\nrm{\chi_{k}'}_{L^\infty} \lesssim 1$ and therefore we obtain \begin{equation*}
\begin{split}
\nrm{ \chi_k (\tht(a_{k}^{-1}\tau_{k}) - \bgtht) }_{H^{s}} = \nrm{ \chi_k \Gmm^{-\frac12}\varphi^\star (a_{k}^{-1}\tau_{k}) }_{H^{s}} \gtrsim  {a_{k}} M_{k}^{s-\beta_{0}}. 
\end{split}
\end{equation*} On the other hand, with a constant independent of $k$, \begin{equation*}
\begin{split}
\nrm{ \chi_k (\tht(a_{k}^{-1}\tau_{k}) - \bgtht) }_{H^{s}} \lesssim \nrm{ (\tht(a_{k}^{-1}\tau_{k}) - \bgtht) }_{H^{s}} \lesssim 1 +A. 
\end{split}
\end{equation*} We obtain a contradiction as $k\to\infty$ since $ {a_{k} M_{k}^{s - \bt_{0}}\to\infty}$. This finishes the proof. \hfill \qedsymbol

\bibliographystyle{amsplain}
\providecommand{\bysame}{\leavevmode\hbox to3em{\hrulefill}\thinspace}
\providecommand{\MR}{\relax\ifhmode\unskip\space\fi MR }
\providecommand{\MRhref}[2]{%
	\href{http://www.ams.org/mathscinet-getitem?mr=#1}{#2}
}
\providecommand{\href}[2]{#2}


\end{document}